\documentclass[reqno]{amsart}

\address{Department of Mathematics, Kyoto
University, Kyoto, Japan } \email{fukaya@math.kyoto-u.ac.jp}
\address{Department of Mathematics, University of
Wisconsin, Madison, WI, USA \& \newline \indent Korea Institute for
Advanced Study, Seoul, Korea} \email{oh@math.wisc.edu}
\address{Graduate School of Mathematics,
Nagoya University, Nagoya, Japan } \email{ohta@math.nagoya-u.ac.jp}
\address{Department of Mathematics,
Hokkaido University, Sapporo, Japan }
\email{ono@math.sci.hokudai.ac.jp}

\usepackage[all]{xy}

\begin{document}
\quad \vskip1.375truein

\def\E{\ifmmode{\mathbb E}\else{$\mathbb E$}\fi} 
\def\N{\ifmmode{\mathbb N}\else{$\mathbb N$}\fi} 
\def\R{\ifmmode{\mathbb R}\else{$\mathbb R$}\fi} 
\def\Q{\ifmmode{\mathbb Q}\else{$\mathbb Q$}\fi} 
\def\C{\ifmmode{\mathbb C}\else{$\mathbb C$}\fi} 
\def\H{\ifmmode{\mathbb H}\else{$\mathbb H$}\fi} 
\def\Z{\ifmmode{\mathbb Z}\else{$\mathbb Z$}\fi} 
\def\P{\ifmmode{\mathbb P}\else{$\mathbb P$}\fi} 
\def\T{\ifmmode{\mathbb T}\else{$\mathbb T$}\fi} 
\def\SS{\ifmmode{\mathbb S}\else{$\mathbb S$}\fi} 
\def\DD{\ifmmode{\mathbb D}\else{$\mathbb D$}\fi} 
\def\K{\ifmmode{\mathbb K}\else{$\mathbb K$}\fi}

\renewcommand{\theequation}{\thesection.\arabic{equation}}

\renewcommand{\a}{\alpha}
\renewcommand{\b}{\beta}
\renewcommand{\d}{\delta}
\newcommand{\D}{\Delta}
\newcommand{\e}{\varepsilon}
\newcommand{\g}{\gamma}
\newcommand{\G}{\Gamma}
\newcommand{\la}{\lambda}
\newcommand{\La}{\Lambda}
\newcommand{\var}{\varphi}
\newcommand{\s}{\sigma}
\newcommand{\Sig}{\Sigma}
\renewcommand{\t}{\tau}
\renewcommand{\th}{\theta}
\renewcommand{\O}{\Omega}
\renewcommand{\o}{\omega}
\newcommand{\z}{\zeta}
\newcommand{\del}{\partial}

\newcommand{\ben}{\begin{enumerate}}
\newcommand{\een}{\end{enumerate}}
\newcommand{\be}{\begin{equation}}
\newcommand{\ee}{\end{equation}}
\newcommand{\bea}{\begin{eqnarray}}
\newcommand{\eea}{\end{eqnarray}}
\newcommand{\beastar}{\begin{eqnarray*}}
\newcommand{\eeastar}{\end{eqnarray*}}
\newcommand{\bc}{\begin{center}}
\newcommand{\ec}{\end{center}}

\newcommand{\IR}{\mbox{I \hspace{-0.2cm}R}}
\newcommand{\IN}{\mbox{I \hspace{-0.2cm}N}}

\theoremstyle{theorem}
\newtheorem{thm}{Theorem}[section]
\newtheorem{cor}[thm]{Corollary}
\newtheorem{lem}[thm]{Lemma}
\newtheorem{sublem}[thm]{Sublemma}
\newtheorem{prop}[thm]{Proposition}
\newtheorem{ax}[thm]{Axiom}
\newtheorem{conj}[thm]{Conjecture}

\theoremstyle{definition}
\newtheorem{defn}[thm]{Definition}
\newtheorem{rem}[thm]{Remark}
\newtheorem{ques}[thm]{Question}
\newtheorem{cons}[thm]{\rm\bfseries{Construction}}
\newtheorem{exm}[thm]{Example}
\newtheorem{conds}[thm]{Condition}

\newtheorem{notation}[thm]{\rm\bfseries{Notation}}

\newtheorem*{thm*}{Theorem}

\numberwithin{equation}{section}

\hsize=5.0truein \hoffset=.25truein \vsize=8.375truein
\voffset=.15truein
\def\R{{\mathbb R}}
\def\osc{{\hbox{\rm osc}}}
\def\Crit{{\hbox{Crit}}}
\def\tr{{\hbox{\bf tr}}}
\def\Tr{{\hbox{\bf Tr}}}
\def\DTr{{\hbox{\bf DTr}}}

\def\E{{\mathbb E}}
\def\Z{{\mathbb Z}}
\def\C{{\mathbb C}}
\def\R{{\mathbb R}}
\def\P{{\mathbb P}}
\def\ralim{{\mathrel{\mathop\rightarrow}}}
\def\N{{\mathbb N}}
\def\oh{{\mathcal O}}
\def\11{{\mathbb I}}

\def\xbar{{\widetilde x}}
\def\ybar{{\widetilde y}}
\def\ubar{{\widetilde u}}
\def\Jbar{{\widetilde J}}
\def\pxo{{ (X_0) }}
\def\n{{\noindent}}
\def\olim{\mathop{\overline{\rm lim}}}
\def\ulim{\mathop{\underline{\rm lim}}}
\def\hat{{\widehat}}
\def\delbar{{\overline \partial}}
\def\dudtau{{\frac{\del u}{\del \tau}}}
\def\dudt{{\frac{\del u}{\del t}}}

\def\V{\mathbb{V}}
\def\C{\mathbb{C}}
\def\Z{\mathbb{Z}}
\def\A{\mathbb{A}}
\def\T{\mathbb{T}}
\def\L{\mathbb{L}}
\def\D{\mathbb{D}}
\def\Q{\mathbb{Q}}

\def\E{\ifmmode{\mathbb E}\else{$\mathbb E$}\fi} 
\def\N{\ifmmode{\mathbb N}\else{$\mathbb N$}\fi} 
\def\R{\ifmmode{\mathbb R}\else{$\mathbb R$}\fi} 
\def\Q{\ifmmode{\mathbb Q}\else{$\mathbb Q$}\fi} 
\def\C{\ifmmode{\mathbb C}\else{$\mathbb C$}\fi} 
\def\H{\ifmmode{\mathbb H}\else{$\mathbb H$}\fi} 
\def\Z{\ifmmode{\mathbb Z}\else{$\mathbb Z$}\fi} 
\def\P{\ifmmode{\mathbb P}\else{$\mathbb P$}\fi} 
\def\SS{\ifmmode{\mathbb S}\else{$\mathbb S$}\fi} 
\def\DD{\ifmmode{\mathbb D}\else{$\mathbb D$}\fi} 

\def\R{{\mathbb R}}
\def\osc{{\hbox{\rm osc}}}
\def\Crit{{\hbox{Crit}}}
\def\E{{\mathbb E}}
\def\Z{{\mathbb Z}}
\def\C{{\mathbb C}}
\def\R{{\mathbb R}}
\def\M{{\mathbb M}}
\def\ralim{{\mathrel{\mathop\rightarrow}}}
\def\N{{\mathbb N}}
\def\TT{{\mathcal T}}
\def\UU{{\mathcal U}}
\def\VV{{\mathcal V}}

\def\xbar{{\widetilde x}}
\def\ybar{{\widetilde y}}
\def\ubar{{\widetilde u}}
\def\Jbar{{J}}
\def\pxo{{ (X_0) }}

\def\olim{\mathop{\overline{\rm lim}}}
\def\ulim{\mathop{\underline{\rm lim}}}
\def\hat{{\widehat}}
\def\delbar{{\overline \partial}}

\def\a{\alpha}
\def\b{\beta}
\def\c{\chi}
\def\d{\delta} \def\D{\Delta}
\def\e{\varepsilon} \def\ep{\epsilon}
\def\f{\phi} \def\F{\Phi}
\def\g{\gamma} \def\G{\Gamma}
\def\k{\kappa}
\def\l{\lambda} \def\La{\Lambda}
\def\m{\mu}
\def\n{\nu}
\def\r{\rho}
\def\vr{\varrho}
\def\o{\omega} \def\O{\Omega}
\def\p{\psi}
\def\s{\sigma} \def\S{\Sigma}
\def\th{\theta} \def\vt{\vartheta}
\def\t{\tau}
\def\w{\varphi}
\def\x{\xi}
\def\z{\zeta}
\def\U{\Upsilon}
\def\CA{{\mathcal A}}
\def\CB{{\mathcal B}}
\def\CC{{\mathcal C}}
\def\CD{{\mathcal D}}
\def\CE{{\mathcal E}}
\def\CF{{\mathcal F}}
\def\CG{{\mathcal G}}
\def\CH{{\mathcal H}}
\def\CI{{\mathcal I}}
\def\CJ{{\mathcal J}}
\def\CK{{\mathcal K}}
\def\CL{{\mathcal L}}
\def\CM{{\mathcal M}}
\def\CN{{\mathcal N}}
\def\CO{{\mathcal O}}
\def\CP{{\mathcal P}}
\def\CQ{{\mathcal Q}}
\def\CR{{\mathcal R}}
\def\CP{{\mathcal P}}
\def\CS{{\mathcal S}}
\def\CT{{\mathcal T}}
\def\CU{{\mathcal U}}
\def\CV{{\mathcal V}}
\def\CW{{\mathcal W}}
\def\CX{{\mathcal X}}
\def\CY{{\mathcal Y}}
\def\CZ{{\mathcal Z}}
\def\EJ{\mathfrak{J}}
\def\EM{\mathfrak{M}}
\def\ES{\mathfrak{S}}
\def\BB{\mathib{B}}
\def\BC{\mathib{C}}
\def\BH{\mathib{H}}

\def\rd{\partial}
\def\grad#1{\,\nabla\!_{{#1}}\,}
\def\gradd#1#2{\,\nabla\!_{{#1}}\nabla\!_{{#2}}\,}
\def\om#1#2{\omega^{#1}{}_{#2}}
\def\vev#1{\langle #1 \rangle}
\def\darr#1{\raise1.5ex\hbox{$\leftrightarrow$}
\mkern-16.5mu #1}
\def\Ha{{1\over2}}
\def\ha{{\textstyle{1\over2}}}
\def\fr#1#2{{\textstyle{#1\over#2}}}
\def\Fr#1#2{{#1\over#2}}
\def\rf#1{\fr{\rd}{\rd #1}}
\def\rF#1{\Fr{\rd}{\rd #1}}
\def\df#1{\fr{\d}{\d #1}}
\def\dF#1{\Fr{\d}{\d #1}}
\def\DDF#1#2#3{\Fr{\d^2 #1}{\d #2\d #3}}
\def\DDDF#1#2#3#4{\Fr{\d^3 #1}{\d #2\d #3\d #4}}
\def\ddF#1#2#3{\Fr{\d^n#1}{\d#2\cdots\d#3}}
\def\fs#1{#1\!\!\!/\,} 
\def\Fs#1{#1\!\!\!\!/\,} 
\def\roughly#1{\raise.3ex\hbox{$#1$\kern-.75em
\lower1ex\hbox{$\sim$}}}
\def\ato#1{{\buildrel #1\over\longrightarrow}}
\def\up#1#2{{\buildrel #1\over #2}}
\def\opname#1{\mathop{\kern0pt{\rm #1}}\nolimits}
\def\Re{\opname{Re}}
\def\Im{\opname{Im}}
\def\End{\opname{End}}
\def\dim{\opname{dim}}
\def\vol{\opname{vol}}
\def\Int{\opname{Int}}

\def\group#1{\opname{#1}}
\def\SU{\group{SU}}
\def\U{\group{U}}
\def\SO{\group{SO}}
\def\pr{\prime}
\def\CPr{{\prime\prime}}
\def\bs{\mathib{s}}
\def\supp{\operatorname{supp}}
\def\doub{\operatorname{doub}}
\def\Dev{\operatorname{Dev}}
\def\Tan{\operatorname{Tan}}
\def\barCal{\overline{\operatorname{Cal}}}
\def\leng{\operatorname{leng}}
\def\Per{\operatorname{Per}}
\def\End{\operatorname{End}}
\def\Reeb{\operatorname{Reeb}}
\def\Aut{\operatorname{Aut}}
\def\dist{\operatorname{dist}}
\def\Index{\operatorname{Index}}

\def\mq{\mathfrak{q}}
\def\mp{\mathfrak{p}}
\def\mH{\mathfrak{H}}
\def\mh{\mathfrak{h}}
\def\ma{\mathfrak{a}}
\def\ms{\mathfrak{s}}
\def\mm{\mathfrak{m}}
\def\mn{\mathfrak{n}}
\def\mz{\mathfrak{z}}
\def\mw{\mathfrak{w}}
\def\Hoch{{\tt Hoch}}
\def\mt{\mathfrak{t}}
\def\ml{\mathfrak{l}}
\def\mT{\mathfrak{T}}
\def\mL{\mathfrak{L}}
\def\mg{\mathfrak{g}}
\def\md{\mathfrak{d}}
\def\mr{\mathfrak{r}}

\title[Lagrangian Floer theory on compact toric manifolds I]
{Lagrangian Floer theory on compact toric manifolds I}

\author[K. Fukaya, Y.-G. Oh, H. Ohta, K.
Ono]{Kenji Fukaya, Yong-Geun Oh, Hiroshi Ohta, Kaoru Ono}
\thanks{KF is supported partially by JSPS Grant-in-Aid for Scientific Research
No.18104001 and Global COE Program G08, YO by US NSF grant \# 0503954, HO by JSPS Grant-in-Aid
for Scientific Research No.19340017, and KO by JSPS Grant-in-Aid for
Scientific Research, Nos. 17654009 and 18340014}

\begin{abstract} The present authors introduced the
notion of \emph{weakly unobstructed} Lagrangian submanifolds and
constructed their \emph{potential function} $\mathfrak{PO}$ purely
in terms of $A$-model data in \cite{fooo06}. In this paper, we carry
out explicit calculations involving $\mathfrak{PO}$ on toric
manifolds and study the relationship between this class of
Lagrangian submanifolds with the earlier work of Givental
\cite{givental1} which advocates that quantum cohomology ring is
isomorphic to the Jacobian ring of a certain function, called the
Landau-Ginzburg superpotential. Combining this study with the
results from \cite{fooo06}, we also apply the study to various
examples to illustrate its implications to symplectic topology of
Lagrangian fibers of toric manifolds. In particular we relate it to
Hamiltonian displacement property of Lagrangian fibers and to
Entov-Polterovich's symplectic quasi-states.
\end{abstract}

\date{Sep. 3, 2009}


\maketitle \tableofcontents
\section{Introduction}

Floer theory of Lagrangian submanifolds plays an important role
in symplectic geometry since Floer's invention \cite{floer:morse}
of the Floer cohomology and subsequent generalization to the class of
\emph{monotone} Lagrangian submanifolds \cite{oh:monotone}. After the introduction of
$A_{\infty}$ structure in Floer theory \cite{fukaya:category} and Kontsevich's homological
mirror symmetry proposal \cite{konts:hms}, it has also played
an essential role in a formulation of mirror symmetry in string theory.

In \cite{fooo00}, we have analyzed the anomaly $\del^2 \neq 0$ and
developed an obstruction theory for the definition of Floer cohomology
and introduced the class of \emph{unobstructed} Lagrangian
submanifolds for which one can deform Floer's original definition of
the `boundary' map by a suitable bounding cochain denoted by $b$.
Expanding the discussion in section 7 \cite{fooo00} and motivated
by the work of Cho-Oh \cite{cho-oh}, we also introduced the notion
of \emph{weakly unobstructed} Lagrangian submanifolds in Chapter 3
\cite{fooo06} which turns out to be the right class of Lagrangian
submanifolds to look at in relation to the mirror symmetry of Fano
toric $A$-model and Landau-Ginzburg $B$-model proposed by physicists
(see \cite{hori}, \cite{hori-vafa}). In the present paper, we study
the relationship between this class of Lagrangian submanifolds with
the earlier work of Givental \cite{givental1} which advocates that
quantum cohomology ring is isomorphic to the Jacobian ring of a
certain function, which is called the Landau-Ginzburg
superpotential. Combining this study with the results from
\cite{fooo06}, we also apply this study to symplectic topology of
Lagrangian fibers of toric manifolds.
\par
While appearance of bounding cochains is natural in the point of view of
deformation theory, explicit computation thereof has not been carried out.
One of the main purposes of the present paper is to perform this
calculation in the case of fibers of toric manifolds and
draw its various applications.
Especially we show that each fiber $L(u)$ at $u \in \mathfrak t^*$
is weakly unobstructed for \emph{any} toric manifold $\pi: X \to \mathfrak t^*$
(see Proposition \ref{unobstruct}), and
then show that the set of the pairs $(L(u),b)$
of a fiber $L(u)$ and a weak bounding cochain $b$ with nontrivial
Floer cohomology can be calculated from the
quantum cohomology of the ambient toric manifold, at least in the Fano case.
Namely the set of such pairs $(L(u),b)$
is identified with the set of ring homomorphisms
from quantum cohomology to the relevant Novikov ring.
We also show by a variational analysis
that for any compact toric manifold there exists at least one pair of
$(u,b)$'s for which the Floer cohomology of $(L(u),b)$ is nontrivial.
\par
We call a Lagrangian fiber (that is, a $T^n$-orbit) {\it balanced},
roughly speaking, if its Floer cohomology is nontrivial.
(See Definition \ref{def:balanced} for its precise definition.)
The main result of this paper is summarized as follows.
\begin{enumerate}
\item When $X$ is a compact Fano toric manifold, we give a
method to locate all the balanced fibers.
\par
\item Even when $X$ is not Fano, we can still apply the same
method to obtain a finite set of Lagrangian fibers. We prove this
set coincides with the set of balanced Lagrangian fibers
under certain nondegeneracy condition. This condition can
be easily checked, when a toric manifold is given.
\end{enumerate}

Now more precise statement of the main results are in order.

Let $X$ be an $n$ dimensional smooth compact toric manifold.
We fix a $T^n$-equivariant K\"ahler form on $X$ and
let $\pi : X \to \mathfrak t^*\cong (\R^n)^*$ be the moment map.
The image $P = \pi(X) \subset (\R^n)^*$ is called the {\it moment polytope}.
For $u \in \text{\rm Int} \,P$, we denote $L(u) = \pi^{-1}(u)$. The fiber $L(u)$
is a Lagrangian torus which is an orbit of the $T^n$ action.
(See section 2. We refer readers to,
for example, \cite{Aud91}, \cite{fulton} for the details on toric manifolds.)
We study the Floer cohomology defined in \cite{fooo06}. According to
\cite{fooo00,fooo06}, we need an extra data, the bounding cochain,
to make the definition of Floer cohomology more flexible to allow
more general class of Lagrangian submanifolds.
In the current context of Lagrangian torus fibers in toric manifolds,
we use \emph{weak bounding cochains}.
Denote by $\CM_{\text{\rm weak}}(L(u);\Lambda_0)$
the moduli space of (weak) bounding cochains
for a weakly unobstructed Lagrangian submanifold $L(u)$.
(See the end of section \ref{sec:statement}.)

In this situation we first show that
each element in $H^1(L(u);\Lambda_{0})$ gives rise to a weak
bounding cochain, i.e., there is a natural embedding
\begin{equation}\label{eq:H1andM}
H^1(L(u);\Lambda_{0}) \hookrightarrow \mathcal M_{\text{\rm weak}}(L(u);\Lambda_0).
\end{equation}
(See Proposition \ref{unobstruct}.)
Here we use the universal Novikov ring
\be
\Lambda = \left\{ \sum_{i=1}^{\infty} a_i T^{\lambda_i} \, \Big\vert \,
a_i \in \Q, \lambda_i \in \R, \lim_{i\to\infty} \lambda_i = \infty\right\}
\ee
where $T$ is a formal parameter. (We do not use the grading
parameter $e$ used in \cite{fooo06} since it will not play much role
in this paper.) Then $\Lambda_0$ is a subring of $\Lambda$ defined by
\be
\Lambda_0 =\left\{\sum_{i=1}^{\infty} a_i T^{\lambda_i}
\in \Lambda \, \Big\vert \,
\lambda_i \ge 0\right\}.
\ee
We also use another subring
\be
\Lambda_+ = \left\{ \sum_{i=1}^{\infty} a_i T^{\lambda_i}
\in \Lambda \, \Big\vert \, \lambda_i > 0\right\}.
\ee
We note $\Lambda$ is the field of fractions of $\Lambda_0$
and $\Lambda_0$ is a local ring with maximal ideal $\Lambda_+$.
Here we take the universal Novikov ring over $\Q$ but we
also use universal Novikov ring over $\C$ or other ring $R$ which we
denote $\Lambda^{\C}$, $\Lambda^R$, respectively.
(In case $R$ does not contain $\Q$, Floer cohomology over
$\Lambda^R$ is defined only in Fano case.)

\begin{rem}
If we strictly follow the way taken in \cite{fooo06}, we only get
the embedding $H^1(L(u);\Lambda_{+}) \hookrightarrow \mathcal M_{\text{\rm
weak}}(L(u))$, not (\ref{eq:H1andM}).
Here
$$\mathcal M_{\text{\rm
weak}}(L(u)) = \mathcal M_{\text{\rm
weak}}(L(u);\Lambda_+)
$$
is defined in \cite{fooo06}. (We note that $\mathcal M_{\text{\rm
weak}}(L(u);\Lambda_+)\ne \mathcal M_{\text{\rm
weak}}(L(u);\Lambda_0)$ where the right hand side is defined at the
end of section \ref{sec:statement}.)
\par
However we can modify the
definition of weak unobstructedness so that (\ref{eq:H1andM})
follows, using the idea of Cho \cite{cho07}. See section
\ref{sec:flat}.
\par
Hereafter we use the symbol $b$ for an element of $\mathcal M_{\text{\rm
weak}}(L(u);\Lambda_+)$ and $\frak x$ for an element of $\mathcal M_{\text{\rm
weak}}(L(u);\Lambda_0)$.
\end{rem}
\par
We next consider the quantum cohomology ring $QH(X;\Lambda)$
with the universal Novikov ring $\Lambda$ as a coefficient ring.
(See section \ref{sec:milquan}.) It is a commutative ring for the
toric case, since $QH(X;\Lambda)$ is generated by
even degree cohomology classes.
\begin{defn}\label{def:MlagAHhom}
\begin{enumerate}
\item We define the set
$
\text{\rm Spec}(QH(X;\Lambda))(\Lambda^{\C})
$
to be the set of $\Lambda$ algebra homomorphisms
$
\varphi : QH(X;\Lambda) \to \Lambda^{\C}
$.
(In other
words it is the set of all $\Lambda^{\C}$ valued points of
the scheme $\text{\rm Spec}(QH(X;\Lambda))$.
\par
\item We next denote by
$
\mathfrak M(\mathfrak{Lag}(X))
$
the set of all pairs $(\mathfrak x,u)$, $u \in \operatorname{Int} P$, $\mathfrak x \in
H^1(L(u);\Lambda_{0}^{\C})/H^1(L(u);2\pi\sqrt{-1}\Z)$ such that
$$
HF((L(u),\mathfrak x),(L(u),\mathfrak x);\Lambda^{\C}) \ne \{0\}.
$$
\end{enumerate}
\end{defn}
\begin{thm}\label{lageqgeopt0}
If $X$ is a Fano toric manifold then
$$
\text{\rm Spec}(QH(X;\Lambda))(\Lambda^{\C}) \cong \mathfrak M(\mathfrak{Lag}(X)).
$$
\par
If $QH(X;\Lambda)$ is semi-simple in addition, we have
\begin{equation}\label{eq:betti}
\sum_d \text{\rm rank}_{\Q} H_d(X;\Q) = \#\left(\mathfrak M(\mathfrak{Lag}(X))\right).
\end{equation}
\end{thm}
We remark that a commutative ring that is a finite dimensional vector space over a field
(e.g., $\Lambda$ in our case) is semi-simple
if and only if it does not contain any nilpotent element. We also remark that a compact
toric manifold is Fano if and only if every nontrivial holomorphic sphere has
positive Chern number.
\par
We believe that (\ref{eq:betti}) still holds in the non-Fano case
but are unable to prove it at the time of writing this paper.
We however can prove that
there exists a fiber $L(u)$ whose Floer cohomology is nontrivial, by a method different
from the proof of Theorem \ref{lageqgeopt0}.
Due to some technical reason, we can only prove the following slightly weaker
statement.
\begin{thm}\label{exitnonvani}
Assume the K\"ahler form $\omega$
of $X$ is rational. Then, there exists $u \in \operatorname{Int} P$ such that for
any ${\CN}\in \R_+$ there exists $\mathfrak x \in H^1(L(u);\Lambda^{\R}_0)$ with
$$
HF((L(u),\mathfrak x),(L(u),\mathfrak x);\Lambda^{\R}_0/(T^{\CN})) \cong H(T^n;{\R}) \otimes_{\R} \Lambda_0^{\R}/(T^{\CN}).
$$
\end{thm}
We suspect that the rationality assumption in Theorem \ref{exitnonvani} can
be removed. It is also likely that we can prove $\mathfrak
M(\mathfrak{Lag}(X))$ is nonempty but its proof at the moment is a bit cumbersome
to write down. We can however derive the following
theorem from Theorem \ref{exitnonvani}, without rationality
assumption.
\begin{thm}\label{toric-intersect}
Let $X$ be an $n$ dimensional compact toric manifold. There exists $u_0 \in \text{\rm Int} P$
such that the following holds for any Hamiltonian diffeomorphism $\psi : X \to X$,
\begin{equation}
\psi(L(u_0)) \cap L(u_0) \ne \emptyset.
\label{eq:Lu_0}\end{equation}
If $\psi(L(u_0))$ is transversal to $L(u_0)$ in addition, then
\label{eq:transLu_0}
\begin{equation}\#(\psi(L(u_0)) \cap L(u_0)) \ge 2^n.\label{eq:Lu_1}\end{equation}
\end{thm}
\par
Theorem \ref{toric-intersect} will be proved in section \ref{sec:floerhom}.
\par
We would like to point out that (\ref{eq:Lu_0}) can be derived from
a more general intersection result, Theorem 2.1
\cite{entov-pol06}, obtained by Entov-Polterovich with a different method using a very interesting
notion of partial symplectic quasi-state constructed out of the spectral
invariants defined in \cite{SchM}, \cite{oh:alan}. (See also
\cite{viterbo}, \cite{oh:cag1} for similar constructions in the
context of exact Lagrangian submanifolds.)

\begin{rem}
Strictly speaking, Theorem 2.1 \cite{entov-pol06} is stated under
the assumption that $X$ is semi-positive and $\omega$ is rational
because the theory of spectral invariant was developed in
\cite{oh:alan} under these conditions. The rationality assumption
has been removed in \cite{oh:minimax}, \cite{usher} and the
semi-positivity assumption of $\omega$ removed by Usher
\cite{usher}. Thus the spectral invariant satisfying all the
properties listed in \cite{entov-pol06} section 5 is now established
for an arbitrary compact symplectic manifold. By the argument of
\cite{entov-pol06} section 7, this implies the existence of a
partial symplectic quasi-state. Therefore the proof of Theorem 2.1
\cite{entov-pol06} goes through without these assumptions
(semi-positivity and rationality) and hence it implies
(\ref{eq:Lu_0}). (See the introduction of \cite{usher}.) But the
result (\ref{eq:Lu_1}) is new.
\end{rem}

Our proof of Theorem \ref{toric-intersect} gives an explicit way of
locating $u_0$, as we show in section \ref{sec:var}. (The method of
\cite{entov-pol06} is indirect and does not provide a way of finding
such $u_0$. See \cite{entov-pol07I}. Below, we will make some
remarks concerning Entov-Polterovich's approach in the perspective
of homological mirror symmetry.) In various explicit examples we can
find more than one element $u_0$ that have the properties stated in
this theorem. Following terminology employed in \cite{cho-oh}, we
call any such torus fiber $L(u_0)$ as in Theorem
\ref{exitnonvani} a \emph{balanced} Lagrangian torus fiber.
(See Definition \ref{def:balanced} for its precise definition.)

A criterion for $L(u_0)$ to be balanced, for the case $\frak x = 0$, is
provided by Cho-Oh \cite{cho-oh} and Cho \cite{cho07} under the Fano
condition. Our proofs of Theorems \ref{exitnonvani},
\ref{eq:transLu_0} are much based on this criterion, and on the idea
of Cho \cite{cho07} of twisting \emph{non-unitary} complex line
bundles in the construction of Floer boundary operator. This
criterion in turn specializes to the one predicted by physicists
\cite{hori-vafa}, \cite{hori}, which relates the location of $u_0$
to the critical points of the Landau-Ginzburg superpotential.

A precise description of balanced Lagrangian fibers including
the data of bounding cochains involves the notion of a \emph{potential function} :
In \cite{fooo06}, the authors have introduced a function
$$
\mathfrak{PO}^L: \CM_{\text{\rm weak}}(L) \to \Lambda_0
$$
for an arbitrary weakly unobstructed Lagrangian submanifold $L \subset (X,\omega)$.
By varying the function $\mathfrak{PO}^L$ over $L \in \{\pi^{-1}(u) \mid u \in \text{Int} \,P\}$,
we obtain the potential function
\be\label{eq:generalPO}
\mathfrak{PO}: \bigcup_{L \in \{\pi^{-1}(u) \mid u \in \text{Int} \,P\}}
\CM_{\text{\rm weak}}(L) \to \Lambda_0.
\ee
This function is constructed purely in terms of $A$-model data of
the general symplectic manifold $(X,\omega)$ \emph{without} using
mirror symmetry.
\par
For a toric $(X,\omega)$, the restriction of $\mathfrak{PO}$ to
$H^1(L(u);\Lambda_+)$ (see (\ref{eq:H1andM})) can be made explicit
when combined with the analysis of holomorphic discs attached to
torus fibers of toric manifolds carried out in \cite{cho-oh}, at
least in the Fano case. (In the non-Fano case we can make it
explicit modulo `higher order terms'.) This function extends to
$H^1(L(u);\Lambda_0)$.
\begin{rem}
In \cite{entov-pol07II} some relationships between quantum
cohomology, quasi-state, spectral invariant and displacement of
Lagrangian submanifolds are discussed : Consider an idempotent
$\text{\bf i}$ of quantum cohomology. The (asymptotic) spectral
invariants associated to $\text{\bf i}$ gives rise to a
partial symplectic quasi-state
via the procedure concocted in \cite{entov-pol07II}, which in turn
detects non-displaceability of certain Lagrangian submanifolds. (The
assumption of \cite{entov-pol06} is weaker than ours.)
\par
In the current context of toric manifolds, we could also relate them
to Floer cohomology and mirror symmetry in the following way :
Quantum cohomology is decomposed into indecomposable factors. (See
Proposition \ref{thm:factorize}.) Let $\text{\bf i}$ be the
idempotent corresponding to one of the indecomposable factors.
Let $L=L(u(1,\text{\bf i}))$ be a Lagrangian torus fiber whose
non-displaceability is detected by the partial symplectic quasi-state obtained from
$\text{\bf i}$. We conjecture that Floer cohomology
$HF(L(u(1,\text{\bf i}),\mathfrak x),(L(u(1,\text{\bf i}),\mathfrak x)))$ is nontrivial
for some $\mathfrak x$. (See Remark \ref{conj:dispfloer}.) This bounding
cochain $\mathfrak x$ in turn is shown to be a critical point of the potential
function $\mathfrak{PO}$ defined in \cite{fooo06}.
\par
On the other hand, $\text{\bf i}$ also determines a homomorphism
$\varphi_{\text{\bf i}} : QH(X;\Lambda) \to \Lambda$. It corresponds
to some Lagrangian fiber $L(u(2,\text{\bf i}))$ by Theorem
\ref{lageqgeopt0}. Then this will imply via Theorem
\ref{homologynonzero} that the fiber $L(u(2,\text{\bf i}))$ is
non-displaceable.
\par
We conjecture that $u(1,\text{\bf i}) = u(2,\text{\bf i})$.
We remark that $u(2,\text{\bf i})$ is explicitly calculable.
Hence in view of the way $u(1,\text{\bf i})$ is found in
\cite{entov-pol06}, $u(1,\text{\bf i}) = u(2,\text{\bf i})$
will give some information on the asymptotic behavior of the
spectral invariant associated with $\text{\bf i}$.
\end{rem}

We fix a basis of the Lie algebra $\mathfrak t$ of $T^n$ which
induces a basis of $\mathfrak t^*$ and hence a coordinate of the
moment polytope $P \subset \mathfrak t^*$. This in turn induces a
basis of $H^1(L(u);\Lambda_{0})$ for each $u \in \operatorname{Int}
P$ and so identification $H^1(L(u);\Lambda_{0}) \cong
(\Lambda_{0})^n$. We then regard the potential function as a
function
$$
\mathfrak{PO}(x_1,\cdots,x_n;u_1,\cdots,u_n) : (\Lambda_{0})^n \times
\operatorname{Int} P \to \Lambda_{0}
$$
and prove in Theorem \ref{homologynonzero} that
Floer cohomology $HF((L(u),\mathfrak x),(L(u),\mathfrak x);\Lambda)$ with
$\mathfrak x = (\mathfrak x_1,\cdots, \mathfrak x_n)$,
$u = (u_1,\cdots, u_n)$ is nontrivial if and only if $(\mathfrak x,u)$ satisfies
\begin{equation}\label{eq:POcrit}
\frac{\partial\mathfrak{PO}}{\partial x_i}(\mathfrak x;u) = 0, \quad i=1,\cdots,n.
\end{equation}
To study (\ref{eq:POcrit}), it is useful to change the variables $x_i$ to
$$
y_i = e^{x_i}.
$$
In these variables we can write potential function as a sum
$$
\mathfrak{PO}(x_1,\cdots,x_n;u_1,\cdots,u_n) = \sum T^{c_i(u)} P_i(y_1,\cdots,y_n)
$$
where $P_i$ are Laurent polynomials which do not depend on $u$, and
$c_i(u)$ are positive real valued functions. When
$X$ is Fano, we can express the right hand side as a finite sum. (See
Theorem \ref{potential}.)

We define a function $\mathfrak{PO}^u$ of $y_i$'s by
$$
\mathfrak{PO}^u(y_1,\cdots,y_n) = \mathfrak{PO}(x_1,\cdots,x_n;u_1,\cdots,u_n)
$$
as a Laurent polynomial of $n$ variables with coefficient in
$\Lambda$. We denote the set of Laurent polynomials by
$$
\Lambda[y_1,\cdots,y_n,y_1^{-1},\cdots,y_n^{-1}]
$$
and consider its ideal generated by
the partial derivatives of $\mathfrak{PO}^u$. Namely
$$
\left(\frac{\partial\mathfrak{PO}^u}{\partial y_i}; i=1,\cdots,n\right).
$$
\begin{defn}
We call the quotient ring
$$
Jac(\mathfrak{PO}^u) =
\frac{\Lambda[y_1,\cdots,y_n,y_1^{-1},\cdots,y_n^{-1}]}
{\left(\frac{\partial\mathfrak{PO}^u}{\partial y_i};
i=1,\cdots,n\right)}
$$
the {\it Jacobian ring} of $\mathfrak{PO}^u$.
\end{defn}
We will prove that the Jacobian ring is independent of the choice of
$u$ up to isomorphism (see the end of section \ref{sec:milquan}) and
so we will just write $Jac(\mathfrak{PO})$ for
$Jac(\mathfrak{PO^u})$ when there is no danger of confusion..
\par

\begin{thm}\label{QHequalMilnor}
If $X$ is Fano, there exists a $\Lambda$ algebra isomorphism
$$
\psi_u : QH(X;\Lambda) \to Jac(\mathfrak{PO})
$$
from quantum cohomology ring to the Jacobian ring such that
$$
\psi_u(c_1(X)) = \mathfrak{PO}^u.
$$
\end{thm}
Theorem \ref{QHequalMilnor} (or Theorem \ref{lageqgeopt} below)
enables us to explicitly determine all the pairs $(\mathfrak x,u)$ with
$HF((L(u),\mathfrak x),(L(u),\mathfrak x);\Lambda) \ne 0$ out of the quantum cohomology
of $X$. More specifically Batyrev's presentation of quantum
cohomology in terms of the Jacobian ring plays an essential role for
this purpose. We will explain how this is done in sections
\ref{sec:exa2}, \ref{sec:exam2}.
\begin{rem}
\begin{enumerate}
\item
The idea that quantum cohomology ring coincides with the Jacobian ring
begins with a celebrated paper by Givental. (See Theorem 5 (1) \cite{givental1}.)
There it was claimed also that the $D$ module
defined by an oscillatory integral with the superpotential as its kernel
is isomorphic to $S^1$-equivariant Floer cohomology of periodic Hamiltonian system.
When one takes its WKB limit, the former becomes the ring of functions on its characteristic variety,
which is nothing but the Jacobian ring. The latter becomes the (small) quantum cohomology ring
under the same limit. \emph{Assuming} the Ansatz that quantum cohomology can be
calculated by fixed point localization, these claims
are proved in a subsequent paper \cite{givental2}
for, at least, toric Fano manifolds.
Then the required fixed point localization is made rigorous later in \cite{grapan}.
See also Iritani \cite{iritani1}.
\par
In physics literature, it has been advocated that Landau-Ginzburg
model of superpotential (that is, the potential function
$\mathfrak{PO}$ in our situation) calculates quantum cohomology of
$X$. A precise mathematical statement thereof is our Theorem
\ref{QHequalMilnor}. (See for example p. 473 \cite{hori-vafa2}.)
\par
Our main new idea entering in the proof of Theorem \ref{lageqgeopt0}
other than those already in \cite{fooo06} is the way how we
combine them to extract information on Lagrangian submanifolds.
In fact Theorem \ref{QHequalMilnor} itself easily follows if we use
the claim made by Batyrev that quantum cohomology of toric Fano
manifold is a quotient of polynomial ring by the ideal of relations, called
quantum Stanley-Reisner relation and linear relation. (This claim is
now well established.) We include this simple derivation in section
\ref{sec:milquan} for completeness' sake, since it is essential to
take the Novikov ring $\Lambda$ as the coefficient ring in our
applications the version of which does not seem to be proven in the
literature in the form that can be easily quoted.
\item
The proof of Theorem \ref{QHequalMilnor} given in this paper does not contain a
serious study of pseudo-holomorphic spheres. The argument which we
outline in Remark \ref{rem:qcanbeused} is based on open-closed
Gromov-Witten theory, and different from other various methods that have
been used to calculate Gromov-Witten invariant in the literature. In
particular this argument does not use the method of fixed point localization. We
will present this conceptual proof of Theorem \ref{QHequalMilnor} in a sequel to
this paper.
\item
The isomorphism in Theorem \ref{QHequalMilnor} may be regarded as a particular
case of the conjectural relation between quantum cohomology and
Hochschild cohomology of Fukaya category. See Remark  \ref{rem:qcanbeused}.
\item In this paper, we only involve small quantum cohomology ring but
we can also include big quantum cohomology ring. Then we expect Theorem \ref{QHequalMilnor}
can be enhanced to establish a relationship between
the Frobenius structure of the deformation theory of quantum cohomology
(see, for example, \cite{Manin:qhm})
and that of Landau-Ginzburg model (which is due to K. Saito \cite{Sai83}).
This statement (and Theorems \ref{lageqgeopt0}, \ref{QHequalMilnor}) can be
regarded as a version of mirror symmetry between the toric A-model
and the Landau-Ginzburg $B$-model.
In various literature on mirror symmetry, such as \cite{abouz}, \cite{AKO04}, \cite{Ued06},
the B-model is dealt for Fano or toric manifolds in which the derived category of coherent sheaves
is studied while the A-model is dealt for Landau-Ginzburg $A$-models
where the directed $A_{\infty}$ category of Seidel \cite{seidel:book} is studied.
\par
\item In \cite{Aur07}, Auroux discussed a mirror symmetry
between the $A$-model side of toric manifolds and the $B$-model
side of Landau-Ginzburg models. The discussion of \cite{Aur07} uses
Floer cohomology with $\C$-coefficients. In this paper
we use Floer cohomology over the Novikov ring which is more suitable for the applications to
symplectic topology.
\par
\item Even when $X$ is not necessarily Fano we can still prove a similar
isomorphism
\begin{equation}\label{nonfanoQHJAC}
\psi_u : QH^{\omega}(X;\Lambda) \cong Jac(\mathfrak{PO}_0)
\end{equation}
where the left hand side is the Batyrev quantum cohomology ring
(see Definition \ref{def:batyrevqcr})
and the right hand side is the Jacobian ring of some function
$\mathfrak{PO}_0$ : it coincides with the actual potential function
$\mathfrak{PO}$ `up to higher order terms'.
(See (\ref{def:PO0}).)
In the Fano case $\mathfrak{PO}_0 = \mathfrak{PO}$.
(\ref{nonfanoQHJAC}) is Proposition \ref{nonfanoadded}.
\item
During the final stage of writing this article, a paper
\cite{chan-leung} by Chan and Leung was posted in the Archive
which studies the above isomorphism via SYZ transformations.
They give a proof of this isomorphism for the case where
$X$ is a product of projective spaces and with
the coefficient ring $\C$, not with Novikov ring.
Leung presented their result \cite{chan-leung}
in a conference held in Kyoto University in January 2008 where
the first named author also presented the content of this paper.
\end{enumerate}
\end{rem}

From our definition, it follows that the leading order potential function
$\mathfrak{PO}_0$ (see (\ref{def:PO0})) can be extended to the whole
product $(\Lambda^\C_0)^n\times \R^n$ in a way that they are
invariant under the translations by elements in $(2\pi
\sqrt{-1}\Z)^n \subseteq (\Lambda^\C_0)^n$. Hence we may regard $\mathfrak{PO}_0$
as a function defined on
$$
\left(\Lambda^\C_0/(2\pi \sqrt{-1}\Z)\right)^n \times \R^n
\cong \left(\Lambda^\C_0/(2\pi \sqrt{-1}\Z)\right)^n \times \R^n.
$$
In the non-Fano case, the function $\mathfrak{PO}$
is invariant under the translations by elements in $(2\pi
\sqrt{-1}\Z)^n \subseteq (\Lambda^\C_0)^n$ but may not extend to
$\Lambda^\C_0/(2\pi \sqrt{-1}\Z))^n \times \R^n$. This is because the infinite sum
appearing in the right hand side of (\ref{eq:weakPO}) may
not converge in non-Archimedean topology for $u \notin \mbox{Int}\,P$.
\begin{defn}\label{frakM}
We denote by
$$
\text{\rm Crit}(\mathfrak{PO}_0), \quad (\mbox{respectively }\,
\text{\rm Crit}(\mathfrak{PO}))
$$
the subset of pairs
$$
(\mathfrak x,u)\in (\Lambda_{0}^{\C}/(2\pi\sqrt{-1}\Z))^n \times \R^n,
\quad (\mbox{respectively }\,
(\mathfrak x,u)\in (\Lambda_{0}^{\C}/(2\pi\sqrt{-1}\Z))^n \times \mbox{Int}\,P)
$$
satisfying the equation
$$
\frac{\partial\mathfrak{PO}_0}{\partial x_i}(\mathfrak x;u) = 0, \quad
\left(\mbox{respectively }\, \frac{\partial\mathfrak{PO}}{\partial
x_i}(\mathfrak x;u) = 0\right)
$$
$i = 1,\cdots,n$.
\end{defn}
\par
We define $\mathfrak M(\mathfrak{Lag}(X))$ in Definition \ref{def:MlagAHhom}.
(We use the same definition in non-Fano case.)
In view of Theorem \ref{lageqgeopt} (2) below, we also introduce the subset
$$
\mathfrak M_0(\mathfrak{Lag}(X))=
\{(\mathfrak x,u)\in \text {\rm Crit}(\mathfrak{PO}_0) \mid u \in \text{\rm Int}\,P\}.
$$
\par
We also remark $\mathfrak{PO}_0 = \mathfrak{PO}$ in case $X$ is
Fano. The following is a
more precise form of Theorem \ref{lageqgeopt0}.
\begin{thm}\label{lageqgeopt}
\begin{enumerate}
\item There exists a bijection :
$$
\text{\rm Spec}(QH^{\omega}(X;\Lambda))(\Lambda^{\C}) \cong
\text{\rm Crit}(\mathfrak{PO}_0).
$$
\item There exists a bijection :
$$
\mathfrak M(\mathfrak{Lag}(X)) \cong
\text {\rm Crit}(\mathfrak{PO}).
$$
\item If $X$ is Fano and $\frak x$ is a critical point of $\frak{PO}_0^u$ then
$u \in \mbox{\rm Int}\, P$.
\item If $QH^{\omega}(X;\Lambda)$ is semi-simple, then
$$
\sum_d \text{\rm rank}_{\Lambda} QH^{\omega}(X;\Lambda) = \#
\left(\text{\rm Crit}(\mathfrak{PO}_0)\right).
$$
\end{enumerate}
\end{thm}
\begin{rem}
(1) Theorem \ref{lageqgeopt} (3) does not hold in non-Fano case. We give a
counter example (Example \ref{ex:Hilexa}) in section \ref{sec:exam2}. In fact,
in the case of Example \ref{ex:Hilexa} some of the critical points of
$\frak{PO}_0$ correspond to a point $u \in \R^n$ which lies outside
the moment polytope.
\par
(2) In the Fano case Theorem \ref{lageqgeopt} (3) implies
$$
\frak M(\frak{Lag}(X)) = \mbox{Crit}(\frak{PO}) = \mbox{Crit}(\frak{PO}_0).
$$
\end{rem}
We would like to point out that $\mathfrak{PO}_0$ is explicitly
computable. But we do not know the explicit form of $\mathfrak{PO}$.
However we can show that elements of $\text {\rm Crit}(\mathfrak{PO}_0)$
and of $\text {\rm Crit}(\mathfrak{PO})$
can be
naturally related to each other under a mild nondegeneracy
condition. (Theorem \ref{thm:elliminate}.) So we can use
$\mathfrak{PO}_0$ in place of $\mathfrak{PO}$ in most of the cases.
For example we can use it to prove that the following :
\begin{thm}\label{thm:kblowup}
For any $k$, there exists a K\"ahler form on $X(k)$, the $k$ points
blow up of $\C P^2$, that is toric and has exactly $k+1$ balanced fibers.
\end{thm}
See Definition \ref{def:balanced} for the definition of balanced fibers.
Balanced fiber satisfies the conclusions (\ref{eq:Lu_0}), (\ref{eq:Lu_1}) of Theorem
\ref{toric-intersect}. We prove Theorem \ref{thm:kblowup} in section \ref{sec:ellim}.
\begin{rem}
The cardinality of $\frak x \in H^1(L(u);\Lambda^{\C}_0)/H^1(L(u);2\pi\sqrt{-1}\Z)$ with nonvanishing
Floer cohomology is an invariant of Lagrangian submanifold $L(u)$.
This is a consequence of \cite{fooo06} Theorem G
(= \cite{fooo06pre} Theorem G).
\end{rem}
\par
The organization of this paper is now in order.
In section \ref{sec:toric} we gather some basic facts on toric manifolds
and fix our notations. Section \ref{sec:floerreview} is a brief review of Lagrangian
Floer theory of \cite{fooo00, fooo06}.
In section \ref{sec:statement}, we describe our main results on the
potential function $\frak{PO}$ and on its relation to the
Floer cohomology. We illustrate these theorems by several examples and derive
their consequences in sections \ref{sec:example} -
\ref{sec:ellim}. We postpone their proofs until sections \ref{sec:calcpot} -
\ref{sec:floerhom}.
\par
In section \ref{sec:example}, we illustrate explicit
calculations involving the potential functions in such examples as
$\C P^n$, $S^2 \times S^2$ and the two points blow up of $\C P^2$.
We also discuss a relationship between displacement energy of
Lagrangian submanifold (see Definition \ref{dispenergy}) and Floer cohomology.
In sections \ref{sec:milquan} and \ref{sec:exa2}, we prove the results that
mainly apply to the Fano case. Especially we prove Theorems
\ref{lageqgeopt0}, \ref{QHequalMilnor} and \ref{lageqgeopt} in these sections.
Section \ref{sec:exa2} contains some applications of Theorem \ref{QHequalMilnor}
especially to the case of monotone torus fibers and to the $\Q$-structure of
quantum cohomology ring.
In section \ref{sec:exam2} we first illustrate usage of (the proof of)
Theorem \ref{lageqgeopt0} to locate balanced fibers by the example of
one point blow up of $\C P^2$.
We then turn to the study of non-Fano cases and discuss
Hirzebruch surfaces. Section \ref{sec:exam2} also contains some discussion
on the semi-simplicity of quantum cohomology.
\par
In sections \ref{sec:var} and \ref{sec:ellim}, we prove the results
that can be used in all toric cases, whether they are Fano or not.
In section \ref{sec:var}, using variational analysis, we prove
existence of a critical point of the potential function, which is an
important step towards the proof of Theorems \ref{exitnonvani} and
\ref{toric-intersect}. Using the arguments of this section, we can
locate a balanced fiber in any compact toric manifold, explicitly
solving simple linear equalities and inequalities finitely many
times. In section \ref{sec:ellim}, we prove that we can find the
solution of (\ref{eq:POcrit}) by studying its reduction to $\C =
\Lambda_0^{\C}/\Lambda_+^{\C}$, which we call the leading term
equation. This result is purely algebraic. It implies that our
method of locating balanced fibers, which is used in the proof of
Theorem \ref{lageqgeopt0}, can be also used in the non-Fano case
under certain nondegeneracy condition. We apply this method to prove
Theorem \ref{thm:kblowup}. We discuss an example of blow up of $\C
P^n$ along the high dimensional blow up center $\C P^m$ in that
section. We also give several other examples and demonstrate various
interesting phenomena which occur in Lagrangian Floer theory. For
example we provide a sequence $((X,\omega_i),L_i)$ of pairs that
have nonzero Floer cohomology for some choice of bounding cochains,
while its limit $((X,\omega),L)$ has vanishing Floer cohomology
for any choice of bounding cochain (Example \ref{counterexamples}).
We also provide an example of
Lagrangian submanifold $L$ such that it has a nonzero Floer
cohomology over $\Lambda^{\C}$ for some choice of bounding cochain,
but vanishing Floer cohomology for any choice of bounding cochain
over the field $\Lambda^F$ with a field $F$ of characteristic $3$
(Example \ref{torsioncount}).
\par
In section \ref{sec:calcpot}, we review the results
on the moduli space of holomorphic discs from \cite{cho-oh}
which are used in the calculation of the potential function. We
rewrite them in the form that can be used for the purpose of
this paper. We also discuss the non-Fano case in this section.
(Our result is less explicit in the non-Fano case,
but still can be used to explicitly locate balanced fibers
in most of the cases.)
In section \ref{sec:flat}, we use the idea of Cho \cite{cho07} to deform
Floer cohomology by an element from $H^1(L(u);\Lambda_0)$ rather than
from $H^1(L(u);\Lambda_+)$. This enhancement is crucial to obtain an optimal
result about the non-displacement of Lagrangian fibers.
In section \ref{sec:floerhom} we use those results to calculate
Floer cohomology and complete the proof of Theorems \ref{exitnonvani}
and \ref{toric-intersect} etc.
\par
We attempted to make this paper largely independent of
our book \cite{fooo00,fooo06} as much as possible and also
to make the relationship of the contents of the paper with the
general story transparent. Here are a few examples :
\begin{enumerate}
\item Our definition of potential
function for the fibers of toric manifolds in this paper is given in a way independent of
that of \cite{fooo06} except the statement on the existence of
compatible Kuranishi structures and multi-sections
on the moduli space of pseudo-holomorphic discs which provides a rigorous definition
of Floer cohomology of single Lagrangian fiber. Such details
are provided in section 3.1 \cite{fooo06} (= section 29 \cite{fooo06pre}).

\item Similarly the definition of $A_\infty$ algebra in this paper on the
Lagrangian fiber of toric manifolds is also independent of the book except the process
going from $A_{n,K}$ structure to $A_\infty$ structure, which we refer to
section 7.2 \cite{fooo06} (= section 30 \cite{fooo06pre}).
However, for all the applications in this paper, only existence of $A_{n,K}$ structures is
needed.

\item The property of the Floer cohomology $HF(L,L)$ detecting Lagrangian intersection of
$L$ with its Hamiltonian deformation relies on the fact that Floer cohomology
of the pair is independent under the Hamiltonian isotopy. This independence is established in
\cite{fooo06}. In the toric case, its alternative proof based on the de Rham version is given in
\cite{fooo08} in a more general form than we need in this paper.
\end{enumerate}

The authors would like to thank H. Iritani and D. McDuff
for helpful discussions.
They would also like to thank the referee for various helpful comments.

\section{Compact toric manifolds}
\label{sec:toric}

In this section, we summarize basic facts on the toric manifolds and
set-up our notations to be consistent with those in \cite{cho-oh}, which
in turn closely follow those in Batyrev \cite{batyrev:qcrtm92}
and M. Audin \cite{Aud91}.

\subsection{Complex structure}
\label{subsec:complex}

In order to obtain an $n$-dimensional compact toric manifold $X$, we
need a combinatorial object $\Sigma$, a {\em complete fan of
regular cones}, in an $n$-dimensional vector space over $\R$.

Let $N$ be the lattice $\Z^n$, and let $M=Hom_\Z(N,\Z)$ be the
dual lattice of rank $n$. Let $N_{\R} = N \otimes \R$ and
$M_{\R} = M \otimes \R$.
\begin{defn}
A convex subset $\sigma \subset N_{\R}$ is called a regular
$k$-dimensional cone $(k\geq 1)$ if there exists $k$ linearly
independent elements $v_1,\cdots,v_k \in N$ such that
$$ \sigma = \{a_1v_1 + \cdots + a_k v_k \mid a_i \in \R, a_i \geq
0\},$$ and the set $\{v_1,\cdots,v_k\}$ is a subset of some
$\Z$-basis of $N$. In this case, we call $v_1,\cdots,v_k \in N$
the integral generators of $\sigma$.
\end{defn}
\begin{defn}
A regular cone $\sigma'$ is called a {\em face} of a regular cone
$\sigma$ (we write $\sigma' \prec \sigma$) if the set of integral
generators of $\sigma'$ is a subset of the set of integral
generators of $\sigma$.
\end{defn}
\begin{defn}
A finite system $\Sigma = {\sigma_1,\cdots,\sigma_s}$ of regular
cones in $N_{\R}$ is called a {\em complete $n$-dimensional fan}
of regular cones, if the following conditions are satisfied.
\begin{enumerate}
\item if $\sigma \in \Sigma$ and $\sigma' \prec \sigma$, then $\sigma'
\in \Sigma$;

\item if $\sigma, \sigma'$ are in $\Sigma$, then $\sigma' \cap
\sigma \prec \sigma$ and $\sigma' \cap \sigma \prec \sigma'$;

\item $N_{\R} = \sigma_1 \cup \cdots \cup \sigma_s$.
\end{enumerate}
\end{defn}
The set of all $k$-dimensional cones in $\Sigma$ will be denoted
by $\Sigma^{(k)}.$

\begin{defn}\label{prim}
Let $\Sigma$ be a complete $n$-dimensional fan of regular cones.
Denote by $G(\Sigma) = \{v_1,\cdots,v_m\}$ the set of all
generators of 1-dimensional cones in $\Sigma$ ($m=$ Card
$\Sigma^{(1)}$). We call a subset $\CP =
\{v_{i_1},\cdots,v_{i_p}\} \subset G(\Sigma)$ a {\it primitive
collection} if $\{v_{i_1},\cdots,v_{i_p}\}$ does not generate
$p$-dimensional cone in $\Sigma$, while for all $k \, (0 \leq k <
p)$ each $k$-element subset of $\CP$ generates a $k$-dimensional
cone in $\Sigma$.
\end{defn}

\begin{defn}
Let $\C^m$ be an $m$-dimensional affine space over $\C$ with the
set of coordinates $z_1,\cdots,z_m$ which are in the one-to-one
correspondence $z_i \leftrightarrow v_i$ with elements of
$G(\Sigma)$. Let $\CP = \{v_{i_1},\cdots,v_{i_p}\}$ be a primitive
collection in $G(\Sigma)$. Denote by $\A(\CP)$ the
$(m-p)$-dimensional affine subspace in $\C^n$ defined by the
equations
$$z_{i_1}= \cdots=z_{i_p}=0.$$
\end{defn}

Since every primitive collection $\CP$ has at least two elements,
the codimension of $\A(\CP)$ is at least 2.

\begin{defn}\label{homo}
Define the closed algebraic subset $Z(\Sigma)$ in $\C^m$ as follows
$$Z(\Sigma) = \bigcup_{\CP} \A(\CP),$$
where $\CP$ runs over all primitive collections in $G(\Sigma)$.
Put
$$U(\Sigma) = \C^m \setminus Z(\Sigma).$$
\end{defn}

\begin{defn}
Let $\K$ be the subgroup in $\Z^m$ consisting of all lattice vectors
$\lambda = (\lambda_1,\cdots,\lambda_m)$ such that
$$\lambda_1 v_1 + \cdots + \lambda_m v_m =0.$$
\end{defn}
Obviously $\K$ is isomorphic to $\Z^{m-n}$ and we have the exact
sequence:
\begin{equation}\label{kexact}
0 \to \K \to \Z^m \stackrel{\pi}{\to} \Z^n \to 0,
\end{equation}
where
the map $\pi$ sends the basis vectors $e_i$ to $v_i$ for
$i=1,\cdots,m$.

\begin{defn}
Let $\Sigma$ be a complete $n$-dimensional fan of regular cones.
Define $D(\Sigma)$ to be the connected commutative subgroup in
$(\C^*)^m$ generated by all one-parameter subgroups
$$a_{\lambda} : \C^* \to (\C^*)^m,$$
$$ t \mapsto (t^{\lambda_1},\cdots,t^{\lambda_m})$$
where $\lambda = (\lambda_1, \cdots, \lambda_m) \in \K$.
\end{defn}

It is easy to see from the definition that $D(\Sigma)$ acts freely
on $U(\Sigma)$. Now we are ready to give a definition of the
compact toric manifold $X_{\Sigma}$ associated with a complete
$n$-dimensional fan of regular cones $\Sigma$.

\begin{defn}
Let $\Sigma$ be a complete $n$-dimensional fan of regular cones.
Then the quotient
$$X_{\Sigma} = U(\Sigma)/D(\Sigma)$$
is called the {\em compact toric manifold associated with $\Sigma$}.
\end{defn}

There exists a simple open coverings of $U(\Sigma)$ by affine
algebraic varieties.

\begin{prop}
Let $\sigma$ be a $k$-dimensional cone in $\Sigma$ generated by
$\{v_{i_1},\cdots,v_{i_k}\}.$ Define the open subset $U(\sigma)
\subset \C^m$ as
$$ U(\sigma) = \{(z_1,\cdots,z_m) \in \C^m \mid z_j \neq 0
\;\;\textrm{for all}\; j \notin \{i_1,\cdots,i_k\}\}.$$
Then the open sets $U(\sigma)$ have the following properties:
\begin{enumerate}
\item $$U(\Sigma) = \bigcup_{\sigma \in \Sigma} U(\sigma);$$
\item if $\sigma \prec \sigma'$, then $U(\sigma) \subset U(\sigma')$;
\item for any two cone $\sigma_1,\sigma_2 \in \Sigma$,
one has $U(\sigma_1) \cap U(\sigma_2) = U(\sigma_1 \cap \sigma_2)$;
in particular,
$$ U(\Sigma) = \sum_{\sigma \in \Sigma^{(n)}} U(\sigma).$$
\end{enumerate}
\end{prop}

\begin{prop}\label{homocord}
Let $\sigma$ be an $n$-dimensional cone in $\Sigma^{(n)}$
generated by $\{v_{i_1},\cdots,v_{i_n}\}$, which spans the lattice
$M$. We denote the dual $\Z$-basis of the lattice $N$ by
$\{u_{i_1},\cdots,u_{i_n}\}$. i.e.
\begin{equation}
\langle v_{i_k},u_{i_l} \rangle = \delta_{k,l}
\end{equation}
where $\langle \cdot,\cdot \rangle $ is the canonical pairing
between lattices $N$ and $M$.

Then the affine open subset $U(\sigma)$ is isomorphic to
$\C^n \times (\C^*)^{m-n}$, the action of
$D(\Sigma)$ on $U(\sigma)$ is free, and the space of
$D(\Sigma)$-orbits is isomorphic to the affine
space $U_{\sigma} = \C^n$ whose coordinate functions
$y_1^\sigma,\cdots,y_n^\sigma$ are $n$ Laurent monomials
in $z_1,\cdots,z_m$:
\begin{equation}\label{homocoordeq}
\begin{cases}
y_1^\sigma = z_1^{\langle v_1,u_{i_1} \rangle }\cdots
z_m^{\langle v_m,u_{i_1} \rangle }\\
\qquad \vdots\\
y_n^\sigma = z_1^{\langle v_1,u_{i_n} \rangle }\cdots
z_m^{\langle v_m,u_{i_n} \rangle }
\end{cases}
\end{equation}

\end{prop}

The last statement yields a general formula for the local affine
coordinates $y_1^\sigma, \cdots,y_n^\sigma$ of a point
$p \in U_{\sigma}$ as functions of its
``homogeneous coordinates'' $z_1,\cdots,z_m$.

\subsection{Symplectic structure}
\label{subsec:forms}

In the last subsection, we associated a compact manifold $X_{\Sigma}$
to a fan $\Sigma$. In this subsection, we review the construction of
symplectic (K\"ahler) manifold associated to a convex polytope
$P$.

Let $M$ be a dual lattice, we consider a convex polytope $P$ in $M_{\R}$
defined by
\begin{equation}\label{eq:PinMR}
\{u \in M_{\R} \mid \langle u,v_j \rangle \geq \lambda_j
\;\textrm{for}\; j=1,\cdots,m\}
\end{equation}
where $\langle \cdot,\cdot \rangle $ is a dot product of $M_{\R}
\cong \R^n$. Namely, $v_j$'s are inward normal vectors to the
codimension 1 faces of the polytope $P$. We associate to it a fan
in the lattice $N$ as follows: With any face $\Gamma$ of $P$, fix
a point $u_0$ in the (relative) interior of $\Gamma$ and define
$$\sigma_{\Gamma} = \bigcup_{r \geq 0} r \cdot (P-u_0).$$
The associated fan is the family $\Sigma(P)$ of dual convex cones
\begin{eqnarray}
\check{\sigma}_{\Gamma} &=&\{ x\in N_\R \mid \langle y,x \rangle
\geq 0 \;\;\forall y \in
\sigma_{\Gamma} \} \\
&=&\{ x\in N_\R \mid \langle u,x \rangle \leq \langle p,x
\rangle \;\;\forall p \in P, u \in \Gamma \}
\end{eqnarray}
where $\langle \cdot,\cdot \rangle $ is dual pairing $M_\R$ and
$N_\R$. Hence we obtain a compact toric manifold $X_{\Sigma(P)}$
associated to a fan $\Sigma(P)$.

Now we define a symplectic (K\"ahler) form on $X_{\Sigma(P)}$ as
follows. Recall the exact sequence :
$$ 0 \to \K \stackrel{i}\to \Z^m \stackrel{\pi}{\to} \Z^n \to 0.$$
It induces another exact sequence :
$$ 0 \to K \to \R^m/\Z^m \to \R^n/\Z^n \to 0.$$
Denote by $k$ the Lie algebra of the real torus $K$. Then we have
the exact sequence of Lie algebras:
$$ 0 \to k \to \R^m \stackrel{\pi}{\to} \R^n \to 0.$$
And we have the dual of above exact sequence:
$$ 0 \to (\R^n)^* \to (\R^m)^* \stackrel{i^*}{\to} k^* \to 0.$$

Now, consider $\C^m$ with symplectic form $\frac{i}{2} \sum dz_k \wedge
d\overline{z}_k$.
The standard action $T^n$ on $\C^n$ is hamiltonian with moment map
\begin{equation}
\mu(z_1,\cdots,z_m) = \frac{1}{2}(|z_1|^2,\cdots,|z_m|^2).
\end{equation}

For the moment map $\mu_K$ of the $K$ action is then given by
$$
\mu_K=i^* \circ \mu : \C^m \to k^*.
$$
If we choose a $\Z$-basis
of $\K \subset \Z^m$ as
$$ Q_1 = (Q_{11},\cdots,Q_{m1}),\cdots, Q_k= (Q_{1k},\cdots,Q_{mk})$$
and $\{q^1, \cdots, q^k\}$ be its dual basis of $\K^*$. Then the
map $i^*$ is given by the matrix $Q^t$ and so we have
\begin{equation}
\mu_K(z_1,\cdots,z_m) = \frac{1}{2} \left(\sum_{j=1}^m
Q_{j1}|z_j|^2,\cdots,\sum_{j=1}^m Q_{jk}|z_j|^2\right) \in \R^k \cong
k^*
\end{equation}
in the coordinates associated to the basis $\{q^1, \cdots, q^k\}$.
We denote again by $\mu_K$ the restriction of $\mu_K$ on
$U(\Sigma) \subset \C^m$.

\begin{prop}[See Audin \cite{Aud91}]
Then for any $r = (r_1, \cdots, r_{m-n}) \in \mu_K(U(\Sigma))
\subset k^* $, we have a diffeomorphism
\begin{equation}
\mu_K^{-1}(r)/K \cong U(\Sigma)/D(\Sigma) = X_{\Sigma}
\end{equation}
And for each (regular) value of $r \in k^*$, we can associate a
symplectic form $\omega_P$ on the manifold $X_\Sigma$ by
symplectic reduction \cite{marwein}.
\end{prop}

To obtain the original polytope $P$ that we started with, we need
to choose $r$ as follows: Consider $\lambda_j$ for $j=1,\cdots,m$
which we used to define our polytope $P$ by the set of
inequalities $\langle u,v_j \rangle \geq \lambda_j$. Then, for
each $a=1,\cdots,m-n$, let
$$r_a = -\sum_{j=1}^m Q_{ja} \lambda_j.$$

Then we have
$$\mu_K^{-1}(r_1, \cdots, r_{m-n})/K \cong
X_{\Sigma(P)}$$ and for the residual $T^n \cong T^m/K$ action on
$X_{\Sigma(P)}$, and for its moment map $\pi$, we have
$$\pi(X_{\Sigma(P)}) = P.
$$
Using Delzant's theorem \cite{delzant}, one can reconstruct the
symplectic form out of the polytope $P$ (up to $T^n$-equivariant
symplectic diffeomorphisms). In fact, Guillemin \cite{Gu} proved the
following explicit closed formula for the $T^n$-invariant K\"ahler
form associated to the canonical complex structure on $X =
X_\Sigma(P)$

\begin{thm}[Guillemin]\label{gullemin}
Let $P$, $X_{\Sigma(P)}$, $\omega_P$ be as above and
$$
\pi: X_{\Sigma(P)} \to (\R^m/k)^*\cong (\R^n)^*
$$
be the associated moment map. Define the functions on
$(\R^n)^*$
\begin{eqnarray}
\ell_i(u) & = &\langle u, v_i \rangle - \lambda_i \, \mbox{ for }\,
i = 1, \cdots, m \label{eq:elli}\\
\ell_\infty(u) & = &\sum_{i=1}^m\langle u, v_i \rangle =\left\langle u,
\sum_{i=1}^m v_i \right\rangle \nonumber.
\end{eqnarray}
Then we have
\begin{equation}\label{eq:omegaP}
\omega_P =
\sqrt{-1}\partial\overline{\partial}\left(\pi^*\Big(\sum_{i=1}^m\lambda_i
(\log \ell_i) + \ell_\infty\Big)\right)
\end{equation}
on $\text{\rm Int}\,P$.
\end{thm}

The affine functions $\ell_i$ will play an important role in our
description of potential function as in \cite{cho-oh} since they
also measure symplectic areas $\omega(\beta_i)$ of the canonical generators $\beta_i$
of $H_2(X,L(u);\Z)$. More precisely we have
\be\label{eq:area}
\omega(\beta_i) = 2 \pi \ell_i(u)
\ee
(see Theorem 8.1 \cite{cho-oh}). We also recall
\be\label{eq:P} P = \{u \in M_\R \mid \ell_i(u) \geq 0, \, i = 1,
\cdots, m \} \ee by definition (\ref{eq:PinMR}).

\section{Deformation theory of filtered $A_\infty$-algebras}
\label{sec:floerreview}

In this section, we provide a quick summary of the
deformation and obstruction theory
of Lagrangian Floer cohomology developed in \cite{fooo00,fooo06} for readers' convenience. We also refer readers to
the third named author's survey paper \cite{ohta01} for a more detailed review, and
refer to \cite{fooo06}
for complete details of the proofs of the results
described in this section.

We start our discussion with the classical unfiltered $A_\infty$ algebra.
Let $C$ be a graded $R$-module where $R$ is the coefficient ring.
We denote by $C[1]$ its suspension defined by $C[1]^k = C^{k+1}$.
Define the {\it bar complex} $B(C[1])$ by
$$
B_k(C[1]) = \underbrace{C[1]\otimes \cdots \otimes C[1]}_{k}, \quad B(C[1]) =
\bigoplus_{k=0}^\infty B_k(C[1]).
$$
Here $B_0(C[1]) = R$ by definition.  $B(C[1])$
has the structure of {\it graded coalgebra}.

\begin{defn} The structure of {\it $A_\infty$ algebra} is a
sequence of $R$ module homomorphisms
$$
\mathfrak m_k: B_k(C[1]) \to C[1], \quad k = 1, 2, \cdots,
$$
of degree +1 such that the coderivation
$\widehat d = \sum_{k=1}^\infty \widehat{\mathfrak m}_k$
satisfies $\widehat d \,\widehat d= 0$, which
is called the \emph{$A_\infty$-relation}.
Here we denote by $\widehat{\mathfrak
m}_k: B(C[1]) \to B(C[1])$ the unique extension of $\mathfrak m_k$
as a coderivation on $B(C[1])$, that is
\begin{equation}\label{hatmk}
\widehat{\mathfrak m}_k(x_1 \otimes \cdots \otimes x_n) =
\sum_{i=1}^{k-i+1} (-1)^{*}
x_1 \otimes \cdots \otimes \frak m_k(x_i,\cdots,x_{i+k-1})
\otimes \cdots \otimes x_n
\end{equation}
where $* = \deg x_1+\cdots+\deg x_{i-1}+i-1$.
\end{defn}
The relation $\widehat d \,\widehat d= 0$ can be written as
$$
\sum_{k=1}^{n-1}\sum_{i=1}^{k-i+1} (-1)^{*} \frak m_{n-k+1}(
x_1 \otimes \cdots \otimes \frak m_k(x_i,\cdots,x_{i+k-1})
\otimes \cdots \otimes x_n) = 0,
$$
where $*$ is the same as above.
In particular, we have $\mathfrak m_1
\mathfrak m_1 = 0$ and so it defines a complex $(C,\mathfrak m_1)$.
\par
A {\it weak} (or {\it curved}) $A_\infty$-algebra is defined in the same way, except
that it also includes the $\mathfrak m_0$-term
$
\mathfrak m_0: R \to B(C[1]).
$
The first two terms of the $A_\infty$ relation for a weak
$L_\infty$ algebra are given as
$$
\mathfrak m_1(\mathfrak m_0(1))  = 0 \label{eq:m1m0=0},
\quad
\mathfrak m_1\mathfrak m_1 (x) + (-1)^{\deg x+1}\mathfrak m_2(x, \mathfrak m_0(1)) +
\mathfrak m_2(\mathfrak m_0(1), x) = 0. \label{eq:m0m1}
$$
In particular, for the case of weak $A_\infty$ algebras, $\mathfrak
m_1$ will not satisfy boundary property, i.e., $\mathfrak m_1\mathfrak m_1
\neq 0$ in general.

We now recall the notion of \emph{unit} in $A_\infty$ algebra.

\begin{defn} An element ${\bf e} \in C^0 = C[1]^{-1}$ is called a unit
if it satisfies
\begin{enumerate}
\item $\mathfrak m_{k+1}(x_1, \cdots,{\bf e}, \cdots, x_k) = 0$ for $k \geq 2$ or $k=0$.
\item $\mathfrak m_2({\bf e}, x) = (-1)^{\deg x}\mathfrak m_2(x,{\bf e}) = x$ for all $x$.
\end{enumerate}
\end{defn}

Combining this definition of unit and \eqref{hatmk}, we have the
following immediate lemma.

\begin{lem} Consider an $A_\infty$ algebra $(C[1],\mathfrak m)$
over a ground ring $R$ for which
$\mathfrak m_0(1) = \lambda {\bf e}$ for some $\lambda \in R$.
Then $\mathfrak m_1 \mathfrak m_1 = 0$.
\end{lem}

Now we explain the notion of the \emph{filtered $A_{\infty}$ algebra}.
We define the {\it universal Novikov ring} $\Lambda_{0,nov}$ by
$$
\Lambda_{0,nov}
= \left.\left\{ \sum_{i=1}^{\infty} a_i T^{\lambda_i}e^{n_i}
\,\,\right\vert\,\, a_i \in R, n_i \in \Z, \lambda_i \in \R_{\ge 0},
\lim_{i\to\infty} \lambda_i = \infty\right\}.
$$
This is a graded ring by defining $\deg T = 0$, $\deg e = 2$.
Let $\Lambda_{0,nov}^+$ be its maximal ideal
which consists of the elements
$\sum_{i=1}^{\infty} a_i T^{\lambda_i}e^{n_i}$ with $\lambda_i >0$.
\par
Let $\bigoplus_{m\in \Z}C^m$ be the free graded $\Lambda_{0,nov}$
module over the basis  $\{\bold v_{i}\}$. We define a filtration
$\bigoplus_{m\in \Z}F^{\lambda}C^m$ on it such that
$\{T^{\lambda}\bold v_{i}\}$ is a free basis of $\bigoplus_{m\in
\Z}F^{\lambda}C^m$. Here $\lambda \in \R$, $\lambda \ge 0$. We call
this filtration the {\it energy filtration}. (Our algebra
$\bigoplus_{m\in \Z}C^m$ may not be finitely generated. So we need
to take completion.) We denote by $C$ the completion of $\bigoplus
_{m\in \Z}C^m$ with respect to the energy filtration. The filtration
induces a natural non-Archimedean topology on $C$.
\par
A {\it filtered $A_\infty$ algebra} $(C,\frak m)$ is a weak $A_\infty$
algebra such that $A_\infty$ operators  $\frak m$ have the following properties :
\begin{enumerate}
\item $\frak m_k$ respect the energy filtration,
\item $\frak m_0(1) \in F^{\lambda}C^1$ with $\lambda >0$,
\item
The reduction
$\frak m_k \mod \Lambda_{0,nov}^+ : B_k\overline C[1] \otimes R[e,e^{-1}] \to
\overline C  \otimes  R[e,e^{-1}]$ does not contain $e$.
More precisely speaking, it has the form
$\overline{\frak m}_k \otimes\,_RR[e,e^{-1}]$
where $\overline{\frak m}_k : B_k\overline C[1] \to \overline C$
is an $R$ module homomorphism.
(Here $\overline C$ is the free $R$ module over the basis $\bold v_{i}$.)
\end{enumerate}
(See Definition 3.2.20 \cite{fooo06} = Definition 7.20 \cite{fooo06pre}.)
\begin{rem}
\begin{enumerate}
\item
In \cite{fooo06} we assume $\frak m_0 =0$ for (unfiltered) $A_{\infty}$
algebra. On the other hand, $\frak m_0 =0$ is not assumed for
filtered $A_{\infty}$ algebra. Filtered $A_{\infty}$ algebra
satisfying $\frak m_0 =0$ is called to be {\it strict}.
\par
\item In this section, to be consistent with the
exposition given in \cite{fooo06}, we use the Novikov ring $\Lambda_{0,nov}$ which
includes the variable $e$. In \cite{fooo06}, the
variable $e$ is used so that the operations $\frak m_k$ become to have
degree one for all $k$ (with respect to the shifted degree.)
But for the applications of this paper,
it is enough to use the $\Z_2$-grading and so
encoding the degree with a formal parameter is not necessary.
Therefore we will use the ring $\Lambda_0$
in other sections, which does not contain $e$.
An advantage of using the ring $\Lambda_0$ is that it is a local ring
while $\Lambda_{0,nov}$ is not. This makes it easier to use some standard results
from commutative algebra in later sections.
We would like to remark that as a $\Z_2$ graded complex
a $\Z$ graded complex over $\Lambda_{0,nov}$ is
equivalent to the complex over $\Lambda_0$.
\end{enumerate}
\end{rem}
Next we explain how one can deform the given filtered $A_\infty$ algebra $(C,\mathfrak m)$ by
an element $b \in F^{\lambda}C[1]^0$ with $\lambda > 0$,
by re-defining the $A_\infty$ operators as
$$
\mathfrak m_k^b(x_1,\cdots, x_k) = \mathfrak m(e^b,x_1, e^b,x_2, e^b,x_3, \cdots,
x_k,e^b).
$$
This defines a new weak $A_\infty$-algebra for arbitrary $b$.
\par
Here we
simplify notations by writing
$
e^b = 1 + b + b\otimes b + \cdots + b \otimes \cdots \otimes b +\cdots.
$
Note that each summand in this infinite sum has degree 0 in $C[1]$.
When the ground ring is $\Lambda_{0,nov}$, the infinite sum
will converge in the non-Archimedean topology since
$b \in F^{\lambda}C[1]^0$ with $\lambda > 0$

\begin{prop} For $A_\infty$ algebra $(C,\mathfrak m_k^b)$,
$\mathfrak m_0^b \equiv 0 \mod \Lambda_{0,nov}\{\bf e\}$ if and only if $b$ satisfies
\begin{equation}\label{eq:MCe}
\sum_{k=0}^\infty\frak m_k(b,\cdots, b)\equiv 0 \mod \Lambda_{0,nov}\{\bf e\}.
\end{equation}
\end{prop}

We call the equation \eqref{eq:MCe} the \emph{$A_\infty$ Maurer-Cartan equation}.

\begin{defn}\label{boundchain}
Let $(C,\frak m)$ be a filtered $A_\infty$ algebra in general and
$BC[1]$ be its bar-complex. An element $b \in F^{\lambda}C[1]^0$ ($\lambda>0$) is called
a \emph{weak bounding cochain} if it satisfies the equation  (\ref{eq:MCe}). If the $b$ satisfies the strict equation
$$
\sum_{k=0}^\infty\frak m_k(b,\cdots, b) = 0
$$
we call it a (strict) \emph{bounding cochain}.
\end{defn}

From now on, we will also call a weak bounding cochain just a bounding cochain
since we will mainly concern weak bounding cochains.
In general a given $A_\infty$ algebra may or may not have a solution
to (\ref{eq:MCe}).

\begin{defn}\label{unobstructed}
A filtered $A_\infty$-algebra is called \emph{weakly unobstructed} if the equation
(\ref{eq:MCe}) has a solution  $b \in F^{\lambda}C[1]^0$
with $\lambda > 0$.
\end{defn}
One can define a notion of gauge equivalence between two
bounding cochains as described in section 4.3 \cite{fooo06} (=section 16 \cite{fooo06pre}).

The way how a filtered $A_\infty$ algebra is attached to a Lagrangian
submanifold $L \subset (M,\omega)$ arises as an $A_\infty$ deformation of
the classical singular cochain complex including the instanton
contributions. In particular, when there is no instanton contribution
as in the case $\pi_2(M,L) = 0$, it will reduce to an $A_\infty$ deformation
of the singular cohomology in the chain level including all possible
higher Massey products.

We now describe the basic $A_\infty$ operators $\mathfrak m_k$
in the context of $A_\infty$ algebra of Lagrangian submanifolds.
For a given compatible almost complex structure $J$, consider the moduli
space $\CM_{k+1}(\beta;L)$ of stable maps of genus zero. It is a
compactification of
$$
\{ (w, (z_0,z_1, \cdots,z_k)) \mid
\overline \partial_J w = 0, \, z_i \in \partial D^2, \, [w] = \beta
\,\,\, \mbox{in}\, \pi_2(M,L) \}/\sim
$$
where $\sim$ is the conformal reparameterization of the disc $D^2$. The
expected dimension of this space is given by
$
n+ \mu(\beta) - 3 + (k+1) = n+\mu(\beta) + k-2
$.
\par
Now given $k$ singular chains
$
[P_1,f_1], \cdots,[P_k,f_k] \in C_*(L)
$
of $L$, we put the cohomological grading
$\mbox{deg} P_i = n - \dim P_i$ and regard
the chain complex $C_*(L)$ as the cochain complex
$C^{\dim L - *}(L)$. We
consider the fiber product
$$
ev_0: \CM_{k+1}(\beta;L) \times_{(ev_1, \cdots, ev_k)}(P_1 \times
\cdots \times P_k) \to L,
$$
where $ev_i([w, (z_0,z_1, \cdots,z_k)]) = w(z_i)$.
\par
A simple calculation shows that we have the expected degree
$$
\mbox{deg}\left[\CM_{k+1}(\beta;L) \times_{(ev_1, \cdots, ev_k)}(P_1
\times \cdots\times P_k),
ev_0\right] = \sum_{j=1}^n(\mbox{deg} P_j -1) + 2- \mu(\beta).
$$
For each given $\beta \in \pi_2(M,L)$ and $k = 0, \cdots$, we
define
\begin{equation}\label{defmbeta}
\mathfrak m_{k,\beta}(P_1, \cdots, P_k)
= \left[\CM_{k+1}(\beta;L) \times_{(ev_1, \cdots, ev_k)}(P_1 \times \cdots \times P_k),
ev_0\right]
\end{equation}
and $\mathfrak m_k = \sum_{\beta \in \pi_2(M,L)} \mathfrak m_{k,\beta}
\cdot T^{\omega(\beta)} e^{\mu(\beta)/2}$.
\par
Now we denote by $C[1]$
the completion of a \emph{suitably chosen countably generated}
(singular) chain
complex with $\Lambda_{0,nov}$ as
its coefficients with respect to the non-Archimedean topology.
(We regard $C[1]$ as a {\it cochain} complex.) Then
by choosing a system of multivalued perturbations
of the right hand side of (\ref{defmbeta})
and a triangulation of its zero sets, the map
$
\mathfrak m_k : B_k(C[1]) \to C[1]
$
is defined, has degree 1 and continuous with respect to non-Archimedean topology.
We extend $\mathfrak m_k$ as a coderivation
$\widehat{\mathfrak m}_k: BC[1] \to BC[1]$
by (\ref{hatmk}).
Finally we take the sum
\begin{equation}\label{hatdfromm}
\widehat d = \sum_{k=0}^\infty \widehat{\frak m}_k : BC[1] \to BC[1].
\end{equation}
A main theorem proven in \cite{fooo00,fooo06} then is the following
coboundary property

\begin{thm}\label{algebra}
{\rm (Theorem 3.5.11 \cite{fooo06} = Theorem 10.11 \cite{fooo06pre})} Let $L$ be an arbitrary compact relatively
spin Lagrangian submanifold of an arbitrary tame symplectic manifold
$(M,\omega)$. The coderivation $\widehat d$ is a continuous map that
satisfies the $A_\infty$ relation $\widehat d \widehat d = 0$.
\end{thm}

The $A_\infty$ algebra associated to $L$ in this way has
the \emph{homotopy unit}, not a \emph{unit}.
In general a filtered $A_\infty$ algebra with homotopy unit canonically induces
another filtered \emph{unital} $A_\infty$-algebra called a \emph{canonical model}
of the given filtered $A_\infty$ algebra. In the geometric context of $A_\infty$
algebra associated to a Lagrangian submanifold $L \subset M$ of
a general symplectic manifold $(M,\omega)$, the canonical model
is defined on the cohomology group $H^*(L;\Lambda_{0,nov})$.
We refer to \cite{fooo-model08} for a quick explanation of
construction by summing over trees of this canonical model.

Once the $A_\infty$ algebra is attached to each Lagrangian submanifold $L$,
we then construct a filtered \emph{$A_\infty$ bimodule} $C(L,L')$ for the transversal pair
of Lagrangian submanifolds $L$ and $L'$.
Here $C(L,L')$ is the free $\Lambda_{0,nov}$ module such that the set of its
basis is identified with $L \cap L'$.
The filtered $A_\infty$ bimodule structure is by definition is a family of operators
$$
\mathfrak n_{k_1,k_2}: B_{k_1}(C(L)[1])\,\,
\widehat{\otimes}_{\Lambda_{0,nov}} \,\, C(L,L')
\,\,\widehat{\otimes}_{\Lambda_{0,nov}} \,\, B_{k_1}(C(L')[1]) \to C(L,L')
$$
for $k_1,k_2\ge 0$.
(Here
$\widehat{\otimes}_{\Lambda_{0,nov}}$ is the completion of
algebraic tensor product.)
Let us briefly describe the definition of $\mathfrak n_{k_1,k_2}$.
A typical element of the tensor product
$B_{k_1}(C(L)[1])\,\,\widehat{\otimes}_{\Lambda_{0,nov}} \,\, C(L,L')
\,\,\widehat{\otimes}_{\Lambda_{0,nov}} \,\, B_{k_2}(C(L')[1])$
has the form
$$
P_{1,1} \otimes \cdots, \otimes P_{1,k_1} \otimes \langle p \rangle \otimes
P_{2,1} \otimes \cdots \otimes P_{2,k_2}
$$
with $p \in L \cap L'$. Then the image $\mathfrak n_{k_1,k_2}$ thereof is given by
$$
\sum_{q, B}T^{\omega(B)}e^{\mu(B)/2}\# \left(\CM(p,q;B;P_{1,1},\cdots,P_{1,k_1};
P_{2,1},\cdots,P_{2,k_2})\right)\langle q \rangle.
$$
Here $B$ denotes homotopy class of Floer trajectories connecting $p$ and $q$,
the summation is taken over all $(q,B)$ with
$$
\text{vir.dim } \CM(p,q;B;P_{1,1},\cdots,P_{1,k_1};
P_{2,1},\cdots,P_{2,k_2}) = 0,
$$
and $\# \left(\CM(p,q;B;P_{1,1},\cdots,P_{1,k_1};P_{2,1},\cdots,P_{2,k_2})\right)$ is
the `number' of elements in
the `zero' dimensional moduli space $\CM(p,q;B;P_{1,1},\cdots,P_{1,k_1};
P_{2,1},\cdots,P_{2,k_2})$. Here the moduli space $\CM(p,q;B;P_{1,1},\cdots,P_{1,k_1};
P_{2,1},\cdots,P_{2,k_2})$ is the Floer moduli space
$
\CM(p,q;B)
$
cut-down by intersecting with the given chains $P_{1,i} \subset L$
and $P_{2,j} \subset L'$.

\begin{thm} {\rm (Theorem 3.7.21 \cite{fooo06} = Theorem 12.21 \cite{fooo06pre})} Let $(L,L')$ be an arbitrary relatively spin pair
of compact Lagrangian submanifolds.
Then the family $\{\mathfrak n_{k_1,k_2}\}$ define
a left $(C(L),\mathfrak m)$ and right $(C(L'),\mathfrak m')$ filtered
$A_\infty$-bimodule structure on $C(L,L')$.
\end{thm}

What this theorem means is explained below as Proposition \ref{bimodule}.
\par
Let $B(C(L)[1])\,\widehat{\otimes}_{\Lambda_{0,nov}} \, C(L,L')
\,\widehat{\otimes}_{\Lambda_{0,nov}} \, B(C(L')[1])$
be the completion of the direct sum of
$
B_{k_1}(C(L)[1])\,\widehat{\otimes}_{\Lambda_{0,nov}} \, C(L,L')
\,\widehat{\otimes}_{\Lambda_{0,nov}} \, B_{k_2}(C(L')[1])$
over $k_1\ge 0$, $k_2\ge 0$.
We define
the boundary operator $\widehat d$ on it by using the maps $\mathfrak n_{k_1,k_2}$
and $\frak m_k$, $\frak m'_k$ as follows :
$$\aligned
&\widehat d((x_1 \otimes \cdots \otimes x_n) \otimes x \otimes
(x'_1 \otimes \cdots \otimes x'_m)) \\
&=
\widehat d(x_1 \otimes \cdots \otimes x_n)
\otimes  x \otimes
(x'_1 \otimes \cdots \otimes x'_m) \\
&\quad + (-1)^{\deg x_1 + \cdots + \deg x_n + \deg x + n+1}
(x_1 \otimes \cdots \otimes x_n)
\otimes  x \otimes
\widehat d(x'_1 \otimes \cdots \otimes x'_m) \\
&\quad + \sum_{k_1\le n}\sum_{k_2\le m}
(-1)^{\deg x_1 + \cdots + \deg x_{n-k_1} + n-k_1} (x_1 \otimes \cdots  \otimes x_{n-k_1}) \\
&\hskip2.3cm \otimes \frak n_{k_1,k_2}((x_{n-k_1+1} \otimes
\cdots \otimes x_n) \otimes x \otimes (x'_1 \otimes \cdots \otimes x'_{k_2}))
\\
&\hskip2.3cm \otimes
(x'_{k_2+1} \otimes \cdots \otimes x'_m).
\endaligned$$
Here $\widehat d$ in the second and the third lines are
induced from $\frak m$ and $\frak m'$ by Formula (\ref{hatdfromm}) respectively.
\begin{prop}\label{bimodule} The map $\widehat d$ satisfies
$\widehat d \, \widehat d = 0$.
\end{prop}

The $A_\infty$ bimodule structure,
which define a boundary operator on the bar complex, induces an
operator $\delta = \frak n_{0,0}$ on a much smaller,
ordinary free
$\Lambda_{0,nov}$ module $C(L,L')$ generated by the intersections
$L \cap L'$. However the boundary property of this Floer's `boundary' map $\delta$
again meets obstruction coming from the obstructions cycles of either $L$, $L'$ or of both.
We need to deform $\delta$ using suitable bounding cochains of $L, \, L'$.
\par
In the case where both $L, \, L'$ are weakly unobstructed, we can carry out this
deformation of $\mathfrak n$ using weak bounding chains $b$
and $b'$ of fibered $A_{\infty}$ algebras associated to $L$ and $L'$ respectively, in
a way similar to $\frak m^b$. Namely we define
$\delta_{b,b'} : C(L,L') \to C(L,L')$ by
$$
\delta_{b,b'}(x) =
\sum_{k_1,k_2} \frak n_{k_1,k_2} (b^{\otimes k_1} \otimes
x \otimes b^{\prime \otimes k_2}) =  \mathfrak{\widehat n}(e^{b},x,e^{b'}).
$$
We can generalize the story to the case where $L$ has clean intersection
with $L'$, especially to the case $L=L'$.
In the case $L=L'$ we have
$\frak n_{k_1,k_2} = \frak m_{k_1+k_2+1}$. So
in this case, we have
$
\delta_{b,b'}(x) = \frak m(e^b,x,e^{b'}).
$
\par
In general $\delta_{b,b'}$ does not satisfy
the equation $\delta_{b,b'}\delta_{b,b'} = 0$.
It turns out that there is an elegant condition for
$\delta_{b,b'}\delta_{b,b'} = 0$ to hold
in terms of the \emph{potential function} introduced in \cite{fooo06},
which we explain in the next section.
In the case $\delta_{b,b'}\delta_{b,b'} = 0$ we define
Floer cohomology by
$$
HF((L,b),(L',b');\Lambda_{0,nov}) = \mbox{Ker}\,\delta_{b,b'}/\mbox{Im}
\,\delta_{b,b'}.
$$
Let $\Lambda_{nov}$ be the field of fraction of $\Lambda_{0,nov}$.
We define $HF((L,b),(L',b');\Lambda_{nov})$ by extending the
coefficient ring $\Lambda_{0,nov}$ to $\Lambda_{nov}$. Then $HF((L,b),(L',b');\Lambda_{nov})$ is invariant under the
Hamiltonian isotopies of $L$ and $L'$. Therefore we can use it
to obtain the following result about non-displacement of
Lagrangian submanifolds.
\begin{thm}{\rm (Theorem G \cite{fooo06} = Theorem G \cite{fooo06pre})}\label{displace1}
Assume $\delta_{b,b'}\delta_{b,b'} = 0$. Let
$\psi : X \to X$ be a Hamiltonian diffeomorphism
such that $\psi(L)$ is transversal to $L'$. Then we have
$$
\# (\psi(L) \cap L')
\ge \text{\rm rank}_{\Lambda_{nov}}HF((L,b),(L',b');\Lambda_{nov}).
$$
\end{thm}
The Floer cohomology $HF((L,b),(L',b');\Lambda_{0,nov})$
with coefficient $\Lambda_{0,nov}$ is not
invariant under the Hamiltonian isotopy. We however can prove the
following Theorem \ref{displace2}. There exists an integer $a$ and positive numbers
$\lambda_i$ ($i=1,\cdots,b$) such that
$$
HF((L,b),(L',b');\Lambda_{0,nov})
= \Lambda_{0,nov}^{\oplus a} \oplus \bigoplus_{i=1}^b
(\Lambda_{0,nov}/T^{\lambda_i}\Lambda_{0,nov}).
$$
(See Theorem 6.1.20 \cite{fooo06} (= Theorem 24.20 \cite{fooo06pre}).)
Let $\psi : X \to X$ be a Hamiltonian diffeomorphism and
$\Vert \psi\Vert$ is its Hofer distance (see \cite{hofer}) from the
identity map.
\begin{thm}{\rm (Theorem J \cite{fooo06} = Theorem J \cite{fooo06pre})}\label{displace2}
If $\psi(L)$ is transversal to $L'$, we have an inequality
$$
\# (\psi(L) \cap L')
\ge
a + 2\,\#\{ i \mid \lambda_i \ge \Vert \psi\Vert\}.
$$
\end{thm}
In later sections, we will apply Theorems \ref{displace1},\ref{displace2}
to study non-displacement of Lagrangian fibers
of toric manifolds.

\section{Potential function}
\label{sec:statement}

The $A_\infty$ structure defined on a countably generated
chain complex $C(L,\Lambda_0)$ itself explained in the previous section
is not suitable for explicit calculations as in our study of toric
manifolds. For this computational purpose, we work with the filtered
$A_\infty$ structure on the canonical model defined on $H(L;\Lambda_0)$
which has a \emph{finite} rank over $\Lambda_0$. Furthermore
this has a strict unit ${\bf e}$ given by the dual of the
fundamental class $PD([L])$. Recall $C(L,\Lambda_0)$ itself
has only a homotopy-unit.

An element $b \in H^1(L;\Lambda_{+})$ is called a \emph{weak bounding cochain}
if it satisfies the $A_\infty$ Maurer-Cartan equation
\be\label{eq:MC}
\sum_{k=0}^\infty \mathfrak m_k(b,\cdots, b) \equiv 0 \mod PD([L])
\ee
where $\{\mathfrak m_k\}_{k =0}^\infty$ is the $A_\infty$ operators
associated to $L$. $[L] \in H_n(L)$ is the fundamental class, and
$PD([L]) \in H^0(L)$ is its Poincar\'e dual.
We denote by $\widehat{\CM}_{\text{weak}}(L)$
the set of weak bounding cochains of $L$.
We say $L$ is {\it weakly unobstructed}
if $\widehat{\CM}_{\text{weak}}(L) \neq \emptyset$.
The moduli space $\CM_{\text{weak}}(L)$ is then defined to be the
quotient space of $\widehat{\CM}_{\text{weak}}(L)$ by suitable gauge
equivalence. (See chapter 3 and 4 \cite{fooo06} for more explanations.)

\begin{lem}[Lemma 3.6.32 \cite{fooo06} = Lemma 11.32 \cite{fooo06pre}]
If $b \in \widehat{\CM}_{\text{\rm weak}}(L)$ then
$\delta_{b,b}\circ
\delta_{b,b} = 0$, where $\delta_{b,b}$ is the deformed Floer
operator defined by
$$
\delta_{b,b}(x) = \mathfrak m_1^b(x) = :\sum_{k,\ell \ge 0} \mathfrak
m_{k+\ell +1}(b^{\otimes k},x,b^{\otimes \ell}).
$$
\end{lem}

For $b \in \widehat{\CM}_{\text{weak}}(L)$, we define
$$
HF((L;b),(L;b)) = \frac{\operatorname{Ker}(\delta_{b,b}:C \to C)}
{\operatorname{Im}(\delta_{b,b}:C \to C)},
$$
where $C$ is an appropriate subcomplex of the singular chain complex of $L$.
When $L$ is weakly unobstructed i.e., $\widehat{\CM}_{\text{weak}}(C) \neq \emptyset$,
we define a function
$$
\mathfrak {PO} : \widehat{\CM}_{\text{weak}}(C) \to \Lambda_+
$$
by the equation
$$
\mathfrak {m}(e^b)=\mathfrak {PO}(b)\cdot PD([L]).
$$
This is the \emph{potential function} introduced in \cite{fooo06}.

\begin{thm}[Proposition 3.7.17 \cite{fooo06} = Proposition 12.17 \cite{fooo06pre}] For each $b \in \widehat{\CM}_{\text{\rm weak}}(L)$ and
$b' \in \widehat{\CM}_{weak}(L')$, the map
$\delta_{b,b'}$ defines a continuous map
$
\delta_{b,b'} : CF(L,L') \to CF(L,L')
$
that satisfies $\delta_{b,b'}\delta_{b,b'} = 0$, provided
\be\label{eq:POb1=POb2}
\mathfrak {PO}(b) = \mathfrak {PO}(b').
\ee
\end{thm}

Therefore for each pair $(b,b')$ of $b \in \widehat{\CM}_{\text{\rm weak}}(L)$ and
$b' \in \widehat{\CM}_{\text{\rm weak}}(L')$ that satisfy \eqref{eq:POb1=POb2},
we define the \emph{$(b,b')$-Floer cohomology} of the pair $(L,L')$ by
$$
HF((L,b), (L',b');\Lambda_{nov}) = \frac{\mbox{Ker} \delta_{b,b'}}
{\mbox{\rm Im } \delta_{b,b'}}.
$$
\medskip

Now in the rest of this section we state the main results concerning the
detailed structure of the potential function
for the case of Lagrangian fibers of toric manifolds.

For the later analysis of examples, we recall from \cite{fooo00,fooo06}
that $\mathfrak m_k$ is further decomposed into
$$
\mathfrak m_k = \sum_{\beta \in \pi_2(M,L)} \mathfrak m_{k,\beta} \otimes T^{\omega(\beta)}e^{\mu(\beta)/2}.
$$
Here $\mu$ is the Maslov index.
\par
Firstly we remove the grading parameter $e$ from
the ground ring. Secondly to eliminate
many appearance of $2\pi$ in front of the affine function $\ell_i$
in the exponents of the parameter $T$ later
in this paper, we redefine $T$ as $T^{2\pi}$. Under this arrangement, we get the
formal power series expansion
\be\label{eq:newmk}
\mathfrak m_k = \sum_{\beta \in \pi_2(M,L)} \mathfrak m_{k,\beta} \otimes T^{\omega(\beta)/2\pi}
\ee
which we will use throughout the paper.
\par
Now we restrict to the case of toric manifold. Let $X = X_\Sigma$ be
associated a complete regular fan $\Sigma$
(in other words $\Sigma$ is the normal fan of $X$), and $\pi: X \to \mathfrak t^*$ be the moment map of
the action of the torus $T^n \cong T^m/K$. We make the identifications
$$
\mathfrak t = Lie(T^n) \cong N_\R^n \cong \R^n, \, \mathfrak t^* \cong M_\R
\cong (\R^n)^*.
$$
We will exclusively use $N_\R$ and $M_\R$ to be consistent with the
standard notations in toric geometry instead of $\mathfrak t$ (or
$\R^n$) and $\mathfrak t^*$ (or $(\R^n)^*$) as much as possible.

Denote the image of $\pi: X \to M_\R$ by $P \subset M_\R$ which is
the moment polytope of the $T^n$ action on $X$.

We will prove the following in section \ref{sec:calcpot}.

\begin{prop}\label{unobstruct} For any $u \in \text{\rm Int} P$, the fiber
$L(u)$ is weakly unobstructed.
Moreover we have the canonical inclusion
$$
H^1(L(u);\Lambda_{+}) \hookrightarrow \CM_{\text{\rm weak}}(L(u)).
$$
\end{prop}

Choose an integral basis $\text{\bf e}^*_i$ of $N$ and $\text{\bf
e}_i$ be its dual basis on $M$. With this choice made, we identify
$M_\R$ with $\R^n$ as long as its meaning is obvious from the
context. Identifying $H_1(T^n;\Z)$ with $N \cong \Z^n$ via $T^n =
\R^n/\Z^n$, we regard $\text{\bf e}_i$ as a basis of $H^1(L(u);\Z)$.
The following immediately follows from definition.

\begin{lem}\label{37.90}
We write $\pi = (\pi_1,\cdots,\pi_n) : X \to M_\R$ using the
coordinate of $M_\R$ associated to the basis $\text{\bf e}_i$. Let
$S^1_i \subset T^n$ be the subgroup whose orbit represents
$\text{\bf e}^*_i \in H_1(T^n;\Z)$. Then $\pi_i$ is proportional to
the moment map of $S^1_i$ action on $X$.
\end{lem}

Let
$$
b = \sum x_i \text{\bf e}_i \in H^1(L(u);\Lambda_{+}) \subset
{\CM}_{\text{\rm weak}}(L(u)).
$$
We study the potential function
$$
\mathfrak{PO}: H^1(L(u);\Lambda_{+}) \to \Lambda_{+}.
$$
Once a choice of the family of bases $\{\text{\bf e}_i\}$ on
$H^1(L(u);\Z)$ for $u \in \text{\rm Int}\,P$ is made as above
starting from a basis on $N$ , then we can regard this function as a
function of $(x_1,\cdots,x_n) \in (\Lambda_{+})^n$ and
$(u_1,\cdots,u_n) \in P \subset M_\R$. We denote its value by
$\mathfrak{PO}(x;u) = \mathfrak{PO}(x_1, \cdots,x_n;u_1,
\cdots,u_n)$. We put
$$
y_i = e^{x_i} = \sum_{k=0}^{\infty}
\frac{x_i^k}{k!} \in \Lambda_{0}.
$$
Let
$$
\partial P = \bigcup_{i=1}^m \partial_iP
$$
be the decomposition of the boundary of the moment polytope into its
faces of codimension one. ($\partial_iP$ is a polygon in an $n-1$
dimensional affine subspace of $M_\R$.)

Let $\ell_i$ be the affine functions
$$
\ell_i(u) = \langle u, v_i \rangle - \lambda_i \, \mbox{ for }\,
i = 1, \cdots, m
$$
appearing in Theorem \ref{gullemin}. Then the followings hold
from construction :
\begin{enumerate}
\item $\ell_i \equiv 0$ on $\partial_iP$.
\item
$ P = \{u \in M_\R \mid \ell_i(u) \ge 0, i = 1,\cdots,m\}. $
\item The coordinates of the vectors $v_i=(v_{i,1}, \cdots, v_{i,n})$ satisfy
\begin{equation}\label{eq:vij}
v_{i,j} = \frac{\partial \ell_i}{\partial u_j}
\end{equation}
and are all integers.
\end{enumerate}

\begin{thm}\label{potential}
Let $L(u) \subset X$ be as in Theorem $\ref{toric-intersect}$
and $\ell_i$ be as above.
Suppose $X$ is Fano. Then we can take the canonical model
of $A_\infty$ structure of $L(u)$ over $u \in \text{\rm Int}\,P$
so that the potential function restricted to
$$
\bigcup_{u \in \text{\rm Int }P} H^1(L(u);\Lambda_{+}) \cong
(\Lambda_+)^n \times \text{\rm Int}\,P
$$
has the form
\bea\label{eq:POx;u} \mathfrak{PO}(x;u)
& = & \sum_{i=1}^m y_1^{v_{i,1}}\cdots y_n^{v_{i,n}}T^{\ell_i(u)} \label{eq:POy;u}\\
& = & \sum_{i=1}^m e^{\langle v_i,x \rangle} T^{\ell_i(u)}
\eea
where
$(x;u) = (x_1,\cdots,x_n;u_1,\cdots,u_n)$ and $v_{i,j}$ is as in
$(\ref{eq:vij})$.
\end{thm}

Theorem \ref{potential} is a minor improvement of a result from
\cite{cho-oh} (see (15.1) of \cite{cho-oh}) and \cite{cho07} : The
case considered in \cite{cho-oh} corresponds to the case where $y_i
\in U(1) \subset \{ z \in \C \mid \vert z\vert = 1\}$ and the case
in \cite{cho07} corresponds to the one where $y_i \in \C \setminus
\{0\}$. The difference of Theorem \ref{potential} from the ones
thereof is that $y_i$ is allowed to contain $T$, the formal
parameter of the universal Novikov ring encoding the energy.

For the non-Fano case, we prove the following slightly weaker
statement. The proof will be given in section \ref{sec:calcpot}.

\begin{thm}\label{weakpotential}
Let $X$ be an arbitrary toric manifold and $L(u)$ be as above.
Then there exist $c_j \in \Q$, $e_j^i \in \Z_{\ge 0}$ and $\rho_j > 0$,
such that $\sum_{i=1}^m e_j^i > 0$ and
\bea\label{eq:weakPO}
\mathfrak{PO}(x_1,\cdots,x_n;u_1,\cdots,u_n)
& - & \sum_{i=1}^m y_1^{v_{i,1}}\cdots
y_n^{v_{i,n}}T^{\ell_i(u)} \nonumber \\
&= &
\sum_{j=1} c_j
y_1^{v'_{j,1}}\cdots y_n^{v'_{j,n}}T^{\ell'_j(u)+\rho_j},
\eea
where
$$
v'_{j,k} = \sum_{i=1}^m e_j^iv_{i,k}, \quad \ell'_j = \sum_{i=1}^m e_j^i \ell_i.
$$
If there are infinitely many non-zero $c_j$'s, we have
$$
\lim_{j\to\infty} \ell'_j(u) + \rho_{j} = \infty.
$$
Moreover $\rho_j = [\omega]\cap \alpha_j$ for some $\alpha_j \in \pi_2(X)$
with nonpositive first Chern number $c_1(X) \cap [\alpha_j]$.
\end{thm}

We note that although $\mathfrak{PO}$ is defined
originally on $(\Lambda_+^{\C})^n \times P$, Theorems \ref{potential} and \ref{weakpotential}
imply that $\mathfrak{PO}$ extends to a function on
$(\Lambda_0^{\C})^n \times P$
in general, and
to one on $(\Lambda_0^\C)^n \times M_\R$ for the Fano case. Furthermore these theorems also imply the periodicity
of $\mathfrak{PO}$ in $x_i$'s,
\begin{equation}\label{eq:periodicx}
\mathfrak{PO}(x_1,\cdots,x_i+2\pi\sqrt{-1},\cdots,x_n;u)
= \mathfrak{PO}(x_1,\cdots,x_n;u).
\end{equation}
\par
We write
\begin{equation}\label{def:PO0}
\mathfrak{PO}_0 = \sum_{i=1}^m y_1^{v_{i,1}}\cdots
y_n^{v_{i,n}}T^{\ell_i(u)}
\end{equation}
to distinguish it from $\mathfrak{PO}$. We call $\mathfrak{PO}_0$
the {\it leading order potential function}.
\par
We will concern the existence of the bounding cochain $\mathfrak x$ for which the Floer cohomology
$HF((L(u),\mathfrak x),(L(u),\mathfrak x))$ is not zero, and prove that critical points of
the $\mathfrak{PO}^u$ (as a function of $y_1,\cdots,y_n$)
have this property. (Theorem \ref{homologynonzero}.)

This leads us to study the equation
\begin{equation}\label{formula:critical}
\frac{\partial \mathfrak{PO}^u}{\partial y_k}(\mathfrak y_1,\cdots,\mathfrak y_n)
= 0,\qquad k=1,\cdots, n,
\end{equation}
where $\mathfrak y_i \in \Lambda_{0}\setminus \Lambda_+$.
\par
We regard $\mathfrak{PO}^u$ as either a function of $x_i$ or of $y_i$.
Since the variable ($x_i$ or $y_i$) is clear from situation, we do not mention it occasionally.

\begin{prop}\label{existcrit}
We assume that the coordinates of the vertices of $P$ are rational.
Then
there exists $u_0 \in \text{\rm Int} P \cap \Q^n$ such that for
each $\CN$ there exists $\mathfrak y_1,\cdots,\mathfrak y_n \in \Lambda_{0} \setminus \Lambda_+$ satisfying :
\begin{equation}\label{formula:criticalweak}
\frac{\partial \mathfrak{PO}^{u_0}}{\partial y_k}(\mathfrak y_1,\cdots,\mathfrak y_n)
\equiv 0, \mod T^{\CN} \qquad k=1,\cdots, n.
\end{equation}
\par
Moreover there exists $\mathfrak y'_1,\cdots,\mathfrak y'_n \in \Lambda_{0} \setminus \Lambda_+$
such that
\begin{equation}\label{formula:critical0}
\frac{\partial \mathfrak{PO}_0^{u_0}}{\partial y_k}(\mathfrak y'_1,\cdots,\mathfrak y'_n)
= 0,\qquad k=1,\cdots, n.
\end{equation}
\end{prop}
We will prove Proposition \ref{existcrit} in section \ref{sec:var} .
\begin{rem}\label{rationalrem}
\begin{enumerate}
\item $u_0$ is independent of $\CN$. But $\mathfrak y_i$ may depend on $\CN$.
(We believe it does not depend on $\CN$, but are unable to prove it
at the time of writing this paper.)
\par
\item If $[\omega] \in H^2(X;\R)$ is contained in $H^2(X;\Q)$ then we may
choose $P$ so that its vertices are rational.
\par
\item We believe that rationality of the vertices of $P$ is superfluous.
We also believe there exists not only a solution of (\ref{formula:criticalweak}) or of
(\ref{formula:critical0}) but also of
(\ref{formula:critical}). However then the proof seems to become more cumbersome.
Since we can reduce the general
case to the rational case by approximation in most of the applications,
we will be content to prove the above weaker statement in this paper.
\end{enumerate}
\end{rem}
We put
$$
\mathfrak x_i = \log \mathfrak y_i \in \Lambda_{0}
$$
and write
\begin{equation}\label{frakxdef}
\mathfrak x = \sum_i \mathfrak x_i \text{\bf e}_i \in H^1(L(u_0);\Lambda_{0}).
\end{equation}
Since $\mathfrak y_i \in \Lambda_0 \setminus \Lambda_+$, $\log \mathfrak y_i$ is
well-defined (by using non-Archimedean topology on $\Lambda_{0}$)
and is contained in $\Lambda_{0}$.
\par
We remark that $\frak x_i$ is determined from $\frak y_i$ up to
addtion by an element of $2\pi\sqrt{-1}\Z$. It will follow from
(\ref{eq:periodicx}) that changing $\frak x_i$ by an element of
$2\pi\sqrt{-1}\Z$ does not change corresponding Floer cohomology. So
we take for example $\text{\rm Im}\,\frak x_i \in [0,2\pi)$. (See
also Definition \ref{def:MlagAHhom} (2).)
\par
Let $\mathfrak y_{i,0} \in \C \setminus \{0\}$ be
the zero-order term of $\mathfrak y_{i}$ i.e., the complex number such that
$$
\mathfrak y_{i} - \mathfrak y_{i,0} \equiv 0 \mod \Lambda_{+}^{\C}.
$$
If we put an additional assumption that
$\mathfrak y_{i,0} = 1$ for $i=1,\cdots,n$, then $\frak x$ lies in
$$
H^1(L(u_0);\Lambda_{+}) \subset H^1(L(u_0);\Lambda_0).
$$
Therefore Proposition \ref{unobstruct} implies the Floer cohomology
$
HF((L(u_0),\frak x),(L(u_0),\frak x);\Lambda_{0})
$
is defined. Then (\ref{formula:criticalweak}) combined with the argument
from \cite{cho-oh} (see section \ref{sec:floerhom}) imply
\begin{equation}\label{Floernonzero}
HF((L(u_0),\mathfrak x),(L(u_0),\mathfrak x);\Lambda_{0}^{\C}/(T^{\CN}))
\cong H(T^n;\Lambda_{0}^{\C}/(T^{\CN})).
\end{equation}
\par
We now consider the case when $\mathfrak y_{i,0} \ne 1$ for some $i$.
In this case, we follow the idea of Cho \cite{cho07} of
twisting the Floer cohomology of $L(u)$ by {\it non-unitary}
flat line bundle and proceed as follows :
\par
We define $\rho : H_1(L(u);\Z) \to \C\setminus \{0\}$ by
\begin{equation}\label{repdef}
\rho(\text{\bf e}_i^*) = \mathfrak y_{i,0}.
\end{equation}
Let $\mathfrak L_{\rho}$ be the flat complex line bundle on $L(u)$
whose holonomy representation is $\rho$.
We use $\rho$ to twist the operator $\mathfrak m_k$ in the same way as
\cite{fukaya:family}, \cite{cho07} to obtain a filtered
$A_{\infty}$ algebra, which we write $(H(L(u);\Lambda_0),\mathfrak m^{{\rho}})$.
It is weakly unobstructed and
${\mathcal M}_{\text{\rm weak}}(H(L(u);\Lambda_0),\mathfrak m^{{\rho}}) \supseteq
H^1(L(u);\Lambda_{+})$. (See section \ref{sec:flat}.)
\par
We deform the filtered $A_{\infty}$
structure $\mathfrak m^{{\rho}}$ to ${\mathfrak m^{{\rho},b}}$ using
$b \in H^1(L(u);\Lambda_{+})$ for which
$\mathfrak m^{\rho,b}_1\mathfrak m^{\rho,b}_1 = 0$ holds. Denote by
$HF((L(u_0),{\rho},b),(L(u_0),{\rho},b);\Lambda_{0}^{\C})$ the cohomology of
$\mathfrak m^{{\rho},b}_1$.
We denote the potential function of $(H(L(u);\Lambda_0),\mathfrak m^{{\rho}})$
by
$$
\mathfrak{PO}^u_{\rho} : H^1(L(u);\Lambda_{+}) \to \Lambda_{+}
$$
which is defined in the same way as $\mathfrak{PO}^u$ by using $\mathfrak m^\rho$
instead of $\mathfrak m$.
\par
Let $\frak x$ be as in (\ref{frakxdef}) and put
\begin{equation}\label{bdef}
\frak x_{i,0} = \log \frak y_{i,0},
\qquad
b = \sum (\frak x_i - \frak x_{i,0})\text{\bf e}_i \in H^1(L(u);\Lambda_{+}).
\end{equation}
From the way how the definition goes, we can easily prove
\begin{lem}\label{37.104}
$
\mathfrak{PO}^u_{\rho}(b) = \mathfrak{PO}^u(\frak x ).
$
\end{lem}
We note from the remark right after Theorem \ref{weakpotential} that $\mathfrak{PO}^u$
has been extended to a function on $(\Lambda_0^\C)^n$ and hence
the right hand side of the identity in this lemma has a well-defined meaning.
Lemma \ref{37.104} will be proved in section \ref{sec:floerhom}.
\par
Now we have :
\begin{thm}\label{homologynonzero}
Let $\mathfrak x_{i}$, $\mathfrak y_i = e^{\mathfrak x_i}$, and $\rho$ satisfy $(\ref{formula:critical})$, $(\ref{frakxdef})$ and
$(\ref{repdef})$. Let $\frak x_{i,0}$ and $b$ be as in $(\ref{bdef})$.
Then we have
\begin{equation}\label{eq:homoiso}
HF((L(u_0),{\rho},b),(L(u_0),{\rho},b),\Lambda_{0}^{\C})
\cong H(T^n;\Lambda_{0}^{\C}).
\end{equation}
If $(\ref{formula:criticalweak})$, $(\ref{frakxdef})$, $(\ref{repdef})$
and $(\ref{bdef})$ are satisfied instead then we have
\begin{equation}\label{eq:homoisomodN}
HF((L(u_0),{\rho},b),(L(u_0),{\rho},b),\Lambda_{0}^{\C}/(T^{\CN}))
\cong H(T^n;\Lambda_{0}^{\C}/(T^{\CN})).
\end{equation}
\end{thm}

Theorem \ref{homologynonzero} is proved in section \ref{sec:floerhom}. Using this we prove
Theorem \ref{toric-intersect} in section \ref{sec:floerhom}.
More precisely, we will also discuss the following two points in that section :
\begin{enumerate}
\item We need to study the case where $\omega$ is not necessarily
rational.
\item We only have (\ref{eq:homoisomodN}) instead of (\ref{eq:homoiso}).
\end{enumerate}

Next we give
\begin{defn}\label{def:balanced}
Let $(X,\omega)$ be a smooth compact toric manifold, $P$ be its
moment polytope. We say that a fiber $L(u_0)$ at $u_0 \in P$ is {\it
balanced} if there exists a sequence $\omega_i$, $u_i$ such that
\begin{enumerate}
\item $\omega_i$ is a $T^n$ invariant K\"ahler structure on $X$ such that
$\lim_{i\to\infty} \omega_i = \omega$.
\par
\item $u_i$ is in the interior of the moment polytope $P_i$ of $(X,\omega_i)$. We
make an appropriate choice of moment polytope $P_i$ so that they
converge to $P$. Then $\lim_{i\to\infty} u_i = u_0$.
\item For each ${\CN}$, there exist a sufficiently large $i$ and
$\mathfrak x_{i,{\CN}} \in H^1(L(u_i);\Lambda^{\C}_0)$ such that
$$
HF((L(u_i),\mathfrak x_{i,{\CN}}),(L(u_i),\mathfrak x_{i,{\CN}});\Lambda^{\C}/(T^{{\CN}})) \cong
H(T^n;\C) \otimes \Lambda^{\C}/(T^{{\CN}}).
$$
\end{enumerate}
\par
We say that $L(u_0)$ is {\it strongly balanced} if there exists
$\mathfrak x
\in H^1(L;\Lambda^{\C}_0)$ such that
$HF((L(u_0),\mathfrak x),(L(u_0),\mathfrak x);\Lambda^{\C}_0) \cong H(T^n;\Q) \otimes
\Lambda^{\C}_0$.
\end{defn}
Obviously `strongly balanced' implies `balanced'.
The converse is not true in general. See Example \ref{counterexamples}.
We also refer readers to Remark \ref{balancerem2} for other characterizations of being
balanced (or strongly balanced).

Theorem \ref{homologynonzero} implies that $L(u_0)$ in Proposition
\ref{existcrit} is balanced. (Proposition \ref{prof:bal}.)
We will prove the next following intersection result in section \ref{sec:floerhom}.
Theorem \ref{toric-intersect} will then be a consequence of
Propositions \ref{existcrit} and \ref{prof:bal2}.

\begin{prop}\label{prof:bal2}
If $L(u_0)$ is a balanced Lagrangian fiber then the following
holds for any Hamiltonian diffeomorphism $\psi : X \to X$.
\begin{equation}
\psi(L(u_0)) \cap L(u_0) \ne \emptyset.
\label{eq:Lu_02}\end{equation}
If $\psi(L(u_0))$ is transversal to $L(u_0)$ in addition, then
\begin{equation}\#(\psi(L(u_0)) \cap L(u_0)) \ge 2^n.\label{eq:Lu_12}\end{equation}
\end{prop}

Denoting $\mathfrak x = b + \sum \mathfrak x_{i,0} \text{\bf e}_i$,
we sometime write $HF((L(u_0),\mathfrak x),(L(u_0),\mathfrak x),\Lambda_{0})$
for $HF((L(u_0),{\rho},b),(L(u_0),{\rho},b),\Lambda_{0})$ from now on.
We also define
$$
\mathcal M_{\text{\rm weak}}(L(u);\Lambda_0) : =
\{(\rho,b) \mid \rho : \pi_1L(u) \to \C \setminus \{0\},
b \in \mathcal M_{\text{\rm weak}}(H(L(u)),\frak m^{\rho})\}.
$$
Namely it is the set of pairs $(\rho,b)$ where $\rho$ is a
holonomy of a flat $\C$ bundle over $L(u)$ and
$b \in H(L(u);\Lambda_+)$ is a weak bounding cochain of the filtered $A_{\infty}$
algebra associated to $L(u)$ and twisted by $\rho$.
With this definition of $\mathcal M_{\text{\rm weak}}(L(u);\Lambda_0)$, we have :
$$
H^1(L(u);\Lambda_0) \subseteq \mathcal M_{\text{\rm weak}}(L(u);\Lambda_0).
$$

\vskip 1.00cm

\section{Examples}
\label{sec:example}

In this section, we discuss various examples of toric manifolds
which illustrate the results of section \ref{sec:statement}.

\begin{exm}
Consider $X=S^2$ with standard symplectic form with area $2\pi$. The
moment polytope of the standard $S^1$-action by rotations along an
axis becomes $P = [0,1]$ after a suitable translation. We have
$\ell_1(u) = u$, $\ell_2(u) = 1-u$ and
$$
\mathfrak{PO}(x;u) = e^x T^u + e^{-x}T^{1-u}
= y T^u + y^{-1}T^{1-u}.
$$
The zero of
$$
\frac{\partial\mathfrak {PO}^u}{\partial y} = T^u - y^{-2} T^{1-u}
$$
is
$
\mathfrak y = \pm T^{(1-2u)/2}.
$
If $u\ne 1/2$ then
$$
\log \mathfrak y = \frac{1-2u}{2} \log(\pm T)
$$
is not an element of universal Novikov ring. In particular, there is
no critical point in $\Lambda^{\C}_{0} \setminus \Lambda^{\C}_{+}$.
\par
If $u=1/2$ then $\mathfrak y = \pm 1$. The case
$\mathfrak y= 1$ corresponds to $\mathfrak x=0$.
We have
$$
HF((L(1/2),0),(L(1/2),0);\Lambda_{0})
\cong H(S^1;\Lambda^{\C}_{0}).
$$
The other case $\mathfrak y=-1$, corresponds to
a nontrivial flat bundle on $S^1$.
\end{exm}
\begin{exm}
We consider $X = \C P^n$. Then
$$
P = \{(u_1,\cdots,u_n) \mid 0\le u_i, u_1+\cdots+u_n \le 1\},
$$
is a simplex. We have
$$
\mathfrak{PO}(x_1,\cdots,x_n;u_1,\cdots,u_n)
= \sum_{i=1}^n e^{x_i} T^{u_i} + e^{-\sum x_i}T^{1 -\sum u_i}.
$$
We put
$
u = u_0 = \left(\frac{1}{n+1}, \cdots, \frac{1}{n+1}\right).
$
Then
$$
\mathfrak{PO}^{u_0} = (y_1 + \cdots + y_n + y_1^{-1}y_2^{-1}\cdots y_n^{-1})
T^{1/(n+1)}.
$$
Solutions of the equation (\ref{formula:critical}) are given by
$$
\mathfrak y_1 = \cdots = \mathfrak y_{n} = e^{2\pi k\sqrt{-1}/{(n+1)}},
\quad k=0,\cdots,n.
$$
Hence the conclusion of Theorem \ref{toric-intersect} holds for our torus.
The case $k=0$ corresponds to $b=0$.
The other
cases correspond to appropriate flat bundles on $T^n$.
\end{exm}
\begin{rem}
The critical values of the potential function is
$(n+1) e^{2\pi\sqrt{-1}k/(n+1)}$, $k=0,\cdots,n$.
\par
We consider the quantum cohomology ring
$$
QH(\C P^n;\Lambda_{0}) \cong \Lambda_{0}[z,T]/(z^{n+1} -
T).
$$
The first Chern class $c_1$ is $(n+1)z$. The eigenvalues of the
operator
$$
c: QH(\C P^n) \to QH(\C P^n), \alpha \mapsto c_1 \cup_Q \alpha
$$
are $(n+1) e^{2\pi\sqrt{-1}k/(n+1)}$, $k=0,\cdots,n$.
It coincides with the set of critical values.
\par
Kontsevich announced this statement at the homological mirror
symmetry conference at Vienna 2006. (According to some physicists,
this statement is known to them before.) See \cite{Aur07}. In our
situation of Lagrangian fiber of compact toric manifolds, we can prove it by using Theorem \ref{QHequalMilnor}.
\end{rem}
In the rest of this subsection, we discuss 2 dimensional examples.
\par
Let $\text{\bf e}_1$, $\text{\bf e}_2$ be the basis of $H^1(T^2;\Z)$
as in Lemma \ref{37.90}. We put $\text{\bf e}_{12} = \text{\bf e}_1
\cup \text{\bf e}_2 \in H^2(T^2;\Z)$. Let $\text{\bf e}_{\emptyset}$
be the standard basis of $H^0(T^2;\Z) \cong \Z$. The proof of the
following proposition will be postponed until section
\ref{sec:floerhom}.

\begin{prop}\label{eq:m1b}
Let $\mathfrak x = \sum \frak x_i \text{\bf e}_i \in H^1(L(u);\Lambda_{0})$, $\frak y_i = e^{\frak x_i}$.
Then the boundary operator $\mathfrak m_1^{\mathfrak x}$
is given as follows :
\begin{equation}\label{spectredif}
\left\{\aligned
\mathfrak m_1^{\mathfrak x}(\text{\bf e}_i)
&= \frac{\partial \mathfrak{PO}^u}{\partial y_i}(\mathfrak y) \text{\bf
e}_{\emptyset}, \\
\mathfrak m_1^{\mathfrak x}(\text{\bf e}_{12}) &= \frac{\partial
\mathfrak{PO}^u}{\partial y_1}(\mathfrak y) \text{\bf e}_{2} -
\frac{\partial \mathfrak{PO}^u}{\partial y_2}(\mathfrak y)
\text{\bf e}_{1}, \\
\mathfrak m_1^{\mathfrak x}(\text{\bf e}_{\emptyset}) &= 0.
\endaligned\right.
\end{equation}
\end{prop}

We remark that we do not use the grading parameter $e$, which was introduced in \cite{fooo06}.
So the boundary operator $\mathfrak m_1^{\mathfrak x}$ is of degree $-1$ rather than
$+1$. (Note we are using cohomology notation.) In other words,
our Floer cohomology is only of $\Z_2$ graded.
\par
With (\ref{spectredif}) in our disposal, we examine various
examples.

\begin{exm}
We consider $M = \C P^2$ again. We put $u_1 = \epsilon + 1/3$,
$u_2 = 1/3$. ($\epsilon >0$.)
Using (\ref{spectredif}) we can easily find the following isomorphism
for the Floer cohomology with $\Lambda_0$ coefficients :
$$
HF^{odd}((L(u),0),(L(u),0)) \cong
HF^{even}((L(u),0),(L(u),0)) \cong \Lambda_{0}/(T^{1/3-\epsilon}).
$$
Let us apply Theorem J \cite{fooo06}
(= Theorem \ref{displace2}) in this situation.
(See also Theorem \ref{thm:eLuEu} below.)
We consider
a Hamiltonian diffeomorphism $\psi : \C P^2 \to \C P^2$.
We denote by $\Vert\psi\Vert$ the Hofer distance of $\psi$
from identity. Then we have
$$
\#(\psi(L(u)) \cap L(u)) \ge 4
$$
if $\Vert\psi\Vert < 2\pi(\frac{1}{3}-\epsilon)$ and $\psi(L(u))$ is transversal
to $L(u)$. We remark $\omega \cap [\C P^1] = 2\pi$ by (\ref{eq:area}).
\par
We remark that this fact was already proved by Chekanov \cite{che:duke}.
(Actually the basic geometric idea behind our proof is the
same as Chekanov's.)
\end{exm}
\begin{exm}\label{s2s2}
Let $M = S^2(\frac{a}{2}) \times S^2(\frac{b}{2})$, where
$S^2(r)$ denotes the 2-sphere with radius $r$.
We assume $a < b$.\par
Then $B = [0,a] \times [0,b]$ and we have :
$$
\mathfrak{PO}(x_1,x_2;u_1,u_2) = y_1T^{u_1} + y_2T^{u_2} +
y_1^{-1}T^{a-u_1} +y_2^{-1}T^{b-u_2}.
$$
Let us take $u_1 = a/2$, $u_2 = b/2$. Then
$$
\frac{\partial\mathfrak{PO}^u}{\partial y_1} = (1-y_1^{-2})T^{a/2},
\quad
\frac{\partial\mathfrak{PO}^u}{\partial y_2} = (1-y_2^{-2})T^{b/2}.
$$
Therefore $y_1 = \pm 1$ $y_2 = \pm 1$ are solutions of
(\ref{formula:critical}). Hence we can apply Theorem \ref{homologynonzero} to our torus.
\par
We next put $u_1 = a/2$, $a < 2u_2 < b$. Then
$$
\frac{\partial\mathfrak{PO}^u}{\partial y_1} = (1-y_1^{-2})T^{a/2},
\quad
\frac{\partial\mathfrak{PO}^u}{\partial y_2} = T^{u_2}-y_2^{-2}T^{b-u_2}.
$$
We put $y_1 = y_2 = 1$. Then $\frac{\partial\mathfrak{PO}^u}{\partial y_1}=0$,
$\frac{\partial\mathfrak{PO}^u}{\partial y_2} \ne 0$.
We find that
$$
HF^{odd}((L(u),0),(L(u),0)) \cong
HF^{even}((L(u),0),(L(u),0)) \cong \Lambda_{0}/(T^{u_2}).
$$
Let $\psi : \C P^2 \to \C P^2$ be a Hamiltonian diffeomorphism.
Then, Theorem J \cite{fooo06} (= Theorem \ref{displace2}) implies that
$$
\#(\psi(L(u)) \cap L(u)) \ge 4
$$
if $\Vert\psi\Vert <2\pi u_2$ and $\psi(L(u))$ is transversal
to $L(u)$. Note there exists a pseudo-holomorphic disc with
symplectic area $\pi a$ ($<2\pi u_2$). Hence our result
improves a result from \cite{che:duke}.
\end{exm}
\begin{exm}\label{37.116}
Let $X$ be the two-point blow up of $\C P^2$.
We may take its K\"ahler form so that the moment polytope
is given by
$$
P = \{(u_1,u_2) \mid -1 \le u_1 \le 1,
-1 \le u_2 \le 1, u_1+u_2 \le 1+\alpha\},
$$
where $-1<\alpha<1$ depends on the choice of K\"ahler form.
We have
\begin{equation}\aligned
\mathfrak{PO}(x_1,x_2;u_1,u_2)
= y_1T^{1+u_1} &+y_2T^{1+u_2} +y_1^{-1}T^{1-u_1} \\
&+
y_2^{-1}T^{1-u_2} +y_1^{-1}y_2^{-1}T^{1+\alpha-u_1-u_2}.
\endaligned\end{equation}
Note $X$ is Fano in our case.
\par\medskip
\noindent(Case 1: $\alpha =0$).
\par
In this case $X$ is monotone.
We put $u_0 = (0,0)$.
$L(u_0)$ is a monotone Lagrangian submanifold.
We have
$$
\frac{\partial\mathfrak{PO}^{u_0}}{\partial y_1} = (1-y_1^{-2}-
y_1^{-2}y_2^{-1})T,
\quad
\frac{\partial\mathfrak{PO}^{u_0}}{\partial y_2} = (1-y_2^{-2}-
y_1^{-1}y_2^{-2})T.
$$
The solutions of (\ref{formula:critical}) are given by $y_2 = \frac{1}{y_1^2-1}$,
$y_1^5 + y_1^4 - 2y_1^3 - 2y_1^2 + 1 = 0$ in $\C$.
(There are 5 solutions.)
\par\medskip
\noindent(Case 2: $\alpha >0$).
\par
We put $u_0 = (0,0)$. Then
$$
\frac{\partial\mathfrak{PO}^{u_0}}{\partial y_1} = (1-y_1^{-2})T-
y_1^{-2}y_2^{-1}T^{1+\alpha},
\quad
\frac{\partial\mathfrak{PO}^{u_0}}{\partial y_2} = (1-y_2^{-2})T-
y_1^{-1}y_2^{-2}T^{1+\alpha}.
$$
We consider, for example, the case $y_1 = y_2 = \tau$.
Then (\ref{formula:critical}) becomes
\begin{equation}\label{thirdegreeeq}
\tau^3 - \tau - T^{\alpha} = 0.
\end{equation}
The solution of (\ref{thirdegreeeq}) with $\tau \equiv 1 \mod \Lambda_{+}$
is given by
$$
\tau = 1 + \frac{1}{2} T^{\alpha} - \frac{3}{8} T^{2\alpha}
+ \frac{1}{2} T^{3\alpha} + \sum_{k=4}^{\infty} c_k T^{k\alpha}.
$$
Let us put
${\mathfrak x} = x_1\text{\bf e}_1 + x_2\text{\bf e}_2$ with
$$
x_1 = x_2 = \log \left(
1 + \frac{1}{2} T^{\alpha} - \frac{3}{8} T^{2\alpha}
+ \frac{1}{2} T^{3\alpha} + \cdots \right)
\in \Lambda_{+}.
$$
Then by Theorem \ref{homologynonzero} we have
\begin{equation}
HF((L(u_0),{\mathfrak x}),(L(u_0),{\mathfrak x});\Lambda_{0})
\cong H(T^2;\Lambda_{0}).
\nonumber\end{equation}
\par
We like to point out that in this example it is essential to deform
Floer cohomology using an element ${\mathfrak x}$ of $H^1(L(u_0);\Lambda_{+})$
containing the formal parameter $T$ to obtain nonzero Floer cohomology.
\par
At $u_0$, there are actually 4 solutions such that
$$
(y_1,y_2) \equiv (1,1),(1,-1),(-1,1),(-1,-1) \mod \Lambda_{+},
$$
respectively.
\par
In the current case there is another point $u'_0 = (\alpha,\alpha)
\in P$ at which $L(u'_0)$ is balanced \footnote{Using the method of
spectral invariants and symplectic quasi-states, Entov and
Polterovich discovered some non-displaceable Lagrangian fiber which
was not covered by the criterion given in \cite{cho-oh} (see section
9 \cite{entov-pol06}). Recently this example, among others, was
explained by Cho \cite{cho07} via Lagrangian Floer cohomology
twisted by non-unitary line bundles.}. In fact at $u'_0 =
(\alpha,\alpha)$ the equation (\ref{thirdegreeeq}) becomes
$$
0 =
-(y_1^{-2}y_2^{-1}+y_1^{-2})T^{1-\alpha} + T^{1+\alpha},
\quad
0 = -(y_1^{-1}y_2^{-2}+y_2^{-2})T^{1-\alpha} + T^{1+\alpha}.
$$
we put $\tau = y_1 = y_2$ to obtain
$$
\tau^3T^{2\alpha} - \tau - 1 = 0
$$
This equation has a unique solution with $\tau \equiv -1 \mod \Lambda_{+}$.
(The other solution is $T^{2\alpha}\tau^3 \equiv 1 \mod \Lambda_{+}$,
for which Theorem \ref{homologynonzero} is not applicable.)
\par
The total number of the solutions $({\mathfrak x},u)$ is $5$.
\par\medskip
\noindent(Case 3: $\alpha <0$).
\par
\par
We first consider $u_0 = (0,0)$. Then
$$
\frac{\partial\mathfrak{PO}^{u_0}}{\partial y_1} =
-
y_1^{-2}y_2^{-1}T^{1+\alpha} + (1-y_1^{-2})T,
\quad
\frac{\partial\mathfrak{PO}^{u_0}}{\partial y_2} =
-
y_1^{-1}y_2^{-2}T^{1+\alpha}+(1-y_2^{-2})T.
$$
We assume $y_i$ satisfies (\ref{formula:critical}). It is then easy to see that $y_{1}^{-1}\equiv 0$, or
$y_{2}^{-1} \equiv 0 \mod \Lambda_{+}$.
In other words, there is no $(y_1,y_2)$ to which we can
apply Theorem \ref{homologynonzero}. Actually it is easy to find
a Hamiltonian diffeomorphism $\psi : X \to X$
such that $\psi(L(u_0)) \cap L(u_0) = \emptyset$.
\par
We next take $u'_0 = (\alpha/3,\alpha/3)$.
Then
$$\aligned
\frac{\partial\mathfrak{PO}^{u'_0}}{\partial y_1} &=
(1-y_1^{-2}y_2^{-1})T^{1+\alpha/3} - y_1^{-2}T^{1-\alpha/3},
\\
\frac{\partial\mathfrak{PO}^{u'_0}}{\partial y_2} &=
(1-y_1^{-1}y_2^{-2})T^{1+\alpha/3} - y_2^{-2}T^{1-\alpha/3}.
\endaligned$$
By putting $y_1 = y_2 = \tau$ for example, (\ref{formula:critical}) becomes
\begin{equation}\label{anotherthird}
\tau^3 - T^{-2\alpha/3} \tau - 1 = 0.
\end{equation}
Let us put
${\mathfrak x} = x_1\text{\bf e}_1 + x_2\text{\bf e}_2$ with
$$
x_1 = x_2 = \log\tau = \log \left(
1 + \frac{1}{3} T^{-2\alpha/3} - \frac{1}{81} T^{-6\alpha/3} + \cdots
\right)
\in \Lambda_{+},
$$
where $\tau$ solves (\ref{anotherthird}).
Theorem \ref{homologynonzero} is applicable.
(There are actually 3 solutions of (\ref{formula:critical}) corresponding to the 3 solutions
of (\ref{anotherthird}).)
\par
There are two more points $u = (\alpha+1,\alpha), (\alpha,\alpha+1)$
where (\ref{formula:critical}) has a solution in
$(\Lambda_{0} \setminus \Lambda_{+})$.
Each $u$ has one solution $b$.
\par
Thus the total number of the pair $({\mathfrak x},u)$ is again 5.
We remark
$$
5 = \sum \text{\rm rank}\, H^k(X;\Q).
$$
This is not just a coincidence but an example of general phenomenon
stated as in Theorem \ref{lageqgeopt0}.
\par\medskip
We remark that as $\alpha \to 1$ our $X$ blows down to
$S^2(1) \times S^2(1)$. On the other hand, as
$\alpha \to -1$ our $X$ blows down
to $\C P^2$. The situation of the case $\alpha > 0$ can be regarded
as a perturbation of the situation of $S^2(1) \times S^2(1)$,
by the effect of exceptional curve corresponding to the
segment $u_1 + u_2 = 1 +\alpha$.
The situation of the case $\alpha < 0$ can be regarded as a
perturbation of the situation of $\C P^2$ by the
effect of the two exceptional curves corresponding to the segments
$u_1 = 1$ and $u_2 = 1$. An interesting phase change occurs
at $\alpha = 0$.
\end{exm}
We remark that $H^2(X;\R)$ is two dimensional. So there are
actually two parameter family of symplectic structures.
We study two points blow up of $\C P^2$ more in Example \ref{counterexamples}.
\par
The discussion of this section strongly suggests that Lagrangian
Floer theory (Theorems G, J \cite{fooo06} =
Theorems \ref{displace1}, \ref{displace2}) gives the optimal result
for the study of non-displacement of Lagrangian fibers in toric
manifolds.
\begin{rem}\label{conj:dispfloer}
Let $X$ be a compact toric manifold and $L(u) = \pi^{-1}(u)$, $u \in
\text{\rm Int} P$.
We consider the following two conditions :
\begin{enumerate}
\item There exists no Hamiltonian diffeomorphism $\psi : X
\to X$ such that
$
\psi(L(u)) \cap L(u) = \emptyset.
$
\par
\item
$L(u)$ is balanced.
\end{enumerate}
$(2) \Rightarrow (1)$ follows from
Proposition \ref{prof:bal2}. In many cases $(1) \Rightarrow (2)$
can be proved by the method of \cite{mcd}. However there is a case
$(1) \Rightarrow (2)$ does not follow, as we will
mention in Remark \ref{rembulk} and will prove
in a sequel of this series of papers, using
the bulk deformation of Lagrangian Floer cohomology.
We conjecture that after including this wider
class of Floer cohomology, we can detect all the
non-displacable Lagrangian fibers in toric manifolds,
by Floer cohomology.
\end{rem}
\par
Using the argument employed in Example \ref{s2s2} we can
discuss the relationship between the Hofer distance and displacement.
First we introduce some notations for this purpose. We denote by
$Ham(X,\omega)$ the group of Hamiltonian diffeomorphisms of $(X,\omega)$.
For a time-dependent Hamiltonian $H: [0,1] \times X \to \R$, we denote by
$\phi_H^t$ the time $t$-map of Hamilton's equation $\dot x = X_H(t,x)$.
The Hofer norm of $\psi \in Ham(X,\omega)$ is defined to be
$$
\|\psi\| = \inf_{H; \phi_H^1=\psi} \int_0^1 (\max H_t - \min H_t)\, dt.
$$
(See \cite{hofer}.)

\begin{defn}\label{dispenergy}
Let $Y \subset X$.
We define the {\it displacement energy} $e(Y) \in [0,\infty]$ by
$$
e(Y) : = \inf \{ \Vert \psi\Vert \mid \psi \in Ham(X,\omega), \,
\psi(Y) \cap \overline Y = \emptyset\}.
$$
We put $e(Y) = \infty$ if there exists no $\psi \in Ham(X,\omega)$
with $\psi(Y) \cap \overline Y = \emptyset$.
\end{defn}
Let us consider $\mathfrak{PO}(y_1,\cdots,y_n;u_1,\cdots,u_n)
: \Lambda_{0}^n \times P
\to
\Lambda_+$ as in Theorem \ref{potential}.
\begin{defn}
We define the number $\mathfrak E(u) \in (0,\infty]$
as the supremum of all $\lambda$ such that
there exists $\mathfrak y_1,\cdots,\mathfrak y_n \in
(\Lambda_{0} \setminus \Lambda_{0}^{+})^n$ satisfying
\begin{equation}\label{modequ}
\frac{\partial\mathfrak{PO}}{\partial y_i}(\mathfrak y_1,\cdots,\mathfrak y_n
; u) \equiv 0 \mod T^\lambda
\end{equation}
for $i=1,\cdots,n$. (Here we consider universal Novikov ring with
$\C$-coefficients.)
We call $\mathfrak E(u)$ the \emph{$\mathfrak{PO}$-threshold} of
the fiber $L(u)$.
We put
$$
\overline{\mathfrak E}(u)
= \limsup_{\omega_i \to \omega, u_i \to u} {\mathfrak E}(u_i).
$$
Here limsup is taken over all sequences $\omega_i$ and $u_i$ such that
$\omega_i$ is a sequence of $T^n$ invariant symplectic structures
on $X$ with $\lim_{i\to\infty}\omega_i = \omega$ and
$u_i$ is a sequence of points of moment polytopes $P_i$ of
$(X,\omega_i)$ such that $P_i$ converges to $P$ and
$u_i$ converges to $u$.
\end{defn}
Clearly $\overline{\mathfrak E}(u)
\ge {\mathfrak E}(u)$.
We will give an example where $\overline{\mathfrak E}(u)
\ne {\mathfrak E}(u)$ in Example \ref{counterexamples} (\ref{Ecounterexa}).
\begin{thm}\label{thm:eLuEu}
For any compact toric manifold $X$ and $L(u) = \pi^{-1}(u)$, $u \in
\text{\rm Int} P$, we have
\begin{equation}\label{displaceineq}
e(L(u)) \ge 2\pi \overline{\mathfrak E}(u).
\end{equation}
\end{thm}
We will prove Theorem \ref{thm:eLuEu} in section \ref{sec:floerhom}.

\begin{rem}\label{diseneconj}
The equality in $(\ref{displaceineq})$ holds in various examples.
However there are cases that the equality in $(\ref{displaceineq})$
does not hold. The situation is the same as Remark \ref{conj:dispfloer}.
\end{rem}
\par
We like to remark $\mathfrak E(u)$, $\overline{\mathfrak E}(u)$ can be calculated
in most of the cases once
the toric manifold $X$ is given explicitly.
In fact the leading order potential function $\frak{PO}_0$ is explicitly
calculated by Theorem \ref{potential}. We can then find the
maximal value $\lambda$ for which the polynomial equations
$$
\frac{\partial {\frak{PO}}_0}{\partial y_i}(u;\frak y_1,\cdots,\frak
y_n) \equiv 0 \mod T^\lambda
$$
has a solution $\frak y_i \in \Lambda_0 \setminus \Lambda_+$.
In a weakly degenerate case this value of $\lambda$ for $\frak{PO}(u;\cdots)$ is the same as
the value for $\frak{PO}_0(u;\cdots)$.
(See section \ref{sec:ellim}.)
\begin{rem}
Appearance of a new family of pseudo-holomorphic discs with Maslov
index 2 after blow up, which we observed in Examples \ref{37.116}
can be related to the operator $\mathfrak q$ that we introduced in
section 3.8 \cite{fooo06} (=
section 13 \cite{fooo06pre})  in the following way.
\par
We denote by $\mathcal M^{\text{\rm main}}_{k+1,l}(\beta)$ the moduli space of stable
maps $f : (\Sigma,\partial\Sigma) \to (X,L)$ from bordered Riemann
surface $\Sigma$ of genus zero with $l$ interior and $k+1$ boundary
marked points and in homology class $\beta$. (See section 3
\cite{fooo00} (= subsection 2.1.2 \cite{fooo06}). The symbol $\text{main}$ means that we require the boundary marked points to respect the
cyclic order of $\partial\Sigma$.) Let us consider the case when Maslov index of
$\beta$ is $2n$.
More precisely we take the following class $\beta$.
We use notation introduced at the beginning of section \ref{sec:calcpot}.
We put $\beta = \beta_{i_1}+\cdots+\beta_{i_n}$,
where $\partial_{i_1} P \cap \cdots \cap \partial_{i_n} P = \overline{p}$ is an vertex of $P$.
We assume $[f] \in \mathcal M^{\text{\rm main}}_{0+1,1}(\beta)$ and
$\mathcal M^{\text{\rm main}}_{0+1,1}(\beta)$ is Fredholm regular at $f$. The virtual
dimension of $\mathcal M^{\text{\rm main}}_{0+1,1}(\beta)$ is $3n$.
Let us take the unique point $p \in X$ such that $\pi(p) = \overline{p}$.
$p$ is the $T^n$ fixed point.
We assume moreover $f(0) = p$. We blow up $X$ at
a point $p = f(0) \in X$
and obtain $\widehat{X}$. Let $[E] \in H_{2n-2}(\widehat{X})$ be the homology
class of the exceptional divisor $E = \pi^{-1}(p)$. Now $f$ induces
a map $\widehat{f} : (\Sigma;\partial \Sigma) \to (\widehat{X},L)$. The
Maslov index of the homology class $[\widehat f] \in H_2(\widehat{X},L)$
becomes $2$. We put $\widehat{\beta} = [\widehat{f}]$.
\par
Since $p$ is a fixed point of $T^n$ action,
a $T^n$-invariant perturbation lifts to a perturbation of the moduli space
$\mathcal M^{\text{\rm main}}_{0+1,0}(\widehat{\beta})$.
Then any $T^n$-orbit of the moduli space $\mathcal M^{\text{\rm main}}_{0+1,0}(X;\beta)$ of holomorphic
discs passing through $p$ corresponds to the $T^n$-orbit of
$\mathcal M^{\text{\rm main}}_{0+1,0}(\widehat X;\widehat{\beta})$
and vice versa. Namely we have an isomorphism
\begin{equation}\label{blowup}
\mathcal M^{\text{\rm main}}_{0+1,1}(\beta) \,\,{}_{ev}\times_{X} \{p\}
\cong
\mathcal M^{\text{\rm main}}_{0+1,0}(\widehat{\beta}).
\end{equation}
Here $ev$ in the left hand side is the evaluation map at the
interior marked point. (Actually we need to work out analytic detail
of gluing construction etc.. It seems very likely that we can do it
in the same way as the argument of Chapter 10 \cite{fooo06pre}. See
also \cite{LiRuan}.)
\par
Using (\ref{blowup}) we may prove :
$$
\mathfrak q_{1,k;\beta}(PD([p]);b,\cdots,b) = \mathfrak m_{k,\widehat{\beta}}(b,\cdots,b),
$$
where
$$
\mathfrak q_{1,k;\beta}(Q;P_1,\cdots,P_k)
= ev_{0*}\left(
\mathcal M^{\text{\rm main}}_{k+1,1}(\beta) \times_{(X \times L^k)} (Q\times P_1\times \cdots \times P_k) \right)
$$
is defined in section 3.8 \cite{fooo06} (=section 13 \cite{fooo06pre}).
(Here $Q$ is a chain in $X$ and $P_i$ are chains in $L(u)$, and
$ev_0 :\mathcal M^{\text{\rm main}}_{k+1,1}(\beta) \to X$ is the evaluation map at the
$0$-th boundary marked point. In the right hand side, we take fiber
product over $X \times L^k$.)
This is an example of a blow-up formula in Lagrangian Floer
theory.
\end{rem}
\section{Quantum cohomology and Jacobian ring}
\label{sec:milquan}

In this section, we prove Theorem \ref{QHequalMilnor}.
Let $\mathfrak{PO}_0$ be the leading order potential function.
(Recall if $X$ is Fano, we have $\mathfrak{PO}_0 = \mathfrak{PO}$.)
We define the monomial
\begin{equation}\label{formdef:overzi}
\overline z_i(u) = y_1^{v_{i,1}}\cdots
y_n^{v_{i,n}}T^{\ell_i(u)}
\in \Lambda_{0}[y_1,\cdots,y_n,y^{-1}_1,\cdots,y^{-1}_n].
\end{equation}
Compare this with (\ref{homocoordeq}). It is also suggestive to
write $\overline z_i$ also as
\be\label{eq:ziexvi}
\overline z_i(u) = e^{\langle x, v_i \rangle} T^{\ell_i(u)}, \, x =
(x_1,\cdots, x_n), \, y_i = e^{x_i}.
\ee
By definition we have
\begin{eqnarray}
\mathfrak{PO}_0^u &=& \sum_{i=1}^m \overline z_i(u) \label{poz}\\
y_j\frac{\partial \overline z_i}{\partial y_j} &=& v_{i,j} \overline z_i(u).
\label{euler}
\end{eqnarray}
The following is a restatement of Theorem \ref{QHequalMilnor}.
Let $z_i \in H^2(X;\Z)$ be the Poincar\'e dual of the
divisor $\pi^{-1}(\partial_iP)$.
\begin{thm}\label{QHequalMilnor2}
If $X$ is Fano, there exists an isomorphism
$$
\psi_u : QH(X;\Lambda) \cong Jac(\mathfrak{PO})
$$
such that $\psi_u(z_i) = \overline z_i(u)$.
\end{thm}

Since $c_1(X) = \sum_{i=1}^m z_i$ (see \cite{fulton}) and
$\mathfrak{PO}_0^u = \sum_{i=1}^m \overline z_i(u)$ by definition,
Theorem \ref{QHequalMilnor} follows from Theorem
\ref{QHequalMilnor2}.
\par

In the rest of this section, we prove Theorem \ref{QHequalMilnor2}. We
remark that $z_i$ ($i=1,\cdots,m$) generates the quantum cohomology
ring $QH(X;\Lambda)$ as a $\Lambda$-algebra (see Theorem
\ref{thm:batyrev} below). Therefore it is enough to prove that the
assignment $\tilde{\psi}_u(z_i) = \overline z_i(u)$ extends to a
homomorphism $ \tilde{\psi}_u : \Lambda[z_1,\cdots,z_m] \to
Jac(\mathfrak{PO}_0^u)$ that induces an isomorphism in
$QH(X;\Lambda)$. In other words, it suffices to show that the
relations among the generators in $\Lambda[z_1,\cdots,z_m]$ and in
$Jac(\mathfrak{PO}_0^u)$ are mapped to each other under the
assignment $\tilde{\psi}_u(z_i) = \overline z_i(u)$. To establish
this correspondence, we now review Batyrev's description of the
relations among $z_i$'s.
\par
We first clarify the definition of quantum cohomology ring over the
universal Novikov rings $\Lambda_{0}$ and $\Lambda$. Let
$(X,\omega)$ be a symplectic manifold and $\alpha \in \pi_2(X)$. Let
$\mathcal M_3(\alpha)$ be the moduli space of stable map with homology class
$\alpha$ of genus
$0$ with $3$ marked points. Let $ ev : \mathcal M_3(\alpha) \to X^3
$ be the evaluation map. We can define the virtual fundamental class
$ ev_*[\mathcal M_3(\alpha)] \in H_d(X^3;\Q) $ where $ d =
2(\dim_{\C} X + c_1(X)\cap \alpha). $ Let $a_i \in H^*(X;\Q)$. We
define $ a_1\cup_Q a_2 \in H^*(X;\Lambda_{0}) $ by the following
formula.
\begin{equation}\label{qcohdef}
\langle a_1\cup_Q a_2, a_3\rangle
= \sum_{\alpha} T^{\omega\cap \alpha/2\pi}
ev_*[\mathcal M_3(\alpha)] \cap (a_1\times a_2\times a_3).
\end{equation}
Here $\langle \cdot, \cdot \rangle$ is the Poincar\'e duality.
Extending this linearly we obtain the quantum product
$$
\cup_Q : H(X;\Lambda_{0}) \otimes
H(X;\Lambda_{0}) \to H(X;\Lambda_{0}).
$$
Extending the coefficient ring further to $\Lambda$, we obtain
the (small) quantum cohomology ring $QH(X;\Lambda)$.
\par
Now we specialize to the case of compact toric manifolds and
review Batyrev's presentation of quantum cohomology ring.
We consider the exact sequence
\begin{equation}
0 \longrightarrow \pi_2(X) \longrightarrow \pi_2(X;L(u))
\longrightarrow \pi_1(L(u)) \longrightarrow 0.
\label{exseq:homtpyeq}\end{equation} We note $\pi_2(X;L(u)) \cong
\Z^m$ and choose its basis adapted to this exact sequence as follows
: Consider the divisor $\pi^{-1}(\partial_iP)$ and take a small disc
transversal to it. Each such disc gives rise to an element
\begin{equation}
[\beta_i] \in H_2(X;\pi^{-1}(\text{\rm Int} P)) \cong H_2(X;L(u))
\cong \pi_2(X,L(u)). \label{def:betai}
\end{equation}
The set of $[\beta_i]$ with $i=1,\cdots,m$ forms a basis of
$\pi_2(X;L(u)) \cong \Z^m$. The boundary map $[\beta] \mapsto
[\partial \beta] : \pi_2(X;L(u)) \to \pi_1(L(u))$ is identified with
the corresponding map $H_2(X;L(u)) \to H_1(L(u))$. Using the basis
chosen in Lemma \ref{37.90} on $H_1(L(u))$ we identify $H_1(L(u))
\cong \Z^n$. Then this homomorphism maps $[\beta_i]$ to
\begin{equation}
[\partial\beta_i] \cong v_i = (v_{i,1},\cdots,v_{i,n}),
\label{form:connecting}\end{equation} where $v_{i,j}$ is as in
(\ref{eq:vij}). By the exactness of (\ref{exseq:homtpyeq}), we have
an isomorphism
\begin{equation}
H_2(X) \cong \{ \beta \in H_2(X;L(u)) \mid [\partial \beta] = 0 \}.
\label{def:pi2}\end{equation}

\begin{lem}\label{omegacap}
We have
\begin{equation}
\omega \cap \left[\sum k_i\beta_i\right] = 2\pi \sum k_i \ell_i(u).
\label{form:symparea}\end{equation}
If $\left[\sum k_i\partial\beta_i\right] = 0$ then
\be\label{eq:kidliduj}
\sum k_i \frac{d\ell_i}{du_j} = 0.
\ee
In particular, the right hand side of $(\ref{form:symparea})$ is independent of $u$.
\end{lem}
\begin{proof} (\ref{form:symparea}) follows from the area formula
(\ref{eq:area}),
$
\omega(\beta_i) = 2\pi \ell_i(u).
$
On the other hand if
$\left[\sum k_i\partial\beta_i\right] = 0$, we have
$$
\sum_{i=1}^m k_i v_i = 0.
$$
By the definition of $\ell_i$, $\ell_i(u) = \langle u,v_i \rangle -
\lambda_i$, from Theorem \ref{gullemin}, this equation is precisely
(\ref{eq:kidliduj}) and hence the proof.
\end{proof}
\par
Let $\mathcal P \subset \{1,\cdots,m\}$ be a primitive collection
(see Definition \ref{prim}). There exists a unique subset
$\mathcal P' \subset \{1,\cdots,m\}$ such that
$\sum_{i\in \mathcal P} v_i$ lies in the interior of the
cone spanned by $\{v_{i'} \mid i' \in \mathcal P'\}$, which is a member of the fan $\Sigma$.
(Since $X$ is compact, we can choose such $\mathcal P'$. See section 2.4 \cite{fulton}.)
We write
\begin{equation}
\sum_{i\in \mathcal P} v_i = \sum_{i' \in \mathcal P'} k_{i'}v_{i'}.
\label{qsrprep}
\end{equation}
Since $X$ is assumed to be nonsingular $k_{i'}$ are all positive integers.
(See p.29 of \cite{fulton}.)
We put
\begin{equation}
\omega(\mathcal P) = \sum_{i\in \mathcal P} \ell_i(u) - \sum_{i' \in \mathcal P'}k_{i'}\ell_{i'}(u).
\label{def:symparea2}\end{equation}
It follows from (\ref{form:symparea}) that $2\pi \omega(\mathcal P)$ is the
symplectic area of the homotopy class
\begin{equation}
\beta(\mathcal P) = \sum_{i\in \mathcal P}\beta_i - \sum_{i' \in \mathcal P'} k_{i'}\beta_{i'} \in \pi_2(X).
\label{def:betaP}
\end{equation}

\begin{lem}\label{omegaP>0} Let $\mathcal P$ be any primitive collection. Then
$\omega(\mathcal P) > 0$.
\end{lem}
\begin{proof} Since the cone spanned by $\{v_{i'} \vert i' \in {\mathcal P}'\}$
belongs to the fan $\Sigma$, we have
\begin{equation}
\bigcap_{i' \in \mathcal P'} \pi^{-1}(\partial_{i'} P) \ne \emptyset.
\nonumber
\end{equation} Then for any $u' \in \bigcap_{i' \in \mathcal
P'} \partial_{i'} P$, we have $\ell_{i'}(u') = 0$.
By the choice of $k_{i'}$ in (\ref{qsrprep}), we have
$\partial \beta(\mathcal P) = 0$. Therefore by Lemma \ref{omegacap} and
by the continuity of the right hand side of (\ref{def:symparea2})
we can evaluate $\omega(\mathcal P)$ at a point $u'
\in \bigcap_{i' \in \mathcal P'} \pi^{-1}(\partial_{i'} P)$.
Then we obtain
$$
\omega(\mathcal P) = \sum_{i\in \mathcal P} \ell_i(u').
$$
Since $\mathcal P$ is a primitive collection, and in particular
does not form a member of the fan, there must be
an element $v_i \in \mathcal P$ such that $\ell_i(u') > 0$ and so
$\omega(\mathcal P) > 0$. This finishes the proof.
\end{proof}

Now we associate formal variables, $z_1,\cdots,z_m$, to $v_1, \cdots, v_m$,
respectively.
\begin{defn}[Batyrev \cite{batyrev:qcrtm92}]
\begin{enumerate}
\item The {\it quantum Stanley-Reisner ideal} $SR_{\omega}(X)$ is the
ideal generated by
\begin{equation}
z(\mathcal P) = \prod_{i\in\mathcal P} z_{i} - T^{\omega(\mathcal P)} \prod_{i'\in \mathcal P'}
z_{i'}^{k_{i'}}
\label{def:generatprqsr}\end{equation}
in the polynomial ring $\Lambda\left[z_1,\cdots,z_m\right]$.
Here $\mathcal P$ runs over all primitive collections.
\par
\item We denote by $P(X)$ the ideal generated by
\begin{equation}
\sum_{i=1}^m v_{i,j} z_i
\label{def:linrelidea}\end{equation}
for $j=1,\cdots,n$. In this paper we call $P(X)$ the {\it linear relation ideal}.
\par
\item
We call the quotient
\begin{equation}
QH^{\omega}(X;\Lambda) = \frac{\Lambda\left[z_1,\cdots,z_m\right]}{(P(X) + SR_{\omega}(X))}
\label{form:batyrevqcr}\end{equation}
the {\it Batyrev quantum cohomology ring}.
\end{enumerate}
\label{def:batyrevqcr}\end{defn}
\begin{rem}
We do not take closure of our ideal $P(X) + SR_{\omega}(X)$ here. See Proposition \ref{prop:closuerideal}.
\end{rem}
\begin{thm}[Batyrev \cite{batyrev:qcrtm92}, Givental \cite{givental2}]
If $X$ is Fano there exists a ring isomorphism from $QH^{\omega}(X;\Lambda)$
to the quantum cohomology ring $QH(X;\Lambda)$ of $X$
such that $z_i$ is sent to the Poincar\'e dual to $\pi^{-1}(\partial_iP)$.
\label{thm:batyrev}\end{thm}
The main geometric part of the proof of Theorem \ref{thm:batyrev} is the following.
\begin{prop}\label{prop:qsrholds}
The Poincar\'e dual to $\pi^{-1}(\partial_iP)$ satisfy the quantum
Stanley-Reisner relation.
\end{prop}
We do not prove Proposition \ref{prop:qsrholds} in this paper.
See Remarks \ref{rem:qcanbeused} and \ref{CSMT}.
However since our choice of the coefficient ring is
different from other literature, we explain here for reader's convenience how
Theorem \ref{thm:batyrev} follows from Proposition \ref{prop:qsrholds}.
\par
Proposition \ref{prop:qsrholds} implies that we can define a ring homomorphism
$h : QH^{\omega}(X;\Lambda) \to QH(X;\Lambda)$
by sending $z_i$ to $PD(\pi^{-1}(\partial_iP))$. Let
$
F^kQH(X;\Lambda)
$
be the direct sum of elements of degree $\le k$.
Let $F^kQH^{\omega}(X;\Lambda)$ be the submodule generated by
the polynomial of degree at most $k/2$ on $z_i$.
Clearly $h(F^kQH^{\omega}(X;\Lambda)) \subset F^kQH(X;\Lambda)$.
\par
Since $X$ is Fano, it follows that,
$$
x\cup_Q y - x\cup y \in F^{\deg x + \deg y -2}QH(X;\Lambda).
$$
We also recall the cohomology ring
$H(X;\Q)$ is obtained by putting $T=0$ in
quantum Stanley-Reisner relation. Moreover
we find that the second product of the right hand side of
(\ref{def:generatprqsr}) has degree strictly
smaller than the first since $X$ is Fano.
\par
Therefore the graded ring
$$
gr(QH(X;\Lambda)) = \bigoplus_{k} F^k(QH(X;\Lambda))/F^{k-1}(QH(X;\Lambda)),
$$
is isomorphic to the (usual) cohomology ring (with $\Lambda$ coefficient) as a ring.
The same holds for $QH^{\omega}(X;\Lambda)$.
It follows that $h$ is an isomorphism.
\qed\par
\medskip
In the rest of this section, we will prove the following Proposition
\ref{nonfanoadded}. Theorem \ref{QHequalMilnor2} follows immediately
from Proposition \ref{nonfanoadded} and Theorem \ref{thm:batyrev}.

\begin{prop}\label{nonfanoadded}
There exists an isomorphism :
$$
\psi_u : QH^{\omega}(X;\Lambda) \cong Jac(\mathfrak{PO}_0^u)
$$
such that $\psi_u(z_i) = \overline z_i(u)$.
\end{prop}
We remark that we do {\it not} assume that $X$ is Fano in
Proposition \ref{nonfanoadded}. We also remark that for our main
purpose to calculate $\mathfrak M_0(\mathfrak{Lag}(X))$, Proposition
\ref{nonfanoadded} suffices. Proposition \ref{nonfanoadded} is a
rather simple algebraic result whose proof does not require study of
pseudo-holomorphic discs or spheres.

\begin{proof}[Proof of Proposition \ref{nonfanoadded}]

We start with the following proposition.

\begin{prop}
\label{prop:psiwelldefined}
The assignment
\begin{equation}\label{formuladef:psitilde}
\widehat{\psi}_u(z_i) =\overline z_i(u)
\end{equation}
induces a well-defined ring isomorphism
\begin{equation}\label{formula:psitilde}
\widehat{\psi}_u : \frac{\Lambda\left[z_1,\cdots,z_m\right]}{SR_{\omega}(X)} \to
\Lambda[y_1,\cdots,y_n,y_1^{-1},\cdots,y_n^{-1}].
\end{equation}
\end{prop}
\begin{proof}
Let $\mathcal P$ be a primitive collection and
$\mathcal P'$, $k_{i'}$ be as in (\ref{qsrprep}). We calculate
\be\label{eq:prodbzrzi}
\prod_{i\in\mathcal P} \overline z_{i}(u)
= \prod_{i\in \mathcal P}y_1^{v_{i,1}}\cdots
y_n^{v_{i,n}}T^{\ell_i(u)}
\ee
by (\ref{formdef:overzi}).
On the other hand,
$$\aligned
\prod_{i'\in \mathcal P'}
\overline z_{i'}^{k_{i'}}(u)
&= \prod_{i'\in \mathcal P'} y_1^{k_{i'}v_{i',1}}\cdots
y_n^{k_{i'}v_{i',n}}T^{k_{i'}\ell_{i'}(u)} \\
&= \prod_{i\in \mathcal P}y_1^{v_{i,1}}\cdots
y_n^{v_{i,n}}
\prod_{i'\in \mathcal P'}T^{k_{i'}\ell_{i'}(u)}
\endaligned$$
by (\ref{qsrprep}). Moreover
$$
\sum_{i\in \mathcal P} \ell_i(u) - \sum_{i\in \mathcal P'} k_{i'}\ell_{i'}(u)
= \omega(\mathcal P)
$$
by (\ref{def:symparea2}). Therefore
$$
\prod_{i\in\mathcal P} \overline z_{i}(u) = T^{\omega(\mathcal P)} \prod_{i'\in \mathcal P'}
\overline z_{i'}^{k_{i'}}(u)
$$
in $\Lambda[y_1,\cdots,y_n,y_1^{-1},\cdots,y_n^{-1}]$.
In other words, (\ref{formuladef:psitilde}) defines a well-defined ring homomorphism
(\ref{formula:psitilde}).
\par
We now prove that $\widehat\psi_u$ is an isomorphism. Let
$$
pr : \Z^m \cong \pi_2(X;L(u)) \longrightarrow \Z^n \cong \pi_1(L(u))
$$
be the homomorphism induced by the boundary map
$
pr([\beta]) = [\partial \beta].
$
(See (\ref{kexact}).)
We remark $pr(c_1, \dots, c_m)=(d_1, \dots, d_n)$ with $d_j = \sum_i c_i v_{i,j}$. Let $A=\sum_ic_i \beta_i$ be an element in the kernel of $pr$.
We write it as
$$
\sum_{i\in I} a_i \beta_i - \sum_{j\in J} b_{j}\beta_{j}
$$
where $a_i$ $b_{j}$ are positive and $I \cap J = \emptyset$.
We define
\be\label{eq:rA}
r(A) = \prod_{i\in I} z_i^{a_i} - T^{\sum_{i} a_i\ell_i(u) - \sum_{j}
b_{j}\ell_{j}(u)}\prod_{j\in J} z_{j}^{b_{j}}.
\ee
We remark that a generator of quantum Stanley-Reisner ideal corresponds to
$r(A)$ for which $I$ is a primitive collection $\mathcal P$ and $J = \mathcal P'$.
We also remark that the case $I = \emptyset$ or $J=\emptyset$ is included.
\begin{lem}\label{lem:pcolgee}
$$
r(A) \in SR_{\omega}(X).
$$
\end{lem}
\begin{proof}
This lemma is proved in \cite{batyrev:qcrtm92}. We include its proof
(which is different from one in \cite{batyrev:qcrtm92}) here for reader's convenience.
We prove the lemma by an induction over the values
$$
E(A) = \sum_{i\in I} a_i \ell_i(u_0) + \sum_{j\in J} b_{j} \ell_{j}(u_0).
$$
Here we fix a point $u_0\in \text{Int} P$ during the proof of Lemma \ref{lem:pcolgee}.
\par
Since $I\cap J = \emptyset$, at least one of $\{ v_i \mid i\in I\}$, $\{ v_j \mid j\in J\}$
can not span a cone that is a member of the fan $\Sigma$. Without loss of
generality, we assume that $\{ v_i \mid i\in I\}$ does not span such a cone.
Then it contains a subset $\mathcal P \subset I$ that is a primitive collection.
We take $\mathcal P'$, $k_{i'}$ as in (\ref{qsrprep}) and define
\be\label{eq:Z}
Z = \prod_{i\in I} z_i^{a_i} - T^{\omega(\mathcal P)}\prod_{i\in I\setminus \mathcal P} z_i^{a_i}
\prod_{i\in \mathcal P} z_i^{a_i-1} \prod_{i''\in\mathcal P'} z_{i''}^{k_{i''}}.
\ee
Then $Z$ lies in $SR_{\omega}(X)$ by construction. We recall from
Lemma \ref{omegacap} that the values
$$
\sum_{i\in \mathcal P} \ell_i(u) - \sum_{i\in \mathcal P'}k_i\ell_{i}(u) = \omega(\mathcal P)
$$
are independent of $u$ and positive. By the definitions (\ref{eq:rA}), (\ref{eq:Z})
of $r(A)$ and $Z$, we can express
$$
r(A) - Z = T^{\omega(\mathcal P) + c} \left(\prod_{h\in K} z_h^{n_h}\right) r(B)
$$
for an appropriate $B$ in the kernel of $pr$ and a constant $c$.
Moreover we have
$$
E(B) + 2\sum_{h \in K} n_h\ell_{h}(u_0) + \omega(\mathcal P) = E(A).
$$
Since $u_0 \in \text{Int}\,P$ it follows that $\ell_h(u_0) > 0$ which
in turn gives rise to $E(B) < E(A)$. The induction hypothesis
then implies $r(B) \in SR_{\omega}(X)$.
The proof of the lemma is now complete.
\end{proof}
\begin{cor}\label{lem:zinvert} $z_i$ is invertible in
$$
\frac{\Lambda\left[z_1,\cdots,z_m\right]}{SR_{\omega}(X)}.
$$
\end{cor}
\begin{proof}
Since $X$ is compact, the vector $-v_i$
is in some cone spanned by $v_j$ ($j\in I$). Namely
$$
-v_i = \sum_{j\in I} k_j v_j
$$
where $k_j$ are nonnegative integers. Then
$$
T^{\ell_i(u) + \sum_j k_j\ell_j(u)} = z_i \prod_{j\in I} z_j^{k_j} \mod SR_\omega(X)
$$
by Lemma \ref{lem:pcolgee}. Since $T^{\ell_i(u) + \sum_j k_j\ell_j(u)}$ is
invertible in the field $\Lambda$, it follows that
$\prod_{j\in I} z_j^{k_j}$ defines the inverse of $z_i$ in the
quotient ring.
\end{proof}

We recall from Lemma \ref{omegacap} that $\ell_i(u) + \sum_j k_j\ell_j(u)$
is independent of $u$. We {\it define}
\begin{equation}\label{def;z-1}
z_i^{-1} = T^{-\ell_i(u) - \sum_j k_j\ell_j(u)}\prod_{j\in I} z_j^{k_j}.
\end{equation}
(Note we have not yet proved that
${\Lambda\left[z_1,\cdots,z_m\right]}/{SR_{\omega}(X)}$
is an integral domain. This will follow later when we prove Proposition \ref{prop:psiwelldefined}.)

Since $v_1,\cdots,v_m$ generates the lattice $\Z^n$, we can always assume
the following by changing the order of $v_i$, if necessary.

\begin{conds}\label{cond:invertible}
The determinant of the $n\times n$ matrix $(v_{i,j})_{i,j=1,\cdots,n}$ is $\pm 1$.
\end{conds}

Let $(v^{i,j})$ be the inverse matrix of $(v_{i,j})$. Namely $
\sum_j v^{i,j}v_{j,k} = \delta_{i,k}. $ Condition
\ref{cond:invertible} implies that each $v^{i,j}$ is an integer.
Inverting the matrix $(v_{i,j})$, we obtain
\begin{equation}\label{eq:ycomesz}
y_i = T^{-c_i(u)}\prod_{j=1}^n \overline z_i^{v^{i,j}}
\end{equation}
from (\ref{eq:prodbzrzi}) where $ c_i(u) = \sum v^{i,j}\ell_j(u). $
We define using Corollary \ref{lem:zinvert}
$$
\widehat{\phi}_u(y_i^{\pm 1}) = T^{- \pm c_i(u)}\prod_{j=1}^n z_j^{\pm v^{i,j}} \in
\frac{\Lambda\left[z_1,\cdots,z_m\right]}{SR_{\omega}(X)}.
$$
More precisely, we plug (\ref{def;z-1}) here if $\pm v^{i,j}$ is negative.
\par
The identity
$
\widehat{\psi}_u \circ \widehat{\phi}_u = id
$
is a consequence of (\ref{eq:ycomesz}). We next calculate
$
(\widehat{\phi}_u\circ \widehat{\psi}_u)(z_{h}) =
\widehat{\phi}_u(\overline z_h(u))
$
and prove
$$
(\widehat{\phi}_u\circ \widehat{\psi}_u)(z_{h})
= T^{\ell_h(u)} \widehat{\phi}_u(y_1^{v_{h,1}}\cdots y_n^{v_{h,n}})
= T^{e(h;u)}\prod_{j=1}^n z_j^{m_j},
$$
where $m_j \ge 0$ and
\begin{equation}\label{equ:enevec}
v_h = \sum m_j v_j, \qquad \ell_h(u) = e(h;u) + \sum m_j \ell_j(u) :
\end{equation}
To see (\ref{equ:enevec}),
we consider any {\it monomial} $Z$ of $y_i,z_i,\overline z_i,T^{\alpha}$.
We define its multiplicative valuation $\mathfrak v_u(Z) \in \R$ by putting
$$
\mathfrak v_u(y_i) = 0, \quad \mathfrak v_u(z_i) = \mathfrak v_u(\overline z_i) = \ell_i(u),
\quad \mathfrak v_u(T^{\alpha}) = \alpha.
$$
We also define a (multiplicative) grading $\rho(Z) \in \Z^n$ by
$$
\rho(y_i) = \text{\bf e}_i, \quad \rho(z_i) = \rho(\overline z_i) = v_i,
\quad \rho(T^{\alpha}) = 0.
$$
and by $\rho(ZZ') = \rho(Z) + \rho(Z')$.
We remark that $\mathfrak v_u$ and $\rho$ are consistent with (\ref{formdef:overzi}).
We next observe that both
$\mathfrak v_u$ and $\rho$ are preserved by
$\widehat{\psi}_u$, $\widehat{\phi}_u$ and by (\ref{def;z-1}).
This implies (\ref{equ:enevec}).
\par
Now we use Lemma \ref{lem:pcolgee} and (\ref{equ:enevec}) to conclude
$$
z_h - T^{e(h;u)}\prod_{j=1}^n z_j^{m_j} \in SR_{\omega}(X).
$$
The proof of Proposition \ref{prop:psiwelldefined} is now complete.
\end{proof}

Next we prove

\begin{lem}\label{lem:critical} Let $P(X)$ be the linear relation
ideal defined in Definition \ref{def:batyrevqcr}. Then
$$
\widehat{\psi}_u (P(X)) = \left(\frac{\partial
\mathfrak{PO}_0^u}{\partial y_i}; i =1, \cdots, n \right).
$$
\end{lem}
\begin{proof}
Let $\sum_{i=1}^m v_{i,j} z_i$ be in $P(X)$.
Then we have
$$
\widehat{\psi}_u\left(\sum_{i=1}^m v_{i,j} z_i\right)
= \sum_{i=1}^m v_{i,j} \overline z_i
= \sum_{i=1}^m y_j\frac{\partial \overline z_i}{\partial y_j}
= y_j\frac{\partial \mathfrak{PO}_0^u}{\partial y_j}
$$
by (\ref{formdef:overzi}) and (\ref{euler}). Since $y_j$'s are
invertible in $\Lambda[y_1,\cdots,y_n,y_1^{-1},\cdots,y^{-1}_n]$,
this identity implies the lemma.
\end{proof}

The proof of Theorem \ref{QHequalMilnor2}
and of Proposition \ref{nonfanoadded} is now complete.
\end{proof}
\begin{rem}\label{5.7rational}
Proposition \ref{nonfanoadded} holds over $\Lambda^R$ coefficient
for arbitrary commutative ring $R$ with unit. The proof is the same.
\end{rem}
\par\medskip
We define
$$
\psi_{u',u} : Jac(\mathfrak{PO}_0^u) \to Jac(\mathfrak{PO}_0^{u'})
$$
by
\begin{equation}\label{eq:psizibar}
\psi_{u',u}(\overline z_i(u)) = \overline z_i(u')= T^{\ell_i(u') - \ell_i(u)} \overline z_i(u).
\end{equation}
It is an isomorphism. We have
$$
\psi_{u',u} \circ {\psi}_u= {\psi}_{u'}.
$$
The well-definedness $\psi_{u',u}$ is proved from this formula or by checking directly.
\par
In case no confusion can occur, we identify
$Jac(\mathfrak{PO}_0^u)$, $Jac(\mathfrak{PO}_0^{u'})$ by
$\psi_{u',u}$ and denote them by $Jac(\mathfrak{PO}_0)$. Since
$\psi_{u',u}(\overline z_i(u)) = \overline z_i(u')$ we write them
$\overline z_i$ when we regard it as an element of
$Jac(\mathfrak{PO}_0)$. Note $\psi_{u',u}(y_i) \ne y_i$. In case we
regard $y_i \in Jac(\mathfrak{PO}_0^u)$ as an element of
$Jac(\mathfrak{PO}_0)$ we write it as $y_i(u):= \psi_{u}(y_i)$.
\begin{rem}\label{rem:qcanbeused}
The above proof of Theorem \ref{QHequalMilnor2} uses Batyrev's
presentation of quantum cohomology ring and is not likely
generalized beyond the case of compact toric manifolds. (In fact the
proof is purely algebraic and do {\it not} contain serious study of
pseudo-holomorphic curve, except Proposition \ref{prop:qsrholds},
which we quote without proof and Theorem \ref{potential}, which
is a minor improvement of an earlier result of \cite{cho-oh}.) There is an alternative way of
constructing the ring homomorphism $\psi_u$ which is less
computational. (This will give a new proof of Proposition
\ref{prop:qsrholds}.) We will give this conceptual proof in a sequel
to this paper.
\par
We use the operations
$$
\mathfrak q_{1,k;\beta} : H(X;\Q)[2] \otimes B_kH(L(u);\Q)[1]
\to H(L(u);\Q)[1]
$$
which was introduced by the authors in 
section 3.8 \cite{fooo06} (= section 13 \cite{fooo06pre}).
Using the class $z_i \in H^2(X;\Z)$ the Poincar\'e dual to
$\pi^{-1}(\partial_iP)$ we put
\begin{equation}\label{for:useqtodefine}
\psi_u(z_i) = \sum_k\sum_{i=1}^m T^{\beta_i\cap
\omega/2\pi}\int_{L(u)}\mathfrak q_{1,k;\beta_i} (z_i \otimes
x^{\otimes k}).
\end{equation}
Here we put $x = \sum x_i \text{\bf e}_i$ and
the right hand side is a formal power series of $x_i$ with coefficients
in $\Lambda$.
\par
Using the description of the moduli space defining the operators
$\mathfrak q_{1,k;\beta}$
(See section \ref{sec:calcpot}.)
it is easy to see that the right hand side of (\ref{for:useqtodefine}) coincides
with the definition of $\overline z_i$ in the current case, when $X$ is Fano toric.
Extending the expression (\ref{for:useqtodefine}) to an arbitrary homology class
$z$ of arbitrary degree we obtain
\begin{equation}\label{for:useqtodefine2}
\psi_u(z) = \sum_k\sum_{\beta;\mu(\beta) = \deg z} T^{\beta \cap
\omega/2\pi}\int_{L(u)}\mathfrak q_{1,k;\beta}(z \otimes x^{\otimes
k}).
\end{equation}
Since $\mu(\beta) = \deg z$, $\mathfrak q_{1,k;\beta}(z \otimes
x^{\otimes k}) \in H^n(L(u);\Q)$. One can prove that
(\ref{for:useqtodefine2}) defines a ring homomorphism from quantum
cohomology to the Jacobian ring $Jac(\mathfrak{PO}^u)$. We may
regard $Jac(\mathfrak{PO}^u)$ as the moduli space of deformations of
Floer theories of Lagrangian fibers of $X$. (Note the Jacobian ring
parameterizes deformations of a holomorphic function up to an
appropriate equivalence. In our case the equivalence is the right
equivalence, that is, the coordinate change of the {\it domain}.)
\par
Thus (\ref{for:useqtodefine2}) is a particular case of the ring homomorphism
$$
QH(X) \to HH(\mathfrak{Lag}(X))
$$
where $HH(\mathfrak{Lag}(X))$ is the Hochschild cohomology of Fukaya category
of $X$.
(We remark that Hochschild cohomology parameterizes deformations
of $A_{\infty}$ category.)
Existence of such a homomorphism is a folk theorem (See \cite{konts:hms}) which is verified by various
people in various favorable situations. (See for example \cite{Aur07}.) It is conjectured to be an isomorphism
under certain conditions by various people including P. Seidel and M. Kontsevich.
\par
This point of view is suitable for generalizing our story to more
general $X$ (to non-Fano toric manifolds, for example) and also for
including big quantum cohomology group into our story. (We will then
also need to use the operators $\mathfrak q_{\ell,k}$ mentioned
above for $\ell\ge 2$.)
\par
These points will be discussed in subsequent papers in this series of papers.
In this paper we follow more elementary approach exploiting
the known calculation of quantum cohomology of toric manifolds,
although it is less conceptual.
\end{rem}

\begin{rem}\label{CSMT}
There are two other approaches towards a proof of Proposition \ref{prop:qsrholds}
besides the fixed point localization. One is written by Cieliebak
and Salamon \cite{ciel-sal} which uses vortex equations (gauged
sigma model) and the other is written by McDuff and Tolman
\cite{mc-tol} which uses Seidel's result \cite{seidel:auto}.
\end{rem}

\section{Localization of quantum cohomology ring at moment polytope}
\label{sec:exa2}

In this section, we discuss applications of Theorem \ref{QHequalMilnor}.
In particular, we prove Theorem \ref{lageqgeopt}.
(Note Theorem \ref{lageqgeopt0} is a consequence of Theorem \ref{lageqgeopt}.)
The next theorem and Theorem \ref{QHequalMilnor} immediately imply (1) of Theorem \ref{lageqgeopt}.

\begin{thm}\label{thm:geometricpt}
There exists a bijection
$$
\text{\rm Crit}(\mathfrak{PO}_0) \cong
Hom(Jac(\mathfrak{PO}_0);\Lambda^{\C}).
$$
\end{thm}
Here the right hand side is the set of unital $\Lambda^{\C}$-algebra homomorphisms.

We start with the following definition

\begin{defn}\label{def:valuation}
For an element $x \in \Lambda \setminus \{0\}$, we define its valuation
$\mathfrak v_T(x)$ as the unique number $\lambda \in \R$ such that
$
T^{-\lambda}x \in \Lambda_{0} \setminus \Lambda_{+}.
$
\end{defn}
\par\smallskip
We note that
$\mathfrak v_T$ is multiplicative non-Archimedean valuation, i.e., satisfies
\beastar
\mathfrak v_T(x + y) &\ge& \min (\mathfrak v_T(x) ,\mathfrak v_T(y)), \\
\mathfrak v_T(x y) &=& \mathfrak v_T(x) + \mathfrak v_T(y).
\eeastar

\begin{lem}\label{lemma:chooseu}
For any $\varphi \in Hom(Jac(\mathfrak{PO}_0);\Lambda^{\C})$ there
exists a unique $u\in M_\R$ such that
\begin{equation}
\mathfrak v_T(\varphi(y_j(u))) = 0
\label{ynorm}\end{equation}
for all $j=1,\cdots,n$.
\end{lem}
\begin{proof}
We still assume Condition \ref{cond:invertible}. By definition
(\ref{formdef:overzi}) of $\overline z_i$, homomorphism property of
$\varphi$ and multiplicative property of valuation, we obtain
\begin{equation}\label{form:normofzi}
\mathfrak v_T(\varphi(\overline z_i)) = \ell_i(u) + \sum_{j=1}^n v_{i,j} \mathfrak v_T(\varphi(y_j(u))),
\end{equation}
for $i = 1,\cdots,m$. On the other hand,
since $\ell_i(u) = \langle u, v_i \rangle - \lambda_i$
and $(v_{i,j})_{i,j=1,\cdots,n}$ is invertible, there is
a unique $u$ that satisfies
\begin{equation}\label{form:normofzi2}
\mathfrak v_T(\varphi(\overline z_i)) = \ell_i(u)
\end{equation}
for $i=1,\cdots,n$. But by the invertibility of
$(v_{i,j})_{i,j=1,\cdots,n}$ and (\ref{form:normofzi}), this is
equivalent to (\ref{ynorm}) and hence the proof.
\end{proof}

We remark that obviously by the above proof the formula (\ref{form:normofzi2})
automatically holds for $i=n+1,\cdots,m$ and $u$ in Lemma \ref{lemma:chooseu} as well.
\par

\begin{proof}[Proof of Theorem \ref{thm:geometricpt}]
Consider the maps
$$
\Psi_1(\varphi) = \sum_{i=1}^n
(\log \varphi(y_i(u))) \text{\bf e}_i \in H^1(L(u);\Lambda_{0}), \qquad
\Psi_2(\varphi) = u \in M_\R$$
where $u$ is obtained as in Lemma \ref{lemma:chooseu}.
Since $y_i(u) \in \Lambda_{0}\setminus \Lambda_{+}$ it follows that
we can define its logarithm on $\Lambda_{0}$ as a convergent power series
with respect to the non-Archimedean norm.
\par
Set $(\mathfrak x,u) = (\Psi_1(\varphi),\Psi_2(\varphi))$. Since $\varphi$ is
a ring homomorphism from $Jac(\mathfrak{PO}_0) \cong
Jac(\mathfrak{PO}_0^u)$ it follows from the definition of the
Jacobian ring that
$$
\frac{\partial \mathfrak{PO}_0^u}{\partial y_i}(\mathfrak x) = 0.
$$
Therefore by Theorem \ref{homologynonzero},
$HF((L(u),\mathfrak x),(L(u),\mathfrak x);\Lambda) \ne 0$. We have thus defined
$$
\Psi : Hom(Jac(\mathfrak{PO}_0);\Lambda^{\C}) \to \text{\rm Crit}(\mathfrak{PO}_0).
$$
Let $(\mathfrak x,u) \in \text{\rm Crit}(\mathfrak{PO}_0)$. We put
$\mathfrak x = \sum \mathfrak x_i \text{\bf e}_i$. We define a
homomorphism $\varphi
: Jac(\mathfrak{PO}_0) \to \Lambda$ by assigning
$$
\varphi(y_i(u)) = e^{\mathfrak x_i}.
$$
It is straightforward to check that $\varphi$ is well-defined. Then
we define $\Phi(\mathfrak x,u): = \varphi$. It easily follows from definition
that $\Phi$ is an inverse to $\Psi$. The proof of Theorem
\ref{thm:geometricpt} is complete.
\end{proof}

We next work with the (Batyrev) quantum cohomology side.

\begin{defn}\label{defn:opeartor} For each $z_i$,
we define a $\Lambda$-linear map
$
\widehat z_i : QH^{\omega}(X;\Lambda^{\C}) \to QH^{\omega}(X;\Lambda^{\C})
$
by
$
\widehat z_i(z) = z_i \cup_Q z,
$
where $\cup_Q$ is the product in $QH^{\omega}(X;\Lambda^{\C})$.
\end{defn}
Since $QH^{\omega}(X;\Lambda)$ is generated by even degree elements it
follows that it is commutative. Therefore we have
\begin{equation}
\widehat z_i \circ \widehat z_j = \widehat z_j \circ \widehat z_i.
\label{formula:commutez}\end{equation}
\begin{defn}\label{def;weight}
For $\mathfrak w = (\mathfrak w_1,\cdots,\mathfrak w_n) \in (\Lambda^{\C})^n$
we put
$$
QH^{\omega}(X;\mathfrak w)
= \{x \in QH^{\omega}(X;\Lambda^{\C}) \mid (\widehat z_i- \mathfrak w_i)^N x = 0 \quad \text{for $i = 1,\cdots,n$ and large $N$.}\}
$$
We say that $\mathfrak w$ is a {\it weight} of $QH^\omega(X)$
if $QH^{\omega}(X;\mathfrak w) $ is nonzero.
We denote by $W(X;\omega)$ the set of weights of $QH^{\omega}(X)$ .
\end{defn}
We remark that $\mathfrak w_i \ne 0$ since $z_i$ is invertible. (Corollary \ref{lem:zinvert}.)
\begin{rem}
Since $z_i$, $i=1,\cdots,n$ generates $QH^{\omega}(X;\Lambda)$ by Condition
\ref{cond:invertible}, we have the following.
For each $\mathfrak w = (\mathfrak w_1,\cdots,\mathfrak w_n)$, there exists $\mathfrak w_{n+1}, \cdots, \mathfrak w_{m}$ depending only on $\frak w$ such that
$
(\widehat z_i- \mathfrak w_i)^N x = 0
$
also holds for $i=n+1,\cdots,m$, if $N$ is sufficiently large and $x \in QH^{\omega}(X;\mathfrak w)$.
\end{rem}
\begin{prop}\label{thm:factorize}
\begin{enumerate}
\item There exists a factorization of the ring
$$
QH^{\omega}(X;\Lambda^{\C}) \cong \prod_{\mathfrak w \in
W(X;\omega)} QH^{\omega}(X;\mathfrak w).
$$
\par
\item
There exists a bijection
$$
W(X;\omega) \cong Hom(QH^{\omega}(X;\Lambda);\Lambda^{\C}).
$$
\par
\item
$QH^{\omega}(X;\mathfrak w)$ is a local ring and
$(1)$ is the factorization to indecomposables.
\end{enumerate}
\end{prop}
\begin{proof}
Existence of decomposition (1) as a $\Lambda^{\C}$-vector space is a standard linear algebra,
using the fact that $\Lambda^{\C}$ is an algebraically closed field.
(We will prove this fact in appendix as Lemma \ref{algclos}.)
If $z \in QH^{\omega}(X;\mathfrak w)$ and $z' \in QH^{\omega}(X;\mathfrak w')$
then
$$\aligned
(z_i- \mathfrak w_i)^N \cup_Q (z \cup_Q z') &= ((z_i- \mathfrak w_i)^N \cup_Q z) \cup_Q z' = 0, \\
(z_i- \mathfrak w'_i)^N \cup_Q (z \cup_Q z') &= ((z_i- \mathfrak w'_i)^N \cup_Q z') \cup_Q z = 0.
\endaligned$$
Therefore $z\cup_Q z' \in QH^{\omega}(X;\mathfrak w) \cap QH^{\omega}(X;\mathfrak w')$.
This implies that the decomposition (1) is a ring factorization.
\par
Let $\varphi : QH^{\omega}(X;\mathfrak w) \to \Lambda^{\C}$ be a
unital $\Lambda^{\C}$ algebra homomorphism. It induces a
homomorphism $QH^{\omega}(X;\Lambda) \to \Lambda^{\C}$ by (1). We
denote this ring homomorphism by the same letter $\varphi$. Let $z
\in QH^{\omega}(X;\mathfrak w)$ be an element such that $\varphi(z)
\ne 0$. Then we have
$$
(\varphi(z_i) - \mathfrak w_i)^N\varphi(z) = \varphi((z_i - \mathfrak w_i)^N \cup_Q z) = 0.
$$
Therefore
\begin{equation}\label{form:misvarphiz}
\mathfrak w_i = \varphi(z_i).
\end{equation}
Since $z_i$ generates $QH^{\omega}(X;\Lambda)$, it follows from
(\ref{form:misvarphiz}) that there is a unique $\Lambda^{\C}$
algebra homomorphism $: QH^{\omega}(X;\mathfrak w) \to
\Lambda^{\C}$. (2) follows.
\par
Since $QH^{\omega}(X;\mathfrak w)$ is a finite dimensional
$\Lambda^{\C}$ algebra and $\Lambda^{\C}$ is algebraically closed,
we have an isomorphism
\begin{equation}\label{form;nilradical}
\frac{QH^{\omega}(X;\mathfrak w)}{\text{rad}} \cong (\Lambda^{\C})^k
\end{equation}
for some $k$. (Here $\text{rad} = \{ z \in QH^{\omega}(X;\mathfrak
w) \mid z^N = 0 \quad \text{for some $N$.}\}$) Since there is a
unique unital $\Lambda^{\C}$-algebra homomorphism $:
QH^{\omega}(X;\mathfrak w) \to \Lambda^{\C}$, it follows that $k=1$.
Namely ${QH^{\omega}(X;\mathfrak w)}$ is a local ring.

It also implies that
${QH^{\omega}(X;\mathfrak w)}$ is indecomposable.
\end{proof}
The result up to here also works for the non-Fano case. But the next
theorem will require the fact that $X$ is Fano since we use the
equality $QH^{\omega}(X;\Lambda) \cong QH(X;\Lambda)$.
\begin{thm}\label{thm:mismplus}
If $X$ is Fano then $\mbox{\rm Crit}(\frak{PO}_0) = \mathfrak M(\mathfrak{Lag}(X))$.
\end{thm}
\begin{proof}
Let $\mathfrak w$ be a weight. We take $z \in QH^{\omega}(X;\mathfrak w) \subset H(X;\Lambda^{\C})
\cong H(X;\C) \otimes \Lambda^{\C}$.
We may take $z$ so that
$$
z \in (H(X;\C) \otimes \Lambda_{0}^{\C})\,\, \setminus (H(X;\C) \otimes \Lambda_{+}^{\C}).
$$
Since
$$
z_i \cup_Q z \equiv z_i \cup z \mod \Lambda_{+}^{\C},
$$
where $\cup$ is the classical cup product.
(We use $QH^{\omega}(X;\Lambda) = QH(X;\Lambda)$ here.)
It follows that
$$
\mathfrak w_i^nz = (\widehat z_i)^n(z) = (z_i)^n \cup_Q z
\equiv (z_i)^n \cup z \mod \Lambda_{+}^{\C}.
$$
Therefore $\mathfrak w_i \in \Lambda_{+}^{\C}$ as $(z_i)^n \cup z = 0$.
(\ref{form:normofzi2}) and (\ref{form:misvarphiz}) then imply
$$
\ell_i(u) = \mathfrak v_T(\mathfrak w_i) > 0.
$$
Namely $u \in \text{Int} P$.
\end{proof}
We are now ready to complete the proof of Theorem \ref{lageqgeopt}.
(1) is Theorem \ref{thm:geometricpt}. (2)
is a consequence of Theorem \ref{homologynonzero}. (3) is Theorem
\ref{thm:mismplus}. If $QH^{\omega}(X;\Lambda^{\C})$ is semi-simple,
then (\ref{form;nilradical}) and $k=1$ there implies
\begin{equation}\label{form;decompsimplecase}
QH^{\omega}(X;\Lambda^{\C}) \cong (\Lambda^{\C})^{\# W(X;\omega)}
\end{equation}
as a $\Lambda^{\C}$ algebra. (4) follows from
(\ref{form;decompsimplecase}), Proposition \ref{thm:factorize} (2),
and Theorem \ref{thm:geometricpt}. The proof of Theorem
\ref{lageqgeopt} is complete. \qed\par\medskip
We next explain the
factorization in Proposition \ref{thm:factorize} (1) from the point
of view of Jacobian ring. Let $(\mathfrak x,u) \in \text{\rm Crit}(\mathfrak{PO}_0)$.
\begin{defn}\label{def:locjac}
We consider the ideal generated by
\begin{equation}\label{jacidealgene}
\frac{\partial }{\partial w_i}\mathfrak{PO}_0^u(\frak y_1+w_1,\cdots,\frak y_n+w_n)
\end{equation}
$i=1,\cdots,n$, in the ring $\Lambda[[w_1,\cdots,w_n]]$ of formal
power series where $\mathfrak x = \sum \mathfrak x_i \text{\bf e}_i$
and $\mathfrak y_i =
e^{\mathfrak x_i}$. We denote its quotient ring by $Jac(\mathfrak{PO}_0 ;
\mathfrak x,u)$.
\end{defn}
\begin{prop}
\begin{enumerate}
\item There is a direct product decomposition :
$$
Jac(\mathfrak{PO}_0) \cong \prod_{(\mathfrak x,u) \in \text{\rm Crit}(\mathfrak{PO}_0)} Jac(\mathfrak{PO}_0 ; \mathfrak x,u),
$$
as a ring.
\par
\item If $(\mathfrak x,u) \in \text{\rm Crit}(\mathfrak{PO}_0)$ corresponds to
$\mathfrak w \in W(X;\omega)$ via the isomorphism given in
Proposition $\ref{thm:factorize}$ $(2)$ and Theorem
$\ref{thm:geometricpt}$, then $\psi_u$ induces an isomorphism
$$
\psi_u : QH^{\omega}(X;\mathfrak w) \cong Jac(\mathfrak{PO}_0 ; \mathfrak x,u).
$$
\par
\item $Jac(\mathfrak{PO}_0 ; \mathfrak x,u)$ is one dimensional (over $\Lambda$) if and only if
the Hessian
$$
\left(\frac{\partial^2 \mathfrak{PO}_0^u}{\partial y_i\partial y_j}\right)_{i,j=1,\cdots,n}
$$
is invertible over $\Lambda$ at $\mathfrak x$.
\end{enumerate}
\end{prop}
\begin{proof}
We put $\frak x = \sum \frak x_i \text{\bf e}_i$ and
$\frak y_i = e^{\frak x_i}$.
Let $\mathfrak m(\mathfrak x,u)$ be the ideal generated by $y_i - \frak y_i$, in
the ring :
$$
Jac(\mathfrak{PO}_0) = \frac{\Lambda[y_1^{\pm},\cdots,y_n^{\pm}]}{(y_i \frac{\partial \frak{PO}_0^u}{\partial y_i})}.
$$ Since $Jac(\mathfrak{PO}_0)$ is finite
dimensional over $\Lambda$ it follows that
$$
Jac(\mathfrak{PO}_0) \cong \prod_{(\mathfrak x,u) \in \text{\rm Crit}(\mathfrak{PO}_0)} Jac(\mathfrak{PO}_0)_{\mathfrak m(\mathfrak x,u)},
$$
where $Jac(\mathfrak{PO}_0)_{\mathfrak m(\mathfrak x,u)}$ is the localization of the ring $Jac(\mathfrak{PO}_0)$ at $\mathfrak m(\mathfrak x,u)$.
Using finite dimensionality of $Jac(\mathfrak{PO}_0)$ again we have
an isomorphism
$Jac(\mathfrak{PO}_0)_{\mathfrak m(\mathfrak x,u)} \cong Jac(\mathfrak{PO}_0 ; \mathfrak x,u)$
which sends $y_i - \frak y_i$ to $w_i$. Here $\frak x = \sum_i \frak x_i \text{\bf e}_i$
and $\frak y_i = e^{\frak x_i}$
(1) follows.
\par
Now we prove (2). If $z \in QH^{\omega}(X;\mathfrak w)$ then $(\widehat z_i -
\mathfrak w_i)^N z = 0$. Let $\pi_{\mathfrak x,u} : Jac(\mathfrak{PO}_0) \to
Jac(\mathfrak{PO}_0 ; \mathfrak x,u)$ be the projection. We then have
\begin{equation}\label{eq:compeigen}
(T^{\ell_i(u)} \mathfrak y_1^{v_{i,1}}\cdots \mathfrak y_n^{v_{i,n}} - \mathfrak w_i)^N
\pi_{\mathfrak x,u}(\psi_u(z)) = 0.
\end{equation}
We remark that
\begin{equation}\label{eq:yandyprime}
\mathfrak w_i = T^{\ell_i(u')} \mathfrak y_1^{\prime v_{i,1}} \cdots
\mathfrak y_n^{\prime v_{i,n}}
\end{equation}
if $\mathfrak w_i$ corresponds $(\frak x',u')$ and $\mathfrak y'_i$ are exponential
of the coordinates of $\frak x'$. We define the operator $\widehat{\mathfrak y}_i :
Jac(\mathfrak{PO}_0 ; \frak x,u) \to Jac(\mathfrak{PO}_0 ; \frak x,u)$ by
$$\widehat{\mathfrak y}_i(c) = {\mathfrak y}_i c.
$$
By definition of $Jac(\mathfrak{PO}_0 ; \mathfrak x,u)$ the eigenvalue of
$\widehat{\mathfrak y}_i$ is $\mathfrak y_i$.
Therefore (\ref{eq:compeigen}) and (\ref{eq:yandyprime}) imply that
$\pi_{\mathfrak x,u}(\psi_u(z)) = 0$ unless $(\mathfrak x,u) = (\mathfrak x',u')$. (2) follows.
\par
Let us prove
(3). We first remark $\Lambda = \Lambda^{\C}$ is an algebraically closed field
(Lemma \ref{algclos}).
Therefore $\dim_{\Lambda} Jac(\mathfrak{PO}_0 ; \mathfrak x,u) = 1$
if and only if the ideal generated by (\ref{jacidealgene}) (for $i = 1,\cdots,n$)
is the maximal ideal $\frak m = (w_1,\cdots,w_n)$.
We remark that $\frak m/\frak m^2 = \Lambda^n$ and elements
(\ref{jacidealgene}) reduces to
$$
\left(\frac{\partial^2 \mathfrak{PO}_0^u}{\partial y_i\partial y_j}\right)_{j=1,\cdots,n}
\in \Lambda^n,
$$
modulo $\frak m^2$.
(3) follows easily.
\end{proof}

We recall that a symplectic manifold $(X,\omega)$ is said to be
(spherically) monotone if there exists $\lambda >0$ such that
$c_1(X) \cap \alpha = \lambda \,[\omega] \cap \alpha$ for all
$\alpha \in \pi_2(X)$. Lagrangian submanifold $L$ of $(X,\omega)$ is
said to be monotone if there exists $\lambda >0$ such that
$\mu(\beta) = \lambda\omega(\beta)$ for any $\beta \in \pi_2(X,L)$.
(Here $\mu$ is the Maslov index.) In the monotone case we have the
following :

\begin{thm}\label{thm:monotone}
If $X$ is a monotone compact toric manifold then there exists
a unique $u_0$ such that
$$
\mathfrak M(\mathfrak{Lag}(X)) \subset \Lambda  \times\{u_0\}
$$
i.e., whenever $(\mathfrak x,u) \in \mathfrak M(\mathfrak{Lag}(X))$, $u = u_0$.
Moreover $L(u_0)$ is monotone.
\end{thm}
\begin{rem}
Related results are discussed in \cite{entov-pol06}.
\end{rem}
\begin{proof}
Since $X$ is Fano, we have $QH^{\omega}(X;\Lambda) = QH(X;\Lambda)$.
We assume $c_1(X) \cap \alpha = \lambda \,[\omega] \cap \alpha$ with $\lambda >0$.
Let $\cup_{\alpha}$ be the contribution to the moduli space of
pseudo-holomorphic curve of homology class $\alpha \in H_2(X;\Z)$ in the
quantum cup product. (See (\ref{qcohdef}).) We have a decomposition :
$$
x \cup_Q y = x \cup y + \sum_{\alpha \in \pi_2(X)\setminus \{0\}}
T^{\alpha\cap [\omega]/2\pi} x \cup_{\alpha} y.
$$
Then
\begin{equation}\label{eq:degpres}
\deg(x \cup_{\alpha} y) = \deg x + \deg y - 2 c_1(X) \cap \alpha
= \deg x + \deg y - 2 \lambda \alpha \cap [\omega] .
\end{equation}
We define
$$
\mathfrak v_{\deg}(T^{1/2\pi}) = 2\lambda, \quad \mathfrak v_{\deg}(x)
= \deg x \quad (\text{for $x \in H(X;\Q)$}).
$$
$\mathfrak v_{\deg}$ is a multiplicative non-Archimedean valuation on $QH(X;\Lambda)$ such that
$
\mathfrak v_{\deg}(a\cup_Q b) = \mathfrak v_{\deg}(a) + \mathfrak v_{\deg}(b),
$
by virtue of (\ref{eq:degpres}). Moreover for $c \in \Lambda$ and $a\in QH(X;\Lambda)$
we have
$
\mathfrak v_{\deg}(ca) = 2\lambda\mathfrak v_T(c) + \mathfrak v_{\deg}(a).
$
Now let $\mathfrak w$ be a
weight and
$x \in QH^{\omega}(X;\mathfrak w)$. Since $\mathfrak v_{\deg}(z_i) = 2$ it follows that
$$
2\lambda\mathfrak v_T(\mathfrak w_i) + \mathfrak v_{\deg}(x) = \mathfrak v_{\deg}(z_i x) = 2 + \mathfrak v_{\deg}(x).
$$
Therefore if $(\mathfrak x,u)$ corresponds to $\mathfrak w$ then
$\ell_i(u) = \mathfrak v_T(\mathfrak w_i) = 1/\lambda$.
Namely $u$ is independent of $\mathfrak w$.
We denote it by $u_0$.
\par
For $\beta_i \in H_2(X,L(u_0))$ ($i=1,\cdots,m$) given by (\ref{def:betai}),
we have $\omega(\beta_i) = 2\pi\ell_i(u_0) = 2\pi/\lambda$.
Hence $\mu(\beta_i) = \lambda \omega(\beta_i)/\pi$. Since $\beta_i$ generates $H_2(X,L(u_0))$,
it follows that $L(u_0)$ is monotone, as required.
\end{proof}

So far we have studied Floer cohomology with
$\Lambda^{\C}$-coefficients. We next consider the case of
$\Lambda^F$ coefficient where $F$ is a finite Galois extension of
$\Q$. We choose $F$ so that each of the weight $\mathfrak w$ lies in
$(\Lambda^F_0)^n$. (Since every finite extension of $\Lambda^{\Q}$
is contained in such $\Lambda^F$ we can always find such an $F$.
See appendix.)
Then we have a decomposition
\begin{equation}
QH^{\omega}(X;\Lambda^F) \cong
\prod_{\mathfrak w \in W(X;\omega)} QH^{\omega}(X;\mathfrak w;F).
\end{equation}
It follows that the Galois group $Gal(F/\Q)$ acts on $W(X;\omega)$.
It induces a $Gal(F/\Q)$ action on $\text{\rm Crit}(\mathfrak{PO}_0)$.
(We use Remark \ref{5.7rational} here.) We write it as $(\mathfrak
x,u) \mapsto (\sigma(\mathfrak x),\sigma(u))$. We remark the
following:
\begin{prop}\label{prop:galois}
\begin{enumerate}
\item $\sigma(u) = u$.
\par
\item We write by $y_i(\mathfrak x)$ the exponential of the coordinates of $\mathfrak x$.
Then $y_i(\mathfrak x) \in \Lambda^F$ and $y_i(\sigma(\mathfrak x)) = \sigma(y_i(\mathfrak x))$.
\par
\item If $QH^{\omega}(X;\Lambda^{\Q})$ is indecomposable, there exists
$u_0$ such that whenever $(\mathfrak x,u) \in \text{\rm Crit}(\mathfrak{PO}_0))$, $u = u_0$.
\end{enumerate}
\end{prop}
\begin{proof}
Let $\mathfrak w_i(\mathfrak x)$ corresponds to $(\mathfrak x,u)$. Then
$$
\ell_i(\sigma(u)) = \mathfrak v_T(\mathfrak w_i(\sigma(\mathfrak x,u)))
= \mathfrak v_T(\sigma\mathfrak w_i(\mathfrak x,u))
= \mathfrak v_T(\mathfrak w_i(\mathfrak x,u)) = \ell_i(u).
$$
(1) follows.
(2) follows from the definition and (1). (3) is a consequence of (1).
\end{proof}
A monotone blow up of $\C P^2$ (at one point) gives an
example where the assumption of Proposition \ref{prop:galois} (3) is
satisfied.
\par
It seems interesting to observe that the ring $QH(X;\Lambda^{\Q})$ jumps
sometimes when we deform symplectic structure of $X$.
The point where this jump occurs is closely related to the
point where the number of balanced Lagrangian fibers jumps.
In the case of Example \ref{37.116} we have
$$
QH(X;\Lambda^{\Q})
\cong
\begin{cases}
(\Lambda^{\Q})^5 & \alpha > 0,\\
(\Lambda^{\Q})^3 \times \Lambda^{\Q(\sqrt{-3})} & \alpha < 0, \\
\Lambda^{\Q(\sqrt{5})} \times \Lambda^{F} & \alpha =0,
\end{cases}
$$
where
$F = \Q [x]/(x^3-x-1).
$
We remark that $x^5+x^4-2x^3-2x^2+1 = (x^2+x-1)(x^3-x-1)$.
\par
We also refer readers to Example \ref{hidimexa} for further example.
\begin{rem}
In sections \ref{sec:calcpot} - \ref{sec:floerhom}, we will use
de Rham cohomology of Lagrangian submanifold to define and study
Floer cohomology. As a consequence, our results on Floer
cohomology is proved over $\Lambda^{\R}_0$ or $\Lambda^\C_0$ but not
over $\Lambda^{\Q}_0$ or $\Lambda^{F}_0$.
(The authors believe that those results can be also proved over $\Lambda^{\Q}_0$
by using the singular cohomology version developed in \cite{fooo06},
although the detail of their proofs could be more complicated.)
\par
On the other hand, Proposition \ref{nonfanoadded} and Theorem
\ref{QHequalMilnor2} are proved over $\Lambda^{\Q}_0$.
Therefore the discussion on quantum cohomology here
works over $\Lambda^{F}_0$.
\par
We also remark that, though Proposition \ref{prop:galois} (3) is
related to Floer cohomology, its proof given above does {\it not}
use Floer cohomology over $\Lambda^{\Q}_0$ but uses only
Floer cohomology over $\Lambda^{\C}_0$ and quantum cohomology
over $\Lambda^{\Q}_0$.
In fact, the proof above implies the following :
If $u \in \text{\rm Int} P$ and $\mathfrak x \in H^1(L(u);\Lambda^{\C}_0)$
satisfy
$$
\frac{\partial \frak{PO}^u_0}{\partial x_i}(\frak x) = 0
$$
then $u=u_0$. This is because
$$Jac(\frak{PO}^u_0;\Lambda^{\C}_0) = Jac(\frak{PO}^u_0;\Lambda^{\Q}_0)\otimes_{\Lambda^{\Q}_0}
\Lambda^{\C}_0$$ and $$Jac(\frak{PO}^u_0;\Lambda^{\Q}_0) \cong QH(X;\Lambda^{\Q}).$$
\end{rem}

\section{Further examples and remarks}
\label{sec:exam2}

In this section we show how we can use the argument of the last 2
sections to illustrate calculations of $\mathfrak
M(\mathfrak{Lag}(X))$ in examples.
\begin{exm}\label{exa:onepointblow}
We consider one point blow up $X$ of $\C P^2$. We choose its
K\"ahler form so that the moment polytope is
$$
P = \left\{ (u_1,u_2) \mid 0 \le u_1,u_2, \,\, u_1+u_2 \le 1, \,\, u_2 \le 1-\alpha\right\},
$$
$0 < \alpha <1$. The potential function is
$$
\mathfrak{PO} = y_1T^{u_1} + y_2 T^{u_2} + (y_1y_2)^{-1} T^{1-u_1-u_2} + y_2^{-1}T^{1-\alpha-u_2}.
$$
We put $\overline z_1 = y_1T^{u_1}$, $\overline z_2 = y_2 T^{u_2}$,
$\overline z_3 =(y_1y_2)^{-1} T^{1-u_1-u_2}$, $\overline z_4 =y_2^{-1}T^{1-\alpha-u_2}$.
\par
The quantum Stanley-Reisner relation is
\begin{equation}\label{eq:qsr}
\overline z_1\overline z_3 = \overline z_4 T^{\alpha},
\quad \overline z_2\overline z_4 = T^{1-\alpha},
\end{equation}
and linear relation is
\begin{equation}\label{eq:lin}
\overline z_1 - \overline z_3 = 0, \quad \overline z_2 - \overline z_3 - \overline z_4 = 0.
\end{equation}
We put $X = \overline z_1$ and $Y = \overline z_2$ and solve (\ref{eq:qsr}), (\ref{eq:lin}).
We obtain
\begin{equation}\label{eq:forX}
X^3(T^{\alpha} + X) = T^{1+\alpha},
\end{equation}
with $Y = X + T^{-\alpha} X^2$. We consider valuations of both sides
of (\ref{eq:forX}).
There are three different cases to consider.
\par\medskip
\noindent{\bf Case 1;} $\mathfrak v_T(X) > \alpha$ : (\ref{eq:forX})
implies $3\mathfrak v_T(X) + \alpha = 1+\alpha$. Namely $\mathfrak
v_T(X) = 1/3$. So $\alpha < 1/3$. Moreover $\mathfrak v_T(Y) = 1/3$.
We have $u_1 = \mathfrak v_T(X) = 1/3$, $u_2 = \mathfrak v_T(Y) =
1/3$. (See Lemma \ref{lemma:chooseu}.) Writing $X = a_1 T^{1/3} +
a_2 T^\lambda + $ higher order terms with $\lambda > \frac{1}{3}$
and substituting this into (\ref{eq:forX}), we obtain $a_1^3 = 1$
which has 3 simple roots. Each of them corresponds to the solution
for $\mathfrak x$ by Hensel's lemma (see Proposition 3 in p 144
\cite{BGR}, for example). (It also follows from Theorem
\ref{thm:elliminate} in section \ref{sec:ellim}.)
\par\smallskip
\noindent{\bf Case 2;} $\mathfrak v_T(X) < \alpha$ : By taking the
valuation of (\ref{eq:forX}) we obtain $u_1 = \mathfrak v_T(X) =
(1+\alpha)/4$. Hence $\alpha > 1/3$. Moreover $u_2 = \mathfrak
v_T(Y) = (1-\alpha)/2$. In the same way as Case 1, we can check that there are four solutions.
\par\smallskip
\noindent{\bf Case 3;} $\mathfrak v_T(X) = \alpha$ :
We put $X = a_1 T^{\alpha} + a_2 T^{\lambda} +$
higher order terms where $\lambda > \alpha$.
\noindent(Case 3-1: $a_1 \ne -1$) :
By taking valuation of (\ref{eq:forX}), we obtain $u_1 = \mathfrak v_T(X) = 1/3$.
Then $\alpha = 1/3$ and $u_2 = \mathfrak v_T(Y) = 1/3$. (\ref{eq:forX}) becomes
\begin{equation}
a_1^4 + a_1^3 - 1 = 0.
\label{eq;thridordera}\end{equation}
(In this case $X = a_1T^{\alpha}$ has no higher term.) There are four solutions. We remark that
(\ref{eq;thridordera}) is irreducible over $\Q$, since it is so
over $\Z_2$. Namely the assumption of
Proposition \ref{prop:galois} (3) is satisfied.
Actually $X$ is monotone in the case $\alpha =1/3$. Hence the
same conclusion (uniqueness of $u$) follows from Theorem \ref{thm:monotone}
also.
\par
\noindent(Case 3-2: $a_1 = -1$) :
By taking valuation of (\ref{eq:forX}), we obtain $\lambda = 1-2\alpha$.
$\lambda>\alpha$ implies $\alpha < 1/3$. $u_2 = \mathfrak v_T(Y) = 1-2\alpha$.
($u_1 = \mathfrak v_T(X) = \alpha$.) There is one solution.
\par\medskip
In summary, if $\alpha < 1/3$ there are two choices of $u =
(\alpha,1-2\alpha), (1/3,1/3)$. On the other hand the numbers of choices of $\mathfrak x$ are
$1$ and $3$ respectively.
\par
If $\alpha \ge 1/3$ there is the unique choice $u =
((1+\alpha)/4,(1-\alpha)/2)$. The number of choices of $\mathfrak x$ is $4$.
\end{exm}
We next study a non-Fano case. We will study Hirzebruch surface
$F_n$. Note $F_1$ is one point blow up of $\C P^2$ which we have
already studied. We leave the case $F_2$ to the reader.

\begin{exm}\label{ex:Hilexa}
We consider Hirzebruch surface $F_n$, $n\ge 3$. We take its
K\"ahler form so that the moment polytope is
$$
P = \left\{ (u_1,u_2) \mid 0 \le u_1,u_2, \,\, u_1+nu_2 \le n, \,\, u_2 \le 1-\alpha\right\},
$$
$0 < \alpha <1$. The leading order potential function is
$$
\mathfrak{PO}_0 = y_1T^{u_1} + y_2 T^{u_2} + y_1^{-1}y_2^{-n} T^{n-u_1-nu_2} + y_2^{-1}T^{1-\alpha-u_2}.
$$
We put $\overline z_1 = y_1T^{u_1}$, $\overline z_2 = y_2 T^{u_2}$,
$\overline z_3 =y_1^{-1}y_2^{-n} T^{n-u_1-nu_2}$, $\overline z_4 =y_2^{-1}T^{1-\alpha-u_2}$.
\par
The quantum Stanley-Reisner relation and linear relation gives
\begin{eqnarray}
\overline z_1\overline z_3 = \overline z_4^n T^{n\alpha},
&\quad \overline z_2\overline z_4 = T^{1-\alpha},
\label{eq:qsrhil}\\
\overline z_1 - \overline z_3 = 0, &\quad \overline z_2 - n\overline z_3 - \overline z_4 = 0.
\label{eq:linhil}\end{eqnarray}
Let us assume $n$ is odd. We put
$$
\overline z_1 = Z^n, \qquad \overline z_4 = Z^2 T^{-\alpha}.
$$
(In case $n = 2n'$ is even we put $\overline z_1 = Z^{n'}$, $\overline z_4 =
\pm Z T^{-\alpha}. $ The rest of the argument are similar and is
omitted.) Then $\overline z_2 = T^{-\alpha}Z^2 + nZ^n$ and
\begin{equation}\label{eq;nordhilt}
Z^4 (n Z^{n-2} + T^{-\alpha}) = T.
\end{equation}
\par\smallskip
\noindent
{\bf Case 1 ;} $(n-2)\mathfrak v_T(Z) > -\alpha$ :
In the first case, we have $\mathfrak v_T(Z) = (\alpha+1)/4$. (Then
$(n-2)\mathfrak v_T(Z) > -\alpha$ is automatically satisfied.)
Therefore $u_1 = \mathfrak v_T(z_1) = n(\alpha+1)/4$, $u_2 = \mathfrak
v_T(z_2) = (1-\alpha)/2$. We also can check that there are $4$
solutions. We remark that we are using $\mathfrak{PO}_0$ in place of
$\mathfrak{PO}$. However we can use Corollary \ref{cor:POPO0}
to prove the following lemma.
This lemma in particular implies that $L(n(\alpha+1)/4,(1-\alpha)/2)$ is balanced
which was already shown above in Example \ref{exa:onepointblow} for the case $n=1$.

\begin{lem}\label{lem:hilellim}
Let $y^{(i)} \in \Lambda_0 \times \Lambda_0$ ($i=1,\cdots,4$) be
critical points $\frak{PO}_0^u$ for
$u = (n(\alpha+1)/4,(1-\alpha)/2)$. Then there exists
$y^{(i) \prime} \in \Lambda_0 \times \Lambda_0$ which is a critical
point of $\frak{PO}^u$ and $y^{(i)} \equiv y^{(i) \prime} \mod \Lambda_+$.
\end{lem}
We will prove Lemma \ref{lem:hilellim} in section \ref{sec:ellim}.
\par\smallskip
\noindent{\bf Case 2 ;} $(n-2)\mathfrak v_T(Z) < -\alpha$ :
We have $v_T(Z) = 1/(n+2)$. This can never occur
since $1/(n+2) > 0 > -\alpha/(n-2)$.
\par\smallskip
\noindent{\bf Case 3 ;} $(n-2)\mathfrak v_T(Z) = -\alpha$ :
We put $Z = a_1 T^{-\alpha/(n-2)} + a_2 T^{\lambda} + $ higher
order term.
\par\smallskip
\noindent(Case 3-1 : $na_1^{n-2} \ne -1$) :
Then $\mathfrak v_T(Z) = (\alpha+1)/4$. Since $(\alpha+1)/4 \ne -\alpha/(n-2)$,
this case never occur.
\par\smallskip
\noindent(Case 3-2 : $na_1^{n-2} = -1$) :
We have
$
4 v_T(Z) + (n-3)v_T(Z) + \lambda = 1.
$
Therefore
$$
\lambda = \frac{n-2 + (n+1)\alpha}{n-2}.
$$
We have
$$
u_1 = v_T(z_1) = - \frac{n\alpha}{n-2}, \qquad
u_2 = v_T(z_2) = 1 - \alpha - v_T(z_4) = \frac{n-2+2\alpha}{n-2}.
$$
Thus $(u_1,u_2)$ is {\it not} in the moment polytope.
\end{exm}
In Example \ref{ex:Hilexa}, we have
$$
\mathfrak M(\mathfrak{Lag}(X)) = \mathfrak M_0(\mathfrak{Lag}(X)) \ne \text{\rm Crit}(\mathfrak{PO}_0).
$$
On the other hand,
the order of $\mathfrak M(\mathfrak{Lag}(X))$ is $4$ and is equal to the sum of Betti numbers.
\begin{rem}\label{conj:ranknumber}
In a sequel of this series of papers, we will prove the equality
$$
\sum_d \text{\rm rank}\, H_d(X;\Q) = \# (\mathfrak
M(\mathfrak{Lag}(X)))
$$
for any compact toric manifold  $X$ (which is not necessarily Fano) such that
$QH(X;\Lambda)$ is semi-simple.
If we count the right hand side with multiplicity, the same equality
holds without assuming semi-simplicity.
\end{rem}

We next discuss the version of the above story where we substitute some explicit
number into the formal variable $T$.
Let $u\in \mbox{Int}P$. We define a Laurent polynomial
$$
\mathfrak{PO}^u_{0,T=t} \in \C[y_1,\cdots,y_n,y_1^{-1},\cdots,y_n^{-1}]
$$
by substituting a complex number $t\in \C \setminus \{0\}$.
In the same way we define the algebra $QH^{\omega}(X;T=t;\C)$ over $\C$ by substituting $T=t$ in
the quantum Stanley-Reisner relation.
The argument of section \ref{sec:milquan} goes through to show
\begin{equation}\label{eq:finiteT}
QH^{\omega}(X;T=t;\C) \cong Jac(\mathfrak{PO}^u_{0,T=t}).
\end{equation}
In particular the right hand side is independent of $u$ up to an isomorphism.
Here the $\C$-algebra in the right hand side of (\ref{eq:finiteT})
is the quotient of the polynomial ring
$\C[y_1,\cdots,y_n,y_1^{-1},\cdots,y_n^{-1}]$ by the ideal generated by
$\partial \mathfrak{PO}_{0,T=t}/\partial
y_i$. ($i=1,\cdots,n$.)
\par
We remark that right hand side of (\ref{eq:finiteT}) is always nonzero,
for small $t$, by Proposition \ref{existcrit}.
It follows that the equation
\begin{equation}\label{eq:finiteTcrieq}
\frac{\partial \mathfrak{PO}^u_{0,T=t}}{\partial y_i} = 0
\end{equation}
has a solution $y_i\ne 0$ for {\it any} $u$. Namely, as far as the
Floer cohomology after $T=t$ substituted, there {\it always} exists $b \in
H^1(X;\C)$ with nonvanishing Floer cohomology
$
HF((L(u),b),(L(u),b);\C)
$
for any $u \in \mbox{Int} P$.
Since the
version of Floer cohomology after substituting $T=t$ is {\it not}
invariant under the Hamiltonian isotopy, this is not useful for the
application to symplectic topology. (Compare this with section 14.2
\cite{cho-oh}.)
\par
The relation between the set of solutions of (\ref{eq:finiteTcrieq}) and that
of (\ref{formula:critical}) is stated as follows : Let
$(y^{(c)}_1(t;u),\cdots,y^{(c)}_n(t;u))$ be a branch of the
solutions of (\ref{eq:finiteTcrieq}) for $t \ne 0$ where
$c$ is an integer with $1 \leq c \leq l$ for some $l \in \N$.
We can easily
show that it is a holomorphic function of $t$ on $\C \setminus
\R_{-}$. We consider its behavior as $t\to 0$. For generic $u$ the
limit either diverges or converges to $0$. However if
$(\frak x,u) \in \mathfrak M_0(\mathfrak{Lag}(X))$ and
$\mathfrak x = \sum x_i
\text{\bf e}_i$ then there
is some $c$ such that
$$
\lim_{t\to 0} y^{(c)}_i(t;u) \in \C \setminus \{0\}
\quad
\text{and that}
\quad
y^{(c)}_i(t;u) = e^{x_i(t)}.
$$
\par
To prove this claim it suffices to show that if $(\mathfrak x,u) \in
\text {\rm Crit}(\mathfrak{PO}_0)$ and $\mathfrak x = \sum x_i \text{\bf
e}_i$, $y_i = e^{x_i}$, $y_i = \sum_j y_{ij} T^{\lambda_{ij}}$ then
$ \sum_j y_{ij} t^{\lambda_{ij}} $ converges for $0 < \vert t\vert <
\epsilon$. (Here $\epsilon$ is sufficiently small positive number.)
This follows from the following Lemma \ref{convclo}. Let
$\Lambda_0^{conv}$ be the ring
$$
\left\{ \sum_i a_{i} T^{\lambda_{i}} \in \Lambda_0^{\C} \,\left\vert\,
\exists \epsilon > 0 \,\,\text{such that}\,\, \sum_i \vert a_{i}\vert \vert t\vert^{\lambda_{i}} \,\,
\text{\rm converges for $\vert t\vert < \epsilon$.}\right\}\right.
$$
and $\Lambda^{conv}$ be its field of fractions.
We put $\Lambda_+^{conv} = \Lambda_0^{conv} \cap \Lambda_+$.
\begin{lem}\label{convclo}
The field $\Lambda^{conv}$ is algebraically closed.
\end{lem}
We will prove Lemma \ref{convclo} in section \ref{sec:appendix}.
\par\medskip
We go back to the discussion on the difference between two sets $\mathfrak M_0(\mathfrak{Lag}(X))$
and $\text{\rm Crit}(\mathfrak{PO}_0)$. (See Definition \ref{frakM} for its
definition.)
The rest of this section owes much to the discussion with
H. Iritani and also to his papers \cite{iritani1}, \cite{iritani2}.
The results we describe below will not be used in the other part of this paper.
\par
We recall that we did {\it not} take closure of the
ideal $(P(X) + SR_{\omega}(X))$ in section \ref{sec:milquan}. This
is actually the reason why we have $\mathfrak M_0(\mathfrak{Lag}(X))
\ne \text{\rm Crit}(\mathfrak{PO}_0)$. More
precisely we have the following Proposition \ref{prop:closuerideal}.
\par
We consider the polynomial ring $\Lambda[z_1,\cdots,z_m]$.
We define its norm $\Vert \cdot \Vert$ so that
$$
\left\Vert\sum_{\vec i} a_{\vec i} z_1^{i_1}\cdots z_m^{i_m}
\right\Vert = \exp\left(- \inf_{\vec i} \mathfrak v_T(a_{\vec
i})\right).
$$
We take the closure of the ideal
$(P(X) + SR_{\omega}(X))$ with respect to this norm
and denote it by $\text{\rm Clos}(P(X) + SR_{\omega}(X))$.
We put
\begin{equation}\label{eq:closure}
\overline{QH}^{\omega}(X;\Lambda)
=
\frac{\Lambda[z_1,\cdots,z_m]}{\text{\rm Clos}(P(X) + SR_{\omega}(X))}.
\end{equation}
Let $W^{\text{\rm geo}}(X;\omega)$ be the set of all weight such that
the corresponding $(\mathfrak x,u)$ satisfies $u \in \text{\rm Int}\,P$.
We remark that $\mathfrak w \in W^{\text{\rm geo}}(X;\omega)$ if and only if
$\mathfrak v_T(\mathfrak w_i) > 0$ for all $i$.

\begin{prop}[Iritani]\label{prop:closuerideal}
There exists an isomorphism
$$
\overline{QH}^{\omega}(X;\Lambda^{\C}) \cong \prod_{\mathfrak w \in W^{\text{\rm geo}}(X;\omega)}
QH^{\omega}(X;\mathfrak w).
$$
\end{prop}
\begin{proof}
Let $\mathfrak w \in W(X;\omega) \setminus W^{\text{\rm geo}}(X;\omega)$.
We first assume $\mathfrak v_T(\mathfrak w_i) = - \lambda < 0$.
(The case $\mathfrak v_T(\mathfrak w_i) = 0$ will be discussed at the end of the proof.)
\par
Then, there exists $f \in \Lambda_0 \setminus \Lambda_+$ such that
$T^{\lambda}f \mathfrak w_i = 1$.
Let $x \in QH^{\omega}(X;\mathfrak w)$.
We assume $x \ne 0$. We take $k$ such that $(\widehat z_i-\frak w_i)^k x \ne 0$,
$(\widehat z_i-\frak w_i)^{k+1} x = 0$ and replace $x$ by $(\widehat z_i-\frak w_i)^k x$.
We then have $T^{\lambda}f \widehat z_i x = x$. Since
$\lim_{N\to\infty}\Vert (fz_iT^{\lambda})^N\Vert = 0$, it follows
that $x=0$ in $\overline{QH}^{\omega}(X;\Lambda^{\C})$.
This is a contradiction.
\par
We next assume $\mathfrak v_T(\mathfrak w_i) > 0$ for all $i$.
We consider the homomorphism
$$
\varphi : \Lambda[z_1,\cdots,z_m] \to Hom_{\Lambda}(QH^{\omega}(X;\mathfrak w),QH^{\omega}(X;\mathfrak w)),
$$
defined by
$$
\varphi(z_i)(x) = z_i\cup_Q x.
$$
We have $\varphi(P(X) + SR_{\omega}(X)) = 0$.
We may choose the basis of
$QH^{\omega}(X;\mathfrak w)$ so that $\varphi(z_i)$ is upper triangular matrix
whose diagonal entries are all $\mathfrak w_i$ and whose off diagonal entries are all $0$ or $1$.
We use it and $\mathfrak v_T(\mathfrak w_i) > 0$ to
show that $\varphi(\text{\rm Clos}(P(X) + SR_{\omega}(X))) = 0$.
Namely $\varphi$ induces a homomorphism
from $\overline{QH}^{\omega}(X;\Lambda)$.
It follows easily that the restriction of the projection
${QH}^{\omega}(X;\Lambda^{\C}) \to \overline{QH}^{\omega}(X;\Lambda^{\C})$
to $QH^{\omega}(X;\mathfrak w)$ is an isomorphism to its image.
\par
We finally show that for $u \in \partial P$, there is no critical
point of $\mathfrak{PO}_0$ on $(\Lambda_0\setminus \Lambda_+)^n$.
Let
$$
u \in \bigcup_{i \in I} \partial_i P \setminus \bigcup_{i\notin I} \partial_iP_i.
$$
Then
$$
\mathfrak{PO}_0^u \equiv \sum_{i\in I} y_1^{v_{i,1}}\cdots y_n^{v_{i,n}} \mod \Lambda_+.
$$
We remark that $v_i$ ($i\in I$) is a part of the $\Z$ basis of $\Z^n$, since
$X$ is nonsingular toric. Hence
by changing the variables to appropriate $y'_i$ it is easy to see that there is no nonzero critical point
of $\sum_{i\in I} y_1^{v_{i,1}}\cdots y_n^{v_{i,n}} = \sum_{i\in I'} y'_i$.
The proof of Proposition \ref{prop:closuerideal} is now complete.
\end{proof}
\par
To further discuss the relationship between the contents of sections
\ref{sec:milquan} and \ref{sec:exa2} and those in \cite{iritani2},
we compare the coefficient rings used here and in \cite{iritani2}.
In \cite{iritani2} (like many of the literatures on quantum
cohomology such as \cite{givental1}) the formal power series ring
$\Q[[q_1,\cdots,q_{m-n}]]$ is taken as the coefficient ring. ($m-n$
is the rank of $H^2(X;\Q)$ and we choose a basis of it.) The
superpotential in \cite{iritani2} (which is the same as the one used
in \cite{givental1}) is given as \footnote{We change the notation so
that it is consistent to ours. $m, n, v_{i,j}$ here corresponds to
$r+N$, $r$, $x_{i,b}$ in \cite{iritani2}, respectively.}
\begin{equation}\label{eq:superiritani}
F_q = \sum_{i=1}^m\left(\prod_{a=1}^{m-n}q_a^{l_{a,i}} \prod_{j=1}^n s_{j}^{v_{i,j}}\right).
\end{equation}
Here $l_{a,i}$ is a matrix element of a splitting of
$H_2(X;\Z) \to H_2(X,T^n;\Z)$.
We will show that (\ref{eq:superiritani}) pulls back to
our potential function $\mathfrak{PO}^u_0$ after a simple
change of variables.
Let $\alpha_a \in H_2(X;\Z)$ be the basis we have chosen.
(We choose it so that $[\omega] \cap \alpha_a$ is positive.)
\begin{lem}
There exists $f_j(u) \in \R$ ($j=1,\cdots,n$) such that
$$
\frac{1}{2\pi}\sum_a l_{a,i} [\omega] \cap \alpha_a = \ell_i(u) - \sum_{j=1}^n v_{i,j}
f_j(u).
$$
\end{lem}
\begin{proof}
We consider the exact sequence
$$
0 \longrightarrow H_2(X;\Z) \overset{i_*}{\longrightarrow} H_2(X,L(u);\Z) \longrightarrow H_1(L(u);\Z) \longrightarrow 0.
$$
$(c_1,\cdots,c_m) \in H_2(X,L(u);\Z)$ is in the image of $H_2(X;\Z)$
if and only if $\sum_ic_iv_i = 0$. (Here $v_i =
(v_{i,1},\cdots,v_{i,n}) \in \Z^n$.) For given $\alpha \in
H_2(X,\Z)$ denote $i_*(\alpha) = (c_1,\cdots,c_m)$. Then we have
$$
\sum_a [\omega]\cap c_i l_{a,i} \alpha_a = [\omega] \cap \alpha =
2\pi \sum c_i\ell_i(u).
$$
This implies the lemma.
\end{proof}

We now put
\begin{equation}
q_a = T^{[\omega] \cap \alpha_a/2\pi}, \qquad
s_j(u) = T^{f_j(u)} y_j. \label{eq;qandT}
\end{equation}
We obtain the identity
\begin{equation}\label{eq:FpulbacktoPO}
F_q(s_1(u),\cdots,s_n(u)) = \mathfrak{PO}_0^u(y_1,\cdots,y_n).
\end{equation}
\par
We remark that if we change the choice of K\"ahler form
then the identification (\ref{eq;qandT}) changes.
In other words, the story over $\Q[[q_1,\cdots,q_{m-n}]]$
corresponds to studying all the symplectic structures simultaneously,
while the story over $\Lambda$ focuses on one particular symplectic
structure.
\par
In \cite{iritani2} Corollary 5.12, Iritani proved semi-simplicity of
quantum cohomology ring of toric manifold with coefficient ring
$\Q[[q_1,\cdots,q_{m-n}]]$. It does not imply the semi-simplicity of
our $QH^{\omega}(X;\Lambda)$ since the semi-simplicity in general is
not preserved by the pull-back. (On the other way round,
semi-simplicity follows from semi-simplicity of the pull-back.)
However it is preserved by the pull-back at a generic point. Namely we
have:
\begin{prop}
The set of $T^n$-invariant symplectic structures on $X$ for which
$Jac(\mathfrak{PO}_0^u)$ is semi-simple is open and dense.
\end{prop}
\begin{proof}
We give a proof for completeness, following the argument in the
proof of Proposition 5.11 \cite{iritani2}. Consider the polynomial
$$
F_{w_1,\cdots,w_m} = \sum_{i=1}^m w_i y_{1}^{v_{i,1}} \cdots y_{n}^{v_{i,n}}
$$
where $w_i \in \C\setminus \{0\}$. By Kushnirenko's theorem
\cite{kushnire} the Jacobian ring of $F_{w_1,\cdots,w_m}$ is
semi-simple for a generic $w_1,\cdots,w_m$. We put
$$
w_i = \exp\left(\frac{1}{2\pi}\sum_a l_{a,i} [\omega]\cap \alpha + \sum_j v_{i,j}f_j(u)\right).
$$
It is easy to see that when we move $[\omega] \cap \alpha_a$ and $u$ (there
are $m-n$, $n$ parameters respectively) then $w_i$ moves in an arbitrary way.
Therefore for generic choice of $\omega$ and $u$, the Jacobian ring
$Jac(\mathfrak{PO}_0^u)$ is
semi-simple.
Since $Jac(\mathfrak{PO}_0^u)$ is independent of $u$ up to isomorphism,
the proposition follows.
\end{proof}
\begin{rem} Combined with Theorem \ref{QHequalMilnor}, this
proposition gives a partial answer to Question in section 3
\cite{entov-pol07I}.
\end{rem}

\par

\section{Variational analysis of potential function}
\label{sec:var}

In this section, we prove Proposition \ref{existcrit}. Let
$\mathfrak{PO}$ be defined as in (\ref{eq:weakPO}).

We define
$$
s_1(u) = \inf\{ \ell_i(u) \mid i=1,\cdots,m\}.
$$
$s_1$ is a continuous, piecewise affine and convex function and
$s_1 \equiv 0$ on $\partial P$. Recall if $u \in \partial_iP$ then
$\ell_i(u) = 0$ by definition.
\par
We put
$$\aligned
S_1 &= \sup\{ s_1(u) \mid u \in P\}, \\
P_1 &= \{ u \in P \mid s_1(u) = S_1 \}.
\endaligned$$
\begin{prop}\label{prop:bsSdef}
There exist $s_k$, $S_k$, and $P_k$ with the following properties.
\begin{enumerate}
\item $P_{k+1}$ is a convex polyhedron in $M_\R$.
$\dim P_{k+1} \le \dim P_{k}$.
\item $s_{k+1} : P_k \to \R$ is a continuous,
convex piecewise affine function.
\item
$
s_{k+1}(u) = \inf\{\ell_i(u) \mid \ell_i(u) > S_k \}
$
for $u \in \text{\rm Int}P_k$.
\item
$s_{k+1}(u) = S_k$ for $u \in \partial P_k$.
\item
$
S_{k+1} = \sup\{ s_{k+1}(u) \mid u \in P_k\}.
$
\item
$
P_{k+1} = \{u \in P_k \mid s_{k+1}(u) = S_{k+1}\}.
$
\item
$P_{k+1} \subset \text{\rm Int} P_k$.
\item
$s_k$, $S_k$, $P_k$ are
defined for $k = 1,2,\cdots, K$ for some $K \in \Z_+$ and
$P_K$ consists of a single point.
\end{enumerate}
\end{prop}
\begin{exm}
Let $P = [0,a] \times [0,b]$ ($a<b$.) Then
$
s_1(u_1,u_2) = \inf\{ u_1,u_2,a-u_1,b-u_2\}.
$
$S_1 = a/2$,
$
P_1 = \{ (a/2,u_2) \mid a/2 \le u_2 \le b-a/2\}$,
$s_2(1/2,u_2) = \inf \{u_2,b-u_2\}$,
$S_2 = b/2$, $P_2 = \{(a/2,b/2)\}$.
\end{exm}
\begin{proof}
We define $s_k$, $S_k$, $P_k$ inductively over $k$.
We assume that $s_k$, $S_k$, $P_k$ are defined for $k=1,\cdots,k_0$
so that (1) - (7) of Proposition \ref{prop:bsSdef} are satisfied for $k=1,\cdots,k_0-1$.
\par
We define $s_{k_0+1}$ by (3) and (4).
We will prove that it satisfies (2).
We use the following lemma for this purpose.
\begin{lem}\label{lem;bdrysmall}
Let $u_j \in \text{\rm Int} P_{k_0}$ and
$\lim_{j\to\infty}u_j = u_{\infty} \in \partial P_{k_0}$. Then
$$
\lim_{j\to\infty} s_{k_0+1}(u_j) = S_{k_0}.
$$
\end{lem}
\begin{proof}
We put
\begin{equation}\label{form;defIk}
I_{k_0}' = \{ \ell_i | \ell_i(u_{\infty})=S_{k_0}\}.
\end{equation}
By (6) for $k = k_0 - 1$, we find that $s_{k_0}(u_{\infty}) =S_{k_0}$.
Then  (3) for $k=k_0-1$ implies that there is $\ell_i$ such
that $\ell_i(u_{\infty})=S_{k_0}$.
Thus $I_{k_0}'$ is non-empty.
We take the affine space $A_{k_0} \subset M_\R$ such that $\text{\rm
Int} P_{k_0}$ is relatively open in $A_{k_0}$.
\par
Now since $u_\infty \in \partial P_{k_0}$, we can take a $\vec u \in
T_{u_{\infty}}A_{k_0}$ such that $u_\infty + \epsilon \vec u \notin P_{k_0}$
for any sufficiently small $\epsilon > 0$. It follows from (7) for $k=k_0-1$ that $u_\infty + \epsilon \vec u  \in \operatorname{Int}P_{k_0-1}$, and hence
$u + \epsilon \vec u  \in \operatorname{Int}P_{k_0-1}\setminus P_{k_0}$.
\par
By definition, we also have $s_{k_0}(u) \leq S_{k_0}$ for all $u \in P_{k_0-1}$.
Therefore we have
$$
s_{k_0}(u_\infty + \epsilon \vec u) < S_{k_0}.
$$
It follows that there exists $\ell_i\in I'_{k_0}$ such that
\begin{equation}\label{form;elliineq}
\ell_i(u_\infty + \epsilon \vec u) < \ell_i(u_{\infty}) = S_{k_0}
< \ell_i(u_\infty - \epsilon \vec u).
\end{equation}
Since (\ref{form;elliineq}) holds for any $\vec u \in T_{u_{\infty}}A_{k_0}$ with
$u_\infty + \epsilon \vec u \notin P_{k_0}$, it holds for $\epsilon \vec u: = u_j - u_\infty$
for any sufficiently large $j$.
We note that since $u_j \in \operatorname{Int}P_{k_0} \subset A_{k_0}$,
$\vec u_j = u_\infty - u_j$ is an `outward' vector as a tangent vector in $T_{u_\infty}A_{k_0}$
at $u_\infty \in \partial P_{k_0}$. Therefore we have
$u_\infty + \vec u_j \notin P_{k_0}$. Because $u_j = u_\infty - \vec u_j$,
it follows from \eqref{form;elliineq} that
\begin{equation}
\ell_i(u_j) > S_{k_0}
\end{equation}
for any sufficiently large $j$. Therefore we have
\begin{equation}
s_{k_0+1}(u_j) = \inf\{ \ell_i(u_j) \mid \ell_i \in I'_{k_0}, \,\, \ell_i(u_j) > S_{k_0}\}
\end{equation}
and $\lim_{j \to \infty} s_{k_0+1}(u_j) = \lim_{j\to \infty}\ell_i(u_j) = \ell_i(u_\infty) = S_{k_0}$.
This finishes the proof of the lemma.
\end{proof}
Lemma \ref{lem;bdrysmall} implies that $s_{k_0+1}$ is continuous and
piecewise linear in a neighborhood of $\partial P_{k_0}$. We can
then check (2) easily.
\par
We define $S_{k_0+1}$ by (5).
Then we can define $P_{k_0+1}$ by (6).
(In other words the right hand side of (6) is nonempty.)
(7) is a consequence of Lemma \ref{lem;bdrysmall}.
We can easily check that $P_{k_0+1}$ satisfies (1).
\par
We finally prove that $P_K$ becomes a point
for some $K$.
Let $\text{\bf u}_k \in \text{\rm Int}\,P_k$ and put
\begin{equation}\label{form;defIk1.5}
I_k = \{\ell_i\mid \ell_i(\text{\bf u}_k) = S_k\}.
\end{equation}
Here $k = 1,\cdots,K$.  We remark that $I_k$ is independent of the
choice of  $\text{\bf u}_k \in \text{\rm Int}\,P_k$.
\par
Note we defined $I'_{k_0}$ by the formula (\ref{form;defIk}).
We have $I_{k_0} \subseteq I'_{k_0}$.  But the equality may not
hold in general.
In fact $u_{\infty}$ in the boundary of $P_{k_0}$
but $\text{\bf u}_{k_0}$ is an interior point  of $P_{k_0}$.
Therefore if $\ell_i \in I_{k_0}$ then $\ell_i$ is constant on $P_{k_0}$.
But element of $I'_{k_0}$ may not have this property.
\par
In case some $\ell_i \in I_{k_0+1}$  is not constant on $P_{k_0}$
it is easy to see that $\dim P_{k_0+1} < \dim P_{k_0}$.
We remark that there exists some $\ell_j \notin \bigcup_{k\le k_0}I_{k_0}$
which is not constant on $S_{k_0}$ unless $S_{k_0}$ is a point.
Therefore if $\dim S_{k_0} \ne 0$, there exists $k' > k_0$
such that  $\dim P_{k'} < \dim P_{k_0}$.
Therefore there exists $K$ such that $P_{K}$ becomes $0$-dimensional
(namely a point). Hence we have achieved (8). The proof of
Proposition \ref{prop:bsSdef} is now complete.
\end{proof}
\begin{rem}
In a recent preprint \cite{mcd}, McDuff pointed out an
error in the statement (1) of Proposition \ref{prop:bsSdef} in the previous version
of this paper. We have corrected the statement and modified
the last paragraph of its proof, following the corresponding
argument in section 2.2 of \cite{mcd}. We thank her for pointing out
this error.
\end{rem}
The next lemma easily follows from construction.
\begin{lem}\label{lem;rationality}
If all the vertices of $P$ lie in $\Q^n$ then $u_0 \in\Q^n$. Here
$\{u_0\} = P_K$.
\end{lem}
By parallelly translating the polytope, we may assume, without loss of generality, that $u_0 = \text{\bf 0}$,
the origin. In the rest of this subsection, we will prove that $\mathfrak{PO}^{\text{\bf 0}}$ has a critical
point on $(\Lambda_{0}\setminus \Lambda_{+})^n$.
More precisely we prove Proposition \ref{existcrit} for $u_0=\text{\bf 0}$.
(We remark that if $P$ and $\ell_i$ are given we can easily locate $u_0$.)
\begin{exm}\label{leadingexample}
Let us consider Example \ref{exa:onepointblow} in the case $\alpha > 1/3$.
At $u_0 = ((1+\alpha)/4,(1-\alpha)/2)$ we have
$$
\mathfrak{PO}^{u_0} = (y_2+y_2^{-1}) T^{(1-\alpha)/2} + (y_1+(y_1y_2)^{-1})T^{(1+\alpha)/4}.
$$
Therefore the constant term $\mathfrak y_{i;0}$ of the coordinate
$y_i$ of the critical point is given by
\begin{equation}\label{ex:leading}
1-\mathfrak y_{2;0}^{-2} = 0, \quad 1- \mathfrak y_{1;0}^{-2}\mathfrak y_{2;0}^{-1} = 0.
\end{equation}
Note the first equation comes from the term of the smallest exponent and
contains only $\mathfrak y_{2;0}$. The second equation comes from the term which has
second smallest exponent and contains both $\mathfrak y_{1;0}$ and $\mathfrak y_{2;0}$. So we need to
solve the equation inductively according to the order of the exponent.
This is the situation we want to work out in general.
\end{exm}
\par
We remark that the affine space $A_i$ defined above in the proof of Lemma \ref{lem;bdrysmall}
$$
M_\R = A_0 \supseteq A_1 \supseteq \cdots \supseteq A_{K-1} \supseteq A_K =
\{\text{\bf 0}\}
$$
is a nonincreasing sequence of linear subspaces such that
$\text{\rm Int } P_k$ is an open subset of $A_k$.
Let
$$
A_{l}^\perp \subset (M_\R)^* \cong N_\R
$$
be the annihilator of $A_l \subset M_\R$. Then we have
$$
\{\text{\bf 0}\} = A_0^\perp \subseteq A_1^\perp \subseteq \cdots
\subseteq A_{K-1}^\perp \subseteq A_K^\perp = N_\R.
$$
\par
We recall:
\begin{equation}\label{form;defIk2}
I_k = \{\ell_i\mid \ell_i(\text{\bf 0}) = S_k\},
\end{equation}
for $k = 1,\cdots,K$.  In fact $\text{\bf 0} \in P_{k+1}
\subseteq \text{\rm Int}\,P_k$ for $k<K$.
\par
We renumber each of $I_k$ in (\ref{form;defIk2}) so that
\begin{equation}\label{formula;Iandlambda}
\{\ell_{k,j} \mid j=1,\cdots,a(k)\} = I_k.
\end{equation}
By construction
\begin{equation}\label{formula:144}
s_k(u) = \inf_j \ell_{k,j}(u)
\end{equation}
in a neighborhood of $\text{\bf 0}$ in $P_{k-1}$. In fact $s_{k-1}(\text{\bf 0}) = S_{k-1} < S_k = s_k(\text{\bf 0})$
and
$$
\{\ell_i(\text{\bf 0}) \mid i =1,\cdots, m\} \cap (S_{k-1},S_k) = \emptyset.
$$
\begin{lem}\label{lem;145}
If $u \in A_k$ then $\ell_{k,j}(u) = S_k$.
\end{lem}
\begin{proof}
We may assume $k < K$. Hence $\text{\bf 0} \in \text{\rm Int} P_k$.
We regard $u \in A_k = T_{\text{\bf 0}}A_k$. By (\ref{formula:144}), we have
$$
s_k(\e u) = \inf\{\ell_{k,j}(\epsilon u) \mid j=1,\cdots,a(k)\}.
$$
Since $s_k(\e u) = S_k$ for $\e u \in P_k$ it follows that
$\ell_{k,j}(u) = S_k$.
\end{proof}
Lemma \ref{lem;145} implies that the linear part $d\ell_{k,j}$ of $\ell_{k,j}$
is an element of $A_k^\perp \subset \mathfrak t =N_\R$. In fact if $\ell_{k,j} = \ell_i$,
we have $d\ell_{k,j} = v_i$ from the definition of $\ell_i$,
$\ell_i(u) = \langle u,v_i \rangle - \lambda_i$ given in Theorem
\ref{gullemin}.

\begin{lem}\label{lem;146}
For any $v \in A_k^\perp$, there exist nonnegative real numbers $c_j \ge 0$, $j=1,\cdots,a(k)$
such that
$$
v - \sum_{j=1}^{a(k)} c_j d\ell_{k,j} \in A_{k-1}^\perp.
$$
\end{lem}
\begin{proof} Suppose to the contrary that
$$
\left\{v - \sum_{j=1}^{a(k)} c_j d\ell_{k,j} \, \Big | \, c_j \geq
0,\,\, j =1, \cdots, a(k) \right \} \cap A_{k-1}^\perp = \emptyset.
$$
Then we can find $u \in A_{k-1} \setminus A_k$ such that
\begin{equation}\label{form;37.147}
d\ell_{k,j}(u) \geq 0
\end{equation}
for all $j = 1, \cdots, a(k)$.

Since $\e u \in A_{k-1} \setminus A_{k}$
it follows that
$$
s_k(\e u) < S_k
$$
for a sufficiently small $\e$. On the other hand,
(\ref{form;37.147}) implies $d\ell_{k,j}(\e u) \geq 0$ for all $\e >
0$ and so $\ell_{k,j}(\e u) \geq \ell_{k,j}(\text{\bf 0}) = S_k$.
Therefore by definition of $s_k$ in Proposition \ref{prop:bsSdef} we have
\beastar
s_k(\e u) & \geq & \inf \{\ell_{k,j}(\e u) \mid j = 1,\cdots, a(k)\} \\
& \ge & \inf \{\ell_{k,j}(\text{\bf 0}) \mid j = 1,\cdots, a(k)\} = S_k.
\eeastar
This is a contradiction.
\end{proof}
Applying Lemma \ref{lem;146} inductively downwards starting from
$\ell = k$ ending at $\ell = 1$, we immediately obtain the following
\begin{cor}\label{Col:37.148}
For any $v \in A_k^\perp$, there exist $c_{l,j} \ge 0$
for $l = 1,\cdots, k$, $j = 1,\cdots,a(l)$ such that
$$
v = \sum_{l=1}^k\sum_{j=1}^{a(l)} c_{l,j} d\ell_{l,j}.
$$
\end{cor}

We denote
\begin{equation}
\mathfrak I = \{\ell_i \mid i=1,\cdots,m\} \setminus \bigcup_{k=1}^K I_k.
\label{37.149}\end{equation}
It is easy to see that
\begin{equation}
\ell \in \mathfrak I \Rightarrow \ell(\text{\bf 0}) > S_K.
\label{37.150}\end{equation}
\par

Now we go back to the situation of (\ref{eq:weakPO}). We use the
notation of (\ref{eq:weakPO}). In this case, for each
$k=1,\cdots,K$, we also associate, in Definition \ref{defn;151}, a set $\mathfrak I_{k}$ consisting
of pairs $(\ell,\rho)$ with an affine map $\ell : M_\R \to \R$ and
$\rho \in \R_+$.

\begin{defn}\label{defn;151}
We say that a pair $(\ell,\rho) = (\ell'_j,\rho_j)$
is an element of $\mathfrak I_{k}$ if the following holds :
\begin{enumerate}
\item If $e_j^i \ne 0$ then
$\ell_i \in \bigcup_{l=1}^k I_{l}$.
(Note $\ell'_j = \sum_i e_j^i \ell_i$.)
\par
\item (1) does not
hold for some $i,j$ if we replace $k$ by $k-1$.
\end{enumerate}
A pair $(\ell,\rho) = (\ell'_j,\rho_j)$ as in (\ref{eq:weakPO}) is, by definition,
an element of $\mathfrak I_{K+1}$ if it is not contained in any of
$\mathfrak I_{k}$, $k=1,\cdots,K$.
\end{defn}
\begin{lem}\label{37.152}
\begin{enumerate}
\item If $(\ell,\rho) \in \mathfrak I_{k}$ then
$d\ell \in A_k^\perp$.
\par
\item If $(\ell,\rho) \in \mathfrak I_{k}$ then
$\ell(\text{\bf 0})+ \rho > S_k$.
\par
\item
If $(\ell,\rho) \in \mathfrak I_{K+1}$ then
$\ell(\text{\bf 0})+ \rho > S_K$.
\end{enumerate}
\end{lem}
\begin{proof}
(1) follows from Definition \ref{defn;151} (1) and Lemma \ref{lem;145}.
\par
If $(\ell,\rho) = (\ell_j',\rho_j) \in \mathfrak I_{k}$ then there exists $e_j^i \ne 0$,
$\ell_i = \ell_{k,j'}$. Then
$$
\ell(\text{\bf 0})+\rho \ge e^i_j\ell_i(\text{\bf 0}) + \rho_j >
\ell_i(\text{\bf 0}) = S_{k}.
$$
(2) follows. The proof of (3) is the same.
\end{proof}
\begin{lem}\label{37.154}
The vector space $A_k$ is defined over $\Q$.
\end{lem}
\begin{proof}
$A_k$ is defined by equalities
of the type $\ell_i = S_k$ on $A_{k-1}$.
Since the linear part of $\ell_i$ has integer coefficients,
the lemma follows by induction on $k$.
\end{proof}
We put $d(k) = \dim A_{k-1} - \dim A_{k} = \dim A^\perp_{k} - \dim A^\perp_{k-1}$.
We choose $\text{\bf e}_{i,j}^* \in Hom(M_\Q,\Q) \cong N_\Q$
($i=1,\cdots,K$, $j=1,\cdots,d(k)$) such that the following condition holds. Here
$M_\Q = M \otimes \Q$ and $N_\Q = N \otimes \Q$

\begin{conds}\label{cond:estar}
\begin{enumerate}
\item $\text{\bf e}_{1,1}^*, \cdots, \text{\bf e}_{k,d(k)}^*$ is a $\Q$ basis
of $A^\perp_k \cap N_{\Q}$.
\par
\item $d\ell_{k,j} = \sum_{k',j'} v_{(k,j),(k',j')}\text{\bf e}_{k',j'}^*$ with
$v_{(k,j),(k',j')} \in \Z$.
\par
\item If $(\ell,\rho) \in \mathfrak I_k$ or $\ell \in \mathfrak I$, then
$d\ell = \sum_{k',j'} v_{\ell,(k',j')}\text{\bf e}_{k',j'}^*$ with $v_{\ell,(k',j')} \in \Z$.
\end{enumerate}
\end{conds}
Note $d(k) = 0$ if $A_k = A_{k-1}$.
\par
We identify $\R^n$ with $H^1(L(u);\R^n)$ in the same way as Lemma
\ref{37.90} and let $x_{k,j} \in Hom(H^1(L(u);\R),\R)$ be the
element corresponding to $\text{\bf e}_{k,j}^*$ by this
identification. In other words, if
$$
\text{\bf e}_{k,j}^* = \sum_{i} a_{(k,j);i} \text{\bf e}_{i}^*,
$$
where $\text{\bf e}_{i}^*$ is as in Lemma \ref{37.90},
then we have
$$
x_{k,j} = \sum_{i} a_{(k,j);i} x_i.
$$
We put
$
y_{k,j} = e^{x_{k,j}}.
$
We define
\begin{equation}
Y(k,j) =\prod_{k'=1}^{K}\prod_{j'=1}^{d(k')} y_{k',j'}^{v_{(k,j),(k',j')}}. \label{37.156.1}
\end{equation}
And for $(\ell,\rho) \in \mathfrak I_k$ or $\ell \in \mathfrak I$, we define
\begin{equation}
Y(\ell) = \prod_{k=1}^{K}\prod_{j=1}^{d(k)} y_{k,j}^{v_{\ell,(k,j)}}.
\label{37.156.2}
\end{equation}
By Theorem \ref{weakpotential} there exists $c_{(\ell,\rho)} \in \Q$
such that :
\begin{equation}
\aligned
\mathfrak{PO}^{\text{\bf 0}}
= &\sum_{k=1}^{K}\left(\sum_{j=1}^{a(k)} Y(k,j)\right) T^{S_{k}}
+\sum_{\ell \in \mathfrak I}Y(\ell)T^{\ell(\text{\bf 0})} \\
&+ \sum_{k=1}^{K+1}\sum_{(\ell,\rho) \in \mathfrak I_{k}}c_{(\ell,\rho)}Y(\ell) T^{\ell(\text{\bf
0})+\rho}
\endaligned\label{37.157}
\end{equation}
where $\mathfrak{PO}^{\text{\bf 0}}$ is $\mathfrak{PO}^u$ with $u = \text{\bf 0}$.
\begin{lem}\label{37.158}
\begin{enumerate}
\item If $k' < k$ then
\begin{equation}
\frac{\partial Y(k',j')}{\partial y_{k,j}} = 0.
\label{37.159.1}\end{equation}
\item If $(\ell,\rho) \in \mathfrak I_{k'}$, $k' < k$ then
\begin{equation}
\frac{\partial Y(\ell)}{\partial y_{k,j}} = 0.
\label{37.159.2}\end{equation}
\par
\item If $(\ell,\rho) \in \mathfrak I_{k}$ then $\ell(\text{\bf 0})+\rho > S_k$.
\par
\item
If $(\ell,\rho) \in \mathfrak I_{K+1}$ then $\ell(\text{\bf 0})+\rho > S_{K}$.
\par
\item If $\ell \in \mathfrak I$ then $\ell(\text{\bf 0}) > S_K$.
\end{enumerate}
\end{lem}
\begin{proof}
Since $d\ell_{k',j'} \in A^\perp_{k'}$ by Lemma \ref{lem;145} it
follows that $v_{(k',j'),(k,j)} = 0$ for $k>k'$. (1) follows. (2)
follows from Lemma \ref{37.152} (1) in the same way. (3) follows
from Lemma \ref{37.152} (2).
(4) follows from Lemma \ref{37.152} (3).
(5) follows from (\ref{37.150}).
\end{proof}
Now equation (\ref{formula:critical}) becomes
$$
0 = \frac{\partial\mathfrak{PO}^{\text{\bf 0}}}{\partial y_{k,j}}.
$$
We calculate this equation using Lemma \ref{37.158} to find that it is equivalent to :

\begin{equation}\aligned
0 = &\sum_{j'=1}^{a(k)}\frac{\partial Y(k,j')}{\partial y_{k,j}}
+ \sum_{k'>k}\sum_{j'=1}^{a(k')}\frac{\partial Y(k',j')}{\partial y_{k,j}} T^{S_{k'} - S_k} \\
&+ \sum_{k'=k}^{K+1}\sum_{(\ell,\rho) \in \mathfrak I_{k'}}
c_{(\ell,\rho)}\frac{\partial Y(\ell)}{\partial y_{k,j}}T^{\ell(\text{\bf 0})+\rho-S_k} + \sum_{\ell \in \mathfrak I}\frac{\partial Y(\ell)}{\partial y_{k,j}}T^{\ell(\text{\bf
0})-S_k}.
\endaligned\label{37.160}\end{equation}
Note the exponents of $T$ in the second, third, and fourth terms of (\ref{37.160}) are all strictly positive.
So after putting $T=0$ we have
\begin{equation}
0 = \sum_{j'=1}^{a(k)}\frac{\partial Y(k,j')}{\partial y_{k,j}}.
\label{37.161}\end{equation}
Note that the equation (\ref{37.161}) does not involve $T$
but becomes a numerical equation.
We call (\ref{37.161}) the {\it leading term equation}.
\begin{lem}\label{37.162}
There exist positive real numbers $\mathfrak y_{k,j;0}$, $k=1,\cdots,K$, $j=1,\cdots,d(k)$,
solving the leading term equations for $k=1,\cdots,K$.
\end{lem}
\begin{proof}
We remark the leading term equation for $k,j$ contain the monomials
involving only $y_{k',j}$ for $k'\le k$.
We first solve the leading term equation for $k= 1$. Denote
$$
f_{1}(x_{1,1},\cdots,x_{1,d(1)}) = \sum_{j=1}^{a(1)}Y(1,j).
$$
It follows from Corollary \ref{Col:37.148} that for any $(x_{1,1},\cdots,x_{1,d(1)}) \ne 0$,
there exists $j$ such that
$$
d\ell_{1,j}(x_{1,1},\cdots,x_{1,d(1)}) > 0.
$$
Therefore, we have
$$
\lim_{t\to\infty} f_{1}(tx_{1,1},\cdots,tx_{1,d(1)})
\ge \lim_{t\to\infty}C\exp(td\ell_{1,j}(x_{1,1},\cdots,x_{1,d(1)})) = +\infty.
$$
Hence $f_{1}(x_{1,1},\cdots,x_{1,d(1)})$ attains its minimum at some point
of $\R^{d(1)}$. Taking its exponential, We obtain $\mathfrak y_{1,j;0} \in \R_+$.
\par
Suppose we have already found $\mathfrak y_{k',j;0}$ for $k'<k$. Then we put
$$
F_k(x_{1,1},\cdots,x_{k,1},\cdots,x_{k,d(k)}) = \sum_{j=1}^{a(k)}Y(k,j)
$$
and
$$
f_k(x_{k,1},\cdots,x_{k,d(k)})
= F_k(\mathfrak x_{1,1;0},\cdots,\mathfrak x_{k-1,d(k-1);0},x_{k,1},\cdots,x_{k,d(k)})
$$
where $\mathfrak x_{k',j;0} = \log \mathfrak y_{k',j;0}$.
Again using Corollary \ref{Col:37.148}, we find
$$
\lim_{t\to\infty} f_k(tx_{k,1},\cdots,tx_{k,d(k)}) = +\infty.
$$
for any $(x_{k,1},\cdots,x_{k,d(k)}) \ne 0$.
Hence $f_k(x_{k,1},\cdots,x_{k,d(k)})$ attains a minimum and we obtain
$\mathfrak y_{k,j;0}$.
Lemma \ref{37.162} now follows by induction.
\end{proof}

We next find the solution of our equation
(\ref{formula:criticalweak}) or (\ref{formula:critical0}).
We take a sufficiently large $\CN$ and put

\begin{equation}\aligned
\mathfrak{PO}^{\text{\bf 0}}_{k,\CN} = &\sum_{j=1}^{a(k)} Y(k,j) +
\sum_{k'>k}\sum_{j'=1}^{a(k')}\frac{\partial Y(k',j')}{\partial y_{k,j}} T^{S_{k'} - S_k}
\\
&+ \sum_{\ell \in \mathfrak I, \,\, \ell(\text{\bf 0}) \le
\CN}Y(\ell) T^{\ell(\text{\bf 0}) -S_k}
\\
&+ \sum_{k'=k+1}^{K+1}\sum_{(\ell,\rho) \in \mathfrak I_{k'}, \ell(\text{\bf 0})
+ \rho \le \CN}
c_{(\ell,\rho)}Y(\ell) T^{\ell(\text{\bf
0})+\rho -S_k}.
\endaligned\label{37.167}\end{equation}
We remark that (\ref{formula:criticalweak}) is equivalent to
\begin{equation}
\frac{\partial \mathfrak{PO}_{k,\CN}^{\text{\bf
0}}}{\partial y_{k,j}}(\mathfrak y_1,\cdots,\mathfrak y_n) \equiv 0
\mod T^{\CN - S_k} \qquad k=1,\cdots,K,\, j=1,\cdots,a(k).
\label{37.168}\end{equation}
We also put
$$
\overline{\mathfrak{PO}}^{\text{\bf 0}}_{k} = \sum_{j=1}^{a(k)}Y(k,j).
$$
It satisfies
\begin{equation}
\overline{\mathfrak{PO}}^{\text{\bf 0}}_{k} \equiv \mathfrak{PO}^{\text{\bf 0}}_{k,\CN} \mod \Lambda_{+}.
\label{37.169}\end{equation}
\par
For given positive numbers $R(1), \cdots, R(K)$ we define the discs
$$
D(R(k)) = \{(x_{k,1}\cdots,x_{k,d(k)}) \mid x_{k,1}^2 + \cdots + x_{k,d(k)}^2 \le
R(k)\} \subset \R^{d(k)}
$$
and the poly-discs
\begin{eqnarray}
D(R(\cdot)) &=& \prod_{k=1}^{K} D(R(k)) \nonumber\\
&=&
\{(x_{1,1}
\cdots,
x_{K,d(K)}) \mid x_{k,1}^2 + \cdots + x_{k,d(k)}^2 \le
R(k),\, k=1,\cdots,K\}.\nonumber
\end{eqnarray}
We factorize
$$
\R^n = \prod_{k=1}^{K} \R^{d(k)}.
$$
Then we consider the Jacobian of $\overline{\mathfrak{PO}}^{\text{\bf 0}}_{k}$
$$
\nabla \overline{\mathfrak{PO}}^{\text{\bf 0}}_{k} : \R^n \to \R^{d(k)}
$$
i.e., the map
\begin{equation}
(\mathfrak x_{1,1},\cdots,\mathfrak x_{K,d(K)}) \mapsto
\left(
\frac{\partial \overline{\mathfrak{PO}}^{\text{\bf 0}}_{k}}{\partial x_{k,j}}(\mathfrak x_{1,1},\cdots,\mathfrak x_{K,d(K)})
\right)_{j=1,\cdots,d(k)}.
\label{37.170}\end{equation}
We remark that $\nabla \overline{\mathfrak{PO}}^{\text{\bf 0}}_{k}$ depends only
on $\R^{d(1)} \times \cdots \times \R^{d(k)}$ components.
\par
Combining all $\nabla \overline{\mathfrak{PO}}^{\text{\bf 0}}_{k}$, $k = 1, \cdots, K$
(\ref{37.170}) induces a map
$$
\nabla \overline{\mathfrak{PO}}^{\text{\bf 0}} : \R^n \to \R^{n}
$$
defined by
$$
\nabla \overline{\mathfrak{PO}}^{\text{\bf 0}}
= (\nabla \overline{\mathfrak{PO}}^{\text{\bf 0}}_{1},\cdots,\nabla \overline{\mathfrak{PO}}^{\text{\bf 0}}_{K}).
$$
The next lemma is closely related to Lemma \ref{37.162}.
\begin{lem}\label{37.171}
We may choose the positive numbers $R(k)$ for $k = 1, \cdots, K$ such that the following holds :
\begin{enumerate}
\item $\nabla \overline{\mathfrak{PO}}^{\text{\bf 0}}$
is nonzero on $\partial(D(R(\cdot)))$.
\par
\item The map $: \partial(D(R(\cdot))) \to S^{n-1}$
$$
\mathfrak x \mapsto \frac{\nabla \overline{\mathfrak{PO}}^{\text{\bf 0}}}
{\Vert\nabla \overline{\mathfrak{PO}}^{\text{\bf 0}}\Vert}
$$
has degree $1$.
\end{enumerate}
\end{lem}
\begin{proof}
We first prove the following sublemma by an upward induction on
$k_0$.
\begin{sublem}\label{37.173}
There exist $R(k)$'s for $1 \leq k \leq K$ such that for any given
$1 \leq k_0 \leq K$ we have
\begin{equation}
\sum_{j=1}^{d(k_0)} x_{k_0,j} \frac{\partial
\overline{\mathfrak{PO}}^{\text{\bf 0}}_{k_0}}{\partial x_{k_0,j}}
(x_{1,1},\cdots,x_{k_0,d(k_0)}) > 0 \label{37.174}\end{equation}if
$(x_{k,1},\cdots,x_{k,d(k)}) \in D(R(k))$ for all $1 \leq k \leq
k_0-1$ and $(x_{k_0,1},\cdots,x_{k_0,d(k_0)}) \in \partial
D(R(k_0))$.
\end{sublem}
\begin{proof}
In case $k_0 = 1$ the existence of $R(1)$ satisfying (\ref{37.174})
is a consequence of Corollary \ref{Col:37.148}.
We assume that the sublemma is proved for $1,\cdots,k_0-1$.
\par
For each fixed $\text{\bf x} = (x_{1,1},\cdots,x_{k_0-1,d(k_0-1)})$
we can find $R(k_0)_{\text{\bf x}}$ such that (\ref{37.174}) holds
for $(x_{k_0,1},\cdots,x_{k_0,d(k_0)}) \in \R^{d(k_0)} \setminus
D(R(k_0)_{\text{\bf x}}/2)$. This is also a consequence of Corollary
\ref{Col:37.148}.
\par
We take supremum of $R(k_0)_{\text{\bf x}}$ over the compact set
$\text{\bf x} \in \prod_{k=1}^{k_0-1}
D(R(k))$ and obtain $R(k_0)$.
The proof of Sublemma \ref{37.173} is complete.
\end{proof}
It is easy to see that Lemma \ref{37.171} follows from
Sublemma \ref{37.173}.
\end{proof}
We now use our assumption that the vertices of $P$ lies in $M_\Q= \Q^n$ and that
$\rho_j \in \Q$.
Replacing $T$ by $T^{1/\mathcal C}$ if necessary, we may assume that
all the exponents of $y_{k,j}$ and $T$ appearing in (\ref{37.167}) are integers.
Then
$$
\mathfrak{PO}^{\text{\bf 0}}_{k,\CN}
= \mathfrak{PO}^{\text{\bf 0}}_{k,\CN}(y_{1,1},\cdots,y_{K,d(K)};T)
$$
are polynomials of $y_{k,j}$, $y_{k,j}^{-1}$ and $T$.
Define the set $\mathfrak X$ by the set consisting of
$$
(\mathfrak y_{1,1},\cdots,\mathfrak y_{K,d(K)};q) \in (\R_+)^{n}\times \R
$$
that satisfy
\begin{equation}
\frac{\partial\mathfrak{PO}^{\text{\bf 0}}_{k,\CN}}{\partial y_{k,j}}(\mathfrak y_{1,1},\cdots,\mathfrak y_{K,d(K)};q)
= 0,
\label{eq:*}\end{equation}
for $k=1,\cdots,K$, $j=1,\cdots,d(k)$.
Clearly $\mathfrak X$ is a real affine algebraic variety.
(Note the equation for $y_i$ are polynomials. So we need to regard $y_i$ (not $x_i$) as
variables to regard $\mathfrak X$ as a real affine algebraic variety. )
\par
Consider the projection
$$
\pi : \mathfrak X \to \R, \quad \pi(\mathfrak y_{1,1},\cdots,\mathfrak y_{K,d(K)};q) = q
$$
which is a morphism of algebraic varieties.
\begin{lem}\label{37.175}
There exists a sufficiently small $\epsilon > 0$ such that if $\vert q\vert < \epsilon$ then
$$
\pi^{-1}(q) \cap \{(e^{x_1},\cdots,e^{x_n}) \mid (x_1,\cdots,x_n) \in D(R(\cdot))\} \ne \emptyset.
$$
\end{lem}
\begin{proof}
We consider the real analytic $q$-family of polynomials
$$
\mathfrak{PO}^{\text{\bf 0}}_{k,\CN,q} (y_{1,1},\cdots,y_{K,d(K)})
= \mathfrak{PO}^{\text{\bf 0}}_{k,\CN}(y_{1,1},\cdots,y_{K,d(K)};q).
$$
Then
\begin{equation}
\mathfrak{PO}^{\text{\bf 0}}_{k,\CN,0} = \overline{\mathfrak{PO}}^{\text{\bf 0}}_{k,\CN}.
\label{37.176}
\end{equation}
Replacing $\overline{\mathfrak{PO}}^{\text{\bf 0}}_{k,\CN}$
by $\mathfrak{PO}^{\text{\bf 0}}_{k,\CN,q}$, we can repeat construction of the map
$$
\nabla \mathfrak{PO}^{\text{\bf 0}}_{\CN,q} : \R^n \to \R^n
$$
for each fixed $q \in \R$ in the same way as we defined
$\nabla \overline{\mathfrak{PO}}^{\text{\bf 0}}$.
Then the conclusion of Lemma \ref{37.171} holds
for $\nabla \mathfrak{PO}^{\text{\bf 0}}_{\CN,q}$
if $|q|$ is sufficiently small.
(This is a consequence of Lemma \ref{37.171} and (\ref{37.176}).)
Lemma \ref{37.175} follows from elementary algebraic topology.
\end{proof}
Lemma \ref{37.175} implies that we can find
$$
\mathfrak y_0 = (\mathfrak y_{1,1;0},\cdots,\mathfrak y_{K,d(K);0}) \in \R_+^n
$$
and a sequence
$$
(\mathfrak y_h,q_h) = (\mathfrak y^h_{1,1;0},\cdots,\mathfrak y^h_{K,d(K);0};q_h) \in \mathfrak X \subset \R^{n+1}
$$
$h=1,2,\cdots$ such that $q_h>0$ and
$
\lim_{h\to \infty} (\mathfrak y_h,q_h) = (\mathfrak y_0,0).
$
Therefore by the curve selection lemma (Lemma 3.1 \cite{Mil68}) there
exists a real analytic map
$$
\gamma : [0,\epsilon) \to \mathfrak X
$$
such that $\gamma(0) = (\mathfrak y_0,0)$
and $\pi(\gamma(t)) > 0$ for $t>0$.
We reparameterize $\gamma(t)$, so that its $q$-component is
$
t^{a/b}
$,
where $a$ and $b$ are relatively prime integers.
We put $T=t^{a/b}$ i.e., $t = T^{b/a}$ and denote the $y_{k,j}$-components of $\gamma(t)$ by
$$
\mathfrak y_{k,j} = \mathfrak y_{k,j;0} + \sum_{\ell=1}^{\infty} \mathfrak y_{k,j;\ell} T^{b\ell/a}.
$$
Since $\gamma(t) \in \mathfrak X$, the element $(\mathfrak y_{k,j})_{k,j} \in (\Lambda^{\R}_{0} \setminus \Lambda^{\R}_{+})^n$ is
the required solution of (\ref{formula:criticalweak}).
\par
Since $\mathfrak{PO}_0$ contains only a finite number of summands,
we can take $\mathfrak{PO}_{0,\CN} = \mathfrak{PO}_{0}$.
Therefore we can find a solution of (\ref{formula:critical0}) for $\mathfrak{PO}_0$.
\par
The proof of Proposition \ref{existcrit} is now complete.
\qed

\section{Elimination of higher order term in nondegenerate cases}
\label{sec:ellim}

In this section, we prove a rather technical (but useful) result, which
shows that solutions of the leading term equation (\ref{37.161})
correspond to actual critical points under certain non-degeneracy condition. For this purpose, we slightly modify
the argument of the last part of section \ref{sec:var}.
This result will be useful to determine
$u \in \text{\rm Int}\, P$
such that $HF((L(u);\frak x),(L(u);\frak x);\Lambda_0) \ne 0$ for some $\frak x$  in
non-Fano cases.
(In other words we study the image of $\mathfrak M(\mathfrak{Lag}(X)) \to \text{\rm Int}\,P$
by the map $(\frak x,u) \mapsto u$.)
In fact it shows that we can use $\mathfrak{PO}_0$
in place of $\mathfrak{PO}$ in most practical cases.
We remark that we explicitly calculate $\mathfrak{PO}_0$ but do not know
the precise form of $\mathfrak{PO}$ in non-Fano cases.
\par
In order to state the result in a general form, we prepare some notations.
Let $u_0 \in \text{\rm Int} \,P$.
(In section \ref{sec:var}, $u_0$ is determined as the unique element of
$P_K$ defined in Proposition \ref{prop:bsSdef}. The present
situation is more general.)
\par
We define positive real numbers $S_1 < S_2 < \cdots$ by
\begin{equation}\label{eq:defSk}
\{\ell_i(u_0) \mid i = 1,\cdots,m\} = \{S_1,S_2,\cdots , S_{m'}\}
\end{equation}
and the sets
\begin{equation}\label{form;defIk29}
I_k = \{\ell_i\mid \ell_i(u_0) = S_k\},
\end{equation}
for $k = 1,\cdots$. We renumber each of $I_k$
so that
\begin{equation}\label{formula;Iandlambda9}
\{\ell_{k,j} \mid j=1,\cdots,a(k)\} = I_k.
\end{equation}
\begin{defn}\label{def:genAl}
Let $A^\perp_l$ be the linear subspace of $N_\R$ spanned by
$d\ell_{k,j}$ $k\le l$, $j \le a(k)$. We define $K$ to be the
smallest number such that $A^\perp_K = N_\R$.
\end{defn}
Note our notations here are consistent with one in section \ref{sec:var} in case
$\{u_0\} = P_K$.
We define $\mathfrak I$ and $\mathfrak I_{k}$ by (\ref{37.149}) and Definition \ref{defn;151}.
Then Lemma \ref{37.152} and (\ref{37.150}) hold.
We choose $\text{\bf e}_{i,j}^* \in Hom(M_\Q,\Q)$ such that Condition \ref{cond:estar}
is satisfied. (Note $A^\perp_l$ is defined over $\Q$.)
$x_{i,j}$ and $y_{i,j}$ then are defined in the same way as section \ref{sec:var}.
We define $Y(k,j)$ by (\ref{37.156.1}) and $Y(\ell)$ by (\ref{37.156.2}).
Then (\ref{37.157}) and Lemma \ref{37.158} hold.
\par
We remark that Corollary \ref{Col:37.148} does {\it not} hold in general in
the current situation.
In fact we can write
$$
v = \sum_{l=1}^k\sum_{j=1}^{a(l)} c_{l,j} d\ell_{l,j}
$$
under the assumption of Corollary \ref{Col:37.148} but we may not be
able to ensure $c_{l,j}\ge 0$.

\begin{defn}\label{def:lteq}
\begin{enumerate}
\item We call
$$
0 = \sum_{j'=1}^{a(k)}\frac{\partial Y(k,j')}{\partial y_{k,j}},
\quad k=1,\cdots,K,\,\, j=1,\cdots,d(k)
$$
the {\it leading term equation} at $u_0$. We regard it as a polynomial equation
for $\mathfrak y_{k,j} \in \C\setminus \{0\}$, $k=1,\cdots,K$, $j=1,\cdots,d(k)$.
\par
\item
A solution $\mathfrak y^0 = (\mathfrak y_{k,j;0})_{k=1,\cdots,K,\,
j=1,\cdots,d(k)}$ of leading term equation is said to be {\it weakly
nondegenerate} if it is isolated in the set of solutions.
\par
\item
A solution $\mathfrak y^0 = (\mathfrak y_{k,j;0})_{k=1,\cdots,K,\,
j=1,\cdots,d(k)}$ of leading term equation is said to be {\it
strongly nondegenerate} if the matrices
$$
\left(\sum_{j'=1}^{a(k)}\frac{\partial^2 Y(k,j')}{\partial y_{k,j_1}\partial y_{k,j_2}}\right)_{j_1,j_2 = 1,\cdots,a(k)}
$$
are invertible for $k=1,\cdots,K$, at $\mathfrak y^0$.
\par
\item
We define the multiplicity of leading term equation in the standard
way of algebraic geometry, in the weakly nondegenerate case.
\end{enumerate}
\end{defn}
\begin{exm}
In Example \ref{leadingexample}, the equation (\ref{ex:leading}) is the leading term equation.
\end{exm}
Let $\mathfrak{PO}^{u_0}_*$ be either $\mathfrak{PO}^{u_0}_0$
or $\mathfrak{PO}^{u_0}$.
\begin{thm}\label{thm:elliminate}
For any strongly nondegenerate solution $\mathfrak y^0 =(\mathfrak
y_{k,j;0})$ of leading term equation, there exists a solution
$\mathfrak y =(\mathfrak y_{k,j}) \in (\Lambda^{\C}_0\setminus
\Lambda^{\C}_+)^n$ of
\begin{equation}\label{eq:cri*}
\frac{\partial \mathfrak{PO}^{u_0}_*}{\partial y_{k,j}} (\mathfrak y)= 0
\end{equation}
such that
$
\mathfrak y_{k,j} \equiv \mathfrak y_{k,j;0} \mod \Lambda^{\C}_+.
$
\par
If all the vertices of $P$ and $u_0$ are rational, the same
conclusion holds for weakly nondegenerate $\mathfrak y^0$.
\end{thm}
We prove the following at the end of section \ref{sec:floerhom}.
\begin{lem}\label{lem:autorat}
We assume $[\omega] \in H^2(X;\Q)$ and choose the moment polytope $P$ such that
its vertices are all rational. Let $u_0 \in \text{\rm Int} P$ such that
$\frak{PO}_0^{u_0}$ has weakly nondegenerate critical point in
$(\Lambda_0 \setminus \Lambda_+)^n$. Then $u_0$ is rational.
\end{lem}
The following corollary is an immediate consequence
of Theorem \ref{thm:elliminate} and Lemma \ref{lem:autorat}.
\begin{cor}\label{cor:POPO0}
Let $(\mathfrak x,u) \in \mathfrak M_0(\mathfrak{Lag}(X))$ and $u \in
\text{\rm Int}\,P$. Assume one of the following conditions :
\begin{enumerate}
\item The corresponding solution of the leading term equation is
strongly nondegenerate.
\par
\item $[\omega] \in H^2(X;\R)$ is rational and the corresponding solution of leading term equation is
weakly nondegenerate.
\end{enumerate}
Then there exists $\mathfrak x'$ such that
$(\mathfrak x',u) \in \mathfrak M(\mathfrak{Lag}(X))$ and $\mathfrak x' \equiv \mathfrak x \mod \Lambda^{\C}_+$.
\end{cor}
\begin{rem}
\begin{enumerate}
\item
Using Proposition \ref{prop:approx1} below, we can also apply Theorem
\ref{thm:elliminate} and Corollary \ref{cor:POPO0} to study
non-displacement of Lagrangian fibers for the weakly
nondegenerate case, without assuming rationality. See the last step of the proof of
Theorem \ref{toric-intersect} given at the end of section
\ref{sec:floerhom}.
\item
The conclusion of Theorem \ref{thm:elliminate} does not hold in general
without weakly nondegenerate assumption. We give an example (Example
\ref{counterexamples}) where both the assumption of weak
nondegeneracy and the conclusion of Theorem \ref{thm:elliminate}
fail to hold.
\par
\item
In this section we work with $\Lambda^{\C}$ coefficients, while in the last section
we work with $\Lambda^{\R}$ coefficients.
We also remark that in the last section, we did {\it not}
assume the weak nondegeneracy condition.
\item
If we define the multiplicity of the element of $\mathfrak M_0(\mathfrak{Lag}(X))$
as the dimension of the Jacobian ring $Jac(\mathfrak{PO}_0 ;
\mathfrak x,u_0)$ in Definition
\ref{def:locjac} (namely as the Milnor number) then the sum of the multiplicities of the solutions of (\ref{eq:cri*})
converging to $\mathfrak y^0$ as $T\to 0$,
is equal to the multiplicity of $\mathfrak y^0$.
\item
In the strongly nondegenerate case, the solution of (\ref{eq:cri*})
with the given leading term is unique.
\end{enumerate}\end{rem}
\begin{prop}\label{prop:approx1}
Let $(X,\omega)$ be a compact toric manifold with moment
polytope $P$ and $u_0 \in \text{\rm Int}\, P$. Then
there exist $(X,\omega^h)$ with moment polytope
$P^h$ and $u_0^h \in \text{\rm Int}\,P^h$ such that the following
holds:
\begin{enumerate}
\item $\lim_{h\to\infty}\omega^h = \omega$. $\lim_{h\to\infty}u^h_0 = u_0$.
\item The vertices of $P^h$ and $u_0^h$ are rational.
\item The leading term equation at $u_0^h$ is the same as the leading
term equation at $u_0$.
\end{enumerate}
\end{prop}
We prove Proposition \ref{prop:approx1} at the end of section \ref{sec:floerhom}.
\par
We first derive Theorem \ref{thm:kblowup} and Lemma \ref{lem:hilellim} from Theorem
\ref{thm:elliminate} before proving Theorem \ref{thm:elliminate}.

\begin{proof}[Proof of Theorem \ref{thm:kblowup}]
We start with $\C P^2$ and blow up a $T^2$ fixed point to obtain $\C P^2 \# (-\C P^2)$.
We take a K\"ahler form so that the volume of the exceptional $\C P^1$ is
$2\pi \epsilon_1$ which is small. We next blow up again at one of the fixed points
so that the volume of the exceptional $\C P^1$ is
$\epsilon_2$ and is small compared with $\epsilon_1$. We continue $k$ times to obtain
$X(k)$, whose K\"ahler structure depends on $\epsilon_1,\cdots,\epsilon_k$.
Note $X(k)$ is non-Fano for $k>3$.

Let $P(k)$ be the moment polytope of $X(k)$ and
$\mathfrak{PO}_{0,k}$ be the leading order potential function of $X(k)$.
We remark that $P(k)$ is obtained by cutting out a vertex of $P(k-1)$.
(See \cite{fulton}.)
\begin{lem}\label{lem:bupproj}
We may choose $\epsilon_i$ ($i=1,\cdots,k$) so that the following holds
for $l\le k$.
\begin{enumerate}
\item The number of balanced fibers of $P(l)$ is $l+1$. We write them as
$L(u^{(l,i)})$ $i=0,\cdots,l$.\par
\item $u^{(l-1,i)} = u^{(l,i)}$ for $i\le l-1$. $u^{(l,0)} = (1/3,1/3)$.
\par
\item $u^{(l,l)}$ is in an $o(\epsilon_l)$ neighborhood of the
vertices corresponding to the point we blow up.
\par
\item The leading term equation of $\mathfrak{PO}_{0,l-1}$ at $u^{(l-1,i)}$ is the
same as the leading term equation of $\mathfrak{PO}_{0,l}$ at $u^{(l,i)}$ for $i \le l-1$.
\par
\item The leading term equations are all strongly nondegenerate.
\end{enumerate}
\end{lem}
\begin{proof}
The proof is by induction on $k$. There is nothing to show for $k=0$.
Suppose that we have proved Lemma \ref{lem:bupproj} up to $k-1$. Let
$w$ be the vertex of the polytope we cut out which
corresponds to the blow up of $X(k-1)$. Let $\ell_i$, $\ell_{i'}$ be the
affine functions associated to the two edges containing $w$. It is easy to see that
$$
P(k) = \{u \in P(k-1) \mid \ell_i(u) + \ell_{i'}(u) \ge \epsilon_k\}.
$$
We also have :
$$
\mathfrak{PO}_{0,k} = \mathfrak{PO}_{0,k-1} + T^{\ell_i(u) + \ell_{i'}(u)-\epsilon_k}
y_1^{v_{i,1}+v_{i',1}} y_2^{v_{i,2}+v_{i',2}}.
$$
Therefore if we choose $\epsilon_k$ sufficiently small, the leading term
equation at $u^{(k-1,i)}$ does not change.
\par
We take $u^{(k,k)}$ such that
$$
\ell_i(u^{(k,k)}) = \ell_{i'}(u^{(k,k)}) = \epsilon_k.
$$
It is easy to see that there exists such $u^{(k,k)}$ uniquely if
$\epsilon_k$ is sufficiently small. We put
$$
y'_1 = y_1^{v_{i,1}} y_2^{v_{i,2}}, \quad y'_2 = y_1^{v_{i',1}} y_2^{v_{i',2}}.
$$
(We remark that $v_i$ and $v_{i'}$ are $\Z$ basis of $\Z^2$, since $X(k-1)$ is
smooth toric.) Then we have
$$
\mathfrak{PO}^{u^{(k,k)}}_{0,k} \equiv (y_1' + y'_2 + y'_1y'_2) T^{\epsilon_k} \mod T^{\epsilon_k}\Lambda_+.
$$
Therefore the leading term equation is
$$
1+y'_1 = 1+y'_2 = 0
$$
and hence has a unique solution $(-1,-1)$. In particular it is
strongly nondegenerate. We can also easily check that there is no
other solution of leading term equation. The proof of Lemma
\ref{lem:bupproj} now follows by Theorem \ref{thm:elliminate}.
\end{proof}
Theorem \ref{thm:kblowup} immediately follows from Lemma \ref{lem:bupproj}.
\end{proof}
Note that Theorem \ref{thm:kblowup} can be generalized to $\C P^n$ by the same proof.

\begin{proof}[Proof of Lemma \ref{lem:hilellim}]
Let $u_0 = (n(\alpha+1)/4,(1-\alpha)/2)$. We calculate
$$
\frak{PO}_0^{u_0} = (y_2+ y_2^{-1})T^{(1-\alpha)/2} + (y_1 + y_1^{-1}y_2^{-1})T^{n(\alpha+1)/4}.
$$
The leading equation is
$$
1 - y^{-2}_2 = 0, \quad 1 - y_1^{-2}y_2^{-1} = 0.
$$
Its solution are $(1,1), (-1,1), (\sqrt{-1},-1), (-\sqrt{-1},-1)$,
all of which are strongly nondegenerate. The lemma them follows from Corollary \ref{cor:POPO0}.
\end{proof}
We give another example which demonstrates the way how one can use
the leading term equation and Theorem \ref{thm:elliminate} to locate
balanced fibers.
\begin{exm}\label{hidimexa}
Let us consider $\C P^n$ with moment polytope
$P = \{(u_1,\cdots,u_n) \mid u_i \ge 0, \sum u_i \le 1\}$.
We take $\C P^{n-\ell} \subset \C P^n$ corresponding to
$u_1 = \cdots = u_\ell = 0$. ($\ell \ge 2$.) We blow up $\C P^n$ along the
center $\C P^{n-\ell}$ and denote the blow-up by $X$.
(The case $\ell=n=2$ is Example \ref{exa:onepointblow}.)
We take $\alpha \in (0,1)$ so that the moment polytope of $X$ is
$$
P_{\alpha} = \left\{(u_1,\cdots,u_n) \in P ~\Big | ~ \sum_{i=1}^\ell u_i \ge \alpha\right \}.
$$
Below we use  $\frak{PO}_0$ in place of the potential function  $\frak{PO}$. 
Since the all the critical points of $\frak{PO}_0$ are weakly nondegenerate, 
they correspond  to 
the critical points of $\frak{PO}$. 
(We thank D. McDuff for pointing out that this example is not Fano.)
The function $\frak{PO}_0$ is given by
$$
\frak{PO}_0 = \sum_{i=1}^n T^{u_i}y_i + T^{1-\sum_{i=1}^n u_i} (y_1\cdots y_n)^{-1}
+ T^{\sum_{i=1}^\ell u_i -\alpha} y_1\cdots y_\ell.
$$
We denote
$$
z_i = T^{u_i}y_i, \quad z_0 = T^{1-\sum_{i=1}^n u_i}(y_1\cdots y_n)^{-1},
\quad z = T^{\sum_{i=1}^\ell u_i -\alpha} y_1\cdots y_\ell.
$$
Then the quantum Stanley-Reisner relations are :
$$
z_1\cdots z_n z_0 = T, \,\,\, z_1\cdots z_\ell = z T^{\alpha}.
$$
By computing the derivatives $y_i\frac{\partial\frak{PO}_0}{\partial y_i}$,
we obtain the linear relations which can be written as
\beastar
z_i - z_0 + z & = & 0 \, \quad \mbox{for }\, i \le \ell \\
z_i - z_0 & = & 0 \,\quad \mbox{for } \, i >\ell.
\eeastar
Putting $X = z_0 - z$, $Y = z$, we obtain
$$
z_i = \begin{cases} X & \quad \mbox{for $1 \le i\le \ell$}\\
X + Y & \quad \mbox{for $\, i > \ell$ or $ i =0$}.
\end{cases}
$$
We also have $X^\ell = Y T^{\alpha}$ and
\begin{equation}\label{hdblowequ}
X^{n+1}(X^{\ell-1}T^{-\alpha} + 1)^{n-\ell+1} = T.
\end{equation}
\par\medskip
\noindent{\bf Case 1;} $(\ell-1)\frak v_T(X) < \alpha$ :
In this case, we obtain
$$
\frak v_T(X) = \frac{1+\alpha(n-\ell+1)}{n+1+(n-\ell+1)(\ell-1)}.
$$
The condition $(\ell-1)\frak v_T(X) < \alpha$ then is equivalent to
\begin{equation}
\alpha > \frac{\ell-1}{n+1}.
\nonumber\end{equation}
And we have
$$
\frak v_T(X+Y) = \frak v_T(Y) = \ell\frak v_T(X) - \alpha < \frak v_T(X)
$$
for $i > \ell$. At the point $\text{\bf u} = (u_1,\cdots,u_n)$
with $u_i = \frak v_T(X)$ for $1 \le i \leq \ell$ and $u_i = \frak v_T(X+Y)$
for $i \ge \ell$, we have
$$
\frak{PO}_0^{\text{\bf u}}
= T^{\frak v_T(Y)} (y_{\ell+1} + \cdots + y_n + (y_1\cdots y_n)^{-1}
+ y_1\cdots y_\ell)
+ T^{\frak v_T(X)}(y_1 + \cdots + y_\ell).
$$
To obtain the leading term equation, it is useful to make a change of
variables from $y_1,\cdots,y_n$ to $y_2,\cdots,y_n,y$ with $y = y_1\cdots y_\ell$.
\par
In fact, using the notation which we introduced at the beginning of this section,
we have
$$
S_1 = \frak v_T(Y), \quad S_2 = \frak v_T(X),
$$
and $A^\perp_1$ is the vector space generated by $\partial/\partial u_i$,
$(i=\ell+1,\cdots,n)$ and $\partial/\partial u_1+\cdots + \partial/\partial u_{\ell}$.
We also have
$A^\perp_2 = \R^n$.
Therefore a basis satisfying Condition \ref{cond:estar} is
$$
\frac{\partial}{\partial u_{2}},\quad \cdots, \quad \frac{\partial}{\partial u_{n}},
\quad
\frac{\partial}{\partial u_{1}} + \cdots +\frac{\partial}{\partial u_{\ell}}.
$$
The variables corresponding to this basis is $y_2,\cdots,y_n,y$.
In these variables, $\frak{PO}_0^{\text{\bf u}}$
has the form
\beastar
\frak{PO}_0^{\text{\bf u}} & = & T^{\frak v_T(Y)}(y_{\ell+1} + \cdots + y_n
+ (y y_{\ell+1}\cdots y_n)^{-1} + y) \\
&{}& \quad +T^{\frak v_T(X)}(y(y_2\cdots y_\ell)^{-1} +
y_2 + \cdots + y_\ell).
\eeastar
Then the leading term equation becomes :
$$
\begin{cases}
0 = 1 - y^{-1}(y y_{\ell+1}\cdots y_n)^{-1} \\
0 = 1 - y_i^{-1}(y y_{\ell+1}\cdots y_n)^{-1} & \mbox{for } i>\ell \\
0 = 1 - y_i^{-1}y(y_2\cdots y_\ell)^{-1} \, & \mbox{for }\, 2 \le i \le \ell.
\end{cases}
$$
Its solutions are
\beastar
y_{\ell+1} = \cdots = y_n = y = \theta, & \qquad \theta^{n-\ell+2} = 1\\
y_2 = \cdots = y_\ell = \rho, & \qquad \rho^\ell = \theta.
\eeastar
It follows that this leading term equation has $\ell(n-\ell+2)$ solutions,
all of which are strongly nondegenerate.
\par\medskip
\noindent{\bf Case 2;} $(\ell-1)\frak v_T(X) > \alpha$ : We have
$$
\frak v_T(X) = \frac{1}{n+1}.
$$
We also have
$\frak v_T(Y) = \ell\frak v_T(X) - \alpha > \frak v_T(X)$ and hence
$$
\frak v_T(X+Y) = \frak v_T(X) = \frac{1}{n+1}.
$$
The condition $(\ell-1)\frak v_T(X) > \alpha$ becomes
$\alpha < \frac{\ell-1}{n+1}.$
If we consider the point $\text{\bf u} = (u_1, \cdots, u_n)$ with
$u_i = \frak v_T(X) = \frac{1}{n+1}$ for $i = 1,\cdots n$, $\frak{PO}_0^{\text{\bf u}}$
has the form
$$
\frak{PO}_0^{\text{\bf u}} = T^{1/n+1}(y_1 + \cdots + y_n + (y_1\cdots y_n)^{-1})
+ T^{\frac{\ell}{n+1} - \alpha}y_1\cdots y_{\ell}
$$
and so the leading term equation is obtained by differentiating
$$
y_1 + \cdots + y_n + (y_1\cdots y_n)^{-1}.
$$
Its solutions are
$$
y_1 = \cdots = y_n = \theta, \qquad \theta^{n+1} =1.
$$
There are $n+1$ solutions all of which are strongly nondegenerate.
\par\medskip
\noindent{\bf Case 3;} $(\ell-1)\frak v_T(X) = \alpha$ :
\par\smallskip
\noindent(Case 3-1; $- X^{\ell-1} \not\equiv T^{\alpha} \mod T^{\alpha}\Lambda_+$.)
We have
$u_i = \frak v_T(X) = \frak v_T(Y) = 1/(n+1)$ ($i=1,\cdots,n$), $\alpha = (\ell-1)/(n+1)$.
In this case
$$
\frak{PO}_0^{\text{\bf u}}
= T^{1/(n+1)}(y_1 + \cdots + y_n + (y_1\cdots y_n)^{-1} + y_1\cdots y_\ell).
$$
\par
This formula implies that the symplectic area of all discs with Maslov index $2$ are
$2\pi/(n+1)$. 
\par
The leading term equation is
$$
\begin{cases}
1 - \frac{1}{y_i}((y_1\cdots y_n)^{-1} - y_1\cdots y_\ell)= 0, & i=1,\cdots,\ell \nonumber\\
1 - \frac{1}{y_i}((y_1\cdots y_n)^{-1} )= 0, & i=\ell+1,\cdots,n. \nonumber
\end{cases}
$$
Its solutions are
$y_1 = \cdots = y_\ell = \rho$, $y_{\ell+1} = \cdots = y_n = \theta$
with
\begin{equation}\label{mntakusanj}
\rho^\ell(\rho + \rho^\ell)^{n-\ell+1} = 1,
\end{equation}
and $\theta = \rho + \rho^\ell$. It looks rather cumbersome to check whether (\ref{mntakusanj}) has multiple root.
Certainly all the solutions are weakly nondegenerate.
The number of solutions counted with multiplicity
is $\ell(n-\ell+2)$.
\par\smallskip
\noindent(Case 3-2; $- X^{\ell-1} \equiv T^{\alpha} \mod T^{\alpha}\Lambda_+$.)
We have
$$
z_0 = X + Y = T^{-\alpha} X ( T^{\alpha} + X^{\ell-1}).
$$
We put
$$
\frak v_T(z_0) = \lambda > \frak v_T(X) = \frac{\alpha}{\ell-1}.
$$
Using (\ref{hdblowequ}) we obtain
$$
\lambda = \frac{\ell-1 - \ell\alpha}{(\ell-1)(n-\ell+1)}.
$$
The condition $\lambda > \frac{\alpha}{\ell-1}$ becomes $\alpha < \frac{\ell-1}{n+1}$.
We have $u_i = \frak v_T(X) = \alpha/(\ell-1)$, $i\le \ell$ and
$u_i = \frak v_T(Y) = \lambda$, $i > \ell$.
We have
$$
\frak{PO}_0^{\text{\bf u}}
= T^{\frak v_T(X)}(y_1 + \cdots + y_\ell
+ y_1\cdots y_\ell)
+ T^{\frak v_T(Y)} (y_{\ell+1} + \cdots + y_n + (y_1\cdots y_n)^{-1}).
$$
The first term gives the leading term equation :
$$
1 + y_1 \cdots \widehat{y}_i \cdots y_\ell = 0, \qquad i=1,\cdots,\ell.
$$
It solution is
$y_1 = \cdots = y_\ell = \rho$
and
$ \rho^{\ell-1} = -1$. The second term of $\frak{PO}_0^{\text{\bf u}} $ gives the
leading term equation
$$
1 - y_i^{-1} \rho^{-\ell} (y_{\ell+1}\cdots y_n)^{-1}= 0, \qquad i=\ell+1,\cdots,n.
$$
Its solutions are
$
y_{\ell+1} = \cdots y_n = \theta,
$ with
$$
\rho^\ell \theta^{n-\ell+1} = 1.
$$
Thus the leading term equation has $(\ell-1)(n+1-\ell)$ solutions all of which are
strongly nondegenerate.
\par
We remark that
$(\ell-1)(n+1-\ell) + (n+1) = \ell(n-\ell+2)$. Hence the number of solutions are
always $\ell(n-\ell+2)$ which coincides with the Betti number of $X$.
There are $2$ balanced fibers for $\alpha < (\ell-1)/(n+1)$ and
$1$ balanced fiber for $\alpha \ge (\ell-1)/(n+1)$.
\par\medskip
By the above discussion and Theorem \ref{QHequalMilnor2}
(see also Remark \ref{5.7rational}),
we can calculate $QH(X;\Lambda^{\Q})$ as follows :
\begin{equation}\label{eq:field}
QH(X;\Lambda^{\Q})
=
\begin{cases}
\Lambda^{R_1} & \alpha > (\ell-1)/(n+1), \\
\Lambda^{R_2} & \alpha = (\ell-1)/(n+1), \\
\Lambda^{R_3} \times \Lambda^{R_4} & \alpha < (\ell-1)/(n+1),
\end{cases}
\end{equation}
where
\begin{eqnarray*}
&&
R_1 = \Q[Z]/(Z^{\ell(n-\ell+2)} - 1), \nonumber\\
&& R_2 = \Q[Z]/(Z^\ell(Z + Z^\ell)^{n-\ell+1} - 1), \nonumber\\
&& R_3 = \Q[Z]/(Z^{n+1} - 1), \\
&& R_4 = \Q[Z_1,Z_2]/(Z_1^{\ell-1} + 1, Z_1^\ell Z_2^{n-\ell+1} - 1).
\nonumber
\end{eqnarray*}
Here we assume  (\ref{mntakusanj}) 
has only a simple root in case $\alpha = (\ell-1)/(n+1)$.
We use the next lemma to show (\ref{eq:field}).
\begin{lem}\label{lemfield}
Let $\mathfrak x = \sum \frak x_i\text{\bf e}_i$ be a critical point of $\frak{PO}^{\text{\bf u}}_*$. We assume
$\frak y_{i;0}= e^{\frak x_{i;0}}$ (where $\frak x_i = \frak x_{i;0} \mod \Lambda_+$) is a
strongly nondegenerate solution of leading term equation.
We also assume that $\frak y_{i;0} \in F$ where $F \subset \C$ is a field.
\par
Then $\frak y_i =e^{\frak x_i}$ is an element of $\Lambda_0^F$.
\end{lem}
We prove Lemma \ref{lemfield} later this section.
\end{exm}

We are now ready to give the proof of Theorem \ref{thm:elliminate}.

\begin{proof}[Proof of Theorem \ref{thm:elliminate}]
We first consider the weakly nondegenerate case. Let $\mathfrak m$
be the multiplicity of $\mathfrak y^0$. We choose $\delta$ such that
the ball $B_{\delta}(\mathfrak y^0)$ centered at $\mathfrak y^0$ and
of radius $\delta$ does not contain a solution of the leading term
equation other than $\mathfrak y^0$. For $y \in \partial
B_{\delta}(\mathfrak y^0)$ we define
$$
\nabla \overline{\mathfrak{PO}}(y) = \left(\sum_{j'=1}^{a(k)}\frac{\partial
Y(k,j')}{\partial y_{k,j}}(y)\right)_{k=1,\cdots,K,\, j = 1,\cdots,d(k)}
\in \C^n.
$$
The map
$$
y \mapsto \frac{\nabla \overline{\mathfrak{PO}}(y)}{\Vert\nabla \overline{\mathfrak{PO}}(y)\Vert} \in S^{2n-1}
$$
is well defined and of degree $\mathfrak m \ne 0$ by the definition of multiplicity.
\par
We define $\mathfrak{PO}^{u_0}_{*,k,\CN}$ in the same way as
(\ref{37.167}). For $q\in \C$, we define
$\mathfrak{PO}^{u_0}_{*,k,\CN}(\cdots;q)$ by substituting $q$ to $T$.
Then in the same way as the proof of Lemma \ref{37.175} we can prove
the following.
\begin{lem}
There exists $\epsilon> 0 $ such that if $\vert q\vert < \epsilon$, the equation
\begin{equation}\label{eq:orderN}
0 = \frac{\partial }{\partial y_{k,j}}\mathfrak{PO}^{u_0}_{*,k,\CN}(\cdots;q)
\end{equation}
has a solution
in $B_{\delta}(\mathfrak y^0)$. The sum of multiplicities of the solutions of
$(\ref{eq:orderN})$
converging to $\mathfrak y_{k,j;0}$ is $\mathfrak m$.
\end{lem}
(\ref{eq:orderN}) is a polynomial equation. Hence multiplicity of its solution
is well-defined in the standard sense of algebraic geometry.
\par
Now we assume that all the vertices of $P$ and $u_0$ are contained
in $\Q^n$. (\ref{eq:orderN}) also depends polynomially on $q' = T'$,
where $T' = T^{1/\mathcal C!}$ for a sufficiently large integer $\mathcal C$. (We remark that
$\mathcal C$ is determined by the denominators of the coordinates of the
vertices of $P$ and of $u_0$. In particular it can be taken to be
independent of $\CN$.)
\par
We denote $y = (y_1,\cdots, y_n)$ and put
$$
\mathfrak X = \{(y,q') \mid y \in
B_{\delta}(\mathfrak y^0), \, q' \, \text{with $\vert q'\vert < \epsilon$ and
$q = (q')^{\mathcal C!}$ satisfying (\ref{eq:orderN})} \}.
$$
We consider the projection
\begin{equation}\label{eq;piqdash}
\pi_{q'} : \mathfrak X \to \{q' \in \C \mid \vert q'\vert < \epsilon\}.
\end{equation}
By choosing a sufficiently small $\epsilon > 0$, we may assume that
(\ref{eq;piqdash}) is a local isomorphism on the punctured disc
$\{ q' \mid 0 < \vert q'\vert < \epsilon\}$. Namely, $\pi_{q'}$
is an \'etale covering over the punctured disc.
\par
We remark that for each $q'$ the fiber consists of at most
$\mathfrak m$ points, since the multiplicity of the leading term
equation is $\mathfrak m$. We put $q'' = (q')^{1/\mathfrak m!}$.
Then the pull-back
\begin{equation}\label{eq;piqdashdash}
\pi_{q''} : \mathfrak X' \to \{q'' \in \C \mid 0 < \vert q''\vert <
\epsilon\}
\end{equation}
of (\ref{eq;piqdash}) becomes a trivial covering space. Namely there
exists a single valued section of $\pi_{q''}$ on $\{ q'' \mid 0 <
\vert q''\vert < \epsilon\}$. It extends to a holomorphic section of
$\{ q'' \mid \vert q''\vert < \epsilon\}$.
\par
In other words there exists a holomorphic family of solutions of (\ref{eq:orderN})
which is parameterized
by $q'' \in \{ q'' \mid \vert q''\vert < \epsilon\}$. We put $T'' = (T')^{1/\mathfrak m!}$.
Then by taking the Taylor series of the $q''$-parameterized family of solutions
at $0$,
we obtain the following :
\begin{lem}\label{lem:modTNexist}
If all the vertices of $P$ and $u_0$ are rational, then for each $\CN$ there exists
$\mathfrak y^{(\CN)} =(\mathfrak y^{(\CN)}_{k,j})$
$$
\mathfrak y^{(\CN)}_{k,j} = \sum_{l=0}^{\CN} \mathfrak y^{(\CN)}_{k,j;l} (T'')^{l}
$$
$(\mathfrak y^{(\CN)}_{k,j;l} \in \C)$
such that
\begin{equation}\label{eq:cri*N}
\frac{\partial \mathfrak{PO}^{u_0}_*}{\partial y_{k,j}}(\mathfrak y^{(\CN)}_{k,j}) \equiv 0
\mod (T'')^{\CN+1}
\end{equation}
and that
$
\mathfrak y^{(\CN)}_{k,j;0} \equiv \mathfrak y_{k,j;0}.
$
\end{lem}
We remark that Lemma \ref{lem:modTNexist} is sufficient for
most of the applications. In fact it implies that $L(u_0)$ is balanced if
there exists a weakly nondegenerate solution of leading term equation at $u_0$.
Hence we can apply Lemma \ref{prof:bal2}.
\par
For completeness we prove the slightly stronger statement made for
the weakly nondegenerate case in Theorem \ref{thm:elliminate}. The
argument is similar to one in subsection 7.2.11 \cite{fooo06} (= subsection 30.11 \cite{fooo06pre}).
\par
For each $\CN$, we denote by
$\widetilde{\mathfrak M}((\mathfrak y_{k,j;0});\CN)$
the set of all
$(\mathfrak y_{k,j;l}^{(\CN)})_{k,j;l}
\in \C^{n\CN}$,
where $k=1,\cdots,K$, $j=1,\cdots,a(k)$, $l=1,\cdots,\CN$, such that
$$
\mathfrak y^{(\CN)}_{k,j} = \mathfrak y_{k,j;0}
+ \sum_{l=1}^{\CN} \mathfrak y^{(\CN)}_{k,j;l} (T'')^{l}
$$
satisfies (\ref{eq:cri*N}).
\par
By definition, $\widetilde{\mathfrak M}((\mathfrak y_{k,j;0});\CN)$
is the set of $\C$-valued points of certain complex
affine algebraic variety (of finite dimension).
Lemma \ref{lem:modTNexist} implies that
$\widetilde{\mathfrak M}((\mathfrak y_{k,j;0});\CN)$
is nonempty.
For $\CN_1>\CN_2$ there exists an obvious morphism
$$
I_{\CN_1,\CN_2} :
\widetilde{\mathfrak M}((\mathfrak y_{k,j;0});\CN_1)
\to \widetilde{\mathfrak M}((\mathfrak y_{k,j;0});\CN_2)
$$
of complex algebraic variety.
\par
To complete the proof of Theorem \ref{thm:elliminate} in the weakly
nondegenerate case, it suffices to show that the projective limit
\begin{equation}\label{eq;prolimit}
\lim_{\longleftarrow} (\widetilde{\mathfrak M}(\mathfrak y_{k,j;0};\CN))
\end{equation}
is nonempty.
\begin{lem}\label{constructible}
$$
\bigcap_{\CN>1} \text{\rm Im}I_{\CN,1} \ne \emptyset.
$$
\end{lem}
\begin{proof}
By a theorem of Chevalley (see Chapter 6 \cite{Ma70}), each
$\text{\rm Im}I_{\CN,1}$ is a constructible set. It is nonempty and
its dimension $\dim\text{\rm Im}I_{\CN,1}$ is nonincreasing as $\CN \to
\infty$. Therefore we may assume $\dim\text{\rm Im}I_{\CN,1} = d$ for
$\CN\ge \CN_1$.
\par
We consider the number of $d$ dimensional irreducible components of
$\text{\rm Im}I_{\CN,1}$. This number is nonincreasing for $\CN\ge \CN_1$.
Therefore, there exists $\CN_2$ such that for $\CN\ge \CN_2$ the number of
$d$ dimensional irreducible components of $\text{\rm Im}I_{\CN,1}$ is
independent of $\CN$. It follows that there exists $X_{\CN}$ a sequence of
$d$ dimensional irreducible components of $\text{\rm Im}I_{\CN,1}$
such that $X_{\CN+1} \subset X_{\CN}$. Since $\dim (X_{\CN} \setminus
X_{\CN+1}) < d$, it follows from Baire's category theorem that
$\cap_{\CN} X_{\CN} \ne \emptyset$. Hence the lemma.
\end{proof}
\begin{lem}\label{projlimitinduction}
There exists a sequence $(\mathfrak y_{k,j;l}^{(n)})_{k,j;l}$
$n=1,2,3,\cdots \, m$
such that
$$
I_{n,n-1}((\mathfrak y_{k,j;l}^{(n)})_{k,j;l}
= (\mathfrak y_{k,j;l}^{(n-1)})_{k,j;l})
$$
for $n=2,\cdots,m$
and that
$$
(\mathfrak y_{k,j;l}^{(m)})_{k,j;l}
\in \bigcap_{\CN>m} \text{\rm Im}I_{\CN,m}.
$$
\end{lem}
\begin{proof}
The proof is by induction on $m$. The case $m=1$ is
Lemma \ref{constructible}. Each of the inductive step
is similar to the proof of Lemma \ref{constructible} and
so it omitted.
\end{proof}
Lemma \ref{projlimitinduction} implies that the projective limit
(\ref{eq;prolimit}) is nonempty. The proof of weakly nondegenerate
case of Theorem \ref{thm:elliminate} is complete.
\par\medskip
We next consider the strongly nondegenerate case. We prove the
following lemma by induction on $\CN$. Let $G$ be a submonoid of
$(\R_{\ge 0},+)$ generated by the numbers appearing in the exponent
of (\ref{37.160}). Namely it is generated by
\begin{equation}\label{eq:genG}
\aligned
S_{k'} - S_k \,\, (k' > k),
\quad
&\ell(u_0) + \rho - S_k\,\, ((\ell,\rho) \in \mathfrak I_{k'}, \,\, k' \ge k),
\\ &\ell(u_0) - S_k \,\,(\ell \in \mathfrak I).
\endaligned
\end{equation}
We define $0< \lambda_1 < \lambda_2 < \cdots$
by
$$
\{ \lambda_i \mid i=1,2,\dots \} = G.
$$
\begin{lem}\label{37.164}
We assume that $\mathfrak y^0 = (\mathfrak y_{k,j;0})_{k=1,\cdots,K,\, j=1,\cdots,d(k)}$ is a
strongly nondegenerate solution of the leading term equation. Then,
there exists
$$
\mathfrak y^{(\CN)}_{k,j} = \mathfrak y_{k,j;0} + \sum_{l=1}^{\CN} \mathfrak y_{k,j;l}T^{\lambda_{l}}
$$
such that
\begin{equation}
\sum_{j'=1}^{a(k)}\frac{\partial
Y(k,j')}{\partial y_{k,j}}(\mathfrak y^{(\CN)}_{k,1},\cdots,\mathfrak y^{(\CN)}_{K,d(K)}) \equiv 0
\mod T^{\lambda_{\CN+1}}.
\end{equation}
Moreover we may choose $\mathfrak y^{(\CN)}_{k,j}$ so that
$$
\mathfrak y^{(\CN)}_{k,j} \equiv \mathfrak y^{(\CN+1)}_{k,j} \mod T^{\lambda_{\CN+1}}.
$$
\label{37.165}\end{lem}
\begin{proof}
The proof is by induction on $\CN$. There is nothing to show in the
case $\CN=0$. Assume we have proved the lemma up to $\CN-1$. Then we
have
$$
\sum_{j'=1}^{a(k)}\frac{\partial
Y(k,j')}{\partial y_{k,j}}(\mathfrak y^{(\CN-1)}_{k,1},\cdots,\mathfrak y^{(\CN-1)}_{K,d(K)}) \equiv c_{k,j,M}T^{\CN}
\mod T^{\lambda_{\CN+1}}.
$$
Consider $\mathfrak y^{(\CN)}_{k,j}$ of the form
$$
\mathfrak y^{(\CN)}_{k,j} = \mathfrak y^{(\CN-1)}_{k,j} + \Delta_{k,j,\CN}
T^{\lambda_{\CN}}
$$
for some $\Delta_{k,j,\CN}$. Then we can write
\begin{equation}\label{indformula}\aligned
\sum_{j'=1}^{a(k)}&\frac{\partial
Y(k,j')}{\partial y_{k,j}}(\mathfrak y^{(M-1)}_{k,1} + \Delta_{k,1,\CN} T^{\lambda_{\CN}},\cdots,
\mathfrak y^{(\CN-1)}_{K,d(K)} + \Delta_{K,d(K),\CN} T^{\lambda_{\CN}})
\\
&\equiv \left(c_{k,j,\CN} +
\sum_{j',j''=1}^{a(k)}\frac{\partial^2 Y(k,j')}{\partial y_{k,j}\partial y_{k,j''}}\Delta_{k,j'',\CN}
\right)T^{\lambda_\CN}
\mod T^{\lambda_{\CN+1}}.
\endaligned\end{equation}
Since $\mathfrak y^0 = (\mathfrak y_{k,j;0})_{k=1,\cdots,K,\, j=1,\cdots,d(k)}$ is
strongly nondegenerate,
we can find $\Delta_{k,j'',\CN} \in \C$ so that
the right hand side become zero module $T^{\lambda_{\CN+1}}$.
The proof of Lemma \ref{37.164} is complete.
\end{proof}
By Lemma \ref{37.164}, the limit $\lim_{\CN\to\infty}\mathfrak
y^{(\CN)}_{k,j}$ exists. We set
$$
\mathfrak y_{k,j}: = \lim_{\CN \to\infty}\mathfrak y^{(\CN)}_{k,j}.
$$
This is the required solution of (\ref{eq:cri*}). The proof of
Theorem \ref{thm:elliminate} is complete
\end{proof}
\begin{proof}[Proof of Lemma \ref{lemfield}]
We put
$$
\mathfrak y_{k,j} = \mathfrak y_{k,j;0} + \sum_{l=1}^{\infty} \mathfrak y_{k,j;l}T^{\lambda_{l}}.
$$
By assumption $\mathfrak y_{k,j;0} \in F$. We remark that (\ref{indformula})
gives a {\it linear} equation which determines $\mathfrak y_{k,j;l}$ inductively on $l$.
We use it to show $\mathfrak y_{k,j;l} \in F$ inductively on $l$.
\end{proof}
We next give an example where weakly nondegeneracy condition
is not satisfied.

\begin{exm}\label{counterexamples}
Consider the 2-point blow up $X(\alpha,\beta)$ of $\C P^2$ with its moment
polytope given by
$$
P = \{(u_1,u_2) \mid
0 \le u_1 \le 1, \,\, 0 \le u_2 \le 1-\alpha, \,\, \beta \le u_1 + u_2 \le 1\}.
$$
We consider the case when $1-\alpha$ is sufficiently small.
The potential function is
$$
\mathfrak{PO} = T^{u_1}y_1 + T^{u_2}y_2 + T^{1-\alpha-u_2}y_2^{-1} + T^{1-u_1-u_2}y_1^{-1}y_2^{-1}
+T^{u_1+u_2-\beta}y_1y_2.
$$
(We remark that $X$ is Fano.) We fix $\alpha$ and move $\beta$ starting from zero.
When $\beta$ is small compared to $1-\alpha$,
there are two balanced fibers.
\par
One is located at $((1+\alpha)/4,(1-\alpha)/2)$. This corresponds to
the location of the balanced fiber of
the one point blow up, which is nothing but the case $\beta=0$.
The other appears near the origin and is $(\beta,\beta)$.
The leading term equation at the first point is
$$
1-y_2^{-2} = 0, \quad 1 - y_1^{-2}y_2^{-1} = 0.
$$
The solutions are $(y_1,y_2) = (\pm 1,1), (\pm\sqrt{-1},-1)$,
all of which are strongly nondegenerate.
The leading term equation at the second point is
$$
1+y_1 = 1+y_2 = 0.
$$
$(-1,-1)$ is the nondegenerate solution.
Thus we have $5$ solutions.
\par\medskip
The situation jumps when $\beta = (1-\alpha)/2$. Denote $\beta_0 =(1-\alpha)/2$
for the simplicity of notation.
In that case the potential function at $(\beta_0,\beta_0)$ becomes
$$
T^{\beta_0}(y_1+y_2+y_1y_2+y_2^{-1}) + T^{1-2\beta_0}y_1^{-1}y_2^{-1}.
$$
The leading term equation is
$$
1+ y_2 = 0,\quad 1+y_1-y_2^{-2}=0.
$$
Its solution is $(0,-1)$. Since $y_1=0$ it follows that there is
no solution in $(\Lambda_0\setminus \Lambda_+)^2$.
Hence there is no weak bounding cochain $\mathfrak x$ for which the
Floer cohomology $HF((L(\beta_0,\beta_0),\mathfrak x),(L(\beta_0,\beta_0),\mathfrak x)
;\Lambda)$ is nontrivial.
In other words, the fiber $L(\beta_0,\beta_0)$ in $X(\alpha,\beta_0)$ is not
strongly balanced.
\par
On the other hand, this fiber $L(\beta_0,\beta_0)$ in $X(\alpha,\beta_0)$
is balanced because by choosing $\beta$ arbitrarily close to $\beta_0$ and
$\beta < \beta_0$,
we can approximate it by the fibers
$$
L(\beta,\beta) \subset X(\alpha,\beta)
$$
for which the Floer cohomology
$HF((L(\beta,\beta),\mathfrak x),(L(\beta,\beta),\mathfrak x))$ is non-trivial.
In particular $L(\beta_0,\beta_0)$ in $X(\alpha,\beta_0)$ is not displaceable.
\par
We can also verify that
\begin{equation}\label{Ecounterexa}
\overline{\mathfrak E}(\beta_0,\beta_0) = \infty \quad \mbox{in } \, X(\alpha,\beta_0)
\end{equation}
while $\mathfrak E(\beta_0,\beta_0) = \beta_0$ in $X(\alpha_0,\beta_0)$.

Now we examine where the missing solutions at $\beta = \beta_0$ have gone.
We consider $(u_1,(1-\alpha)/2)$ where $\beta_0 = (1-\alpha)/2<u_1<(1+\alpha)/4$.
The potential function is
\begin{equation}\label{degpofunc}
T^{\beta_0}(y_2 + y_2^{-1}) + T^{\beta_0+\lambda_1}(y_1 + y_1y_2)
+ T^{\beta_0+\lambda_2}y_1^{-1}y_2^{-1}.
\end{equation}
Here
$$
\lambda_1 = u_1 - \beta_0 < \lambda_2 = (1+\alpha)/2 - u_1-\beta_0.
$$
The leading term equation is
\begin{equation}\label{degenerate}
1- y_2^{-2} = 0, \quad 1 + y_2 =0.
\end{equation}
The solution is $y_2 =-1$ and $y_1$ is arbitrary.
Thus there are infinitely many solutions of the leading term equation. Therefore
these solutions of (\ref{degenerate}) are not weakly nondegenerate.
\par
So we need to study the critical point of (\ref{degpofunc})
more carefully. The condition that
$(y_1,y_2)$ is a critical point of (\ref{degpofunc}) is written as
\begin{equation}\label{criticaleqexam}
\left\{\aligned
&1 - y_2^{-2} + T^{\lambda_1} y_1 - T^{\lambda_2} y_1^{-1}y_2^{-2} = 0 \\
&1+y_2 - T^{\lambda_2-\lambda_1}y_1^{-2}y_2^{-1} = 0.
\endaligned\right.
\end{equation}
The leading order term of $y_2$ should be $-1$. We need to
study also the second order term. We can write
$$
y_2 = -1 + c T^{\mu}, \quad y_1 = d,
$$
where $c,d \in \Lambda_{0} \setminus \Lambda_+$.
Then we have
\begin{eqnarray}
&-2cT^{\mu} + d T^{\lambda_1}\equiv 0 \mod T^{\min\{\mu,\lambda_1\}}
\Lambda_+,
\label{mulambda1}\\
&c T^{\mu} + d^{-2}T^{\lambda_2-\lambda_1}
\equiv 0 \mod T^{\min\{\mu,\lambda_2-\lambda_1\}}\Lambda_+.
\label{mulambda2}\end{eqnarray}
(\ref{mulambda1}) implies $\mu = \lambda_1$.
(\ref{mulambda2}) then implies $\lambda_2-\lambda_1 = \lambda_1$.
It implies $u_1 = 1/3$.
Furthermore
\begin{equation}\label{LTE}
c^3 \equiv -1/4 \mod \Lambda_+, \quad
d \equiv 2c \mod \Lambda_+.
\end{equation}
Since the three solutions of the $\C$-reduction of (\ref{LTE}) are all simple,
we can show, by the same way as that of the proof of Theorem \ref{thm:elliminate},
that all solutions correspond to solutions of the equation
(\ref{criticaleqexam}). Therefore, $L(1/3,\beta_0)$ is a strongly balanced fiber.
\par
We remark that solutions of the
leading term equation (\ref{degenerate}) do not
lift to solutions of (\ref{criticaleqexam})
unless $u_1 = 1/3$ and $y_1 = -1$. This shows that
the weakly nondegeneracy assumption in Theorem \ref{thm:elliminate}
is essential.
\par
We remark that at $((1+\alpha)/4,(1-\alpha)/2)$ the leading term equation
becomes
$$
1-y_2^{-2} = 0, \quad 1 + y_2 - y_1^{-2}y_2^{-1} = 0.
$$
Its solutions in $(\C \setminus \{0\})^2$ are $(\pm 1/\sqrt{2},1)$.
The number of solutions jumps from 4 to 2 here.
$2+ 3 = 5$. So this is consistent with Theorem \ref{lageqgeopt0}.
\par
In summary for the case of $(\alpha,\beta_0)$ with $\beta_0 = (1-\alpha)/2$,
there are 3 balanced fibers $(1/3,\beta_0)$, $((1+\alpha)/4,\beta_0)$
and $(\beta_0,\beta_0)$. The first
two of them are strongly balanced and the last is not strongly balanced.
\par
The balanced fiber $L(1/3,\beta_0) \subset X(\alpha,\beta_0)$ disappears as
we deform $X(\alpha,\beta_0)$ to $X(\alpha,\beta)$ as $\beta_0$
moves to nearby $\beta$.
To see this let us take $\beta$
which is slightly bigger than $\beta_0 = (1-\alpha)/2$. Then $((\alpha+\beta)/2,(1-\alpha)/2)$,
and $(1-\alpha-\beta,\beta)$
are the balanced fibers. The leading term equation at the first point is
$$
1 - y_2^{-2} = 0, \quad y_2 - y_1^{-2}y_2^{-1} = 0.
$$
Hence there are 4 solutions $(\pm 1,\pm 1)$.
The leading term equation at the second point is
$$
1+ y_2 = 0, \quad y_1 - y_2^{-2} = 0.
$$
Hence the solution is $(1,-1)$. Total number is again 5.
\par
We remark that the $\Q$-structure of quantum cohomology also jumps at
$\beta = (1-\alpha)/2$. Namely
$$
QH(X(\alpha,\beta);\Lambda^{\Q})
=
\begin{cases}
\Lambda^{\Q(\sqrt{-1})} \times (\Lambda^{\Q})^3 &\text{$\beta$ is slightly smaller than $(1-\alpha)/2$}, \\
\Lambda^{\Q(\sqrt{2})} \times \Lambda^{ \Q((-2)^{1/3})} &\text{$\beta=(1-\alpha)/2$}, \\
(\Lambda^{\Q})^5 &\text{$\beta$ is slightly larger than $(1-\alpha)/2$}.
\end{cases}
$$
\end{exm}
\begin{rem}\label{rembulk}
In a sequel of this series of papers, we will prove that
$L(u_1,(1-\alpha)/2))$ is not displacable for any $u_1 \in
((1-\alpha)/2,1/3) \cup (1/3,(1+\alpha)/4)$, in the case
$\beta = \beta_0$. We will
use the bulk deformation introduced by 
\cite{fooo06} section 3.8 (= \cite{fooo06pre} section 13)
to prove it.
\end{rem}
The next example shows that Theorems \ref{lageqgeopt0}, \ref{exitnonvani} and
\ref{thm:elliminate} can not be generalized
to the case of a positive characteristic.

\begin{exm}\label{torsioncount}
Consider the $2$-point blow up $X$ of $\C P^2$ with moment
polytope
$$
P = \{(u_1,u_2) \mid
0 \le u_i \le 1-\epsilon, \sum u_i \le 1\}.
$$
Since $X$ is monotone for $\epsilon = 1/3$, it follows that $X$ is Fano.
We will assume $\epsilon > 0$ is sufficiently small.
Then the fiber at $u_0= (1/3,1/3)$ is balanced.

Now we consider the Novikov ring $\Lambda^F$ with $F=\mathbb{F}_3$
a field of characteristic $3$. We will prove that
there exists {\it no} element $\mathfrak x \in H(L(u);\Lambda_{0}^{F})$
such that $HF((L(u_0),\mathfrak x),(L(u_0),\mathfrak x);\Lambda^F) \ne 0$.
\par
The potential function at $u_0$ is
$$
\frak{PO}^{u_0} = T^{1/3}(y_1+y_2+ 1/(y_1y_2))+
T^{2/3-\epsilon}(y_1^{-1}+y_2^{-1}).
$$
Therefore the critical point equation is given by
\begin{equation}\label{eqcrit}
1-1/(y_iy_1y_2)-ty_i^{-2} = 0 \quad i = 1, 2
\end{equation}
where $t = T^{1/3-\epsilon}$. From this it follows that
$y_i\equiv 1 \mod \Lambda_+$.
In fact the leading term equation is $a_ia_1a_2 = 1$ for $i = 1, 2$
which is reduced to
$$
a_1 = a_2 = a, \quad a^3 = 1.
$$
Obviously this equation has the unique solution $a_1 = a_2 =1$ in $\mathbb{F}_3$.

Going back to the study of solutions of the critical point equation (\ref{eqcrit}),
we first prove $y_1 = y_2$. We put $z_i = y_i^{-1}$.
We assume $z_i - z_j \ne 0$ and put $z_i - z_j \equiv t^{\lambda}c
\mod t^{\lambda}\Lambda_+$
with $c\in \mathbb{F}_3\setminus\{ 0\}$. Then by (\ref{eqcrit}) we have
$$
(z_i - z_j) z_1z_2 + t(z_i-z_j)(z_i+z_j) = 0.
$$
This is a contradiction since the left hand side is congruent to
$ct^{\lambda}$ modulo $t^{\lambda}\Lambda_+^{\mathbb{F}_3}$.
\par
We now put $x = y_i$ and obtain
\begin{equation}\label{eq:x}
x^3 - t x - 1=0.
\end{equation}
We prove:
\begin{lem}\label{nosolutionF}
$(\ref{eq:x})$ has no solution in $\Lambda_0^{\mathbb{F}_3}$.
\end{lem}
\begin{proof}
We put $x = 1+ t^{1/3}x'$ and obtain
$$
(x')^3 - t^{1/3}x' - 1 = 0.
$$
This equation resembles (\ref{eq:x}) except that
$t$ is replaced by $t^{1/3}$. We now put
$$
x_N \equiv 1 + \sum_{k=1}^N t^{\sum_{i=1}^k3^{-i}}.
$$
Then
$
x_N^3 = 1 + t x_{N-1}.
$
Therefore
$$
(x_N)^3 - t x_N - 1 = - t^{1+\sum_{i=1}^N3^{-i}}.
$$
Namely $x_N$ is a solution of (\ref{eq:x}) modulo $t^{1+\sum_{i=1}^N3^{-i}}$.
It is easy to see that there are no other solution of (\ref{eq:x}) modulo
$t^{1+\sum_{i=1}^N3^{-i}}$.
\par
However since
$$
1+\sum_{i=1}^{\infty}3^{-i} = \frac {3}{2} < \infty,
$$
it follows that
$$
\lim_{N\to \infty} x_N
$$
does {\it not} converge in $\Lambda_{0}^{\mathbb{F}_3}$.
Thus there is no solution of (\ref{eq:x}) over a field of characteristic $3$.
\end{proof}
\end{exm}
Lemma \ref{nosolutionF} implies that the field of fraction of the Puiseux series ring
with coefficients in an algebraically closed field of
\emph{positive} characteristic
is {\it not} algebraically closed.
It is well known that this phenomenon does not occur in the case of characteristic
zero. See, for example, Corollary 13.15 \cite{Eisen}.
Since we could not find a proof of a similar result for universal Novikov ring in the
literature, we will provide its proof in the appendix for completeness.
(We used it in the proof of Theorem \ref{lageqgeopt} in section \ref{sec:exa2}.)

\section{Calculation of potential function}
\label{sec:calcpot}

In this section, we prove Theorems \ref{potential} and
\ref{weakpotential}. We begin with a review of \cite{cho-oh}. Let
$\pi : X \to P$ be the moment map and $\partial P = \bigcup_{i=1}^m
\partial_i P$ be the decomposition of the boundary of $P$ into $n-1$
dimensional faces. Let $\beta_i \in H_2(X,L(u);\Z)$ be the elements
such that
$$
\beta_i \cap [\pi^{-1}(\partial P_j)] =
\begin{cases}
1 &\text{if $i=j$,}\\
0 &\text{if $i\ne j$.}
\end{cases}
$$
The Maslov index $\mu(\beta_i)$ is 2. (Theorem 5.1 \cite{cho-oh}.)
\par
Let $\beta \in \pi_2(X,L(u))$ and
$\mathcal M_{k+1}^{\text{\rm main}}(L(u),\beta)$ be the moduli space
of stable maps from bordered Riemann surfaces of genus zero with $k+1$
boundary marked points in homology class $\beta$.
(See \cite{fooo00} section 3 
= \cite{fooo06} subsection 2.1.2. We require the boundary marked points
to respect the cyclic order of $S^1 = \partial D^2$. (In other words we consider
the main component in the sense of \cite{fooo00} section 3.))
Let $\mathcal M^{\text{\rm main},\text{\rm reg}}_{k+1}(L(u),\beta)$ be its
subset consisting of all maps from a disc. (Namely the stable
map without disc or sphere bubble.)
The next theorem easily follows from the results of \cite{cho-oh}.
In (3) Theorem \ref{37.177} we use the spin structure
of $L(u)$ which is induced by the diffeomorphism of $L(u) \cong T^n$
by the $T^n$ action and the standard trivialization of the tangent bundle of $T^n$.
\begin{thm}\label{37.177}
\par
\begin{enumerate}
\item If $\mu(\beta) < 0$, or $\mu(\beta) = 0$, $\beta \ne 0$, then $\mathcal M^{\text{\rm main},\text{\rm
reg}}_{k+1}(L(u),\beta)$ is empty.
\par
\item If $\mu(\beta) = 2$, $\beta \ne \beta_1,\cdots,\beta_m$, then $\mathcal M^{{\text{\rm main}},\text{\rm
reg}}_{k+1}(L(u),\beta)$ is empty.
\par
\item For $i=1,\cdots,m$, we have
\begin{equation}\label{178}
\mathcal M^{{\text{\rm main}},\text{\rm reg}}_{1}(L(u),\beta_i) =
\mathcal M^{\text{\rm main}}_{1}(L(u),\beta_i).
\end{equation}
Moreover $\mathcal M_{1}^{\text{\rm main}}(L(u),\beta_i)$ is Fredholm regular.
Furthermore the evaluation map
$$
ev : \mathcal M^{\text{\rm main}}_{1}(L(u),\beta_i) \to L(u)
$$
is an orientation preserving diffeomorphism.
\par
\item For any $\beta$, the moduli space $\mathcal M^{{\text{\rm main}},\text{\rm reg}}_{1}(L(u),\beta)$ is
Fredholm regular. Moreover
$$
ev : \mathcal M^{\text{\rm main,reg}}_{1}(L(u),\beta) \to L(u)
$$
is a submersion.
\par
\item If $\mathcal M^{\text{\rm main}}_{1}(L(u),\beta)$ is not empty then
there exists $k_i \in \Z_{\ge 0}$ and $\alpha_j \in H_2(X;\Z)$ such that
$$
\beta = \sum_i k_i\beta_i + \sum_j \alpha_j
$$
and $\alpha_j$ is realized by holomorphic sphere.
There is at least one nonzero $k_i$.
\end{enumerate}
\end{thm}
\begin{proof}
For reader's convenience and completeness, we explain how to deduce
Theorem \ref{37.177} from the results in \cite{cho-oh}.
\par
By Theorems 5.5 and 6.1 \cite{cho-oh}, $\mathcal M^{{\text{\rm main}},\text{\rm reg}}_{k+1}(L(u),\beta)$
is Fredholm regular for any $\beta$. Since the complex structure is invariant under the $T^n$ action
and $L(u)$ is $T^n$ invariant, it follows that $T^n$ acts on
$\mathcal M^{{\text{\rm main}},\text{\rm reg}}_{k+1}(L(u),\beta)$ and
$$
ev : \mathcal M^{{\text{\rm main}},\text{\rm reg}}_{k+1}(L(u),\beta) \to L(u)
$$
is $T^n$ equivariant. Since the $T^n$ action on $L(u)$ is free and
transitive, it follows that $ev$ is a submersion
if $\mathcal M^{{\text{\rm main}},\text{\rm reg}}_{k+1}(L(u),\beta)$ is nonempty.
(4) follows.
\par
We assume $\mathcal M^{{\text{\rm main}},\text{\rm reg}}_{k+1}(L(u),\beta)$ is nonempty.
Since $ev$ is a submersion it follows that
$$
n = \dim L(u) \le \dim \mathcal M^{{\text{\rm main}},\text{\rm reg}}_{k+1}(L(u),\beta) =
n + \mu(\beta) -2
$$
if $\beta \ne 0$. Therefore $\mu(\beta) \ge 2$. (1) follows.
\par
We next assume $\mu(\beta) = 2$, and $\mathcal M^{{\text{\rm main}},\text{\rm reg}}_{k+1}(L(u),\beta)$ is nonempty.
Then by Theorem 5.3 \cite{cho-oh}, we find $\beta = \beta_i$ for some $i$.
(2) follows.
\par
We next prove (5). It suffices to consider
$$[f] \in \mathcal M^{\text{\rm main}}_{1}(L(u),\beta) \setminus \mathcal M^{{\text{\rm main}},\text{\rm reg}}_{1}(L(u),\beta).$$
We decompose the domain of $u$ into irreducible components and
restrict $f$ there. Let $f_j : D^2 \to M$ and $g_k : S^2 \to M$ be
the restriction of $f$ to disc or sphere components respectively.
We have
$$
\beta = \sum [f_j] + \sum [g_k].
$$
Theorem 5.3 \cite{cho-oh} implies that each of $f_j$ is homologous to the sum of the element of
$\beta_i$.
It implies (5).
\par
We finally prove (3). The fact that $ev$ is a
diffeomorphism for $\beta = \beta_i$, follows dirctly from Theorem 5.3 \cite{cho-oh}.
We next prove that $ev$ is orientation preserving.
Since $L(u), u \in \text{\rm Int}\, P,$ is a principal homogeneous space of
$T^n$, the tangent bundle $TL(u)$ is trivialized once we fix
an isomorphism $T^n \equiv S^1 \times \dots \times S^1$.
Using the orientation and
the spin structure on $L(u)$ induced by such a trivialization,
we orient the moduli space $\mathcal M_1(\beta)$ of holomorprhic discs.
If we change the identification $T^n \equiv S^1 \times \dots
\times S^1$ by an orientation preserving, (resp. reversing), isomorphism,
the corresponding orientations on $L(u)$ and $\mathcal M_1(\beta)$
is preserved, (resp. reversed). Therefore whether
$ev :\mathcal M(\beta_i) \to L(u)$ is orientation preserving or not
does not depend on the choice of the identification $T^n$ and $S^1 \times \dots
\times S^1$.
\par
For each $i=1, \dots, m$, we can find an automorphism $\phi$ of
$(\C^*)^n$ and a biholomorphic map
$f:X \setminus \cup_{j \neq i} \pi^{-1}(\partial_jP) \to
\C \times (\C^*)^{n-1}$ such that
\begin{enumerate}
\item $f$ is $\phi$-equivariant,
\item $f(L(u))=L_{\text{\rm std}}$, where $L_{\text{\rm std}}=\{(w_1, \dots, w_n) \in (\C)^n
\vert \vert w_1 \vert = \dots = \vert w_n \vert = 1 \}.$
\end{enumerate}
Under this identification, $\mathcal M_1(\beta_i)$ is identified with
the space of holomorphic discs
$$
z \in D^2 \mapsto (\zeta \cdot z, w_2, \dots, w_n) \in \C \times
(\C^*)^{n-1}, \quad \zeta \in S^1 \subset \C^*,
$$
where $w_k \in \C, k=2, \dots, n$ with $\vert w_k \vert =1$.
Therefore it is enough to check the statement that $ev$ is orientation
preserving in a single example.
Cho [Cho] proved it in the case of the Clifford torus in $\C P^n$,
hence the proof.

\par
To prove (\ref{178}) and complete the proof of Theorem \ref{37.177}, it remains to prove
$\mathcal M^{{\text{\rm main}},\text{\rm reg}}_{1}(L(u),\beta_{i_0}) =
\mathcal M^{\text{\rm main}}_{1}(L(u),\beta_{i_0})$.
(Here $i_0 \in \{1,\cdots,m\}$.) Let
$[f] \in
\mathcal M^{\text{\rm main}}_{1}(L(u),\beta_{i_0})$.
We take $k_i$ and $\alpha_j$ as in (5). (Here $\beta = \beta_{i_0}$).
We have
$$
\partial \beta_{i_0} = \sum_{i} k_i \partial \beta_i.
\label{37.179}$$
Using the convexity of $P$, (5) and $k_i\ge 0$, we show the inequality
\begin{equation}
\beta_{i_0}\cap \omega \le \sum_i k_i \beta_i\cap \omega
\label{37.180}\end{equation} holds and that the equality holds only
if $k_i = 0$ ($i\ne i_0$), $k_{i_0} = 1$, as follows : By (5) we
have
\begin{equation}
\ell_{i_0} = \sum_{i=1}^m k_i \ell_i + c
\nonumber\end{equation}
where $c$ is a constant. Since $k_i \ge 0$ and $\ell_{i_0}(u') = 0$
for $u' \in \partial_{i_0}P$, it follows that $c\le 0$. (Note
$\ell_i \ge 0$ on $P$.) Since $\beta_i\cap \omega = \ell_i(u)$, we
have the inequality (\ref{37.180}). Let us assume that the equality holds. If there
exists $i\ne j$ with $k_i,k_j > 0$ then
$$
\partial_{i_0}\,P = \{u' \in P \mid \ell_{i_0}(u') = 0\} \subseteq \{u' \in P \mid \ell_i(u')
= \ell_j(u') =
0\}
\subseteq \partial_iP \cap \partial_jP.
$$
This is a contradiction since $\partial_{i_0}P$ is codimension 1. Therefore
there is only one nonzero $k_i$. It is easy to see that $i=i_0$ and
$k_{i_0} = 1$.
\par
On the other hand since $\alpha_j \cap \omega > 0$ it follows that
\begin{equation}
\beta_{i_0}\cap \omega \ge \sum_i k_i \beta_i\cap \omega.
\nonumber
\end{equation}
Therefore there is no sphere bubble (that is $\alpha_j$).
Moreover the equality holds in (\ref{37.180}). Hence the domain of our
stable map is
irreducible. Namely
$$
\mathcal M^{\text{\rm main},\text{\rm reg}}_{1}(L(u),\beta_{i_0}) =
\mathcal M^{\text{\rm main}}_{1}(L(u),\beta_{i_0}).
$$
The proof of Theorem \ref{37.177} is now complete.
\end{proof}
Next we discuss one delicate point to apply Theorem \ref{37.177} to the proofs of
Theorems \ref{potential} and \ref{weakpotential}. (This point was already mentioned in section 16
\cite{cho-oh}.)
Let us consider the case where there exists a holomorphic sphere
$
g : S^2 \to X
$
with
$$
c_1(X) \cap g_*[S^2] < 0.
$$
We assume moreover that there exists a holomorphic disc
$f : (D^2,\partial D^2) \to (X,L(u))$ such that
$$
f(0) = g(1).
$$
We glue $D^2$ and $S^2$ at $0\in D^2$ and $1\in S^2$ to obtain $\Sigma$.
$f$ and $g$ induce a stable map $h : (\Sigma,\partial \Sigma) \to (X,L(u))$.
\par
In general $h$ will {\it not} be Fredholm regular since $g$ may not be Fredholm regular
or the evaluation is not transversal at the interior nodes.
In other words, elements of $\mathcal M^{\text{\rm main}}_1(L(u),\beta) \setminus \mathcal M^{{\text{\rm main}},\text{\rm reg}}_{1}(L(u),\beta)$
may not be Fredholm regular in general.
Moreover replacing $g$ by its multiple cover, we obtain an element
of $\mathcal M^{\text{\rm main}}_1(L(u),\beta) \setminus
\mathcal M^{{\text{\rm main}},\text{\rm reg}}_{1}(L(u),\beta)$ such that $\mu(\beta)$ is negative.
Theorem \ref{37.177} says that all the holomorphic disc without any bubble are
Fredholm regular. However we can not expect that all stable maps in
$\mathcal M^{{\text{\rm main}}}_1(L(u),\beta)$ are Fredholm regular.
\par
In order to prove Theorem \ref{weakpotential}, we need to find appropriate
perturbations of those stable maps.
For this purpose we use the $T^n$ action and proceed as follows.
(We remark that many of the arguments below are much simplified in the Fano case, where
there exists no holomorphic sphere $g$ with
$
c_1(M) \cap g_*[S^2] \le 0.
$
)
\par
We equip each of $\mathcal M_1(L(u),\beta)$ with Kuranishi structure.
(See \cite{FO} for the general theory of Kuranishi structure and
section 17-18 \cite{fooo00} (or 
section 7.1 \cite{fooo06} = section 29 \cite{fooo06pre}) 
for its construction in the context we
currently deal with.)
We may construct Kuranishi neighborhoods and obstruction bundles that
carry $T^n$ actions induced by the $T^n$ action on $X$,
and choose $T^n$-equivariant Kuranishi maps.
(See Definition \ref{eqkurastr}.)
We note that the evaluation map
$$
ev : \mathcal M_1(L(u),\beta) \to L(u)
$$
is $T^n$-equivariant. We use the fact that the complex structure of $X$
is $T^n$-invariant and $L(u)$ is a free $T^n$-orbit to
find such a Kuranishi structure.
(See Proposition \ref{equikurast} for the detail.)
\par
We remark that the $T^n$ action on the Kuranishi neighborhood is
free since the $T^n$ action on $L(u)$ is free and $ev$ is $T^n$ equivariant.
We take a perturbation (that is, a multisection) of the Kuranishi map
that is $T^n$ equivariant.
We can find such a multisection which is also transversal to $0$ as follows :
Since the $T^n$ action is free, we can take the quotient of Kuranishi neighborhood, obstruction bundle
etc. to obtain a space with Kuranishi structure.
Then we take a transversal multisection of the quotient
Kuranishi structure and lift it to a multisection of the Kuranishi
neighborhood of $\mathcal M_1(L(u),\beta)$.
(See Corollary \ref{equiperturb} for detail.)
Let $\mathfrak s_{\beta}$ be such a multisection and let
$\mathcal M_1(L(u),\beta)^{\mathfrak s_{\beta}}$ be its zero set.
We remark that the evaluation map
\begin{equation}
ev : \mathcal M_1(L(u),\beta)^{\mathfrak s_{\beta}} \to L(u)
\label{37.182}\end{equation}
is a submersion. This follows from the $T^n$ equivariance.
This makes our construction of system of multisections much simpler
than the general one in 
section 7.2 \cite{fooo06} (= section 30 \cite{fooo06pre}) 
since the fiber product appearing in the
inductive construction is automatically transversal.
(See subsection 7.2.2 \cite{fooo06} = section 30.2 \cite{fooo06pre}) for the reason why this is crucial.)
More precisely we prove the following Lemma \ref{37.184}.
Let
\begin{equation}
\mathfrak{forget}_0 : \mathcal M^{\text{\rm
main}}_{k+1}(L(u),\beta) \to \mathcal M^{\text{\rm
main}}_1(L(u),\beta) \label{37.183}\end{equation} be the forgetful
map which forgets the first, \dots, $k$-th marked points. (In other
words, only the $0$-th marked point remains.) We can construct our
Kuranishi structure so that it is compatible with
$\mathfrak{forget}_0$ in the same sense as 
Lemma 7.3.8 \cite{fooo06} (= Lemma 31.8 \cite{fooo06pre}).
\par
\begin{lem}\label{37.184}
For each given $E > 0$, there exists a system of multisections
$\mathfrak s_{\beta,k+1}$ on $\mathcal M^{\text{\rm main}}_{k+1}(L(u),\beta)$
for $\beta \cap \omega < E$ with the following
properties :
\begin{enumerate}
\item They are transversal to $0$.
\par
\item They are invariant under the $T^n$ action.
\par
\item The multisection $\mathfrak s_{\beta,k+1}$
is the pull-back of the multisection $\mathfrak s_{\beta,1}$
by the forgetful map $(\ref{37.183})$.
\par
\item The restriction of $\mathfrak s_{\beta,1}$
to the boundary of $\mathcal M^{\text{\rm main}}_1(L(u),\beta)$ is the fiber product
of the multisections $\mathfrak s_{\beta',k'}$ with respect to the
identification of the boundary
$$
\partial{\mathcal M}_1^{\text{\rm main}}(L(u),\beta)
= \bigcup_{\beta_1+\beta_2=\beta}
{\mathcal M}_1^{\text{\rm main}}(L(u),\beta_1) {}_{ev_0}\times_{ev_1}
{\mathcal M}_2^{\text{\rm main}}(L(u),\beta_2).
$$
\item We do not perturb $\mathcal M^{\text{\rm main}}_1(L(u),\beta_i)$
for $i=1,\cdots,m$.
\end{enumerate}
\end{lem}
\begin{proof}
We construct multisections inductively over $\omega \cap \beta$.
Since (2) implies that fiber products of the perturbed moduli spaces
which we have already constructed
in the earlier stage of induction are automatically transversal, we can extend
them so that (1), (2), (3), (4) are satisfied
by the method we already explained above.
We recall from Theorem \ref{37.177} (3) that
$$
\mathcal M^{\text{\rm main}}_1(L(u),\beta_i)
= \mathcal M_1^{{\text{\rm main}},\text{\rm reg}}(L(u),\beta_i)
$$
and it is Fredholm regular and its evaluation map is surjective
to $L(u)$. Therefore when we
perturb the multisection we do not need to worry about compatibility of it
with other multisections we have already constructed in the earlier stage
of induction. This enable us to leave
the moduli space $\mathcal M^{\text{\rm main}}_1(L(u),\beta_i)$ unperturbed
for all $\beta_i$. The proof of Lemma \ref{37.184} is complete.
\end{proof}
\begin{rem}\label{37.186}
We need to fix $E$ and stop the inductive construction of multisections at some finite stage.
Namely we define $s_{\beta,k+1}$ only for $\beta$ with $\beta \cap \omega < E$.
The reason is explained in subsection 7.2.3 \cite{fooo06} (= section 30.3 \cite{fooo06pre}).
We can go around this trouble in the same
way as explained in section 7.2 \cite{fooo06} (= section 30 \cite{fooo06pre}).
See Remark \ref{runningoutresolve}.
\end{rem}
\begin{rem}\label{sphcompatible}
We explain one delicate point of the proof of Lemma \ref{37.184}.
Let $\alpha \in \pi_2(X)$ be represented by
a holomorphic sphere with $c_1(X) \cap \alpha < 0$. We consider the
moduli space $\mathcal M_1(\alpha)$ of holomorphic sphere with one
marked point and in homology class $\alpha$.
Let us consider $\beta \in \pi_2(X;L(u))$ and
the moduli space $\mathcal M_{k+1,1}^{\text{\rm main}}(\beta)$ of
holomorphic discs with one interior and $k+1$ boundary marked
points and of homotopy class $\beta$. The fiber product
$$
\mathcal M_1(\alpha) \times_X \mathcal M_{k+1,1}^{\text{\rm main}}(\beta)
$$
taken by the evaluation maps at interior marked points are
contained in $\mathcal M_{k+1,1}^{\text{\rm main}}(\beta+\alpha)$.
If we want to define a multisection compatible with the
embedding
\begin{equation}\label{eq:interiorbubble}
\mathcal M_1(\alpha) \times_X \mathcal M_{k+1,1}^{\text{\rm main}}(\beta) \subset \mathcal M_{k+1}^{\text{\rm main}}(\beta+\alpha)
\end{equation}
then it is impossible to make it both transversal and $T^n$ equivariant
in general : This is because the
nodal point of such a singular curve could be
contained in the part of $X$ with non-trivial isotropy group.
\par
Our perturbation constructed above satisfies (1) and (2) of Lemma \ref{37.184} and so
may {\it not} be compatible with the embedding (\ref{eq:interiorbubble}).
Our construction of the perturbation given in Lemma \ref{37.184} exploits
the fact that the $T^n$ action acts freely on the Lagrangian fiber $L(u)$
and carried out by induction on the number of {\it disc} components (and of energy)
only, regardless of the number of sphere components.
\end{rem}
The following corollary is an immediate consequence of Lemma \ref{37.184}.
\begin{cor}\label{37.187}
If $\mu(\beta) < 0$ or $\mu(\beta) = 0$, $\beta\ne 0$, then
$\mathcal M^{\text{\rm main}}_1(L(u),\beta)^{s_{\beta}}$ is empty.
\end{cor}

Now we consider $\beta \in \pi_2(X;L)$ with $\mu(\beta) = 2$ and
$\beta \cap \omega < E$, where $E$ is as in Lemma \ref{37.184}. One
immediate consequence of Corollary \ref{37.187} is that the virtual
fundamental chain of $\mathcal M^{\text{\rm main}}_1(L(u),\beta)$
becomes a \emph{cycle}. More precisely, we introduce

\begin{defn}\label{cbeta} Let $\beta \in \pi_2(X;L)$ with $\mu(\beta) = 2$ and
$\beta \cap \omega < E$, where $E$ is as in Lemma \ref{37.184}. We
define a homology class $c_\beta \in H_n(L(u);\Q) \cong \Q$ by the
pushforward
$$
c_{\beta} = ev_*([\mathcal M^{\text{\rm
main}}_1(L(u),\beta)^{s_{\beta}}]).
$$
\end{defn}

\begin{lem}\label{cbetawell}
The number $c_{\beta}$ is independent of the choice of the system of
multisections $\mathfrak s_{\beta,k+1}$ satisfying $(1)$ - $(5)$ of
Proposition $\ref{37.184}$.
\end{lem}
\begin{proof}
If there are two such systems, we can find a $T^n$ invariant
homotopy between them which is also transversal to $0$. By a
dimension counting argument applied to the parameterized version of
$\mathcal M_1^{\text{\rm main}}(L(u),\beta)$ and its perturbation, we
will have the parameterized version of Corollary \ref{37.187}. This
in turn implies that the perturbed (parameterized) moduli space
defines a compact cobordism between the perturbed moduli spaces of
$\mathcal M_1^{\text{\rm main}}(\beta)$ associated to the two such
systems. The lemma follows.
\end{proof}

We remark $c_{\beta_i} = 1$ where $\beta_i$ ($i=1, \cdots, m$) are the classes corresponding
to each of the irreducible components of the divisor $\pi^{-1}(\partial P)$.
If $X$ is Fano, then $c_{\beta} = 0$, for $\beta \ne \beta_i$. But this may not
be the case if $X$ is not Fano.

We now will use our perturbed moduli space to define
a structure of filtered $A_{\infty}$ algebra on
the de Rham cohomology $
H(L(u);\Lambda_0^{\R}) \cong (H(L(u),\R)) \otimes \Lambda_0$.
We will write it $\frak m^{can}$.
\par
We take a $T^n$ equivariant Riemannian metric on $L(u)$.
We observe that a differential form $\rho$ on $L(u)$
is harmonic if and only if $\rho$ is $T^n$ equivariant.
So we identify $H(L(u),\R)$ with the set of $T^n$ equivariant
forms from now on.
\par
We consider the evaluation map
$$
ev = (ev_1,\cdots,ev_k,ev_0) : \mathcal M^{\text{\rm main}}_{k+1}(L(u),\beta)^{\mathfrak s_{\beta}} \to L(u)^{k+1}.
$$
Let $\rho_1, \cdots, \rho_k$ be
$T^n$ equivariant differential forms on $L(u)$.
We define
\begin{equation}
\mathfrak m^{can}_{k,\beta}(\rho_1,\cdots,\rho_k)
= (ev_0)_! (ev_1,\cdots,ev_k)^*(\rho_1 \wedge \cdots \wedge \rho_k).
\label{37.188}\end{equation}
We remark that integration along fiber $(ev_0)_!$ is well defined and
gives a smooth form, since $ev_0$ is a submersion. (It is a consequence of $T^n$
equivariance.)
More precisely, we apply
Definition \ref{cordefinitionss} in section \ref{sec:integration} as follows.
We put $\mathcal M = \mathcal M^{\text{\rm main}}_{k+1}(L(u),\beta)$,
$L_s = L^k$, $L_t = L$.
$ev_s = (ev_1,\cdots,ev_k) : \mathcal M \to L_s$,
$ev_t = ev_0 : \mathcal M \to L_t$.
Thus we are in the situation we formulate at the beginning of
section \ref{sec:integration}.
Then using Lemma \ref{correspondfinalformula}
and Remark \ref{detasdenote} (1),
we put
\begin{equation}\label{mcandef}
\mathfrak m^{can}_{k,\beta}(\rho_1,\cdots,\rho_k)
=(\mathcal M;ev_s,ev_t,\frak s_{\beta})_* (\rho_1\times \cdots \times \rho_k).
\end{equation}
We remark that the right hand side of (\ref{mcandef})
is again $T^n$ equivariant since $\frak s_{\beta}$ etc.
are $T^n$ equivariant.
\par
Lemma \ref{37.184} (4)
implies
$$\aligned
&\partial  \mathcal M^{\text{\rm main}}_{k+1}(L(u),\beta)
\\
&= \bigcup_{k_1+k_2=k+1}\bigcup_{\beta_1+\beta_2 = \beta}\bigcup_{l=1}^{k_2}
\mathcal M^{\text{\rm main}}_{k_1+1}(L(u),\beta_1)
{}_{ev_0} \times_{ev_l}\mathcal M^{\text{\rm main}}_{k_2+1}(L(u),\beta_2).
\endaligned$$
Therefore
using Lemmata \ref{stokes}, \ref{compformula},
we have the following formula.
\begin{equation}\label{Ainftyformula}
\sum_{\beta_1+\beta_2=\beta}\sum_{k_1+k_2=k+1}\sum_{l=1}^{k_1}
(-1)^*\mathfrak m^{can}_{k_1,\beta_1}(\rho_1,\cdots,\mathfrak m^{can}_{k_2,\beta_2}(\rho_l,\cdots),\cdots,\rho_k).
\end{equation}
Here $* = \sum_{i=1}^l (\deg\rho_i + 1)$.
(See the end of section
\ref{sec:integration} for sign.) We now put
\begin{equation}\label{Ainftystru}
\frak m_k^{can}(\rho_1, \cdots, \rho_k) =
\sum_{\beta} T^{\beta\cap\omega/(2\pi)}\frak m_{k,\beta}^{can}(\rho_1, \cdots, \rho_k).
\end{equation}
We extend (\ref{Ainftystru})  to $\rho \in H(L(u);\Lambda^{\R}_0)$
such that it is $\Lambda_0$ multilinear. Then
(\ref{Ainftyformula}) implies that it
defines a structure of a filtered $A_{\infty}$ structure
on $H(L(u);\Lambda_0^{\R})$ in the sense of section \ref{sec:floerreview}.
\par
We also remark that our filtered $A_{\infty}$ algebra is
unital and the constant $0$ form $1 \in H^0(L;\R)$ is a
unit. This is a consequence of Lemma \ref{37.184} (3).
\par
We next calculate our filtered $A_{\infty}$
structure in the case when $\rho_i$ are degree $1$ forms.
\begin{lem}\label{37.189}
For $\mathfrak x \in H^1(L(u),\Lambda_{0})$ and $\beta \in
\pi_2(X,L)$ with $\mu(\beta) = 2$, we have
$$
\mathfrak m_{k,\beta}^{can}(\mathfrak x,\cdots,\mathfrak x) = \frac{c_{\beta}}{k!}
(\partial \beta\cap \mathfrak x)^k \cdot PD([L(u)]).
$$
Here $PD([L(u)])$ is the Poincar\'e dual to the fundamental class.
In other words it is the $n$ form with $\int_{L(u)}PD([L(u)]) = 1$.
\end{lem}
\begin{proof}
It suffices to consider the case $\mathfrak x = \rho \in H^1(L(u);\R)$,
and show
\begin{equation}\label{eq:mkonharmonic}
\int_{L(u)} \mathfrak m^{can}_{k,\beta}(\rho,\cdots,\rho)
=
\frac{c_{\beta}}{k!}(\partial \beta\cap \mathfrak x)^k.
\end{equation}
Let
\begin{equation}\label{eq:configS1}
C_k = \{ (t_1,\cdots,t_k) \mid 0 \le t_1 \le \cdots \le t_k \le 1\}.
\end{equation}
We define an iterated blow-up, denoted by $\widehat C_k$, of $C_k$
in the following way. Let  $S = \del D$ be the boundary of the
unit disc $D=D^2 \subset \C$ and $\beta_D \in H_2(\C,S)$ be the
homology class of the unit disc. We consider the moduli
space $\mathcal M_{k+1}(\C,S;\beta_D)$ and the evaluation
map
$
\vec{ev} = (ev_0,\cdots,ev_k) : \mathcal M_{k+1}(\C,S;\beta_D)
\to (S^1)^{k+1}.
$
We fix a point $p_0 \in S \subset \C$ and put
$$
\widehat C_k := ev_0^{-1}(p_0) \subset \mathcal
M_{k+1}(\C,S;\beta_D).
$$
We make the identification $S^1 \setminus \{p_0\}\cong (0,1)$. Then $\vec{ev}$
induces a diffeomorphism
$$
\widehat C_k \cap \mathcal M_{k+1}^{\rm reg}(\C,S;\beta_D) \to \text{\rm Int }C_k
$$
given by
$$
[w,z_0, \cdots, z_k] \mapsto (w(z_1) - w(z_0), \cdots, w(z_k) - w(z_0))
$$
where
$$
\text{\rm Int }C_k= \{ (t_1,\cdots,t_k) \mid 0 < t_1 < \cdots < t_k < 1\} \subset C_k.
$$
In this sense $\widehat C_k$ is regarded as an iterated blow up of $C_k$
along the diagonal (that is the set of points where $t_i = t_{i+1}$
for some $i$). We identify $\del D= S \cong \R/\Z \cong S^1$.
We have
\begin{equation}\label{eq:Mk+1}
\mathcal M^{\text{\rm main}}_{k+1}(L(u),\beta)^{\frak s} \cong \mathcal
M^{\text{\rm main}}_{1}(L(u),\beta)^{\frak s} \times \widehat C_k.
\end{equation}
In fact,
Corollary \ref{37.187} implies
$\mathcal M^{\text{\rm main}}_{1}(L(u),\beta)^{\frak s}$ consists of finitely many
free $T^n$ orbits
(with multiplicity $\in \Q$) and $\mathcal M^{\text{\rm main}}_{1}(L(u),\beta)^{\frak s}
= \mathcal M^{\text{\rm main,reg}}_{1}(L(u),\beta)^{\frak s}$.
By Lemma \ref{37.184} (3) we have a map
$\mathcal M^{\text{\rm main,reg}}_{k + 1}(L(u),\beta)^{\frak s}
\to \mathcal M^{\text{\rm main,reg}}_{1}(L(u),\beta)^{\frak s}$.
It is easy to see that the fiber can be identified with $\widehat C_k$.
\par
Under this identification, the evaluation map $\vec{ev}$ is induced by
\begin{equation}\label{eq:evi}
ev_i(\mathfrak u;t_1,\cdots,t_k) = [t_i\partial \beta] \cdot
ev(\mathfrak u)
\end{equation}
for $(\frak u;t_1, \cdots, t_k) \in \mathcal
M^{\text{\rm main}}_{1}(L(u),\beta) \times \text{\rm Int }C_k \subset \mathcal
M^{\text{\rm main}}_{1}(L(u),\beta) \times \widehat C_k$.

Here $\partial\beta \in H_1(L(u);\Z)$ is identified to an element of
the universal cover $\widetilde L(u) \cong \R^n$ of $L(u)$ and
$[t_i\partial \beta] \in L(u)$ acts as a multiplication on the
torus. $ev(\mathfrak u)$ is defined by the evaluation map $ev :
\mathcal M^{\text{\rm main}}_{1}(L(u),\beta) \to L(u)$. We also have
:
\begin{equation}\label{eq:evi0}
ev_0(\mathfrak u;t_1,\cdots,t_k) = ev(\mathfrak u).
\end{equation}
We remark that $ev : \mathcal M^{\text{\rm main}}_{1}(L(u),\beta_i)
\to L(u)$ is a diffeomorphism (See Theorem \ref{37.177} (3)). Now we
have
$$
\int_{L(u)} \mathfrak m^{dR}_{k,\beta}(\rho,\cdots,\rho) =
c_{\beta}\text{\rm Vol}(C_k) \left(\int_{\partial \beta}
\rho\right)^k = \frac{c_{\beta}}{k!}(\partial \beta\cap \mathfrak
x)^k.
$$
The proof of Lemma \ref{37.189} is now complete.
\end{proof}
\begin{rem}\label{37.190}
We can prove that our filtered $A_{\infty}$ algebra
$(H(L(u);\Lambda_0^{\R}),\mathfrak m_{*}^{can})$
is homotopy equivalent to the one in
\cite{fooo06} Theorem {\rm A} (= \cite{fooo06pre} Theorem A).
The proof is a straight forward generalization of 
section 7.5 \cite{fooo06} (= section 33 of \cite{fooo06pre}) and is omitted here.
In fact we do not need to use this fact to prove Theorem \ref{toric-intersect} if
we use de Rham version in all the steps of the proof of Theorem \ref{toric-intersect}
without involving the singular homology version.
\end{rem}
\begin{rem}\label{deramversion}
We constructed our filtered $A_{\infty}$ structure directly
on de Rham cohomology group $H(L(u);\Lambda_0^{\R})$.
The above construction uses the fact that wedge product of
harmonic forms are again harmonic. This is a special feature of
our situation where our Lagrangian submanifold $L$ is a torus.
(In other words we use the fact that rational homotopy type
of $L$ is formal.)
\par
Alternatively we can construct filtered $A_{\infty}$ structure
on the de Rham complex $\Omega(L(u)) \widehat{\otimes}_{\R}
\Lambda_0^{\R}$ and reduce it to the de Rham cohomology
by homological algebra.
Namely we consider smooth forms $\rho_i$ which are not necessarily
harmonic and use (\ref{37.188}) and (\ref{Ainftystru})
to define $\frak m_k^{dR}(\rho_1,\cdots,\rho_k)$.
(The proof of $A_{\infty}$ formula is the same.)
Using the formality of $T^n$, we can show that
the canonical model of $(\Omega(L(u)) \widehat{\otimes}_{\R} \Lambda_0^{\R},\frak m^{dR}_*)$ is the same as $(H(L(u);\Lambda_0^{\R}),\frak m^{can})$.
We omit its proof since we do not use it.
\par
Using the continuous family of perturbations,
this construction can be generalized to the case of arbitrary
relatively spin Lagrangian submanifold in a symplectic manifold.
See \cite{fukaya;operad}.
\end{rem}
\par
\begin{proof}[Proof of Proposition \ref{unobstruct}]
Proposition \ref{unobstruct} immediately follows from Corollary \ref{37.187},
Lemma \ref{37.189} and Lemma \ref{37.190} : We just take the sum
\bea\label{eq:summcan}
\sum_{k=0}^\infty \mathfrak m^{can}_k(b,\cdots,b) & = &
\sum_{k=0}^\infty \sum_{\beta \in \pi_2(X,L(u))} T^{\omega\cap \beta/2\pi}\mathfrak m^{can}_{k,\beta}(b,\cdots,b)
\nonumber\\
& = & \sum_{k=0}^\infty \sum_{\beta} T^{\omega\cap \beta/2\pi}
\mathfrak m^{can}_{k,\beta}(b,\cdots,b)\nonumber\\
& = & \sum_{\beta} \sum_{k=0}^\infty \frac{c_{\beta}}{k!}(\partial \beta
\cap b)^k T^{\beta \cap \omega/2\pi} \cdot PD([L(u)]). \eea
Note, by the degree reason, we need to take sum over $\beta$ with $\mu(\beta) =2$.
\par
Since $b$ is assumed
to lie in $H^1(L(u),\Lambda_+)$ not just in $H^1(L(u),\Lambda_0)$,
the series appearing as the scalar factor in (\ref{eq:summcan})
converges in non-Archimedean topology of $\Lambda_0$ and so the sum
$\sum_{k=0}^\infty \mathfrak m^{can}_k(b,\cdots,b)$ is a multiple of
$PD([L(u)])$. Hence $b \in \widehat \CM_{\text{\rm weak}}(L(u))$ by
definition (\ref{eq:MC}). We remark that the gauge equivalence
relation in Chapter 4 \cite{fooo06} is trivial on
$H^1(L(u);\Lambda_0)$ and so $H^1(L(u);\Lambda_+) \hookrightarrow
\CM_{\text{\rm weak}}(L(u))$. We omit the proof of this fact since we do not
use it in this paper.
\end{proof}
\begin{proof}[Proof of Theorem \ref{potential}]
Suppose that there is no nontrivial holomorphic sphere whose Maslov index is nonpositive.
Then Theorem \ref{37.177} (5) implies that
if $\mu(\beta) \le 2$, $\beta \ne \beta_i$, $\beta \ne 0$ then $\mathcal M^{\text{\rm main}}_{1}(L(u),\beta)$ is empty.
Therefore again by dimension counting as in Corollary \ref{37.187}, we obtain
$$
\sum_{k=0}^\infty\mathfrak m_{k}^{dR,can}(x,\cdots,x) = \sum_{i=1}^m \sum_{k=0}^{\infty}T^{\omega\cap \beta_i/2\pi}\mathfrak m_{k,\beta_i}^{dR,can}(x,\cdots,x)
$$
for $x \in H^1(L(u),\Lambda_+)$.
On the other hand, we obtain
\beastar
\mathfrak{PO}(x;u) & = &
\sum_{i=1}^m \sum_{k=0}^\infty \frac{1}{k!}(\partial \beta_i \cap x)^k
T^{\ell_i(u)} \\
& = & \sum_{i=1}^m \sum_{k=0}^\infty \frac{1}{k!}\langle v_i, x\rangle^k
T^{\ell_i(u)}
= \sum_{i=1}^m e^{\langle v_i, x \rangle} T^{\ell_i(u)}
\eeastar
from (\ref{form:connecting}), (\ref{eq:summcan}) and the definition of $\mathfrak{PO}$.
Writing $x = \sum_{i=1}^n x_i \text{\bf e}_i$ and
recalling $y_i = e^{x_i}$, we obtain $e^{\langle v_i, x \rangle} = y_1^{v_{i,1}}\cdots y_n^{v_{i,n}}$
and hence the proof of Theorem \ref{potential}.
\end{proof}
\begin{proof}[Proof of Theorem \ref{weakpotential}]
Let $\beta \in \pi_2(X)$, $\mu(\beta) =2$ and $\mathcal M_{\text{\rm weak}}(\beta) \ne
\emptyset$.
Theorem \ref{37.177} (5) implies that
\begin{equation}
\partial \beta = \sum k_i \partial \beta_i,
\quad \beta = \sum_i k_i \beta_i + \sum_j \alpha_j.
\nonumber\end{equation}
Hence
$$
\sum_k T^{\beta\cap \omega/2\pi} \mathfrak m_{k,\beta}^{can}(b,\cdots,b)
$$
becomes one of the terms of the right hand side of (\ref{eq:weakPO}).
We remark that class $\beta$ with $\mu(\beta) \ge 4$ does not
contribute to $\mathfrak m_{k}^{can}(b,\cdots,b)$ by the degree reason.
\par
When all the vertices of $P$ lie in $\Q^n$, then
the symplectic volume of all $\alpha_j$ are in $2\pi\Q$.
Moreover $\omega \cap \beta_i \in 2\pi\Q$.
Therefore the exponents $\beta\cap \omega/2\pi$ are rational.
\par
The proof of Theorem \ref{weakpotential} is complete.
\end{proof}
\begin{rem}\label{runningoutresolve}
We remark that in Lemma \ref{37.184} we constructed a system of
multisections only for $\mathcal M_{k+1}^{\text{main}}(L(u),\beta)$
with $\beta \cap \omega < E$. So we obtain only an $A_{n,K}$
structure instead of a filtered $A_{\infty}$ structure.
Here $(n,K) = (n(E),K(E))$ depends on $E$ and
$\lim_{E \to \infty} (n(E),K(E)) = (\infty,\infty)$.
It induces an $A_{n,K}$
structure $\frak m^{(E)}$ on
$H(L;\Lambda_0)$. (See subsection 7.2.7 \cite{fooo06} =
subsection 30.7 \cite{fooo06pre}.)
In the same way as section 7.2 \cite{fooo06} (= section 30 \cite{fooo06pre}), we can find
$(n'(E),K'(E))$ such that $(n'(E),K'(n)) \to (\infty,\infty)$ as
$E \to \infty$ and the following holds.

If $E_1 < E_2$ then $A_{n(E_1),K(E_1)}$ structure
$\frak m^{(E_1)}$ is $(n'(E),K'(E))$-homotopy equivalent
to $\frak m^{(E_2)}$.

This implies that we can extend $\frak m^{(E_1)}$
(regarded as $A_{n'(E_1),K'(E_1)}$ structure) to a
filtered $A_{\infty}$ structure by 
Theorem 7.2.72 \cite{fooo06} (= Theorem 30.72 \cite{fooo06pre}).
(We also remark that for all the applications in this paper
we can use filtered $A_{n,K}$ structure for sufficiently large $n,K$,
in place of filtered $A_{\infty}$ structure.)

Moreover we
can use Lemma \ref{cbetawell} to show the following.
If $\frak x_i \in H^1(L(u);\Q)$ then
$
\frak m_{k,\beta}^{(E)}(\frak x_1,\cdots,\frak x_k)
$
is independent of $E$. So in particular it coincides to one
of filtered $A_{\infty}$ structure we define as above.
\par
In other words,
since the number $c_{\beta}$ is independent of
the choice of the system of $T^n$ invariant multisections it follows
that the potential function in Theorem \ref{weakpotential}
is independent of it.
However we do not know how to calculate it.
\end{rem}

\begin{rem}\label{37.191}
We used de Rham cohomology to go around the problem of
transversality among chains in the classical cup product. One
drawback of this approach is that we lose control of the rational
homotopy type. Namely we do {\it not} prove here that the filtered
$A_{\infty}$ algebra (partially) calculated above is homotopy
equivalent to the one in Theorem A \cite{fooo06} 
(= Theorem A \cite{fooo06pre}) over $\Q$. (Note
all the operations we obtain is defined over $\Q$, however.) We
however believe that they are indeed homotopy equivalent over $\Q$.
There may be several possible ways to prove this statement, one of
which is to use the rational de Rham forms used by Sullivan.
\par
Moreover since the number $c_{\beta}$ is independent of the choices
we made, the structure of filtered $A_{\infty}$ algebra on
$H(L(u),\Lambda^{\Q})$ is well defined, that is independent of the
choices involved. The $\Q$-structure is actually interesting in our
situation. See for example Proposition \ref{prop:galois}. However
homotopy equivalence of the $\Q$-version of Lemma \ref{37.190} is
not used in the statement of Proposition \ref{prop:galois} or in its
proof.
\end{rem}

\section{Non-unitary flat connection on $L(u)$}
\label{sec:flat}

In this section we explain how we can include (not necessarily unitary)
flat bundles on Lagrangian submanifolds in Lagrangian Floer theory
following \cite{fukaya:family}, \cite{cho07}.
\begin{rem}\label{37.192}
We need to use flat complex line bundle for our purpose
by the following reason. In \cite{fooo06} we assumed that our
bounding cochain $b$ is an element of $H(L;\Lambda_{+})$
since we want the series
$$
\mathfrak m_1^{b}(x) = \sum_{k,\ell} \mathfrak m_{k+\ell+1}(b^{\otimes k}, x ,b^{\otimes \ell})
$$
to converge. There we used convergence with respect to the non-Archimedean norm.
For the case of Lagrangian fibers in toric manifold, the above series
converges for $b \in H^1(L;\Lambda_{0})$. The convergence is the
usual (classical Archimedean) topology on $\C$ on each coefficient of
$T^{\lambda}$.
\par
This is not an accident and was expected to happen in general.
(See 
Conjecture 3.6.46 \cite{fooo06} = Conjecture 11.46 \cite{fooo06pre}.) However for this convergence to occur, we need to
choose the perturbations on $\mathcal M_{k+1}^{\text{\rm main}}(L,\beta)$
so that it is consistent with $\mathcal M_{k'+1}^{\text{\rm main}}(L,\beta)$
($k' \ne k$)
via the forgetful map. We can make this choice for the current
toric situation by Lemma \ref{37.184} (3). In a more general situation,
we need to regard $\mathcal M_{1}^{\text{\rm main}}(L,\beta)$ as
a chain in the free loop space. (See \cite{fukaya:loop}.)
\par
On the other hand, if we use the complex structure other than the
standard one, we do not know whether Lemma \ref{37.184} (3)
holds or not. So in the proof of independence of Floer cohomology
under the various choices made, there is a trouble to use a bounding
cochain $b$ lying in $H^1(L;\Lambda_{0})$.
The idea, which is originally due to Cho \cite{cho07},
is to change the leading order term of $\mathfrak x$ by twisting the construction
using \emph{non-unitary} flat bundles on $L$.
\end{rem}

Let $X$ be a symplectic manifold and $L$ be its relatively spin
Lagrangian submanifold. Let $\rho : H_1(L;\Z) \to \C\setminus\{0\}$
be a representation and $\mathfrak L_{\rho}$ be the flat $\C$ bundle
induced by $\rho$.
\par
We replace the formula (\ref{Ainftystru}) by
$$
\mathfrak m_{k}^{\rho,can} = \sum_{\beta \in H_2(M,L)} \rho(\partial \beta)\,\, \mathfrak m^{can}_{k,\beta}
\otimes T^{\omega(\beta)/2\pi}.
$$
(Compare this with (\ref{eq:newmk}) in section \ref{sec:statement}.)
\begin{prop}\label{37.193}
$(H(L(u);\Lambda_0^{\R}),\mathfrak m_{k}^{\rho,can})$ is a filtered $A_{\infty}$ algebra.
\end{prop}
\begin{proof}
Suppose that $[f] \in \mathcal M_{k+1}^{\text{\rm main}}(L,\beta)$ is a
fiber product of
$[f_1] \in \mathcal M_{\ell+1}^{\text{\rm main}}(L,\beta_1)$
and
$[f_2] \in \mathcal M_{k-\ell}^{\text{\rm main}}(L,\beta_2)$.
Namely $\beta_1 + \beta_2 = \beta$ and
$ev_0(f_2) = ev_i(f_1)$ for some $i$. Then it is easy to see that
\begin{equation}\label{rhosum}
\rho(\partial\beta) = \rho(\partial\beta_1)\rho(\partial\beta_2).
\end{equation}
(\ref{rhosum}) and
(\ref{Ainftyformula}) imply the filtered $A_{\infty}$ relation.
\end{proof}
\par
The unitality can also be proved in the same way. The well
definedness (that is, independence of various choices up to homotopy
equivalence) can also be proved in the same way.
\begin{rem}\label{37.194}
We have obtained our twisted filtered $A_{\infty}$ structure on the
(untwisted) cohomology group $H^*(L;\Lambda_{0})$. This is because
the flat bundle $Hom(\mathfrak L_{\rho},\mathfrak L_{\rho})$ is
trivial. In more general situation where we consider a flat bundle
$\mathfrak L$ of higher rank, we obtain a filtered $A_{\infty}$
structure on cohomology group with local coefficients with values in
$Hom(\mathfrak L,\mathfrak L)$.
\par
The filtered $A_{\infty}$ structure $\mathfrak m_{k}^{\rho,can}$ is
different from $\mathfrak m_{k}^{can}$ in general as we can see from
the expression of the potential function given in Lemma \ref{37.104}.
\end{rem}
In the rest of this section, we explain the way how the Floer cohomology detects
the Lagrangian intersection. Namely we sketch the proof of
Theorem \ref{displace1} in our case and its generalization
to the case where we include the nonunital flat connection $\rho$.

\par
Let $\psi_t : X \to X$ be a Hamiltonian isotopy with $\psi_0 =$ identity.
We put $\psi_1 = \psi$. We consider the pair
$$
L^{(0)}=L(u), \quad L^{(1)} = \psi(L(u))
$$
such that $L^{(1)}$ is transversal to $L^{(0)}$.
We then take a one-parameter family $\{J_t\}_{t\in[0,1]}$ of
compatible almost complex structures such that
$J_0 = J$ which is the standard complex structure of $X$ and
$J_1 = \psi_*(J)$.

Let $p,q \in L^{(0)}\cap L^{(1)}$. We consider the
homotopy class of a maps
\begin{equation}\label{conncorbitmap}
\varphi : \R \times [0,1] \to X
\end{equation}
such that
\begin{enumerate}
\item
$\lim_{\tau \to -\infty} \varphi(\tau,t) = p$,
$\lim_{\tau \to +\infty} \varphi(\tau,t) = q$.
\item
$\varphi(\tau,0) \in L^{(0)}$, $\varphi(\tau,1) \in L^{(1)}$.
\end{enumerate}
We denote by $\pi_2(L^{(1)},L^{(0)};p,q)$ the set of all such homotopy
classes. There are obvious maps
\begin{equation}
\aligned
& \pi_2(L^{(1)},L^{(0)};p,r) \times \pi_2(L^{(1)},L^{(0)};r,q)
\to \pi_2(L^{(1)},L^{(0)};p,q), \\
& \pi_2(X;L^{(1)}) \times \pi_2(L^{(1)},L^{(0)};p,q)
\to \pi_2(L^{(1)},L^{(0)};p,q), \\
&  \pi_2(L^{(1)},L^{(0)};p,q)  \times \pi_2(X;L^{(0)})
\to \pi_2(L^{(1)},L^{(0)};p,q).
\endaligned
\end{equation}
We denote them by $\#$.
\begin{defn}
We consider the moduli space of maps (\ref{conncorbitmap})
satisfying (1), (2) above, in homotopy class
$B \in \pi_2(L^{(1)},L^{(0)};p,q)$, and satisfies the following
equation :
\begin{equation}\label{CReqn-Jt}
\frac{\partial\varphi}{\partial \tau}
+ J_t \left(\frac{\partial\varphi}{\partial t}\right) = 0.
\end{equation}
We denote it by
$
\widehat{\mathcal M}^{\text{reg}}(L^{(1)},L^{(0)};p,q;B).
$
We take its stable map compactification and denote it by
$
\widehat{\mathcal M}(L^{(1)},L^{(0)};p,q;B).
$
We divide this space by the $\R$ action induced by the translation of
$\tau$ direction to obtain
$
{\mathcal M}(L^{(1)},L^{(0)};p,q;B).
$
We define evaluation maps
$
ev_{L^{(i)}} : \widehat{\mathcal M}(L^{(1)},L^{(0)};p,q;B)
\to L^{(i)}
$
by
$
ev_{L^{(i)}}(\varphi) = \varphi(0,i),
$
for $i=0,1$.
\end{defn}

\begin{lem}\label{boundaryN}
${\mathcal M}(L^{(1)},L^{(0)};p,q;B)$ has an oriented
Kuranishi structure with corners.
Its boundary is isomorphic to the
union of the following three kinds of fiber products as spaces with
Kuranishi structure.
\begin{enumerate}
\item
$${\mathcal M}(L^{(1)},L^{(0)};p,r;B')
\times
{\mathcal M}(L^{(1)},L^{(0)};r,q;B'').$$
Here $B' \in \pi_2(L^{(1)},L^{(0)};p,r)$,
$B'' \in \pi_2(L^{(1)},L^{(0)};r,q)$.
\item
$${\mathcal M}_{1}(L(u);\beta')
\,\,{}_{ev_0}\times_{ev_{L^{(1)}}} \,\,
\widehat{\mathcal M}(L^{(1)},L^{(0)};p,q;B'').$$
Here $\beta' \in \pi_2(X;L^{(1)}) \cong \pi_2(X;L(u))$,
$\beta' \# B'' = B$. The fiber product is taken over
$L^{(1)} \cong L(u)$ by using $ev_0 : {\mathcal M}(L(u);\beta')
\to L(u)$ and
$ev_{L^{(1)}} : \widehat{\mathcal M}(L^{(1)},L^{(0)};p,q;B'')
\to L^{(1)}$.
\item
$$
\widehat{\mathcal M}(L^{(1)},L^{(0)};p,q;B')
\,\,{}_{ev_{L^{(0)}}}\times_{ev_0} \,\,
{\mathcal M}(L(u);\beta'') .$$
Here $\beta' \in \pi_2(X;L^{(1)}) \cong \pi_2(X;L(u))$,
$B' \# \beta'' = B$. The fiber product is taken over
$L^{(0)} \cong L(u)$ by using $ev_0 : {\mathcal M}(L(u);\beta'')
\to L(u)$ and
$ev_{L^{(0)}} : \widehat{\mathcal M}(L^{(1)},L^{(0)};p,q;B')
\to L^{(0)}$.
\end{enumerate}
We have
$$
\dim{\mathcal M}(L^{(1)},L^{(0)};p,q;B) =  \mu(B) - 1
$$
where
$$\aligned
\mu(B_1\#B_2) &= \mu(B_1) + \mu(B_2),
\\
\mu(\beta'\# B'') &= \mu(\beta') + \mu(B''),
\quad \mu(B'\# \beta'') = \mu(B') + \mu(\beta'').
\endaligned$$
Here the notations are as in $(1)$, $(2)$, $(3)$ above.
\end{lem}
Lemma \ref{boundaryN} is proved in 
subsection 7.1.4 \cite{fooo06} (= section 29.4 \cite{fooo06pre}).
\begin{lem}\label{multisectionpair}
There exists a system of multisections on ${\mathcal M}(L^{(1)},L^{(0)};p,q;B)$
such that
\begin{enumerate}
\item It is transversal to $0$.
\item It is compatible at the boundaries described in
Lemma $\ref{boundaryN}$.
Here we pull back multisection of ${\mathcal M}(L^{(1)},L^{(0)};p,q;B)$
to one on $\widehat{\mathcal M}(L^{(1)},L^{(0)};p,q;B)$ and use the
multisection in Lemma $\ref{37.184}$ on ${\mathcal M}(L(u);\beta)$.
\end{enumerate}
\end{lem}
\begin{proof}
We can find such system of multisections inductively over $\int_B \omega$
by using the fact that $ev_0 : {\mathcal M}(L(u);\beta)^{\frak s} \to L(u)$
is a submersion and Lemma \ref{boundaryN} (1) is a direct product.
\end{proof}
In the case when $\mu(B) = 1$, we define
$$
n(B) = \# {\mathcal M}(L^{(1)},L^{(0)};p,q;B)^{\frak s} \in \Q
$$
that is the number of zeros of our multisection counted with
sign and multiplicity. We use it to define Floer cohomology.
\par
Let
$
CF(L^{(1)},L^{(0)})
$
be the free $\Lambda_{0,nov}$ module generated by $L^{(0)} \cap L^{(1)}$.
We will define boundary operator on it.
\par
Let $\rho_i : \pi_1(L^{(i)}) \to \C^*$ be the representation.
We take harmonic $1$ form $h_{i} \in H^1(L^{(i)};\C)$ such that
$$
\rho_i(\gamma) = \exp\left( \int_{\gamma} h_{i}\right).
$$
Let $b_{i,+} \in H^1(L^{(0)};\Lambda_+) \subset \mathcal M(L^{(0)})$.
An element $B \in \pi_2(L^{(1)},L^{(0)};p,q)$
induces a path $\partial_{i} : [0,1] \to L^{(i)}$
joining $p$ to $q$ in $L^{(i)}$ for $i=0,1$. We define
$$
C(B;(h_{0},h_{1}),(b_{0,+},b_{1,+})) = \exp\left( \int_{\partial_{0}B} (h_{0}+b_{0,+}) \right)
\exp\left( -\int_{\partial_{1}B} (h_{1}+b_{1,+}) \right).
$$
This is an element of $\Lambda_0 \setminus \Lambda_+$.
It is easy to see that
\begin{equation}\label{Chomo1}
\aligned
&C(B_1\#B_2;(h_{0},h_{1}),(b_{0,+},b_{1,+})) \\
&= C(B_1;(h_{0},h_{1}),(b_{0,+},b_{1,+}))\,
C(B_2;(h_{0},h_{1}),(b_{0,+},b_{1,+})),
\endaligned
\end{equation}
and
\begin{equation}\label{Chomo2}
\aligned
&C(\beta'\# B'';(h_{0},h_{1}),(b_{0,+},b_{1,+})) \\
&= C(B'';(h_{0},h_{1}),(b_{0,+},b_{1,+})) \exp\left(\int_{\partial \beta'}h_{1}+b_{0,+}\right).
\endaligned\end{equation}
(Here $B'$ and $\beta''$ are as in Lemma \ref{boundaryN} (2).)
Now we define
\begin{equation}
\aligned
&\langle  \delta_{(h_{0},h_{1}),(b_{0,+},b_{1,+})}(p),q  \rangle \\
&= \sum_{B \in \pi_2(L^{(1)},L^{(0)};p,q); \mu(B) = 1}
n(B)C(B;(h_{0},h_{1}),(b_{0,+},b_{1,+})).
\endaligned
\end{equation}
For the case $h_{1} = \psi_*(h_{0})$,
$b_{1,+} = \psi_*(b_{+})$, $b_{0,+} = b_+$, we write
$C(B;h_{0},b_{+})$ and
$\delta^{h_{0},b_{+}}$, in place of
$C(B;(h_{0},h_{1}),(b_{0,+},b_{1,+}))$
and
$\delta_{(h_{0},h_{1}),(b_{0,+},b_{1,+})}$, respectively.
\begin{lem}\label{sqarzero}
$$
\delta_{(h_{0},h_{1}),(b_{0,+},b_{1,+})} \circ
\delta_{(h_{0},h_{1}),(b_{0,+},b_{1,+})}
= (\frak{PO}(b(1)) - \frak{PO}(b(0))) \cdot id
$$
where $b(i) = h_{i}+b_{i,+} \in H^1(L(u);\Lambda_0)$.
\end{lem}
\begin{proof}
Let $p,q \in L^{(1)} \cap L^{(0)}$. We consider $B
\in \pi_2(L^{(1)},L^{(0)};p,q)$ with
$\mu(B) = 2$. We consider
the boundary of ${\mathcal M}(L^{(1)},L^{(0)};p,q;B)$.
We put
$$
\delta_B =
\begin{cases}
1  &\text{$B \in \pi_2(L^{(1)},L^{(0)};p,q)$, $p=q$, and $B$ is the class of
constant map,} \\
0  &\text{otherwise.}
\end{cases}
$$
Then using the classification of the boundary of
${\mathcal M}(L^{(1)},L^{(0)};p,q;B)$ in Lemma \ref{boundaryN}
we have the following equality :
\begin{equation}\label{boundary0}
0 = \sum_{r,B,B''} n(B')n(B'')
+ \sum_{\beta',B''} c_{\beta'} \delta_{B''}
- \sum_{\beta'',B'} c_{\beta''} \delta_{B'},
\end{equation}
where the summention in the first, second, third terms of right hand side is
taken over the set described in Lemma \ref{boundaryN} (1), (2), (3), respectively,
and $c_{\beta}$ is defined in Definition \ref{cbeta}.

We multiply (\ref{boundary0}) by $C(B;(h_{0},h_{1}),(b_{0,+},b_{1,+}))$ and calculate the right hand side by
using Formulae (\ref{Chomo1}), (\ref{Chomo2}) and Lemma \ref{37.189}.
It is easy to see that the first term gives
$$
\langle\delta_{(h_{0},h_{1}),(h_{0,+},b_{1,+})} \circ
\delta_{(h_{0},h_{1}),(b_{0,+},b_{1,+})}(p),q\rangle.
$$
The second term is $0$ if $p \ne q$, and is
$$
\sum_{\beta} c_{\beta} \exp\left(\int_{\partial \beta}h_{0}+b_{0,+}\right)
= \frak{PO}(b(0))
$$
if $p=q$. The third term gives $\frak{PO}(b(1)) \cdot id$ in a similar way.
\end{proof}
\begin{defn}
Let $b(0) = h_{0}+b_{+}$ and $b(1) = \psi_*(b(0))$. We define
$$
HF((L^{(0)},b(0)),(L^{(1)},b(1));\Lambda_0) \cong
\frac{\text{Ker} \,\,\delta_{h_{0},h_{+}}}{\text{Im}\,\, \delta_{h_{0},h_{+}}}.
$$
This is well defined by Lemma \ref{sqarzero} and
by the identity $\frak{PO}(\psi_*(b(0))) =
\frak{PO}((b(0)))$.
\end{defn}
Now we have
\begin{lem}\label{invariance}
$$
HF((L^{(0)},b(0)),(L^{(1)},b(1));\Lambda_0) \otimes_{\Lambda_0} \Lambda
\cong
\frac{\text{\rm Ker} \,\,\frak m_1^{\rho,can,b_+}}{\text{\rm Im} \,\, \frak m_1^{\rho,can,b_+}}
\otimes_{\Lambda_0} \Lambda.
$$
Here $
\rho(\gamma) = \exp\left( \int_{\gamma} h_{0}\right)
$
and
$$
\frak m_1^{\rho,can,b_+}(x) = \sum_{k_1,k_2=0}^{\infty}
\frak m^{\rho,can}_{k_1+k_2+1} (b_+^{\otimes k_1},x,b_+^{\otimes k_2}).
$$
\end{lem}
Lemma \ref{invariance} implies (the $\rho$ twisted version of) Theorem \ref{displace1}
in our case. We omit the proof of Lemma \ref{invariance} and refer
section 5.3 \cite{fooo06} (= section 22 \cite{fooo06pre}) or section 8 \cite{fooo08}.

\section{Floer cohomology at a critical point of potential function}
\label{sec:floerhom}

In this section we prove Theorem \ref{homologynonzero} etc. and complete the proof of
Theorem \ref{toric-intersect}.

\begin{proof}[Proof of Lemma \ref{37.104}] Let $\beta \in H_2(X,L(u_0))$
with $\mu(\beta) =2$ and $\mathcal M_1^{\text{\rm
main}}(L(u_0),\beta)$ be nonempty. We have $\beta = \sum_{i=1}^m c_i
\beta_i + \sum_j\alpha_j$ by Theorem \ref{37.177} (5). Let $\rho$ be
as in (\ref{repdef}). We have $\rho(\partial \beta) = \prod
\rho(\partial \beta_i)^{c_i}$. Note $\partial \beta_i = \sum_j
v_{i,j}\text{\bf e}_j^*$. Thus we have
\begin{equation}
\left\{
\aligned
\rho(\partial \beta_i) &= \mathfrak y_{1,0}^{v_{i,1}} \cdots \mathfrak y_{n,0}^{v_{i,n}}, \\
\rho(\partial \beta) &= \prod_i\prod_j\mathfrak y_{j,0}^{c_iv_{i,j}}.
\endaligned
\right.
\label{37.197}\end{equation}
Therefore for $b \in H^1(L(u_0);\Lambda_{+}^{\C})$
we have
$$\aligned
\sum_{k=0}^\infty\mathfrak m_{k,\beta}^{\rho,can}(b,\cdots,b) &=
\sum_{k=0}^\infty \mathfrak y_{1,0}^{v_{i,1}} \cdots \mathfrak
y_{n,0}^{v_{i,n}}
\mathfrak m_{k,\beta}^{can}(b,\cdots,b) \\
& = \sum_{k=0}^\infty e^{\mathfrak x_{1,0}v_{i,1}} \cdots e^{\mathfrak
x_{n,0}v_{i,n}} \frac{c_{\beta}}{k!} (b\cap \partial \beta)^k \cdot[PD(L)]\\
&= \sum_{k=0}^\infty\mathfrak m^{can}_{k,\beta}\left(b+ \sum_{j=1}^n \mathfrak
x_{j,0}\text{\bf e}_j,\cdots,b+ \sum_{j=1}^n \mathfrak x_{j,0}\text{\bf
e}_j\right).
\endaligned$$

On the other hand, it follows from Theorems \ref{potential} and
\ref{weakpotential} that the left and the right sides of this
identity correspond to those in Lemma \ref{37.104} respectively.
This finishes the proof of Lemma \ref{37.104}. \end{proof}
\begin{proof}[Proof of Theorem \ref{homologynonzero}]
Let $x^+=(x^+_1,\cdots,x^+_n)$, $x^+_1,\cdots,x^+_n \in \Lambda_{+}$. We put
$$
x(x^+) = \sum_{i} (\mathfrak x_{i,0} + x^+_i)\text{\bf e}_i,
\quad
b(x^+) = \sum_{i} x^+_i\text{\bf e}_i.
$$
From Lemma \ref{37.104} we derive
$$
\aligned
\mathfrak{PO}^{u_0}_{\rho}(b(x^+)) &=
\sum \mathfrak m_k^{\rho,can}(b(x^+),\cdots,b(x^+)) \cap [L(u_0)] \\
&= \sum \mathfrak m_k^{can}(x(x^+),\cdots,x(x^+)) \cap [L(u_0)] =
\mathfrak{PO}^{u_0}(x(x^+)).
\endaligned$$
Let $\frak x$ be as in (\ref{frakxdef}).
Then we have
$$
\frac{\partial}{\partial x^+_i} \mathfrak{PO}^{u_0}_{\rho}(b(x^+)) \Big
\vert_{x(x^+) = \mathfrak x} = \frac{\partial}{\partial x^+_i}
\mathfrak{PO}^{u_0}(x(x^+)) \Big\vert_{x(x^+) = \mathfrak x} = \frac{\partial
\mathfrak{PO}^{u_0}}{\partial x_i}(\mathfrak x) = 0,
$$
where the last equality follows from the assumption (\ref{formula:critical0}).
(In case when (\ref{formula:criticalweak}) is assumed we have
`$\equiv 0 \mod T^{\mathcal N}$' in place of `$=0$'.)
On the other hand, we have
\begin{equation}
\aligned \frac{\partial}{\partial x^+_i}
&\mathfrak{PO}^{u_0}_{\rho}(b(x^+))\Big \vert_{x(x^+) = \mathfrak x} \\
& = \sum_k\sum_{\ell}
\mathfrak m^{\rho,can}_{k}
(b^{\otimes \ell} ,\text{\bf e}_i, b^{\otimes (k-\ell-1)}) \cap [L(u_0)] \\
&= \mathfrak m_1^{\rho,can,b}(\text{\bf e}_i) \cap [L(u_0)].
\endaligned
\label{calcm1b}
\end{equation}
Note here and hereafter we write $b$ in place of
$b(x^+)$ with $x(x^+) = \mathfrak x$.
Namely $b = \mathfrak x - \sum \mathfrak x_{i,0}\text{\bf e}_i$.
\par
Hence we obtain
\begin{equation}
\mathfrak m_1^{\rho,can,b}(\text{\bf e}_i)
\begin{cases}
= 0 &\text{if (\ref{formula:critical0}) is satisfied} \\
\equiv 0 \mod T^{\CN} &\text{if (\ref{formula:criticalweak}) is satisfied}.
\end{cases}
\label{37.198}\end{equation} We remark that by the degree reason
$\mathfrak m^{\rho,can,b}_1(\text{\bf e}_i)$ is proportional to
$PD[L(u_0)]$.
\par
We next prove the vanishing of $\mathfrak m_1^{\rho,can,b}(\text{\bf f})$
for the classes $\text{\bf f}$ of higher degree.
Namely we prove
\begin{lem}\label{37.199}
For $\text{\bf f} \in H^*(L(u_0);\Lambda_{0}^{\C})$ we have :
$$
\mathfrak m_{1,\beta}^{\rho,can,b}(\text{\bf f})
\begin{cases}
= 0 &\text{if $(\ref{formula:critical0})$ is satisfied} \\
\equiv 0 \mod T^{\CN} &\text{if $(\ref{formula:criticalweak})$ is satisfied}.
\end{cases}
$$
\end{lem}
\begin{proof}
Let $d = \deg \text{\bf f}$ and
$2\ell =\mu(\beta)$. We say $(d,\ell) < (d',\ell')$ if $\ell<\ell'$
or $\ell=\ell'$, $d<d'$. We prove the lemma by upward induction on
$(d,\ell)$. The case $d=1$ is (\ref{37.198}). We remark that
$\mathfrak m_{k,\beta} = 0$ if $\mu(\beta) \le 0$.
\par
We assume that the lemma is proved for $(d',\ell')$ smaller than
$(d,\ell)$ and will prove the case of $(d,\ell)$. Since the case
$d=1$ is already proved, we may assume that $d\ge 2$. Let $\text{\bf
f} = \text{\bf f}_1 \cup \text{\bf f}_2$ where $\deg\text{\bf
f}_i\ge 1$. By the $A_{\infty}$-relation we have
$$\aligned
\mathfrak m_{1,\beta}^{\rho,can,b}(\text{\bf f}_1 \cup
\text{\bf f}_2) = &\sum_{\beta_1+\beta_2=\beta} \pm\mathfrak
m^{\rho,can,b}_{2,\beta_1}(\mathfrak m^{\rho,can,b}_{1,\beta_2}(\text{\bf f}_1),
\text{\bf f}_2)\\
&+
\sum_{\beta_1+\beta_2=\beta} \pm\mathfrak m^{\rho,can,b}_{2,\beta_1}(
\text{\bf f}_1, \mathfrak m_{1,\beta_2}(\text{\bf f}_2)) \\
&+
\sum_{\beta_1+\beta_2=\beta, \beta_2\ne 0} \pm\mathfrak m^{\rho,can,b}_{1,\beta_1}(\mathfrak m_{2,\beta_2}(\text{\bf f}_1,
\text{\bf f}_2)).
\endaligned$$
We remark that $\mathfrak m^{\rho,can,b}_{1,\beta_0} =
0$ since we are working on a canonical model.
\par
The first two terms of the right hand side vanishes by the induction
hypothesis since $\deg \text{\bf f}_i < \deg \text{\bf f}$ and
$\mu(\beta_i) \le \mu(\beta)$. The third term also vanishes since
$\mu(\beta_1) < \mu(\beta)$. The proof of Lemma \ref{37.199} is
complete.
\end{proof}
Lemma \ref{37.199} immediately implies Theorem
\ref{homologynonzero}. \end{proof}
\begin{proof}[Proof of Proposition \ref{eq:m1b}]
Let us specialize to the case of 2 dimension. In case $\dim L(u_0) =
2$, we can prove $\mathfrak m_{1,\beta}^{\rho,can,b} = 0$ for
$\mu(\beta) \ge 4$ also by dimension counting. We can use that to
prove Proposition \ref{eq:m1b} in the same way as above.
\end{proof}
\par\medskip

\begin{proof}[Proof of Proposition \ref{prof:bal2}]
Let $\omega_i$, $u_i$, $\mathfrak x_{i,\mathcal N}$ be as in Definition
\ref{def:balanced}. We assume $\psi : X \to X$ does not satisfy
(\ref{eq:Lu_02}) or (\ref{eq:Lu_12}) and will deduce a
contradiction. We use the same (time dependent) Hamiltonian as
$\psi$ to obtain $\psi_{i} : (X,\omega_{i}) \to (X,\omega_{i})$.
Take an integer $\CN$ such that $\Vert\psi_i\Vert < 2\pi \CN$ for large
$i$. Then for sufficiently large $i$, $L(u_0^{i})$ and $\psi_{i}$
does not satisfy (\ref{eq:Lu_02}) or (\ref{eq:Lu_12}). In fact if
$\psi(L(u_0)) \cap L(u_0) = \emptyset$ then for sufficiently large
$i$, we have $\psi_{i}(L(u^{i}_0)) \cap L(u^{i}_0) = \emptyset$. If
$\psi(L(u_0))$ is transversal to $L(u_0)$ and if (\ref{eq:Lu_12}) is
not satisfied, then
$$
\#(\psi(L(u_0)) \cap L(u_0)) \ge \#(\psi_{\e_i}(L(u^{i}_0)) \cap L(u^{i}_0)).
$$
On the other hand, by
Theorem \ref{homologynonzero} we have
$$
HF((L(u_i),\mathfrak x_{i,k}),(L(u_i),\mathfrak x_{i,k});\Lambda_{0}^{\C}/(T^{\CN}))
\cong H(T^n;\Lambda_{0}^{\C}/(T^{\CN})).
$$
It follows from Universal Coefficient
Theorem that
\begin{equation}
HF((L(u_i),\mathfrak x_{i,k}),(L(u_i),\mathfrak x_{i,k});\Lambda_{0}^{\C})
\cong \Lambda_{0}^{\oplus a} \oplus \bigoplus_{i=1}^b \Lambda_{0}/(T^{c(i)})
\label{HFmodN}\end{equation}
such that $c(i) \ge \CN$ and $a+2b \ge 2^n$.
This contradicts to Theorem J
\cite{fooo06}.
(In fact Theorem J
\cite{fooo06} (= Theorem \ref{displace2}) and (\ref{HFmodN}) imply that (\ref{eq:Lu_02}) and (\ref{eq:Lu_12}) hold for $L(u_i)$
and $\psi_i$ with $\Vert\psi_i\Vert <2\pi \CN$.)
Proposition \ref{prof:bal2} is proved.
\end{proof}
\begin{proof}[Proof of Theorem \ref{thm:eLuEu}]
The proof is the same as the proof of Proposition \ref{prof:bal2} above.
\end{proof}
\par
Now we are ready to complete the proof of Theorem
\ref{toric-intersect}.
\begin{proof}[Proof of Theorem \ref{toric-intersect}]
In the case where the vertices of $P$ are contained in
$\Q^n$,
Proposition \ref{existcrit} and Theorem
\ref{homologynonzero} imply that $L(u_0)$ is
balanced in the sense of Definition \ref{def:balanced}.
Therefore Proposition \ref{prof:bal2} implies Theorem \ref{toric-intersect}
in this case.
If the leading term equation is strongly nondegenerate, Theorem
\ref{toric-intersect} also follows from Theorem
\ref{homologynonzero}, Theorem \ref{thm:elliminate} and Proposition
\ref{prof:bal2}.
\par
We finally present an argument to remove the rationality assumption.
In view of Lemma \ref{prof:bal2}, it suffices to prove the
following Proposition \ref{prof:bal}.
\begin{prop}\label{prof:bal}
In the situation of Theorem $\ref{toric-intersect}$, there exists $u_0$ such that
$L(u_0)$ is a balanced Lagrangian fiber.
\end{prop}
\begin{proof}
Let
$\pi : X \to P$ be as in Theorem \ref{toric-intersect}. Let us
consider $s_k$, $S_k$, $P_k$ as in section \ref{sec:var}. We obtain
$u_0 \in P$ such that $\{u_0\} = P_K$. We will prove that $L(u_0)$ is balanced.
\par
We perturb the K\"ahler form $\omega$ of $X$ a bit so that we
obtain $\omega'$.
Let $P'$ be the corresponding moment polytope and $s_k^{\omega'}$,
$S_k^{\omega'}$, $P_k^{\omega'}$, be the corresponding piecewise affine
function, number, subset of $P_{\omega'}$ obtained for $\omega'$,
$P_{\omega'}$ as in
section \ref{sec:var}.
\begin{prop}\label{prop:approx}
We can choose $\omega_h$ so that
$\omega_h$ is rational and
$\lim_{h\to\infty} s_k^{\omega_{h}} = s_k$,
$\lim_{h\to\infty}S_k^{\omega_{h}}=S_k$,
$\lim_{h\to\infty}P_k^{\omega_{h}}=P_k$,
$\dim P_k^{\omega_{h}}= \dim P_k$ for all $k$.
\end{prop}
\begin{proof}
We write $I_k^{\omega'}$ the set $I_k$ defined in (\ref{form;defIk2}) section \ref{sec:var}
for $\omega'$, $P'$. We will prove the following lemma.
We remark that the set $\mathfrak K$ of $T^n$ invariant K\"ahler structure $\omega'$ is regarded as an open set of an
affine space defined on $\Q$ (that is the K\"ahler cone).
We may regard $\mathfrak K$ as a moduli space of moment polytope as follows :
We consider a polyhedron $P'$ each of whose edges is parallel to a corresponding edge of
$P$. Translation defines an $\R^n$ action on the set of such $P'$.
The quotient space can be identified with $\mathfrak K$.
\begin{lem}\label{inductiongeneric}
There exists a subset $\mathfrak P_k$ of $\mathfrak K$ which is a
nonempty open subset of an affine subspace defined over $\Q$ such
that any element $\omega' \in \mathfrak P_k$ has the following
properties :
\par
\begin{enumerate}
\item $\dim P_l^{\omega'} = \dim P_l^{\omega}$ for $l\le k$.
\item $I_l^{\omega'} = I_l^{\omega}$ for $l\le k$.
\end{enumerate}
\end{lem}
\begin{rem}
In the case of Example \ref{exa:onepointblow}, the set $P^{\omega'}_k$ etc.
jumps at the point $\alpha = 1/3$ in the K\"ahler cone.
Hence the set $\mathfrak P_k$ may have strictly smaller dimension than
$\mathfrak K$.
\end{rem}
\begin{proof}
Let $A_l^{\omega'}$ be the affine space defined in section \ref{sec:var}.
(We put ${\omega'}$ to specify the symplectic form.)
We write $\ell_i^{\omega}$, $\ell_i^{\omega'}$ in place of $\ell_i$
to specify symplectic form and moment polytope.
We remark that the linear part of $\ell_i^{\omega}$ is
equal to the linear part of $\ell_i^{\omega'}$.
\par
The proof of Lemma \ref{inductiongeneric} is by induction on $k$.
Let us first consider the case $k=1$.
We put
$$
\widehat A_1^{\omega'} = \{ u \in M_\R \mid \ell_{1,1}^{\omega'}(u)
= \cdots = \ell_{1,a_1}^{\omega'}(u)\}.
$$
We remark that $\{\ell^{\omega}_{1,1},\cdots,\ell^{\omega}_{1,a_1}\}
= I_1^{\omega}$ and so $\widehat A_1^{\omega} = A_1^{\omega}$.
\par
We put
$$
\mathfrak P'_1 = \{ \omega' \mid \dim \widehat A_1^{\omega'} =
\dim A_1^{\omega} \}.
$$
It is easy to see that $\mathfrak P'_1$ is a nonempty affine subset of
$\mathfrak K$ and is defined over $\Q$.
\begin{sublem}\label{sublem;pphat}
If $\omega' \in \mathfrak P'_1$ and is sufficiently close to $\omega$,
then $P_1^{\omega'}$ is an equi-dimensional polyhedron in $\widehat
A_1^{\omega'}$. In particular $\widehat A_1^{\omega'} =
A_1^{\omega'}$.
\end{sublem}
\begin{proof}
The tangent space of $\widehat A_1^{\omega'}$ is parallel to
the tangent space of $A_1^{\omega}$. Therefore
$\ell^{\omega'}_{1,j}$ is constant on $\widehat A_1^{\omega'}$.
We put
$$
\widehat S_1^{\omega'} = \ell^{\omega'}_{1,1}(u)
$$
for some $u \in \widehat A_1^{\omega'}$.
\par
On the other hand, if $\ell_i^{\omega} \notin I_1^{\omega}$ then
$\ell_i^{\omega}(u) > S_1^{\omega}$ on $\text{Int}\,P_1^{\omega}$. Therefore if
$\omega'$ is sufficiently close to $\omega$ we have
$
\ell_i^{\omega'}(u) > \widehat S_1^{\omega'}
$
on a neighborhood of a compact subset of $\text{Int}\,\,P_1^{\omega}$,
which we identify with a subset of $P'$.
This implies the sublemma.
\end{proof}
The Condition (1), (2) of Lemma \ref{prop:approx} in the case
$k=1$ follows from Sublemma \ref{sublem;pphat} easily.
\par
Let us assume that Lemma \ref{inductiongeneric} is proved up to
$k-1$. We remark
$\{\ell^{\omega}_{k,1},\cdots,\ell^{\omega}_{k,a_k}\} =
I_k^{\omega}$. We put
$$
\widehat A_k^{\omega'} =
\{ u \in A_{k-1}^{\omega'} \mid \ell_{k,1}^{\omega'}(u) = \cdots = \ell_{k,a_k}^{\omega'}(u)\}
$$
and
$$
\mathfrak P'_k = \{ \omega' \in \mathfrak P'_{k-1} \mid \dim \widehat A_k^{\omega'} =
\dim A_k^{\omega} \}.
$$
$\mathfrak P'_k$ is a nonempty affine subset of $\mathfrak K$ and is
defined over $\Q$. We can show that a sufficiently small open
neighborhood $\mathfrak P_k$ of $\omega$ in $\mathfrak P'_k$ has the
required properties in the same way as the first step of the
induction. The proof of Lemma \ref{inductiongeneric} is complete.
\end{proof}
Proposition \ref{prop:approx} follows immediately from
Lemma \ref{inductiongeneric}. In fact the set of rational points
are dense in $\mathfrak P_K$.
\end{proof}
Proposition \ref{prof:bal} follows from Proposition
\ref{prop:approx}, Proposition \ref{existcrit} and Theorem
\ref{homologynonzero}.
\end{proof}
\par
The proof of Theorem \ref{toric-intersect} is now complete.
\end{proof}
\begin{proof}[Proof of Proposition \ref{prop:approx1}]
The proof is similar to the proof of Proposition \ref{prop:approx}.
Let $I_k$ be as in (\ref{form;defIk29}).
We write it as $I_k(P,u_0)$, where $P$ is the
moment polytope of $(X,\omega)$.
We define $I_k(P',u'_0)$ as follows.
\par
Let $P'$ be a polytope which is a
small perturbation of $P$ and such that each of the faces are parallel to the
corresponding face of $P$. Let $u'_0 \in \text{\rm Int}\,\, P'$.
Let us consider the set $\mathfrak K^+$ of all such pairs $(P',u'_0)$.
It is an open set of an affine space defined over $\Q$.
Each of such $P'$ is a moment
polytope of certain K\"ahler form on $X$. (We remark that K\"ahler
form on $X$ determine $P'$ only up to translation.)
\par
For each $P'$, we take corresponding K\"ahler form on $X$
and it determines a potential function. Therefore $I_k(P',u'_0)$
is determined by (\ref{form;defIk29}).
We define $A^\perp_l(P',u'_0)$ in the same way as
Definition \ref{def:genAl}.
\begin{lem}\label{inductiongeneric2}
There exists a subset $\mathfrak Q_k$ of $\mathfrak K^+$ which is a
nonempty open set of an affine subspace defined over $\Q$. All the
elements $(P',u'_0)$ of $\mathfrak Q_k$ have the following
properties.
\par
\begin{enumerate}
\item $\dim A^\perp_l(P',u'_0) = \dim A^\perp_l(P,u_0)$ for $l\le k$.
\item $I_l(P',u'_0) = I_l(P,u_0)$ for $l\le k$.
\end{enumerate}
\end{lem}
The proof is the same as the proof of Lemma \ref{inductiongeneric} and
is omitted.
\par
Now we take a sequence of rational points $(P_h,u^h_0) \in \mathfrak Q_k$
converging to $(P,u_0)$.
Lemma \ref{inductiongeneric2} (2) implies that the leading term equation
at $u^h_0$ is the same as the leading term equation at
$u_0$. The proof of Proposition \ref{prop:approx1} is complete.
\end{proof}
\begin{proof}[Proof of Lemma \ref{lem:autorat}]
Let $[\omega] \in H^2(X;\Q)$. We may take the moment
polytope $P$ such that its vertices are rational.
(This time we do not change $P$ throughout the proof.)
Let $u_0 \in \text{\rm Int}\, P$ and we assume that
$\frak{PO}_0^{u_0}$ has a nondegenerate critical
point in $(\Lambda_0 \setminus \Lambda_+)^n$.
\par
We define $I_k(P,u)$ as above.
In the same way as the proof of Lemma \ref{inductiongeneric2}
we can prove the following.
\begin{sublem}\label{sublemu0}
The set $\mathcal P_k$ of all $u \in \text{\rm Int}\, P$
satisfying the following conditions $(1)$,$(2)$ contains
an open neighborhoods $u_0$ in certain affine
subspace $A$ of $\R^n$. $A$ is defined on $\Q$.
\begin{enumerate}
\item $\dim A^\perp_l(P,u) = \dim A^\perp_l(P,u_0)$ for all $l<k$.
\item $I_l(P,u) = I_l(P,u_0)$ for all $l<k$.
\end{enumerate}
\end{sublem}
We omit the proof.
We take $K$ such that
$\{ d \ell_i \mid \ell_i \in I_1(P,u_0) \cup \cdots \cup I_K(P,u_0)\}$
generates $N_{\R}$. (Note $P \subset N_{\R} = M_{\R}^*$.)
By Sublemma \ref{sublemu0} there exists a sequence $\{u_i\}_{i=1,2,\cdots}$ of rational points $u_i$
in $\mathcal P_K$ which converges to $u_0$.
\par
(1), (2) above implies $\frak{PO}_0^{u_0}$ and $\frak{PO}_0^{u_i}$
has the same leading term equation.
Therefore by assumption $\frak{PO}_0^{u_i}$ has a critical point
on $(\Lambda_0 \setminus \Lambda_+)^n$.
Since $Jac(\frak{PO}_0;\Lambda)$ is finite dimensional,
it follows that we may take a subsequence $u_{k_i}$ such that
$u_{k_1} = u_{k_2} = \cdots$. Hence
$u_0=u_{k_i}$, is rational as required.
\end{proof}

\begin{rem}\label{balancerem2}
We can replace Definition \ref{def:balanced} (3) by
$$
HF((L(u_i),\mathfrak x_{i,\CN}),(L(u_i),\mathfrak x_{i,\CN});\Lambda^{\C}/(T^{\CN})) \supseteq \Lambda^{\C}/(T^{\CN}).
$$
In fact the following three conditions are equivalent to one another
:
\begin{enumerate}
\item
$ HF((L(u),\mathfrak x),(L(u),\mathfrak x);\Lambda^{\C}/(T^{\CN})) \cong H(T^n;\C) \otimes
\Lambda^{\C}/(T^{\CN}). $
\item
$ HF((L(u),\mathfrak x),(L(u),\mathfrak x);\Lambda^{\C}/(T^{\CN})) \supseteq
\Lambda^{\C}/(T^{\CN}). $
\item
$
\frac{\partial \mathfrak{PO}^{u}}{\partial y_k}
\equiv 0, \mod T^{\CN} \qquad k=1,\cdots, n,
$
at $\mathfrak x$.
\end{enumerate}
(1) $\Rightarrow$ (2) is obvious. (3) $\Rightarrow$ (1) is Theorem
\ref{homologynonzero}. Let us prove (2) $\Rightarrow$ (3). Suppose
(3) does not hold. We put $\frac{\partial
\mathfrak{PO}^{u}}{\partial y_k} \equiv c T^{\lambda} \mod
T^{\lambda}\Lambda_+$, where $c \in \C \setminus \{0\}$ and $0 \le
\lambda < N$. Then (\ref{calcm1b}) implies $T^{\CN - \lambda}PD[L(u)]
= 0$ in $HF((L(u),\mathfrak x),(L(u),\mathfrak x);\Lambda^{\C}/(T^{\CN}))$. Since
$PD[L(u)]$ is a unit, (2) does not hold.
\end{rem}
\begin{rem}\label{37.200}
The proof of Theorem \ref{homologynonzero} (or of Lemma \ref{37.199}) does {\it not} imply
\begin{equation}
\mathfrak m_{k,\beta}(\rho_1,\cdots,\rho_k)= 0
\label{37.201}\end{equation} for $\mu(\beta) \ge 4$.
So it does not imply that the numbers $c_{\beta}$ (Definition \ref{cbeta})
determine the filtered $A_{\infty}$ algebra
$(H(L(u);\Lambda_0),\frak m)$ up to homotopy equivalence.
We however
believe that this is indeed the case. In fact the
homology group $H(\mathcal L(T^{n});\Q)$ of the free loop space
$\mathcal L(T^{n})$ is trivial of degree $> n$. On the other
hand, $\dim \mathcal M_{1}^{\text{\rm main}}(L(u_0);\beta) = n +
\mu(\beta) -2$. Hence if $\mu(\beta) \ge 4$ there is no nonzero homology
class on the corresponding degree in the free loop space. Using the
argument of \cite{fukaya:loop} it may imply that the contribution of those classes
to the homotopy type of filtered $A_{\infty}$ structure
is automatically determined from the contribution of the
classes of Maslov index $2$.
\par
On the other hand, pseudo-holomorphic disc with Maslov index $\ge 4$
certainly contributes to the operator $\mathfrak q_{\ell,k,\beta}$
introduced in 
section 3.8 \cite{fooo06} (=
section 13 \cite{fooo06pre}) : Namely $\mathfrak q_{\ell,k,\beta}$
is the operator that involves a cohomology class of the ambient
symplectic manifold $X$. (See Remark \ref{rem:qcanbeused})
It seems that tropical geometry
will play a role in this calculation.
\end{rem}
\section{Appendix 1 : Algebraically closedness of Novikov fields}
\label{sec:appendix}

\begin{lem}\label{algclos}
If $F$ is an algebraically closed field of characteristic $0$, then
$\Lambda^F$ is algebraically closed.
\end{lem}
\begin{proof}
Let $ \sum_{k=0}^n a_k x^k = 0 $ be a polynomial equation with
$\Lambda^F$ coefficients. We will prove that it has a solution in
$\Lambda^F$ by induction on $n$. We may assume that $a_n = 1$.
Replacing $x$ by $x - \frac{a_{n-1}}{n}$ we may assume $a_{n-1} =
0$. (Here we use the fact that the characteristic of $F$ is $0$.) We
may assume furthermore $a_0 \ne 0$, since otherwise $0$ is a
solution. We put
$$
c = \inf_{k=0,\cdots,n-2} \frac{\frak v_T(a_k)}{n-k}.
$$
We put $x = T^cy$, $b_k = T^{c(k-n)}a_k$. Then our equation is equivalent to
$
P(y) = \sum_{k=0}^n b_k y^k = 0
$.
We remark that $b_n = 1$, $b_{n-1} = 0$, $b_0 \ne 0$. Moreover
\begin{equation}\label{polyeq3}
\frak v_T(b_k) = c(k-n) + \frak v_T(a_k) \ge 0.
\nonumber\end{equation} Namely $b_k \in \Lambda_0$. We define
$\overline b_k \in F$ to be the zero order term of $b_k$ i.e., to
satisfy $b_k \equiv \overline b_k \mod \Lambda_+$. We consider the
equation $ \overline P(\overline y) = \sum_{k=0}^n \overline b_k
\overline y^k = 0. $ By our choice of $c$ there exists $k<n-1$ such
that $\overline b_k \ne 0$ and $ \overline b_n = 1$, $\overline
b_{n-1}=0$. Therefore $\overline P$ has at least two roots. (We use
the fact that the characteristic of $F$ is 0 here.) Since $F$ is
algebraically closed we can decompose $ \overline P = \overline Q
\overline R $ where $\overline Q$ and $\overline R$ are monic, of
nonzero degree and coprime. Therefore by Hensel's lemma, there
exists $Q, R \in \Lambda_0[y]$ such that $P = QR$ and $\deg Q = \deg
\overline Q$, $Q \equiv \overline Q \mod \Lambda_+$, $R \equiv
\overline R \mod \Lambda_+$\footnote{A proof of Hensel's lemma, in
the case when valuation is not necessarily discrete, is given, for
example, in p 144 \cite{BGR}. See also the proof of Lemma \ref{convclo}
given below.}.
\par
Since the degree of $Q$ is smaller than the degree of $P$, it follows
from induction hypothesis that $Q$ has a root in
$\Lambda^{F}$. The proof of Lemma \ref{algclos}
is now complete.
\end{proof}
By a similar argument we can show that if $F$ is characteristic $0$
then a finite algebraic extension of $\Lambda^F$ is contained in
$\Lambda^K$ for some finite extension $K$ of $F$. (We used this fact
in section \ref{sec:exa2}.)
\par\medskip
\begin{proof}[Proof of Lemma \ref{convclo}]
In view of the proof of Lemma \ref{algclos}, it suffices to show that
$\Lambda_0^{conv}$ is Henselian. (Namely Hensel's lemma holds for it.)
Let
$$
P(X) = \sum_{i=0}^{n-1} a_i X^i + X^n \in \Lambda_0^{conv}[X].
$$
We assume that its $\C$-reduction $\overline P \in \C[X]$ is decomposed
as $\overline P = \overline Q \overline R$, where $\overline Q$ and $\overline R$ are monic and coprime. We put
$$
a_i = a_{i,0} + a_{i,+}, \qquad a_{i,0} \in \C, \quad a_{i,+} \in \Lambda_+^{conv},
$$
and
$$
\widetilde P(X) = \sum_{i=0}^{n-1} (a_{i,0} + Z_i) X^i + X^n \in \C[Z_0,\cdots,Z_{n-1}][X],
$$
where $Z_i$ are formal variables.
\par
The {\it convergent} power series ring $\C\{Z_0,\cdots,Z_{n-1}\}$ is
Henselian (see section 45 \cite{Na62}). Therefore there exist monic polynomials
$\widetilde Q, \widetilde R \in \C\{Z_0,\cdots,Z_{n-1}\}[X]$ such that
$$
\widetilde Q\widetilde R = \widetilde P
$$
and the $\C$-reduction of $\widetilde Q$, $\widetilde R$ are $\overline Q$, $\overline R$ respectively.
On the other hand, it is easy to see that
$Z_i \mapsto a_{i,+}$ induces a continuous ring homomorphism
$\C\{Z_0,\cdots,Z_{n-1}\} \to \Lambda_0^{conv}$.
Thus $\widetilde Q$, $\widetilde R$ induce $Q,R \in \Lambda_0^{conv}[X]$ such that
$QR = P$. Hence $\Lambda_0^{conv}$ is Henselian, as required.
\end{proof}
\section{Appendix 2 : $T^n$-equivariant Kuranishi structure}
\label{sec:equivariant}

In this section we define the notion of $T^n$ equivariant
Kuranishi structure and prove that our moduli space
$\mathcal M_{k+1}^{\text{\rm main}}(\beta)$ has one.
We also show the existence of $T^n$ equivariant perturbation
of the Kuranishi map.
\par
Let $\mathcal M$ be a compact space with Kuranishi structure. The
space ${\mathcal M}$ is covered by a finite number of Kuranishi
charts $(V_{\alpha},E_{\alpha},
\Gamma_{\alpha},\psi_{\alpha},s_{\alpha})$, $\alpha \in \mathfrak A$
which satisfy the following :
\begin{conds}
\begin{enumerate}
\item $V_{\alpha}$ is a smooth manifold (with boundaries or corners) and
$\Gamma_{\alpha}$ is a finite group acting effectively on $V_{\alpha}$.
\item $\text{pr}_{\alpha} : E_{\alpha} \to V_{\alpha}$ is a finite dimensional vector bundle on which
$\Gamma_{\alpha}$ acts so that $\text{pr}_{\alpha}$ is $\Gamma_{\alpha}$- equivariant.
\item $s_{\alpha}$ is a $\Gamma_{\alpha}$ equivariant section of $E_{\alpha}$.
\item $\psi_{\alpha} : s_{\alpha}^{-1}(0)/\Gamma_{\alpha} \to \mathcal M$ is a homeomorphism
to its image.
\item
The union of $\psi_{\alpha}(s_{\alpha}^{-1}(0)/\Gamma_{\alpha})$ for various $\alpha$ is $\mathcal M$.
\end{enumerate}
\end{conds}
\par
We assume that
$\{(V_{\alpha},E_{\alpha},\Gamma_{\alpha},
\psi_{\alpha},s_{\alpha})\mid \alpha \in \frak A\}$ is a good coordinate system,
in the sense of Definition 6.1 \cite{FO} or Lemma A1.11 \cite{fooo06}
(= Lemma A1.11 \cite{fooo06pre}).
This means the following :
The set $\mathfrak A$ has a partial order $<$, where
either $\alpha_1 \le \alpha_2$ or $\alpha_2 \le \alpha_1$ holds
for $\alpha_1, \alpha_2 \in \mathfrak A$
if
$$
\psi_{\alpha_1}(s_{\alpha_1}^{-1}(0)/\Gamma_{\alpha_1}) \cap
\psi_{\alpha_2}(s_{\alpha_2}^{-1}(0)/\Gamma_{\alpha_2}) \ne \emptyset.
$$
Let $\alpha_1, \alpha_2 \in \mathfrak A$ and $\alpha_1 \le
\alpha_2$. Then, there exists a $\Gamma_{\alpha_1}$-invariant open
subset $V_{\alpha_2,\alpha_1} \subset V_{\alpha_1}$, a smooth
embedding
$
\varphi_{\alpha_2,\alpha_1} : V_{\alpha_2,\alpha_1}
\to V_{\alpha_2}
$
and a bundle map
$
\widehat{\varphi}_{\alpha_2,\alpha_1} : E_{\alpha_1}\vert_{V_{\alpha_2,\alpha_1}}
\to E_{\alpha_2}.
$
which covers $\varphi_{\alpha_2,\alpha_1}$.
Moreover there exists an injective homomorphism
$
\widehat{\widehat{\varphi}}_{\alpha_2,\alpha_1}
: \Gamma_{\alpha_1} \to \Gamma_{\alpha_2}.
$
We require that they satisfy the following
\begin{conds}\label{coordinatechange}
\begin{enumerate}
\item
The maps $\varphi_{\alpha_2,\alpha_1}$,
${\widehat{\varphi}}_{\alpha_2,\alpha_1}$ are
$\widehat{\widehat{\varphi}}_{\alpha_2,\alpha_1}$-equivariant.
\item
$\varphi_{\alpha_2,\alpha_1}$ and $\widehat{\widehat{\varphi}}_{\alpha_2,\alpha_1}$ induce
an embedding of orbifold
\begin{equation}\label{orbemb}
\overline{\varphi}_{\alpha_2,\alpha_1} :
\frac{V_{\alpha_2,\alpha_1}}{\Gamma_{\alpha_1}}
\to \frac{V_{\alpha_2}}{\Gamma_{\alpha_2}}.
\end{equation}
\item We have
$
s_{\alpha_2} \circ \varphi_{\alpha_2,\alpha_1}
= {\widehat{\varphi}}_{\alpha_2,\alpha_1} \circ s_{\alpha_1}.
$
\item We have
$
\psi_{\alpha_2}\circ \overline{\varphi}_{\alpha_2,\alpha_1}
= \psi_{\alpha_1}
$
on
$
\frac{V_{\alpha_2,\alpha_1} \cap s_{\alpha_1}^{-1}(0)}{\Gamma_{\alpha_1}}.
$
\item If $\alpha_1 < \alpha_2< \alpha_3$ then
$
\varphi_{\alpha_3,\alpha_2}\circ \varphi_{\alpha_2,\alpha_1}
= \varphi_{\alpha_3,\alpha_1},
$
on $\varphi_{\alpha_2,\alpha_1}^{-1}(V_{\alpha_3,\alpha_2})$.
$
{\widehat{\varphi}}_{\alpha_3,\alpha_2}\circ {\widehat{\varphi}}_{\alpha_2,\alpha_1}
= {\widehat{\varphi}}_{\alpha_3,\alpha_1},
$
and
$$
\widehat{\widehat{\varphi}}_{\alpha_3,\alpha_2}\circ \widehat{\widehat{\varphi}}_{\alpha_2,\alpha_1}
= \widehat{\widehat{\varphi}}_{\alpha_3,\alpha_1},
$$
hold in the similar sense.
\item
$V_{\alpha_2,\alpha_1}/\Gamma_{\alpha_1}$ contains
$\psi_{\alpha_1}^{-1}(\psi_{\alpha_1}(s_{\alpha_1}^{-1}(0)/\Gamma_{\alpha_1}) \cap
\psi_{\alpha_2}(s_{\alpha_2}^{-1}(0)/\Gamma_{\alpha_2}))$.
\end{enumerate}
\end{conds}
\par
\begin{conds}\label{normalcompatibility}
$\mathcal M$ has a tangent bundle : i.e., the differential of
$s_{\alpha_2}$ in the direction of the normal bundle induces a
bundle isomorphism
\begin{equation}\label{tangent}
ds_{\alpha_2} : \frac{{\varphi}_{\alpha_2,\alpha_1}^*TV_{\alpha_2}}{TV_{\alpha_2,\alpha_1}}
\to \frac{\widehat{\varphi}_{\alpha_2,\alpha_1}^*E_{\alpha_2}}{E_{\alpha_1}}.
\nonumber\end{equation}
We say $\mathcal M$ is oriented if $V_{\alpha}$, $E_{\alpha}$ is oriented,
the $\Gamma_{\alpha}$ action is orientation preserving, and
$ds_{\alpha}$ is orientation preserving.
\end{conds}
\par
\begin{defn}\label{eqkurastr}
Suppose that $\mathcal M$ has a $T^n$ action. We say our Kuranishi
structure on $\mathcal M$ is {\it $T^n$ equivariant in the strong
sense} if the following holds :
\begin{enumerate}
\item $V_{\alpha}$ has a $T^n$ action which commutes
with the given $\Gamma_{\alpha}$ action.
\item $E_{\alpha}$
is a $T^n$ equivariant bundle.
\item
The Kuranishi map $s_{\alpha}$
is $T^n$ equivariant and $\psi_{\alpha}$ is a $T^n$ equivariant map.
\item
The coordinate changes $\varphi_{\alpha_2,\alpha_1}$ and
$\varphi_{\alpha_2,\alpha_1}$ are $T^n$ equivariant.
\end{enumerate}
\end{defn}
\begin{rem}
We remark that Condition (1) above is more restrictive than
the condition that the orbifold $V_{\alpha}/\Gamma_{\alpha}$
has a $T^n$ action. This is the reason why we use the
phrase {\it in the strong sense} in the above definition.
Hereafter we will say $T^n$ equivariant instead for simplicity.
\end{rem}
Let $L$ be a smooth manifold.
A strongly continuous smooth map $ev : \mathcal M \to L$ is a
family of $\Gamma_{\alpha}$ invariant smooth maps
\begin{equation}\label{evalpha}
ev_{\alpha} : V_{\alpha} \to L
\end{equation}
which induce
$
\overline{ev}_{\alpha} : V_{\alpha}/\Gamma_{\alpha} \to L
$
such that
$
\overline{ev}_{\alpha_2} \circ \overline{\varphi}_{\alpha_2,\alpha_1}
= \overline{ev}_{\alpha_1}
$
on $V_{\alpha_2,\alpha_1}/\Gamma_{\alpha}$. (Note $\Gamma_{\alpha}$
action on $L$ is trivial.)
\par
We say that $ev$ is weakly submersive
if each of ${ev}_{\alpha}$ in (\ref{evalpha}) is
a submersion.
\begin{defn}
Assume that there exists $T^n$ actions on $L$ and $\mathcal M$. We
say that $ev : \mathcal M \to L$ is $T^n$ equivariant if
${ev}_{\alpha}$ in (\ref{evalpha}) is $T^n$ equivariant.
\end{defn}
Now we show
\begin{prop}\label{equikurast}
The moduli space $\mathcal M_{k+1}(\beta)$ has a $T^n$
equivariant Kuranishi structure such that $ev_0 : \mathcal M_{k+1}(\beta)
\to L$ is $T^n$ equivariant strongly continuous weakly
submersive map.
\end{prop}
\begin{proof}[Proof]
Except the $T^n$ equivariance this is proved in 
section 7.1 \cite{fooo06} (= section 29
\cite{fooo06pre}). Below we explain how we choose our Kuranishi
structure so that it is $T^n$ equivariant. We also include the case
of the moduli space $\mathcal M_{k+1,\ell}(\beta)$ with interior
marked points.
\par
We first review the construction of the Kuranishi neighborhood from
section 7.1 \cite{fooo06} (= section 29
\cite{fooo06pre}).
\par
Let $\text{\bf x} =((\Sigma,\vec z),w) \in \mathcal M_{k+1,\ell}(\beta)$.
Let $\Sigma_a$ be an irreducible component of $\Sigma$. We consider
the operator
\begin{equation}\label{linearlized}
D_{w,a}\overline{\partial} : W^{1,p}(\Sigma_a;w^*(TX);L,\vec z_a) \to
W^{0,p}(\Sigma_a;w^*(TX) \otimes \Lambda^{0,1}).
\end{equation}
Here $W^{1,p}(\Sigma_a;w^*(TX);L,\vec z_a)$
is the space of section $v$ of $w^*(TX)$
of $W^{1,p}$ class with the following properties.
\begin{enumerate}
\item The restriction of $v$ to $\partial \Sigma_a$ is in $w^*(TL)$.
\item $\vec z_a$ is the set of $\Sigma$ which is either singular or marked.
We assume that $v$ is $0$ at those points.
\end{enumerate}
$\Lambda^{0,1}$ is the bundle of $(0,1)$ forms on $\Sigma_a$ and
$W^{0,p}(\Sigma_a;w^*(TX) \otimes \Lambda^{0,1})$ is the set of
sections of $W^{0,p}$ class of $w^*(TX) \otimes \Lambda^{0,1}$.
$D_{w,a}\overline{\partial}$ is the linearization of (nonlinear)
Cauchy-Riemann equation. (See \cite{floer:morse}.) The operator
(\ref{linearlized}) is Fredholm by ellipticity thereof.
\par
We choose open subsets $W_a$ of $\Sigma_a$ whose closure is disjoint
from the boundary of each of the irreducible component $\Sigma_a$ of
$\Sigma$ and from the singular points and marked points. By the unique continuation theorem, we can choose a finite
dimensional subset $E_{0,a}$ of $C_0^{\infty}(W_a;w^{*}TX)$ (the set
of smooth sections of compact support in $W_a$) such that
$$
\text{\rm Im}D_{w,a}\overline{\partial} + E_{0,a}
= W^{0,p}(\Sigma_a;w^*(TX) \otimes \Lambda^{0,1}).
$$
When $\text{\bf x}$ has nontrivial automorphisms, we choose
$\bigoplus_a E_{0,a}$ to be invariant under the automorphisms.
\par
We next associate a finite dimensional subspace
$E_{0,a}((\Sigma,\vec z),w')$ for $w'$ which is `$C^0$ close to $w$.' We need
some care to handle the case where some component $(\Sigma_a,\vec
z_a)$ is not stable, that is the case for which the automorphism
group $G_a$ of $(\Sigma_a,\vec z_a)$ is of positive dimension. (Note
$G_a$ is different from the automorphism group of $((\Sigma_a,\vec
z_a),w\vert_{\Sigma_a})$. The latter group is finite.) We explain
this choice of $E_{0,1}$ below following Appendix \cite{FO}.
\par
For each unstable components $\Sigma_a$,
pick points
$z_{a,i} \in \Sigma$, $(i=1,\cdots,d_a)$ with
the following properties.
\begin{conds}\label{addmarked}
\begin{enumerate}
\item
$w$ is immersed at $z_{a,i}$.
\item $z_{a,i}$
in the interior of $\Sigma_a$,
$z_{a,i} \ne z_{a,j}$ for $i\ne j$ and
$z_{a,i} \notin \vec z$.
\item
$(\Sigma_a;(\vec z\cap\Sigma_a)\cup (z_{a,1},
\cdots,z_{a,d_a}))$
is stable.
\item
Let
$\Gamma = \Gamma(\text{\bf x})$ be the group of automorphisms of
$\text{\bf x} =((\Sigma,\vec z),w)$.
If $\gamma \in \Gamma$, $\gamma(\Sigma_a) = \Sigma_{a'}$
then $d_a = d_{a'}$  and
$\{\gamma(z_{a,i})\mid i=1,\cdots,d_a\}
= \{z_{a',i'}\mid i'=1,\cdots,d_a\}$.
\end{enumerate}
\end{conds}
For each $a,i$ we choose a submanifold $N_{a,i}$ of
codimension $2$ in $X$ that transversely intersects with
$(\Sigma_a,w)$ at $w_a(z_{a,i})$.
In relation to Condition \ref{addmarked} (4) above we choose
$N_{a,i} = N_{a',i'}$ if $\gamma(z_{a,i}) = z_{a',i'}$.
\par
We add extra interior marked points
$\{z_{a,i} \mid a,i\}$ in addition to $\vec z$ on $(\Sigma,\vec z)$, and obtain a stable curve
$(\Sigma,\vec z^+)$. (Namely $\vec z^+ = \vec z \sqcup \{z_{a,i} \mid a,i\}$.) (We choose an order of the added marked points and fix it.)
\par
We consider a neighborhood $\frak U_0$ of $[\Sigma,\vec z^+]$ in
$\mathcal M_{k+1,\ell'}^{\text{\rm main}}$, that is the moduli space
of bordered stable  curve of genus $0$ with $k+1$ boundary and
$\ell'$ interior marked points. Let $\Gamma_0$ be the group
of automorphisms of $((\Sigma,\vec z^+),w)$. Now both
$\Gamma$ and $\Gamma_0$ are finite groups and $\Gamma \supseteq
\Gamma_0$.
\par
An element $\gamma \in \Gamma$ induces an
automorphism $\gamma : \Sigma \to \Sigma$
which fixes marked points in $\vec z$ and permutes
$\ell - \ell'$ marked points $\{z_{a,i}\mid a,i \}$,
by Condition \ref{addmarked} (4).
Therefore we obtain an element of
$[\gamma_*(\Sigma,\vec z^+)]$ that is different from
$[\Sigma,\vec z^+]$ only by reordering of marked points.
We take the union of neighborhoods of $[\gamma_*(\Sigma,\vec z^+)]$
for various $\gamma\in \Gamma$ in
$\mathcal M_{k+1,\ell'}^{\text{\rm main}}$  and denote it by
$\frak U$.
\par
$\frak U_0$ is written as $\frak V_0/\Gamma_0$ where $\frak V_0$ is a
manifold. Moreover
there exists a manifold $\frak V$ on which $\Gamma$ acts
such that $\frak V/\Gamma_0 = \frak U$, $\frak V/\Gamma \cong \frak U_0$.
\par
Furthermore there is a `universal family' $\frak M \to \frak V$
where the fiber $\Sigma(\text{\bf v})$ of $\text{\bf v} \in \frak V$
is identified with the bordered marked stable curve that represents
the element $[\text{\bf v}] \in \frak V/\Gamma \subset \mathcal
M_{k+1,\ell'}^{\text{\rm main}}$. There is a $\Gamma$ action on
$\frak M$ such that $\frak M \to \frak V$ is $\Gamma$ equivariant.
\par
By the construction, each member $\Sigma(\text{\bf v})$ of our
universal family is diffeomorphic to $\Sigma$ away from singularity.
More precisely we have the following :
\par
Let $\Sigma_0 = \Sigma \setminus S$ where $S$ is a small
neighborhood of the union of the singular point sets and the marked
point sets. Then for each $\text{\bf v}$ there exists a smooth
embedding $i_{\text{\bf v}} : \Sigma_0 \to \Sigma(\text{\bf v})$.
The error of this embedding for becoming a biholomorphic map goes to
$0$ as $\text{\bf v}$ goes to $0$. We may assume $\text{\bf v}
\mapsto i_{\text{\bf v}}$ is $\Gamma$ invariant in an obvious sense
and $i_{\text{\bf v}}$ depends smoothly on $\text{\bf v}$.
\par
We may choose $W_a$ so
that $W_a \subset \Sigma_0$ for each $a$. We moreover
assume that $\bigoplus_a E_{0,a}$ is invariant under the $\Gamma$ action in the following sense : if
$\gamma \in \Gamma$ and $\Sigma_{a'}= \gamma(\Sigma_a)$ then
the isomorphism induced by $\gamma$ sends $E_{0,a}$
to $E_{0,a'}$.
\par
Now we consider a pair $(w',\text{\bf v})$ where
$$
w' : (\Sigma(\text{\bf v}),\partial \Sigma(\text{\bf v})) \to (X,L).
$$
We assume :
\begin{conds}\label{condnbh}
There exists $\epsilon > 0$ depending only on $\text{\bf x}$ such that the
following holds.
\begin{enumerate}
\item $\sup_{x \in \Sigma_0} \text{\rm dist} (w'(i_{\text{\bf v}}(x)),w(x))
\le \epsilon$.
\item
For any connected component $D_c$ of $\Sigma(\text{\bf v}) \setminus \text{\rm Im}(
i_{\text{\bf v}})$, the diameter of $w'(D_c)$
in $X$ (with a fixed Riemannian metric on it)
is smaller than $\epsilon$.
\end{enumerate}
\end{conds}
For each point $x \in W_a$ we use the parallel transport to make the
identification
$$
T_{w(x)} X \otimes \Lambda^{0,1}_x(\Sigma)
\equiv
T_{w'(i_{\text{\bf v}}(x))} X \otimes \Lambda^{0,1}_{i_{\text{\bf v}}(x)}(\Sigma(\text{\bf v})).
$$
Using this identification we obtain an embedding
$$
I_{(\text{\bf v},w')} : \bigoplus_a E_{0,a}
\longrightarrow w^{\prime *}(TX) \otimes \Lambda^{0,1}
(\Sigma(\text{\bf v})).
$$
Via this embedding $I_{(\text{\bf v},w')}$, we consider the
equation
\begin{equation}\label{perturbedeq}
\overline{\partial} w' \equiv 0
\mod \bigoplus_a I_{(\text{\bf v},w')}(E_{0,a}),
\end{equation}
together with the additional conditions
\begin{equation}\label{markedadd}
w'(z_{a,i}) \in N_{a,i}.
\end{equation}
Let $V(\text{\bf x})$ be the set of solution of (\ref{perturbedeq}),
(\ref{markedadd}). This is a smooth manifold (with boundary and corners) by the implicit
function theorem and a gluing argument. (See \cite{fooo06} section A1.4 
= \cite{fooo06pre} section A1.4 for the smoothness
at boundary and corner.) Since we can make all the construction above
invariant under the $\Gamma(\text{\bf x})$ action, the space
$V(\text{\bf x})$ has a $\Gamma(\text{\bf x})$ action. (Note that we
may choose $N_{a,i}$ so that $\{N_{a,i} \mid a,i \}$ is invariant
under the action of $\Gamma(\text{\bf x})$.)
\par
The obstruction bundle $E$ is the space $\bigoplus_a E_{0,a}$
at $\text{\bf x}$ and $\bigoplus_a I_{(\text{\bf v},w')}(E_{0,a})$ at
in $(\text{\bf v},w')$.
\par
We omit the construction of coordinate changes. See 
 \cite{fooo06} section 7.1 (= \cite{fooo06pre}
section 29) which in turn is similar to section 15 \cite{FO}.
\par
The Kuranishi map is given by
\begin{equation}\label{kuramap}
((\Sigma',\vec z'),w') \mapsto
\overline{\partial} w' \in \bigoplus_a E_{0,a}.
\end{equation}
We have thus finished our review of the construction of Kuranishi
charts.
\par
Now we assume that $X$ has a $T^n$ action, which preserves both the
complex and the symplectic structures on $X$. We also assume that
$L$ is a $T^n$ orbit.
\par
We want to construct a family of vector spaces $(\text{\bf v},w')
\mapsto \bigoplus_a I_{(\text{\bf v},w')}(E_{0,a})$ so that it is
invariant under the $T^n$ action. We need to slightly modify the
above construction for this purpose. In fact it is not totally
obvious to make the condition (\ref{markedadd}) $T^n$-invariant.

For this purpose we proceed in the following way. We fix a point
$p_0 \in L$ and consider an element
$$
\text{\bf x} \in ev_0^{-1}(p_0) \cap \mathcal M_{k+1,\ell}(\beta).
$$
We are going to construct a $T^n$ equivariant Kuranishi neighborhood
of the $T^n$ orbit of $\text{\bf x}$.
Let $\text{\bf v} \in \frak V$ and $w' : (\Sigma(\text{\bf v}),
\partial \Sigma(\text{\bf v})) \to (X,L)$ be as before.
We replace Condition \ref{condnbh} by the following.

\begin{conds} Let $z_0$ be the 0-th boundary marked point and let $g \in T^n$ be
the unique element satisfying $w'(z_0) = g(p_0)$.
There exists $\epsilon > 0$ depending only on $\text{\bf x}$ such that the
following holds.
\begin{enumerate}
\item $\sup_{x \in \Sigma_0} \operatorname{dist} (w'(i_{\text{\bf v}}(x)),g(w(x)))
\le \epsilon$.
\item
For any connected component $D_c$ of $\Sigma(\text{\bf v}) \setminus \text{\rm Im}(
i_{\text{\bf v}})$, the diameter of $w'(D_c)$
in $X$ (with a fixed Riemannian metric on it)
is smaller than $\epsilon$.
\end{enumerate}
\end{conds}
Now we define an embedding
$$
I_{(\text{\bf v},w')} : \bigoplus_a E_{0,a}
\longrightarrow w^{\prime *}(TX) \otimes \Lambda^{0,1}
(\Sigma(\text{\bf v}))
$$
as follows : We first use the $g$ action to define an isomorphism
$$
g_* : T_{w(x)} X \otimes \Lambda^{0,1}_x(\Sigma)
\cong T_{g(w(x))} X \otimes \Lambda^{0,1}_x(\Sigma).
$$
Then we use the parallel transport in the same way as before to
define
$$
\bigoplus_a g_*(E_{0,a})
\longrightarrow w^{\prime *}(TX) \otimes \Lambda^{0,1}
(\Sigma(\text{\bf v})).
$$
By composing the two we obtain the embedding $I_{(\text{\bf
v},w')}$. Now we consider the equation
\begin{equation}\label{perturbedeqeq}
\overline{\partial} w' \equiv 0
\mod \bigoplus_a I_{(\text{\bf v},w')}(E_{0,a}),
\end{equation}
together with the additional conditions
\begin{equation}\label{markedaddeq}
w'(z_{a,i}) \in g(N_{a,i})
\end{equation}
as before. Clearly these equations are $T^n$ invariant.
It is also $\Gamma(\text{\bf x})$ invariant.
\par
Note the action of the automorphism group $\Gamma(\text{\bf x})$ of
$\text{\bf x}$ of our Kuranishi structure, which is a finite group,
acts on the source while $T^n$ acts on the target. Therefore it is
obvious that two actions commute.
\par
By definition the obstruction bundle has a $T^n$ action. Moreover
(\ref{kuramap}) is $T^n$ equivariant. It is fairly obvious from the
construction that coordinate changes of the constructed Kuranishi
structure is also $T^n$ equivariant.
\par
The proof of Proposition \ref{equikurast} is now complete.
\end{proof}
\begin{rem}
\begin{enumerate}
\item From the above construction, it is easy to see that our
$T^n$ equivariant Kuranishi structure is compatible with the gluing
at the boundary marked points. Namely under the embedding
$$
\mathcal M_{k_1+1}(\beta_1) {}_{ev_0}\times_{ev_i}
\mathcal M_{k_2+1}(\beta_2)
\subset
\mathcal M_{k_1+k_2}(\beta_1+\beta_2)
$$
the restriction of the Kuranishi structure of the right hand side
coincides with the fiber product of the Kuranishi structure of the
left hand side. More precisely speaking, we can construct the system
of our Kuranishi structures so that this statement holds : this
Kuranishi structure is constructed inductively over the number of
disc components and the energy of $\beta$.
\item
On the other hand, when we construct the $T^n$ invariant
Kuranishi structure in the way we described above it may not
be compatible with the gluing at the interior marked point. Namely
by the embedding
\begin{equation}\label{eq:interiorbubble2}
\mathcal M_1(\alpha) \times_X \mathcal M_{k+1,1}(\beta) \subset \mathcal M_{k+1}(\beta+\alpha)
\end{equation}
the restriction of the Kuranishi structure of the right hand sides
may not coincide with the fiber product of the Kuranishi structure
of the left hand side. This compatibility is not used in this paper
and hence not required. See Remark \ref{sphcompatible} where similar
point is discussed for the choice of multisections.
\par
However contrary to the choice of multisections mentioned in Remark
\ref{sphcompatible}, we remark that it is possible to construct
$T^n$ equivariant Kuranishi structure compatible with
(\ref{eq:interiorbubble2}). Since we do not use this point in the
paper, we will not elaborate it here.
\end{enumerate}
\end{rem}
We next review on the multisections. (See section 3 \cite{FO}.) Let
$(V_{\alpha},E_{\alpha},\Gamma_{\alpha},\psi_{\alpha},s_{\alpha})$
be a Kuranishi chart of $\mathcal M$. For $x \in V_{\alpha}$ we
consider the fiber $E_{\alpha,x}$ of the bundle $E_{\alpha}$ at $x$.
We take its $l$ copies and consider the direct product $
E_{\alpha,x}^{l} $. We take the quotient thereof by the action of
symmetric group of order $l!$ and let $\mathcal S^l(E_{\alpha,x})$
be the quotient space. There exists a map
$
tm_m : \mathcal S^l(E_{\alpha,x})
\to \mathcal S^{lm}(E_{\alpha,x}),
$
which sends $[a_1,\cdots,a_l]$ to
$
[\,\underbrace {a_1,\cdots,a_1}_{\text{$m$ copies}},\cdots,
\underbrace {a_l,\cdots,a_l}_{\text{$m$ copies}}].
$
A smooth {\it multisection} $s$ of the bundle
$
E_{\alpha} \to V_{\alpha}
$
consists of an open covering
$
\bigcup_iU_i = V_{\alpha}
$
and $s_i$ which maps $x\in U_i$ to $s_i(x) \in \mathcal
S^{l_i}(E_{\alpha,x})$. They are required to have the following
properties :
\begin{conds}
\begin{enumerate}
\item $U_i$ is $\Gamma_{\alpha}$-invariant. $s_i$ is $\Gamma_{\alpha}$-equivariant.
(We remark that there exists an obvious map
$
\gamma : \mathcal S^{l_i}(E_{\alpha,x}) \to \mathcal S^{l_i}(E_{\alpha,\gamma x})
$
for each $\gamma \in \Gamma_{\alpha}$.)
\item If $x \in U_i \cap U_j$ then we have
$$
tm_{l_j}(s_i(x)) = tm_{l_i}(s_j(x)) \in \mathcal S^{l_il_j}(E_{\alpha,\gamma x}).
$$
\item $s_i$ is {\it liftable and smooth} in the following sense.
For each $x$ there exists a smooth section $\tilde s_i$ of
$\underbrace{E_{\alpha} \oplus \cdots \oplus E_{\alpha}}_{\text{$l_i$ times}}$
in a neighborhood of $x$ such that
\begin{equation}\label{locallift}
\tilde s_i(y) = (s_{i,1}(y),\cdots,s_{i,l_i}(y)),
\quad
s_i(y) = [s_{i,1}(y),\cdots,s_{i,l_i}(y)].
\end{equation}
\end{enumerate}
\end{conds}
We identify two multisections $(\{U_i\},\{s_i\},\{l_i\})$, $(\{U'_i\},\{s'_i\},\{l'_i\})$
if
$$
tm_{l_i}(s_i(x)) = tm_{l'_j}(s'_j(x)) \in \mathcal S^{l_il'_j}(E_{\alpha,\gamma x})
$$
on $U_i \cap U'_j$.
We say $s_{i,j}$ to be a {\it branch} of $s_i$ in the situation of (\ref{locallift}).
\par

We next prove the following lemma which we used in section \ref{sec:calcpot}.
\begin{lem}\label{quotientkura}
Let $\mathcal M$ have a $T^n$ action and a $T^n$
equivariant Kuranishi structure. Suppose that the $T^n$ action on each
of the Kuranishi neighborhood is free. Then we can descend the
Kuranishi structure to $\mathcal M/T^n$ in a canonical way.
\end{lem}
\begin{proof}
Let $(V_{\alpha},E_{\alpha},\Gamma_{\alpha},
\psi_{\alpha},s_{\alpha})$ be a Kuranishi chart. Since $T^n$ action
on $V_{\alpha}$ is free $V_{\alpha}/T^n$ is a smooth manifold and
$E_{\alpha}/T^n \to V_{\alpha}/T^n$ is a vector bundle. Since
$\Gamma_{\alpha}$ action commutes with $T^n$ action, it follows that
it acts on this vector bundle. The $T^n$ equivariance of
$s_{\alpha}$ implies that we have a section $\overline s_{\alpha}$
of $E_{\alpha}/T^n \to V_{\alpha}/T^n$. Moreover since
$\psi_{\alpha} : s_{\alpha}^{-1}(0) \to \mathcal M$ is $T^n$
equivariant it induces a map $\overline s_{\alpha}^{-1}(0) \to
\mathcal M/T^n$. Thus we obtain a Kuranishi chart. It is easy to
define the coordinate changes.
\end{proof}
\begin{defn}\label{equisec}
In the situation of Proposition \ref{quotientkura}, we say that a
system of multisections of the Kuranishi structure of {\it $\mathcal
M$ is $T^n$ equivariant}, if it is induced from the multisection of
the Kuranishi structure on $\mathcal M/T^n$ in an obvious way.
\end{defn}
\begin{cor}\label{equiperturb}
In the situation of Proposition $\ref{quotientkura}$,
we assume that we have a $T^n$ equivariant multisection
at the boundary of $\mathcal M$, which is transversal to $0$.
Then it extends to a $T^n$ equivariant multisection of
$\mathcal M$.
\end{cor}
This is an immediate consequence of Lemma 3.14 \cite{FO}, that is
the non-equivariant version.
\section{Appendix 3 : Smooth correspondence via the zero set of
multisection}
\label{sec:integration}
In this section we explain the way how we use the zero
set of multisection to define smooth correspondence,
when appropriate submersive properties are satisfied.
\par
Such a construction is a special case of the techniques of using a continuous
family of multisections and integration along the fiber on their
zero sets so that smooth correspondence by spaces with Kuranishi
structure induces a map between de Rham complex.
\par
This (more general) technique is not new and is known to some
experts. In fact \cite{Rua99}, section 16 \cite{fukaya;abel} use a
similar technique and section 7.5  \cite{fooo06} (= section 33 \cite{fooo06pre}), \cite{fukaya:loop},
\cite{fukaya;operad} contain the details of this more general
technique.
\par
We explain the special case (namely the case we use a single
multisection) in our situation of toric manifold, for the sake of
completeness and of reader's convenience.
\par
Let $\mathcal M$ be a space with Kuranishi structure and $ev_s :
\mathcal M \to L_s$, $ev_t : \mathcal M \to L_t$ be strongly
continuous smooth maps. (Here $s$ and $t$ stand for source and target,
respectively.) We assume our smooth manifolds $L_s, L_t$ are compact
and oriented without boundary. We also assume $\mathcal M$ has a
tangent bundle and is oriented in the sense of Kuranishi structure.
\par
Suppose that $L_t$ has a free and transitive $T^n$ action, and $L_s$
and $\mathcal M$ have $T^n$ action. We assume that the Kuranishi
structure on $\mathcal M$ is $T^n$ equivariant and the maps $ev_s$,
$ev_t$ are $T^n$ equivariant.
\par
Let $\{(V_{\alpha},E_{\alpha},\Gamma_{\alpha},
\psi_{\alpha},s_{\alpha})\}$ be a $T^n$ equivariant Kuranishi coordinate system
(good coordinate system) of $\mathcal M$.
We use Corollary \ref{equiperturb} to find a
$T^n$ equivariant (system) of multisection $\frak s_{\alpha} :
V_{\alpha} \to E_{\alpha}$ which is transversal to $0$.
\par
Let $\theta_{\alpha}$ be a smooth differential form of compact
support on $V_{\alpha}$. We assume that $\theta_{\alpha}$ is
$\Gamma_{\alpha}$-invariant. Let $f_{\alpha} : V_{\alpha} \to L_s$ be
a $\Gamma_{\alpha}$ equivariant submersion. (The $\Gamma_{\alpha}$
action on $L_s$ is trivial.) We next define {\it integration along the fiber}
$$
((V_{\alpha},E_{\alpha},\Gamma_{\alpha},\psi_{\alpha},s_{\alpha}),
\frak s_{\alpha},f_{\alpha})_* (\theta_{\alpha})
\in \Omega^{\deg \theta_{\alpha} + \dim L_t - \dim \mathcal M}(L_t).
$$
We first fix $\alpha$.
Let $(U_i,\frak s_{\alpha,i})$ be a representative of $\frak s_{\alpha}$. Namely $\{U_i\mid i\in I\}$ is an open covering of
$V_{\alpha}$ and $\frak s_{\alpha}$ is represented by
$\frak s_{\alpha,i}$ on $U_i$. By the definition of the
multisection, $U_i$ is $\Gamma_{\alpha}$-invariant. We may shrink
$U_i$, if necessary, so that there exists a lifting $\tilde{\frak
s}_{\alpha,i} = (\tilde{\frak s}_{\alpha,i,1},\cdots,\tilde{\frak
s}_{\alpha,i,l_i})$ as in (\ref{locallift}).
\par
Let $\{\chi_i \mid i \in I\}$ be a partition of unity subordinate to
the covering $\{U_i\mid i\in I\}$. By replacing $\chi_i$ with its
average over $\Gamma_{\alpha}$ we may assume $\chi_i$ is
$\Gamma_{\alpha}$-invariant.
\par
We put
\begin{equation}
\tilde{\frak s}_{\alpha,i,j}^{-1}(0)
= \{x \in U_i \mid \tilde{\frak s}_{\alpha,i,j}(x) = 0\}.
\end{equation}
By assumption $\tilde{\frak s}_{\alpha,i,j}^{-1}(0)$ is a smooth manifold.
Since the $T^n$ action on $L_t$ is free and transitive it follows that
\begin{equation}\label{restrictsubmersion}
ev_{t,\alpha}\vert_{\tilde{\frak s}_{\alpha,i,j}^{-1}(0)} :
\tilde{\frak s}_{\alpha,i,j}^{-1}(0) \to L_t
\end{equation}
is a submersion.
\begin{defn}
We define a differential form on $L_t$ by
\begin{equation}\label{intfibermainformula}
\aligned
&((V_{\alpha},E_{\alpha},\Gamma_{\alpha},\psi_{\alpha},s_{\alpha}),
\frak s_{\alpha},ev_{t,\alpha})_* (\theta_{\alpha})\\
&=
\frac{1}{\#\Gamma_{\alpha}}\sum_{i=1}^I\sum_{j=1}^{l_i} \frac{1}{l_i}
(ev_{t,\alpha})_!\left(\chi_i \theta_{\alpha}\vert_{\tilde{\frak s}_{\alpha,i,j}^{-1}(0)}\right).
\endaligned
\end{equation}
Here $(ev_{t,\alpha})_!$ is the
integration along fiber of the smooth submersion (\ref{restrictsubmersion}).
\end{defn}
\begin{lem}
The right hand side of $(\ref{intfibermainformula})$ depends only on
$(V_{\alpha},E_{\alpha},\Gamma_{\alpha},\psi_{\alpha},s_{\alpha})$,
$\frak s_{\alpha}$, $ev_{t,\alpha}$,
and $\theta_{\alpha}$ but independent of the following choices :
\begin{enumerate}
\item The choice of representatives $(\{U_i\},\frak s_{\alpha,i})$ of $\frak s_{\alpha}$.
\item
The partition of unity $\chi_i$.
\end{enumerate}
\end{lem}
\begin{proof}
The proof is straightforward generalization of the proof of
well-definedness of integration on manifold, which can be found in
the text book of manifold theory, and is left to the leader.
\end{proof}
\par
So far we have been working on one Kuranishi chart
$(V_{\alpha},E_{\alpha},\Gamma_{\alpha},\psi_{\alpha},s_{\alpha})$.
We next describe the compatibility conditions of multisections for
various $\alpha$. During the construction we need to shrink
$V_{\alpha}$ a bit several times. We will not explicitly mention
this point henceforth.
\par
Let $\alpha_1 < \alpha_2$.
For $\alpha_1 < \alpha_2$, we take the normal
bundle $N_{V_{\alpha_2,\alpha_1}}V_{\alpha_2}$ of
$\varphi_{\alpha_2,\alpha_1}(V_{\alpha_2,\alpha_1})$ in
$V_{\alpha_2}$.
We use an appropriate $\Gamma_{\alpha_2}$ invariant Riemannian
metric on $V_{\alpha_2}$ to define
the exponential map
\begin{equation}\label{normalexp}
\text{Exp}_{\alpha_2,\alpha_1} :
B_{\e}N_{V_{\alpha_2,\alpha_1}}V_{\alpha_2} \to V_2.
\end{equation}
(Here $B_{\e}N_{V_{\alpha_2,\alpha_1}}V_{\alpha_2}$ is the $\epsilon$ neighborhood of the
zero section of $N_{V_{\alpha_2,\alpha_1}}V_{\alpha_2}$.)
\par
Using $\text{Exp}_{\alpha_2,\alpha_1}$, we identify
$B_{\e}N_{V_{\alpha_2,\alpha_1}}V_{\alpha_2}/{\Gamma_{\alpha_1}}
$  to an open subset of $V_{\alpha_1}/{\Gamma_{\alpha_1}}$
and denote it by $U_{\epsilon}({V_{\alpha_2,\alpha_1}}/{\Gamma_{\alpha_1}})$.
\par
Using the projection
$$
\text{\rm Pr}_{V_{\alpha_2,\alpha_1}} :
U_{\epsilon}({V_{\alpha_2,\alpha_1}}/{\Gamma_{\alpha_1}}) \to {V_{\alpha_2,\alpha_1}}/{\Gamma_{\alpha_1}}
$$
we extend the orbibundle $E_{\alpha_1}$ to $U_{\epsilon}({V_{\alpha_2,\alpha_1}}/{\Gamma_{\alpha_1}})$.
Also we extend the embedding $E_{\alpha_1} \to {\widehat{\varphi}_{\alpha_2,\alpha_1}^*E_{\alpha_2}}$,
(which is induced by $\widehat{\varphi}_{\alpha_2,\alpha_1}$) to
$U_{\epsilon}({V_{\alpha_2,\alpha_1}}/{\Gamma_{\alpha_1}})$.
\par
We fix a $\Gamma_{\alpha}$-invariant inner product of the bundles
$E_{\alpha}$. We then have a bundle isomorphism
\begin{equation}\label{bdliso}
E_{\alpha_2} \cong E_{\alpha_1} \oplus \frac{\widehat{\varphi}_{\alpha_2,\alpha_1}^*E_{\alpha_2}}{E_{\alpha_1}}
\end{equation}
on $U_\e(V_{\alpha_2,\alpha_1}/\Gamma_{\alpha_1})$. We can use
Condition \ref{normalcompatibility} to modify
$\text{Exp}_{\alpha_2,\alpha_1}$ in (\ref{normalexp}) so that the
following is satisfied.
\begin{conds}\label{normalcompcond}
If $y = \text{Exp}_{\alpha_2,\alpha_1}(\tilde y) \in U_{\epsilon}({V_{\alpha_2,\alpha_1}}/{\Gamma_{\alpha_1}})$ then
\begin{equation}\label{compatiformula}
ds_{\alpha_2}(\tilde y \mod TV_{\alpha_1}) \equiv s_{\alpha_2}(y) \mod E_{\alpha_1}.
\end{equation}
\end{conds}
Let us explain the notation of (\ref{compatiformula}).
We remark that $\tilde y \in T_{{\varphi}_{\alpha_2,\alpha_1}(x)}V_{\alpha_2}$ for $x =
\text{\rm Pr}(\tilde y) \in V_{\alpha_2,\alpha_1}$. Hence
$$
\tilde y \mod TV_{\alpha_1} \in \frac{T_{{\varphi}_{\alpha_2,\alpha_1}(x)}V_{\alpha_2}}{T_xV_{\alpha_1}}.
$$
Therefore
$$
ds_{\alpha_2}(\tilde y \mod TV_{\alpha_1})
\in \frac{(E_{\alpha_2})_{{\varphi}_{\alpha_2,\alpha_1}(x)}}{(E_{\alpha_1})_{x}}.
$$
(\ref{compatiformula}) claims that it coincides with $s_{\alpha_2}$ modulo $(E_{\alpha_1})_{x}$.
\par
We remark that Condition \ref{normalcompatibility} implies that
$$
[\tilde y] \mapsto ds_{\alpha_2}(\tilde y \mod TV_{\alpha_1}) :
\frac{T_{{\varphi}_{\alpha_2,\alpha_1}(x)}V_{\alpha_2}}{T_xV_{\alpha_1}}
\longrightarrow \frac{(E_{\alpha_2})_{{\varphi}_{\alpha_2,\alpha_1}(x)}}{(E_{\alpha_1})_{x}}
$$
is an isomorphism.
Therefore we can use the implicit function theorem to modify
$\text{Exp}_{\alpha_2,\alpha_1}$ so that Condition
\ref{normalcompcond} holds.
\begin{defn}\label{compdefini}
A multisection $\frak s_{\alpha_2}$ of $V_{\alpha_2}$
is said to be {\it compatible} with $\frak s_{\alpha_1}$ if
the following holds for each
$y = \text{Exp}_{\alpha_2,\alpha_1}(\tilde y) \in U_{\epsilon}({V_{\alpha_2,\alpha_1}}/{\Gamma_{\alpha_1}})$.
\begin{equation}\label{compatibilityforalpha}
\frak s_{\alpha_2}(y)
= \frak s_{\alpha_1}(\text{\rm Pr}(\tilde y)) \oplus ds_{\alpha_2}(\tilde y \mod TV_{\alpha_1}).
\end{equation}
\end{defn}
We remark that $\frak s_{\alpha_1}(w,\text{\rm Pr}(\tilde y))$ is a
multisection of $\pi_{\alpha_1}^*E_{\alpha_1}$ and
$ds_{\alpha_2}(\tilde y \mod TV_{\alpha_1})$ is a (single valued)
section. Therefore via the isomorphism (\ref{bdliso}) the right hand
side of (\ref{compatibilityforalpha}) defines an element of
$\mathcal S^{l_i}(E_{\alpha_2})_x$ ($x = \text{\rm Pr}(\tilde y)$),
and hence is regarded as a multisection of
$\pi_{\alpha_2}^*E_{\alpha_2}$. In other words, we omit
$\widehat{\varphi}_{\alpha_2,\alpha_1}$ in
(\ref{compatibilityforalpha}).
\par
We next choose a partition of unity $\chi_{\alpha}$ subordinate to
our Kuranishi charts. To define the notion of partition of unity, we
need some notation.  Let $\text{Pr}_{\alpha_2,\alpha_1} :
N_{V_{\alpha_2,\alpha_1}}V_{\alpha_2} \to V_{\alpha_2,\alpha_1}$ be
the projection. We fix a $\Gamma_{\alpha_1}$-invariant positive
definite metric of $N_{V_{\alpha_2,\alpha_1}}V_{\alpha_2}$ and let
$r_{\alpha_2,\alpha_1} : N_{V_{\alpha_2,\alpha_1}}V_{\alpha_2} \to
[0,\infty)$ be the norm with respect to this metric. We fix a
sufficiently small $\delta$ and let $\chi^{\delta} : \R \to [0,1]$
be a smooth function such that
$$
\chi^{\delta}(t)
=
\begin{cases}
0 &t \ge \delta \\
1 &t \le \delta/2.
\end{cases}
$$

Let $U_{\delta}(V_{\alpha_2,\alpha_1}/\Gamma_{\alpha_1})$
be the image of the exponential map. Namely
$$
U_{\delta}(V_{\alpha_2,\alpha_1}/\Gamma_{\alpha_1}) = \{\text{Exp}(v) \mid v \in
N_{V_{\alpha_2,\alpha_1}}V_{\alpha_2}/
\Gamma_{\alpha_1} \mid r_{\alpha_2,\alpha_1}(v) < \delta\}.
$$
We push out our function $r_{\alpha_2,\alpha_1}$ to
$U_{\delta}(V_{\alpha_2,\alpha_1}/\Gamma_{\alpha_1})$ and denote it
by the same symbol. We call $r_{\alpha_2,\alpha_1}$ a {\it tubular
distance function}. We require $r_{\alpha_2,\alpha_1}$ to satisfy
the compatibility conditions for various tubular
neighborhoods and tubular distance functions, which is
formulated in sections 5 and 6 in \cite{Math73}.
\par
Let $x \in V_{\alpha}$.
We put
$$\aligned
\frak A_{x,+} &= \{ \alpha_+ \mid x \in V_{\alpha_+,\alpha}, \, \alpha_+ > \alpha\} \\
\frak A_{x,-} &= \{ \alpha_- \mid [x \mod\Gamma_{\alpha}] \in
U_{\delta}(V_{\alpha,\alpha_-}/\Gamma_{\alpha_-}),\, \alpha_- <
\alpha\}.
\endaligned$$
For $\alpha_- \in \frak A_{x,-}$ we take $x_{\alpha_-}$ such that
$\text{Exp}(x_{\alpha_-}) = x$.
\begin{defn}\label{partitionofunity}
A system $\{\chi_{\alpha} \mid \alpha \in \frak A\}$ of
$\Gamma_{\alpha}$-equivariant smooth functions $\chi_{\alpha} :
V_{\alpha} \to [0,1]$ of compact support is said to be a partition
of unity subordinate to our Kuranishi chart if :
$$
\chi_{\alpha}(x) + \sum_{\alpha_- \in \frak A_{x,-}}
\chi^{\delta}(r_{\alpha,\alpha_-}(x_{\alpha_-}))
\chi_{\alpha_-}(\text{Pr}_{\alpha,\alpha_-}
(x_{\alpha_-}))
+ \sum_{\alpha_+ \in \frak A_{x,+}} \chi_{\alpha_+}(\varphi_{\alpha_+,\alpha}(x))
= 1.
$$
\end{defn}
\begin{lem}
There exists a partition of unity subordinate to our Kuranishi chart.
We may choose them so that they are $T^n$ equivariant.
\end{lem}
\begin{proof}
We may assume that $\frak A$ is a finite set since $\mathcal M$ is compact.
By shrinking $V_{\alpha}$ if necessary we may assume that there exists
$V^-_{\alpha}$ such that $V^-_{\alpha}$ is a relatively compact subset of
$V_{\alpha}$ and that $E_{\alpha}$, $\varphi_{\alpha_2,\alpha_1}$, $s_{\alpha}$,
etc restricted to $V^-_{\alpha}$ still defines a good coordinate system.
We take a $\Gamma_{\alpha}$ invariant smooth function $\chi'_{\alpha}$ on $V_{\alpha}$
which has compact support and satisfies $\chi'_{\alpha}=1$ on $V^-_{\alpha}$.
We define
$$
h_{\alpha}(x)
= \chi'_{\alpha}(x) + \sum_{\alpha_- \in \frak A_{x,-}}
\chi^{\delta}(r_{\alpha,\alpha_-}(x_{\alpha_-}))
\chi'_{\alpha_-}(\text{Pr}_{\alpha,\alpha_-}
(x_{\alpha_-}))
+ \sum_{\alpha_+ \in \frak A_{x,+}} \chi'_{\alpha_+}(\varphi_{\alpha_+,\alpha}(x)).
$$
Using compatibility of tubular neighborhoods and tubular
distance functions, we can show that $h_{\alpha}$ is $\Gamma_{\alpha}$ invariant and
$$
h_{\alpha_2}(\varphi_{\alpha_2,\alpha_1}(x)) = h_{\alpha_1}(x)
$$
if $x \in V_{\alpha_2,\alpha_1}$.
Therefore
$$
\chi_{\alpha}(x) = \chi'_{\alpha}(x)/h_{\alpha}(x)
$$
has the required properties.
\end{proof}
Now we consider the situation we start with. Namely we have two
strongly continuous $T^n$ equivariant smooth maps
$$
ev_{s} : \mathcal M \to L_s, \qquad ev_{t} : \mathcal M \to L_t
$$
and $ev_{t}$ is weakly submersive. (In fact $T^n$ action on $L_t$ is
transitive and free.)

Let a differential form $h$ on $L_s$ be given. We choose a $T^n$
equivariant good coordinate system
$\{(V_{\alpha},E_{\alpha},\Gamma_{\alpha},\psi_{\alpha},s_{\alpha})\}$
of $\mathcal M$ and a $T^n$ equivariant multisection represented by
$\{\frak s_{\alpha}\}$ in this Kuranishi chart. Assume that the
multisection is transversal to zero.

We also choose a partition of unity $\{\chi_{\alpha}\}$ subordinate
to our Kuranishi chart. Then we put
\begin{equation}\label{thetaalphadef}
\theta_{\alpha} = \chi_{\alpha}(ev_{s,\alpha})^* h
\end{equation}
which is a differential form on $V_{\alpha}$.
\begin{defn}\label{cordefinitionss}
Define
\begin{equation}\label{correspondfinalformula}
(\mathcal M;ev_s,ev_t)_* (h)
= \sum_{\alpha} ((V_{\alpha},\Gamma_{\alpha},E_{\alpha},\psi_{\alpha},s_{\alpha}),\frak s_{\alpha},ev_{t,\alpha})_* (\theta_{\alpha}).
\end{equation}
This is a smooth differential form on $L_t$.
It is $T^n$ equivariant if $h$ is $T^n$ equivariant.
\end{defn}
\begin{rem}\label{detasdenote}
\begin{enumerate}
\item
Actually the right hand side of (\ref{correspondfinalformula}) {\it depends}
on the choice of $(V_{\alpha},E_{\alpha},\Gamma_{\alpha},\psi_{\alpha},s_{\alpha})$,
$\frak s_{\alpha}$.
We write $\frak s$ to demonstrate this choice and write
$(\mathcal M;ev_s,ev_t,\frak s)_* (h)$.
\item The right hand side of (\ref{correspondfinalformula}) is independent of the choice of
partition of unity. The proof is similar to the
well-definiedness of integration on manifolds.
\end{enumerate}
\end{rem}
In case $\mathcal M$ has a boundary $\partial \mathcal M$, the choices
$(V_{\alpha},E_{\alpha},\Gamma_{\alpha},\psi_{\alpha},s_{\alpha})$,
$\frak s_{\alpha}$ on $\mathcal M$ induces one for
$\partial \mathcal M$. We then have the following :
\begin{lem}[Stokes' theorem]\label{stokes}
We have
\begin{equation}\label{stoformula}
d((\mathcal M;ev_s,ev_t,\frak s)_* (h))
= (\mathcal M;ev_s,ev_t,\frak s)_* (dh) +
(\partial\mathcal M;ev_s,ev_t,\frak s)_* (h).
\end{equation}
\end{lem}
We will discuss the sign at the end of this section.
\begin{proof}
Using the partition of unity $\chi_{\alpha}$ it suffices to consider
the case when $\mathcal M$ has only one Kuranishi chart
$V_{\alpha}$. We use the open covering $U_i$ of $V_{\alpha}$ and the
partition of unity again to see that we need only to study on one
$U_i$. In that case (\ref{stoformula}) is immediate from the usual
Stokes' formula.
\end{proof}

We next discuss composition of smooth correspondences.
We consider the following situation.
Let
$$
ev_{s;st} : \mathcal M_{st} \to L_s, \qquad ev_{t;st} : \mathcal M_{st} \to L_t
$$
be as before such that $T^n$ action
on $L_t$ is free and transitive.
Let
$$
ev_{r;rs} : \mathcal M_{rs} \to L_r, \qquad ev_{s;rs} : \mathcal M_{rs} \to L_s
$$
be a similar diagram such that $T^n$ on $L_s$ is free and
transitive. Using the fact that $ev_{s;rs}$ is weakly submersive, we
define the fiber product
$$
\mathcal M_{rs}\, {}_{ev_{s;rs}}\times_{ev_{s;st}} \mathcal M_{st}
$$
as a space with Kuranishi structure.
We write it as $\mathcal M_{rt}$. We have a diagram of strongly continuous
smooth maps
$$
ev_{r;rt} : \mathcal M_{rt} \to L_r, \qquad ev_{t;rt} : \mathcal M_{rt} \to L_t.
$$
\par
We next make choices $\frak s^{st}$, $\frak s^{rs}$ for $\mathcal M_{st}$ and
$\mathcal M_{rs}$. It is easy to see that it determines a
choice $\frak s^{rt}$ for $\mathcal M_{rt}$.
\par
Now we have :
\begin{lem}[Composition formula]\label{compformula}
We have the following formula for each differential form $h$ on
$L_r$.
\begin{equation}
\aligned
&(\mathcal M_{rt};ev_{r;rt},ev_{t;,rt},\frak s^{rt})_* (h)\\
&= ((\mathcal M_{st};ev_{s;st},ev_{t;,st},\frak s^{st})_* \circ
(\mathcal M_{rs};ev_{r;rs},ev_{s;,rs},\frak s^{rs})_*)(h).
\endaligned
\end{equation}
\end{lem}
\begin{proof}
Using a partition of unity it suffices to study locally on $\mathcal
M_{rs}$, $\mathcal M_{st}$. In that case it suffices to consider the
case of usual manifold, which is well-known.
\end{proof}
We finally discuss the signs in Lemmas \ref{stokes} and
\ref{compformula}. It is rather cumbersome to fix appropriate sign
convention and show those lemmata with sign. So, instead, we use the
trick of 
subsection 8.10.3 \cite{fooo06} (= section 53.3 \cite{fooo06pre}) (see also section 13
\cite{fukaya;operad}) to reduce the orientation problem to the case
which is already discussed in 
Chapter 8 \cite{fooo06} (= Chapter 9 \cite{fooo06pre}), as follows.
\par
The correspondence $h \mapsto (\mathcal M;ev_s,ev_t,\frak s)_* (h)$
extends to the currents $h$ that satisfy appropriate transversality
properties about its wave front set. (See \cite{Hor71}.) We can also
represent the smooth form $h$ by an appropriate average (with
respect to certain smooth measure) of a family of currents realized
by smooth singular chains. So, as far as sign concerns, it suffices
to consider a current realized by a smooth singular chain. Then the
right hand side of (\ref{intfibermainformula}) turn out to be a
current realized by a smooth singular chain which is obtained from a
smooth singular chain on $L_s$ by a transversal smooth
correspondence. In fact, we may assume that all the fiber products
appearing here are transversal, since it suffices to discuss the
sign in the generic case. Thus the problem reduces to find a sign
convention (and orientation) for the correspondence of the singular
chains by a smooth manifold. In the situation of our application,
such sign convention (singular homology version) was determined and
analyzed in detail in Chapter 8 \cite{fooo06} (= Chapter 9 \cite{fooo06pre}). Especially, existence
of an appropriate orientation that is consistent with the sign
appearing in $A_{\infty}$ formulae etc. was proved there. Therefore
we can prove that there is a sign (orientation) convention which
induces all the formulae we need with sign, in our de Rham version,
as well. See subsection 8.10.3 \cite{fooo06} (= section 53.3 \cite{fooo06pre}) or section 13
\cite{fukaya;operad} for the detail.

\bibliographystyle{amsalpha}

\end{document}